\documentclass{amsart}
\usepackage{etex}
\usepackage{graphicx}
\usepackage{bm,caption,amsbsy,enumitem,amsmath,amsthm,
amssymb,mathtools,amsfonts,multirow,verbatim,tikz,tikz-cd,thmtools,thm-restate,xr,tensor,xcolor}
\usepackage[alphabetic]{amsrefs}
\usetikzlibrary{matrix,arrows,decorations.pathmorphing}
\usepackage[top=1.25in, bottom=1in, left=1.25in, right=1.25in,marginpar=1in]{geometry}
\usepackage{url}
\DefineSimpleKey{bib}{myurl}
\usepackage[hidelinks,colorlinks=true,linkcolor=blue, citecolor=black,linktocpage=true]{hyperref}
\usepackage{chngcntr}
\usepackage[all]{xy}
\usepackage{leftidx}
\counterwithin{figure}{section}

\newcommand\myurl[1]{\url{#1}}

\numberwithin{equation}{section}

\newtheorem{innercustomthm}{Theorem}
\newenvironment{customthm}[1]
  {\renewcommand\theinnercustomthm{#1}\innercustomthm}
  {\endinnercustomthm}

\newtheorem{innercustomcor}{Corollary}
\newenvironment{customcor}[1]
  {\renewcommand\theinnercustomcor{#1}\innercustomcor}
  {\endinnercustomcor}

\newtheorem{thm}{Theorem}[section]
\newtheorem{prop}[thm]{Proposition}

\newtheorem{cor}[thm]{Corollary}
\newtheorem{lem}[thm]{Lemma}

\newtheorem{subclaim*}{Subclaim}

\newtheorem{subclaim}{Subclaim}[thm]
\newtheorem{step}{Step}

\theoremstyle{definition}
\newtheorem{define}[thm]{Definition}

\theoremstyle{remark}
\newtheorem{rem}[thm]{Remark}

\newcommand{\ve}[1]{\boldsymbol{\mathbf{#1}}}
\newcommand{\R}{\mathbb{R}}

\newcommand{\Z}{\mathbb{Z}}
\newcommand{\N}{\mathbb{N}}
\newcommand{\Q}{\mathbb{Q}}
\newcommand{\C}{\mathbb{C}}

\renewcommand{\d}{\partial}
\renewcommand{\subset}{\subseteq}

\renewcommand{\tilde}{\widetilde}
\renewcommand{\hat}{\widehat}
\newcommand{\iso}{\cong}

\DeclareMathOperator{\codim}{{codim}}

\DeclareMathOperator{\Crit}{{Crit}}

\DeclareMathOperator{\Diff}{{Diff}}

\DeclareMathOperator{\ev}{{ev}}

\DeclareMathOperator{\gr}{{gr}}

\DeclareMathOperator{\id}{{id}}
\DeclareMathOperator{\ind}{{ind}}
\DeclareMathOperator{\Int}{{int}}

\DeclareMathOperator{\im}{{im}}

\DeclareMathOperator{\MCG}{{MCG}}

\DeclareMathOperator{\red}{{red}}

\DeclareMathOperator{\Spin}{{Spin}}

\DeclareMathOperator{\Sing}{{Sing}}

\DeclareMathOperator{\Span}{{Span}}

\DeclareMathOperator{\Sym}{{Sym}}

\DeclareMathOperator{\Tors}{{Tors}}

\newcommand{\bD}{\mathbb{D}}

\newcommand{\bF}{\mathbb{F}}

\newcommand{\bK}{\mathbb{K}}

\newcommand{\bS}{\mathbb{S}}
\newcommand{\bT}{\mathbb{T}}

\newcommand{\cD}{\mathcal{D}}

\newcommand{\cG}{\mathcal{G}}
\newcommand{\cH}{\mathcal{H}}

\newcommand{\cK}{\mathcal{K}}

\newcommand{\cM}{\mathcal{M}}
\newcommand{\cN}{\mathcal{N}}

\newcommand{\cQ}{\mathcal{Q}}
\newcommand{\cR}{\mathcal{R}}

\newcommand{\cT}{\mathcal{T}}
\newcommand{\cU}{\mathcal{U}}
\newcommand{\cV}{\mathcal{V}}
\newcommand{\cW}{\mathcal{W}}

\newcommand{\cY}{\mathcal{Y}}

\newcommand{\frA}{\mathfrak{A}}
\newcommand{\frB}{\mathfrak{B}}

\newcommand{\frS}{\mathfrak{S}}

\newcommand{\frd}{\mathfrak{d}}

\newcommand{\frj}{\mathfrak{j}}

\newcommand{\frs}{\mathfrak{s}}
\newcommand{\frt}{\mathfrak{t}}

\newcommand{\as}{\ve{\alpha}}
\newcommand{\bs}{\ve{\beta}}
\newcommand{\gs}{\ve{\gamma}}
\newcommand{\ds}{\ve{\delta}}

\newcommand{\vs}{\ve{v}}
\newcommand{\xs}{\ve{x}}
\newcommand{\ys}{\ve{y}}
\newcommand{\ws}{\ve{w}}
\newcommand{\zs}{\ve{z}}

\newcommand{\GrCob}{\mathsf{GrCob}}
\newcommand{\FlGr}{\mathsf{FlGr}}

\newcommand{\CF}{\mathit{CF}}
\newcommand{\HF}{\mathit{HF}}

\renewcommand{\a}{\alpha}
\renewcommand{\b}{\beta}
\newcommand{\g}{\gamma}

\newcommand{\PD}{\mathit{PD}}

\newcommand{\bmP}{{\bm{P}}}

\renewcommand{\bar}{\overline}

\DeclareMathOperator{\Sw}{Sw}

\title{Graph cobordisms and Heegaard Floer homology}
\author{Ian Zemke}
\address{Department of Mathematics\\Princeton University\\  Princeton, NJ 08544, USA}
\email{izemke@math.princeton.edu}
\begin{document}

\begin{abstract}We construct a graph TQFT for the minus flavor of Heegaard Floer homology. Our graph TQFT extends Ozsv\'{a}th and Szab\'{o}'s TQFT for closed and connected 3-manifolds, and allows for cobordisms with disconnected ends. As an application, we give an explicit formula for the chain homotopy type of the $\pi_1$-action on Heegaard Floer homology. We show that on homology the $\pi_1$-action is trivial on the plus, minus and infinity flavors, but give examples where it is non-trivial on the hat flavor.
\end{abstract}

\maketitle

\setcounter{tocdepth}{1}
\tableofcontents

\section{Introduction}Heegaard Floer homology, introduced by Ozsv\'{a}th and Szab\'{o}, associates functorial invariants to 3- and 4-manifolds. To an oriented, closed and connected 3-manifold $Y$, with a $\Spin^c$ structure $\frs\in \Spin^c(Y)$, they construct groups
\[
\HF^-(Y,\frs),\quad \HF^\infty(Y,\frs), \quad \HF^+(Y,\frs) \quad \text{and}\quad  \hat{\HF}(Y,\frs),
\]
which are modules over the ring $\bF_2[U]$ \cite{OSDisks}.  Ozsv\'{a}th and Szab\'{o} showed that these modules fit into the framework of a (3+1)-dimensional TQFT:
\begin{thm}[\cite{OSTriangles}]\label{thm:OScobordismtheorem} Suppose $W$ is an oriented, compact 4-manifold with boundary $\d W=-Y_0\sqcup Y_1$, and $W$, $Y_0$ and $Y_1$ are all non-empty and connected. If $\frs\in \Spin^c(W)$,  and $\circ\in \{+,-,\infty, \wedge\}$, then there is a functorial map
\[
F_{W,\frs}\colon \HF^\circ(Y_0,\frs|_{Y_0})\to \HF^\circ(Y_1,\frs|_{Y_1}),
\]
which depends on a choice of path in $W$, connecting basepoints in $Y_0$ and $Y_1$.
\end{thm}

 The dependence on a choice of path was not explicitly stated in Ozsv\'{a}th and Szab\'{o}'s original work, though their construction implicitly depends on a choice of  path. Indeed, we will show that the maps appearing in Theorem~\ref{thm:OScobordismtheorem} are not invariants without some extra choice of data; see Corollary~\ref{cor:E}, below.

The goal of this paper is two-fold. Firstly, we determine precisely the dependence of the maps in Theorem~\ref{thm:OScobordismtheorem} on a choice of path. Secondly, and more broadly, we provide a TQFT framework for Heegaard Floer homology which incorporates the basepoints in a natural way and extends Ozsv\'{a}th and Szab\'{o}'s framework to disconnected 3- and 4-manifolds.

Throughout the paper we work over the field $\bF_2:=\Z/2 \Z$.

\subsection{The graph TQFT}

Extending  their original construction for 3-manifolds \cite{OSDisks}, Ozsv\'{a}th and Szab\'{o} described Heegaard Floer complexes $\CF^-(Y,\ws,\frs)$ for  3-manifolds with collections of basepoints $\ws\subset Y$ \cite{OSLinks}. For the purposes of the introduction, the complex $\CF^-(Y,\ws,\frs)$ is a free, finitely generated chain complex over the polynomial ring $\bF_2[U]$. Later, we will consider an algebraic generalization over a more general polynomial ring.

In this paper, we define a notion of cobordism between manifolds with collections of basepoints. A \emph{ribbon graph cobordism} $(W,\Gamma)$ from $(Y_0,\ws_0)$ to $(Y_1,\ws_1)$ is a cobordism $W$ from $Y_0$ to $Y_1$ which contains an embedded graph $\Gamma$, whose intersection with $Y_i$ is $\ws_i$. Furthermore, $\Gamma$ is decorated with a \emph{formal ribbon structure}, i.e. a choice of cyclic ordering of the edges adjacent to each vertex of $\Gamma$.

Our present paper centers on proving the following:

\begin{customthm}{A}\label{thm:A}Suppose $(W,\Gamma)\colon (Y_0,\ve{w}_0)\to (Y_1,\ve{w}_1)$ is a ribbon graph cobordism and $\frs\in \Spin^c(W)$. The construction in this paper yields two chain maps,
\[
F_{W,\Gamma,\frs}^A,\, F_{W,\Gamma,\frs}^B \colon \CF^-(Y_0,\ve{w}_0,\frs|_{Y_0})\to \CF^-(Y_1,\ve{w}_1,\frs|_{Y_1}),
\]
which  are diffeomorphism invariants of $(W,\Gamma)$, up to $\bF_2[U]$-equivariant chain homotopy.
\end{customthm}

Theorem~\ref{thm:A} also applies to $\CF^\infty$, $\CF^+$ and $\hat{\CF}$, since they can be recovered via tensor products with $\CF^-$; see Section~\ref{sec:HFcomplexes}.

The type-$A$ and type-$B$ cobordism maps satisfy the following symmetry:
 \begin{equation}
 F_{W,\Gamma,\frs}^A\simeq F_{W,\bar{\Gamma},\frs}^B, \label{eq:changecyclicordersAB}
 \end{equation}
  where $\bar{\Gamma}$ denotes $\Gamma$ with cyclic orders reversed \cite{HMZConnectedSum}*{Lemma~5.9}. Despite this symmetry, it is natural to describe the $A$ and $B$ versions separately.

Unlike Ozsv\'{a}th and Szab\'{o}'s original construction,  Theorem~\ref{thm:A} applies to cobordisms which  are disconnected, or which have disconnected 3-manifolds appearing in their ends, as long as each component contains a basepoint.

\emph{Diffeomorphism invariance} in Theorem~\ref{thm:A} amounts to the following. Suppose that $(W,\Gamma)\colon (Y_0,\ws_0)\to (Y_1,\ws_1)$ and $(W',\Gamma')\colon (Y_0',\ws_0')\to (Y_1',\ws_1')$ are two graph cobordisms, and $D\colon (W,\Gamma)\to (W',\Gamma')$ is an orientation preserving diffeomorphism. The map $D$ restricts to give diffeomorphisms $d_0\colon Y_0\to Y_0'$ and $d_1\colon Y_1\to Y_1'$. Let $\frs\in \Spin^c(W)$ and $\frs':=D_*(\frs)$. The following diagram, up to chain homotopy:
\begin{equation}
\begin{tikzcd}
\CF^-(Y_0,\ws_0,\frs|_{Y_0})\arrow{r}{(d_0)_*}\arrow{d}{F_{W,\Gamma,\frs}^A} & \CF^-(Y_0',\ws_0',\frs'|_{Y_0'})\arrow{d}{F_{W',\Gamma',\frs'}^A}\\
\CF^-(Y_1,\ws_1,\frs|_{Y_1})\arrow{r}{(d_1)_*}& \CF^-(Y_1',\ws_1',\frs'|_{Y_1'}).\label{eq:diffeomorphism-invariance}
\end{tikzcd}
\end{equation}

Concerning graph cobordisms where the graph consists of a collection of paths from $Y_0$ to $Y_1$, we prove the following:

\begin{customthm}{B}\label{thm:B}Suppose that $(W,\Gamma)\colon (Y_0,\ve{w}_0)\to (Y_1,\ve{w}_1)$ is a graph cobordism, and $\Gamma$ consists of a collection of paths, each connecting $\ws_0$ to $\ws_1$.
\begin{enumerate}
\item The $A$ and $B$ versions of the maps coincide:
\[
F_{W,\Gamma,\frs}^A\simeq F_{W,\Gamma,\frs}^B.
\]
 \item  If $\phi\colon (Y,\ws)\to (Y,\ws)$ is an orientation preserving diffeomorphism, let $W(\phi)$ denote the \emph{mapping cylinder} (i.e. $[0,1]\times Y$, with $\{0\}\times Y$ identified with $Y$ via $\id_Y$ and $\{1\}\times Y$ identified with $Y$ via $\phi$). Then
 \[
F_{W(\phi), [0,1]\times \ws, \frs}^A\simeq F_{W(\phi), [0,1]\times \ws, \frs}^B\simeq \left(\phi_*\colon \CF^-(Y,\ws, \frs)\to \CF^-(Y,\ws, \phi_*\frs)\right). 
 \]
 \item Suppose $(W,\gamma)\colon (Y_0,w_0)\to (Y_1,w_1)$ is a cobordism such that $W,$ $Y_0$ and $Y_1$ are nonempty and connected, and $\gamma$ is a path from $w_0$ to $w_1$. Then $F_{W,\gamma,\frs}^A\simeq F_{W,\gamma,\frs}^B$, and both maps coincide with the map defined by Ozsv\'{a}th and Szab\'{o}.
\end{enumerate}
\end{customthm}

Analogous to the cobordism maps defined by Ozsv\'{a}th and Szab\'{o}, the graph cobordism maps satisfy a $\Spin^c$ composition law:

\begin{customthm}{C}\label{thm:C}Suppose that $(W,\Gamma)$ is ribbon graph cobordism which decomposes as a composition $(W,\Gamma)=(W_2,\Gamma_2)\cup (W_1,\Gamma_1)$. If $\frs_1$ and $\frs_2$ are $\Spin^c$ structures on $W_1$ and $W_2$, then
\[
F_{W_2,\Gamma_2,\frs_2}^A\circ F_{W_1,\Gamma_1,\frs_1}^A\simeq \sum_{\substack{\frs\in \Spin^c(W)\\
\frs|_{W_2}=\frs_2\\
\frs|_{W_1}=\frs_1}} F_{W,\Gamma,\frs}^A.
\]
The same relation holds for the type-$B$ graph cobordism maps.
\end{customthm}

\subsection{Moving basepoints and the $\pi_1$-action}

If $w\in Y$, there is a fibration
\begin{equation}
\Diff(Y,w)\to \Diff(Y)\xrightarrow{\ev_w} Y.\label{eq:LESfibration}
\end{equation}
The long exact sequence of homotopy groups for the fibration in Equation~\eqref{eq:LESfibration} gives a homomorphism
\[
\pi_1(Y,w)\to \MCG(Y,w),
\]
where $\MCG(Y,w)$ denotes the \emph{based mapping class group} of orientation preserving diffeomorphisms modulo smooth isotopies which fix the point $w$.

If $\gamma\in \pi_1(Y,w)$, we write
\[
\gamma_*\colon \CF^-(Y,w,\frs)\to \CF^-(Y,w,\frs)
\]
for the induced diffeomorphism map.

 If $w\in \ws$ is a chosen of basepoint, there is a $+1$ graded endomorphism
\[
\Phi_w\colon \CF^-(Y,\ws,\frs)\to \CF^-(Y,\ws,\frs).
\]

%

One interpretation of the map $\Phi_w$ is as the map for the \emph{broken path cobordism}, shown in Figure~\ref{fig::59}. Alternatively, $\Phi_w$ can be interpreted as the map for a count of holomorphic disks on a Heegaard diagram; see Equation~\eqref{eq:Phi-map-def}.

\begin{figure}[ht!]
	\centering
\begingroup%
  \makeatletter%
  \providecommand\color[2][]{%
    \errmessage{(Inkscape) Color is used for the text in Inkscape, but the package 'color.sty' is not loaded}%
    \renewcommand\color[2][]{}%
  }%
  \providecommand\transparent[1]{%
    \errmessage{(Inkscape) Transparency is used (non-zero) for the text in Inkscape, but the package 'transparent.sty' is not loaded}%
    \renewcommand\transparent[1]{}%
  }%
  \providecommand\rotatebox[2]{#2}%
  \newcommand*\fsize{\dimexpr\f@size pt\relax}%
  \newcommand*\lineheight[1]{\fontsize{\fsize}{#1\fsize}\selectfont}%
  \ifx\svgwidth\undefined%
    \setlength{\unitlength}{133.00322434bp}%
    \ifx\svgscale\undefined%
      \relax%
    \else%
      \setlength{\unitlength}{\unitlength * \real{\svgscale}}%
    \fi%
  \else%
    \setlength{\unitlength}{\svgwidth}%
  \fi%
  \global\let\svgwidth\undefined%
  \global\let\svgscale\undefined%
  \makeatother%
  \begin{picture}(1,0.74058408)%
    \lineheight{1}%
    \setlength\tabcolsep{0pt}%
    \put(0,0){\includegraphics[width=\unitlength,page=1]{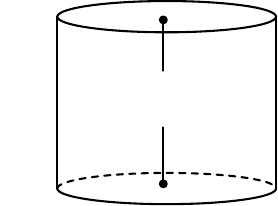}}%
    \put(0.10661187,0.36535524){\color[rgb]{0,0,0}\makebox(0,0)[rt]{\lineheight{2.50000024}\smash{\begin{tabular}[t]{r}$\Phi_w$\end{tabular}}}}%
    \put(0,0){\includegraphics[width=\unitlength,page=2]{fig59.pdf}}%
    \put(0.60659211,0.03884388){\color[rgb]{0,0,0}\makebox(0,0)[lt]{\lineheight{2.50000024}\smash{\begin{tabular}[t]{l}$w$\end{tabular}}}}%
    \put(0.60659211,0.68329664){\color[rgb]{0,0,0}\makebox(0,0)[lt]{\lineheight{2.50000024}\smash{\begin{tabular}[t]{l}$w$\end{tabular}}}}%
  \end{picture}%
\endgroup%

	\caption{\textbf{The broken path graph cobordism for the endomorphism $\Phi_w$.} The underlying 4-manifold is $[0,1]\times Y$. When $w$ is not the only basepoint, there are additional, unbroken strands from the bottom to the top.}\label{fig::59}
\end{figure}

 An analog of the map $\Phi_w$ for link Floer homology was discovered by Sarkar~ \cite{SarkarMovingBasepoints}  in the context of basepoint moving diffeomorphisms on link Floer homology.

\begin{customthm}{D}\label{thm:D}Suppose $(Y,\ve{w})$ is a multi-pointed 3-manifold, $w\in \ve{w}$ and $\gamma\in \pi_1(Y,w)$. Then the diffeomorphism map $\gamma_*$ on $\CF^-(Y,\ws,\frs)$ satisfies
\[
\gamma_*\simeq \id+\Phi_w\circ A_\gamma,
\]
 where $\Phi_w$ is the broken path cobordism map   and $A_\gamma$ denotes the action of the class  $[\gamma]\in H_1(Y,\Z)/\Tors$.
\end{customthm}

Using Theorem~\ref{thm:D}, we exhibit 3-manifolds $Y$ where the induced map 
\[
\gamma_*\colon \hat{\HF}(Y,w,\frs)\to \hat{\HF}(Y,w,\frs)
\]
is not the identity:
\begin{customcor}{E}\label{cor:E}Let $Y$ be a 3-manifold and $\frs\in \Spin^c(Y)$.
\begin{enumerate}
\item  If $\frs$ is torsion and there is an $x\in \HF^+(Y,w,\frs)$ such that 
\[
U\cdot x=0 \quad \text{ and }\quad x\not \in U\cdot \HF^+(Y,w,\frs),
\]
then $\pi_1(Y\# S^1\times S^2,w)$ acts non-trivially on
\[
\hat{\HF}(Y\# S^1\times S^2,w).
\]
\item Suppose $[\gamma]\in H_1(Y;\Z)$ is a class whose action on $\HF^+(Y,w,\frs)$ does not vanish. If $|\ws|\ge 2$, then the diffeomorphism map $\gamma_*$ acts non-trivially on the $\bF_2[U]$-module
\[
\HF^-(Y,\ws,\frs).
\]
\end{enumerate}
\end{customcor}

Note that $Y=\Sigma(2,3,7)$ satisfies the conditions of Part (1) of Corollary~\ref{cor:E}. The manifold $Y=S^1\times S^2$ satisfies Part (2).

When we restrict Theorem~\ref{thm:D} to 3-manifolds with a single basepoint $w$, the formula defining $\Phi_w$ has the following algebraic interpretation. View $\CF^-(Y,w,\frs)$ as a free module over $\bF_2[U]$ with basis equal to the set of intersection points. The differential can be written as a matrix over this basis, with entries in $\bF_2[U]$. The endomorphism $\Phi_w$ is obtained by differentiating each entry of this matrix with respect to $U$. Using this basis, we can also define a derivative map $d/d U$ as an endomorphism of $\CF^-(Y,w,\frs)$, which is not $U$-equivariant. The Leibniz rule implies
\begin{equation}
\Phi_w=\d\circ \frac{d}{d U}+\frac{d}{d U} \circ \d.\label{eq:chainhomotopyddU}
\end{equation}

As a consequence, we prove the following folklore result known to experts:

\begin{customcor}{F}\label{cor:F}If $(Y,w)$ is a singly based 3-manifold and $\gamma\in \pi_1(Y,w)$, then the induced map $\gamma_*$ is equal to the identity on the homology groups $\HF^-(Y,w,\frs)$, $\HF^\infty(Y,w,\frs)$ and $\HF^+(Y,w,\frs)$, but not necessarily $\hat{\HF}(Y,w,\frs)$. Consequently, if $(W,\gamma)$ is path cobordism between two singly based 3-manifolds $(Y_0,w_0)$ and $(Y_1,w_1)$, then the cobordism map
\[
F_{W,\gamma,\frs}\colon \HF^\circ(Y_0,w_0,\frs|_{Y_0})\to \HF^\circ(Y_1,w_1,\frs|_{Y_1})
\]
is independent of $\gamma$ if $\circ\in \{-,\infty,+\}$.
\end{customcor}

The chain homotopy $H=d/d U$ appearing in Equation~\eqref{eq:chainhomotopyddU} is not $\bF_2[U]$ equivariant, and hence does not induce a chain homotopy on $\hat{\CF}$.

In another direction, if $\lambda$ is a path between two basepoints $w_1,w_2\in \ws\subset Y$, there is a diffeomorphism $\Sw_\lambda\colon (Y,\ws)\to (Y,\ws)$ which swaps the two basepoints $w_1$ and $w_2$ along the path $\lambda$ (this is well defined, up to isotopy, since $Y$ is 3-dimensional). Using the graph TQFT, we prove  in Proposition~\ref{prop:basepointswappingmap} that
\[
\Sw_\lambda\simeq \Phi_{w_1}A_\lambda+A_\lambda\Phi_{w_2}\simeq A_{\lambda} \Phi_{w_1}+\Phi_{w_2}A_{\lambda}.
\]

\subsection{Outline of the construction of the graph TQFT}

We now describe the main ingredients of our graph TQFT. The first ingredients are the 4-dimensional handle attachment maps,  similar to the maps defined by Ozsv\'{a}th and Szab\'{o} \cite{OSTriangles}, with the exception that we define maps for 0-handles and 4-handles, and our construction of 1-handle maps and 3-handle maps is more flexible than their construction. There are two additional, novel ingredients of the graph TQFT: the \emph{free-stabilization maps}, and the \emph{relative homology maps}. 

If $w\not \in \ws$, we define two free-stabilization maps
\[
S_w^+\colon \CF^-(Y,\ws,\frs)\to \CF^-(Y,\ws\cup \{w\}, \frs), \text{ and }
\]
\[
S_w^-\colon \CF^-(Y,\ws\cup \{w\}, \frs)\to \CF^-(Y,\ws,\frs).
\]
The maps $S_w^+$ and $S_{w}^-$ are induced by the graph cobordisms shown in Figure~\ref{fig::58}.

\begin{figure}[ht!]
	\centering
\begingroup%
  \makeatletter%
  \providecommand\color[2][]{%
    \errmessage{(Inkscape) Color is used for the text in Inkscape, but the package 'color.sty' is not loaded}%
    \renewcommand\color[2][]{}%
  }%
  \providecommand\transparent[1]{%
    \errmessage{(Inkscape) Transparency is used (non-zero) for the text in Inkscape, but the package 'transparent.sty' is not loaded}%
    \renewcommand\transparent[1]{}%
  }%
  \providecommand\rotatebox[2]{#2}%
  \newcommand*\fsize{\dimexpr\f@size pt\relax}%
  \newcommand*\lineheight[1]{\fontsize{\fsize}{#1\fsize}\selectfont}%
  \ifx\svgwidth\undefined%
    \setlength{\unitlength}{323.03502203bp}%
    \ifx\svgscale\undefined%
      \relax%
    \else%
      \setlength{\unitlength}{\unitlength * \real{\svgscale}}%
    \fi%
  \else%
    \setlength{\unitlength}{\svgwidth}%
  \fi%
  \global\let\svgwidth\undefined%
  \global\let\svgscale\undefined%
  \makeatother%
  \begin{picture}(1,0.30492066)%
    \lineheight{1}%
    \setlength\tabcolsep{0pt}%
    \put(0,0){\includegraphics[width=\unitlength,page=1]{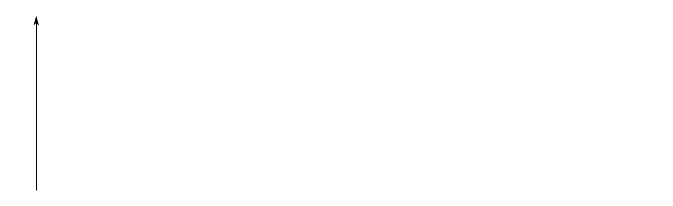}}%
    \put(0.04077399,0.1504277){\color[rgb]{0,0,0}\makebox(0,0)[rt]{\lineheight{2.5}\smash{\begin{tabular}[t]{r}$S_{w}^+$\end{tabular}}}}%
    \put(0,0){\includegraphics[width=\unitlength,page=2]{fig58.pdf}}%
    \put(0.26705809,0.26998452){\color[rgb]{0,0,0}\makebox(0,0)[lt]{\lineheight{2.50000024}\smash{\begin{tabular}[t]{l}$w$\end{tabular}}}}%
    \put(0,0){\includegraphics[width=\unitlength,page=3]{fig58.pdf}}%
    \put(0.63866757,0.1504277){\color[rgb]{0,0,0}\makebox(0,0)[rt]{\lineheight{2.5}\smash{\begin{tabular}[t]{r}$S_{w}^-$\end{tabular}}}}%
    \put(0,0){\includegraphics[width=\unitlength,page=4]{fig58.pdf}}%
    \put(0.86495241,0.01896075){\color[rgb]{0,0,0}\makebox(0,0)[lt]{\lineheight{2.50000024}\smash{\begin{tabular}[t]{l}$w$\end{tabular}}}}%
    \put(-0.10680102,0.43022568){\color[rgb]{0,0,0}\makebox(0,0)[lt]{\begin{minipage}{0.53355605\unitlength}\centering \end{minipage}}}%
  \end{picture}%
\endgroup%

	\caption{\textbf{Graph cobordisms for the free-stabilization maps.} The underlying 4-manifold is $[0,1]\times Y$.}\label{fig::58}
\end{figure}

Another new ingredient of our graph TQFT is the \emph{relative homology action}. To a path $\lambda$ between two basepoints $w_1, w_2\in \ws$ we construct an endomorphism
\[
 A_\lambda\colon \CF^-(Y,\ws,\frs)\to \CF^-(Y,\ws,\frs).
\]

 The definition of the map $A_\lambda$ is asymmetrical with respect to the $\as$ and $\bs$ curves. By switching the roles of $\as$ and $\bs$ in the construction, we obtain another map $B_\lambda$, with the same domain and codomain, which is also a chain map.

The maps $A_{\lambda}$ and $B_{\lambda}$ turn out to be the graph cobordism maps for $H$-shaped graphs in $[0,1]\times Y$ for different ribbon structures, though we will not make use of this fact in this paper; see \cite{ZemCFLTQFT}*{Lemma~14.11}.

As an intermediate step towards defining the graph cobordism maps, we construct maps for graphs embedded in a fixed 3-manifold $Y$. We say  $\cG=(\Gamma,\ws_0,\ws_1)$ is an embedded \emph{flow-graph} in $Y$ if $\Gamma\subset Y$ is an embedded ribbon graph, each point of $\ws_0$ and $\ws_1$ has valence 1 in $\Gamma$, and $\Gamma$ has no valence 0 vertices.

As a key step towards the full construction of the graph TQFT, we construct $\bF_2[U]$-equivariant chain maps
\[
\frA_{\cG},\, \frB_{\cG}\colon \CF^-(Y,\ws_0,\frs)\to \CF^-(Y,\ws_1,\frs),
\]
which we call the \emph{type-$A$} and \emph{type-$B$ graph action maps}. The two maps $\frA_\cG$ and $\frB_\cG$ are related by the same symmetry as the cobordism maps in Equation~\eqref{eq:changecyclicordersAB}.
 
If $\cG=(\Gamma,\ws_0,\ws_1)$ is a ribbon flow-graph in $Y$, then by definition
\begin{equation}
\frA_{\cG}\simeq F_{[0,1]\times Y,\Gamma',\frs}^{A}\qquad \text{and} \qquad \frB_{\cG}\simeq F_{[0,1]\times Y,\Gamma',\frs}^B,
\end{equation}
where $\Gamma'\subset [0,1]\times Y$ is a ribbon graph which projects to $\Gamma$ and such that $\Gamma'\cap (\{0\}\times Y)=\ws_0$ and $\Gamma'\cap (\{1\}\times Y)=\ws_1$.

\subsection{Further developments}

This paper focuses on defining the graph cobordism maps, $F_{W,\Gamma,\frs}^A$ and $F_{W,\Gamma,\frs}^B$, and proving invariance. In subsequent papers, further properties and applications have been explored, which we briefly summarize.

 Hendricks, Manolescu and the author showed that pair-of-pants graph cobordisms containing a trivalent graph induce chain homotopy equivalences between $\CF^-(Y_1\# Y_2)$ and $\CF^-(Y_1)\otimes  \CF^-(Y_2)$ \cite{HMZConnectedSum}*{Proposition~5.2}, giving a cobordism perspective on Ozsv\'{a}th and Szab\'{o}'s connected sum formula \cite{OSProperties}*{Theorem~1.5}. We used this fact to prove a connected sum formula for involutive Heegaard Floer homology \cite{HMZConnectedSum}*{Theorem~1.1}.
 
 In another subsequent paper, we prove that the graph cobordism maps for the \emph{trace cobordism}
 \[
 ([0,1]\times Y,[0,1]\times \{w\})\colon (-Y,w)\sqcup (Y,w)\to \emptyset,
 \]
 are chain homotopic to the canonical trace pairing between $\CF^-(Y)$ and $\CF^-(-Y)$ \cite{ZemDualityMappingTori}*{Theorem~1.6}. A similar result holds for the \emph{cotrace cobordism}, obtained by turning around the trace cobordism. As an application, the author computed the Heegaard Floer mixed invariants of mapping tori in terms of Lefschetz numbers on $\HF_{\red}^+(Y)$ \cite{ZemDualityMappingTori}*{Theorem~1.1}.

In another direction, the author described a TQFT for link Floer homology \cite{ZemCFLTQFT}. It turns out that the link Floer TQFT recovers the graph TQFT in a natural way; see \cite{ZemCFLTQFT}*{Theorem~C} for a precise statement. One interpretation of the dependence on cyclic orders from this paper can be seen as an artifact of the relation with graphs embedded on surfaces.

The graph cobordism maps are invariant under a more general equivalence than ambient diffeomorphism of the 4-manifold, as stated in Equation~\eqref{eq:diffeomorphism-invariance}. The cobordism maps are also invariant under homotopies of the graph inside $W$ which are only required to be smooth on each edge, and hence may not extend to a smooth isotopy of $W$; see Definition~\ref{def:smooth-graph-isotopy}. Using the aforementioned connection with link cobordisms, it turns out that the graph cobordisms are invariant under modifications of the graph which preserve the \emph{ribbon equivalence class} of the ribbon graph; see \cite{ZemCFLTQFT}*{Corollary~D} for a precise statement.

\subsection{Further commentary}

It is worthwhile to compare the cobordism maps in Heegaard Floer homology to the cobordism maps in Kronheimer and Mrowka's construction of  monopole Floer homology  \cite{KMMonopole}. It follows from work of Kutluhan, Lee and Taubes \cite{KLTHF=HM1}, \cite{KLTHF=HM2}, \cite{KLTHF=HM3}, \cite{KLTHF=HM4}, and \cite{KLTHF=HM5}, and independently Colin, Ghiggini and Honda \cite{CGHHF=ECH0}, \cite{CGHHF=ECH1}, \cite{CGHHF=ECH2} and \cite{CGHHF=ECH3}, that certain versions of Heegaard Floer homology, monopole Floer homology, and embedded contact homology are isomorphic. A proof that these isomorphisms extend to the level of 4-manifold invariants has not yet appeared in the literature.

In monopole Floer homology, as well as embedded contact homology  the chain complexes are defined independently of a basepoint, but the action of $U$ requires a choice of basepoint, and a path can be used to construct a chain homotopy between the two $U$ maps; see \cite{HutchingsTaubes}*{Section~2.5}.

In a different direction, we note that an early version of Theorem~\ref{thm:D} appeared in \cite{ZemHatGraphTQFT},  for $\hat{\HF}$. The proof in \cite{ZemHatGraphTQFT} used Juh\'{a}sz's sutured TQFT \cite{JDisks} \cite{JCob} as well as the contact gluing map of Honda, Kazez and Mati\'{c} \cite{HKMTQFT}, and hence had a different flavor than the one we explore in this paper. It is an interesting question whether the graph cobordism maps described in this paper have an interpretation in terms of the contact gluing map, perhaps using a limiting construction as in \cite{EVZLimit} \cite{GollaLimit} or using a minus version of sutured Floer homology described by Alishahi and Eftekhary \cite{AErefinedSFH}.

\subsection{Outline of the paper}

In Section~\ref{categories}, we define the category of ribbon graph cobordisms, and some related cobordism categories which appear in our paper. In Section~\ref{sec:background}, we describe some background material on Heegaard Floer homology, and  prove several results about admissibility for multi-pointed Heegaard Floer diagrams. In Section~\ref{sec:relhommap}, we construct the relative homology action. In Section~\ref{sec:freestab}, we construct the free-stabilization maps, for adding or removing a basepoint. In Section~\ref{sec:graphaction}, we combine the free-stabilization maps with the relative homology action to construct a restricted version of our TQFT, which we call the graph action map. In Sections~\ref{sec:1-handles} and~\ref{sec:2-handles}, we describe maps for 4-dimensional handles. In Sections~\ref{sec:constructionI} and~\ref{sec:constructionII}, we define the graph cobordism maps and prove invariance (Theorem~\ref{thm:A}). In Section~\ref{sec:compositionlaw}, we prove the composition law (Theorem~\ref{thm:C}). In Section~\ref{sec:normalizationI}, we give a summary of the proof of the normalization axiom (Theorem~\ref{thm:B}). In Section~\ref{sec:movingbasepoints}, we perform several technical, holomorphic curve arguments to finish the proof of the normalization axiom, and also give several formulas for basepoint moving diffeomorphism maps (in particular, Theorem~\ref{thm:D} and Corollaries~\ref{cor:E} and \ref{cor:F}).

\section{Acknowledgments}
I would like to thank Antonio Alfieri, Andr\'{a}s Juh\'{a}sz, Kristen Hendricks, Robert Lipshitz, Yajing Liu,  Ciprian Manolescu, Mike Miller, and Matthew Stoffregen for helpful conversations. I would like to thank Ciprian Manolescu and Peter Ozsv\'{a}th for helping me with the holomorphic geometry related to the 1- and 3-handle maps.

\section{Categorical Preliminaries}
\label{categories}
In this section, we define a category of graph cobordisms $\GrCob$, and a related category $\FlGr(Y)$ of immersed flow-graphs in a fixed 3-manifold $Y$.

\subsection{The graph cobordism category}

The objects of $\GrCob$ are the following:

\begin{define}
 A \emph{multi-pointed 3-manifold} is a pair $(Y,\ws)$ consisting of a closed, oriented 3-manifold $Y$ (not necessarily connected), together with a finite collection of basepoints $\ws\subset Y$, such that each component of $Y$ contains at least one basepoint.
\end{define}

Morphisms in $\GrCob$ are the following:

\begin{define} A \emph{ribbon graph cobordism} from $(Y_0,\ws_0)$ to  $(Y_1,\ws_1)$ is a pair $(W,\Gamma)$ satisfying the following:
\begin{enumerate}
\item $W$ is a cobordism from $Y_0$ to $Y_1$.
\item $\Gamma$ is an embedded graph in $W$ such that $\Gamma\cap Y_i=\ws_i$. Furthermore, each point of $\ws_i$ has valence 1 in $\Gamma$.
\item $\Gamma$ has finitely many edges and vertices, and no vertices of valence 0.
\item  The embedding of $\Gamma$ is smooth on each edge.
\item $\Gamma$ is decorated with a \emph{formal ribbon structure}, i.e. a formal choice of cyclic ordering of the edges adjacent to each vertex. 
\end{enumerate}
\end{define}

\begin{rem}
 In $\GrCob$, the embedding of a graph need not respect the formal ribbon structure (e.g. the embedding of $\Gamma$ near a vertex need not map the adjacent edges into a 2-plane centered at the vertex in a way which respects the cyclic ordering). 
\end{rem}

The identity graph cobordism from $(Y,\ws)$ to itself is $([0,1]\times Y, [0,1]\times \ws)$.

As always in cobordism categories, for $\GrCob$ to form an honest category, we must also include parametrizing diffeomorphisms of the boundary as data of a morphism, and we must also quotient by diffeomorphisms which respect these parametrizations. (Otherwise, there is no identity morphism). To simplify exposition, we suppress this fact.

Some care is required when defining equivalences of graph cobordisms, since vertices of valence greater than 1 are inherently non-smooth. In the morphism sets of $\GrCob$, we quotient by diffeomorphisms of the ambient 4-manifold, as well as the following notion of graph isotopy:

\begin{define}\label{def:smooth-graph-isotopy} Suppose that $W$ is a 4-manifold and $\Gamma$ is an abstract graph. We say that a continuous map
\[
i\colon [0,1]\times \Gamma\to W
\]
is a \emph{smooth isotopy of embedded graphs} if there is a finite subset $\ve{v}\subset \Gamma$, consisting of valence 1 vertices, such that the following are satisfied:
\begin{enumerate}
\item For each $t$, the map $i|_{\{t\}\times \Gamma}$ is a topological embedding.
\item $i^{-1}(\d W)=[0,1]\times \ve{v}$, and $i$ is constant on $\ve{v}$.
\item If $e\subset \Gamma$ is an edge, then the restriction $i|_{[0,1]\times e}$ is smooth.
\end{enumerate}
\end{define}

\subsection{The flow-graph category}
\label{sec:flow-graph}
There is a related category of interest to us, which we call the \emph{flow-graph} category of $Y$. The objects of $\FlGr(Y)$ are collections of basepoints $\ws\subset Y$ such that each component of $Y$ contains at least one basepoint. The morphisms in $\FlGr$ are the following:

\begin{define} Suppose $Y$ is a closed, oriented 3-manifold and $\ws_0$ and $\ws_1$ are two collections of basepoints in $Y$. We say a tuple $\cG=(\Gamma,i,\vs_0,\vs_1)$ is an immersed \emph{ribbon flow-graph} from $\ws_0$ to $\ws_1$ if the following are satisfied:
\begin{enumerate}
\item $\Gamma$ is an abstract ribbon graph. Furthermore, $i\colon \Gamma\to Y$ is an immersion.
\item Each vertex of $\Gamma$ has valence at least 1.
\item $\ve{v}_0$ and $\ve{v}_1$ disjoint collections of valence 1 vertices in $\Gamma$. Furthermore $i$ maps $\ve{v}_0$ bijectively to $\ws_0$ and $i$ maps $\ve{v}_1$ bijectively to $\ws_1$.
\end{enumerate}
\end{define}

We identify two ribbon flow-graphs $(\Gamma,i,\ve{v}_0,\ve{v}_1)$ and $(\Gamma',i',\ve{v}_0',\ve{v}_1')$ if there is a homeomorphism $h\colon \Gamma\to \Gamma'$ such that $i'=h\circ i$.

If $\cG=(\Gamma,i,\ve{v}_0,\ve{v}_1)$ is a ribbon flow-graph, we say that $\ve{v}_0$ and $\ve{v}_1$ are the \emph{boundary} vertices of $\Gamma$. We say all other vertices are \emph{interior} vertices. We will usually omit the immersion $i$ from the notation, and write simply $\cG=(\Gamma,\ws_0,\ws_1)$ for a flow-graph. For convenience, we will usually assume that our flow-graphs are embedded. 

We  use the following notion of equivalence of immersed flow-graphs:

\begin{define} Suppose that $\Gamma$ is an abstract graph with pairwise disjoint sets of valence 1 vertices $\ve{v}_0$ and $\ve{v}_1$. If $\ws_0,\ws_1$ are two collections of basepoints in $Y$, we say a continuous map
\[
h\colon [0,1]\times \Gamma\to Y
\]
is a \emph{smooth homotopy of immersed flow-graphs} from $\ws_0$ to $\ws_1$ if the following are satisfied:
\begin{enumerate}
\item For each $t$, the map $h|_{\{t\}\times \Gamma}$ is locally an embedding.
\item  $h(\{t\}\times \vs_0)=\ws_0$ and $h(\{t\}\times \vs_1)=\ws_1$ for all $t$.
\item If $e\subset \Gamma$ is an edge, then $h|_{[0,1]\times e}$ is smooth.
\end{enumerate}
\end{define}

The identity flow-graph from $\ws$ to itself consists of the pair $(\Gamma,i)$ where $\Gamma$ is the disjoint union of $|\ws|$ copies of the interval $[0,1]$, and $i$ is a small perturbation of the map which sends a copy of $[0,1]$ to the corresponding basepoint in $\ws$.

There is a functor from $\FlGr(Y)$ to $\GrCob$, which sends an immersed flow-graph $(\Gamma,i,\ve{v}_0,\ve{v}_1)$ to a graph cobordism $([0,1]\times Y,\Gamma')\colon (Y, i(\ve{v}_0))\to (Y, i(\ve{v}_1))$, where $\Gamma'$ is an embedded graph in $[0,1]\times Y$ which projects to $\Gamma$ (up to ambient isotopy).

\section{Heegaard Floer homology}
 \label{sec:background}

In this section, we recall Ozsv\'{a}th and Szab\'{o}'s construction of the Heegaard Floer complexes for multi-pointed 3-manifolds.

\subsection{Multi-pointed Heegaard diagrams and the Heegaard Floer complexes}

\begin{define}\label{def:multipointeddiagram}
Suppose $(Y,\ws)$ is a multi-pointed 3-manifold. A \emph{multi-pointed Heegaard diagram} $\cH=(\Sigma, \ve{\alpha},\bs,\ve{w})$ for $(Y,\ve{w})$ is a tuple satisfying the following:
\begin{enumerate}
\item $\Sigma$ is a closed, oriented surface, embedded in $Y$, such that $\ve{w}\subset \Sigma\setminus (\as\cup \bs)$. Furthermore, $\Sigma$ splits $Y$ into two handlebodies, $U_{\a}$ and $U_{\b}$, oriented so that $\Sigma=\d U_\a=-\d U_\b$.
\item $\ve{\alpha}=\{\alpha_1,\dots, \alpha_{n}\}$ is a collection of $n=g(\Sigma)+|\ws|-1$ pairwise disjoint, simple, closed curves on $\Sigma$, bounding pairwise disjoint compressing disks in $U_{\a}$. Each component of $\Sigma\setminus \as$ is planar and contains a single basepoint.
\item $\ve{\beta}=\{\beta_1,\dots, \beta_n\}$ is a collection of pairwise disjoint, simple, closed curves on $\Sigma$ bounding pairwise disjoint compressing disks in $U_{\b}$. Each component of $\Sigma\setminus \bs$ is planar and contains a single basepoint.
\end{enumerate}
\end{define}
We will also require Heegaard diagrams to satisfy an admissibility requirement; see Section \ref{sec:admissibility}.

\subsection{The  Heegaard Floer chain complexes}
\label{sec:HFcomplexes}

In this section, we describe Ozsv\'{a}th and Szab\'{o}'s Heegaard Floer chain complexes.

If $\ws=\{w_1,\dots, w_n\}$ is a set of basepoints in $Y$, we define the free polynomial ring
\[
\bF_2[U_{\ws}]:=\bF_2[U_{w_1},\dots, U_{w_n}].
\]
Write $\bF_2[U_{\ws},U_{\ws}^{-1}]$ for the ring obtained by formally inverting each of the variables $U_{w_i}$. 

If $\ve{k}=(k_1,\dots ,k_n)$ is an $n$-tuple, write
\[
U_{\ws}^{\ve{k}}:=U_{w_1}^{k_1}\cdots U_{w_n}^{k_n}.
\]



Suppose $(Y,\ws)$ is a connected, multi-pointed 3-manifold and $\frs\in \Spin^c(Y)$. Pick a diagram $\cH=(\Sigma,\as,\bs,\ws)$ of $(Y,\ws)$, and  consider the two tori 
\[
\bT_{\alpha}:=\alpha_1\times \cdots \times \alpha_{n}\quad \text{and} \quad \bT_{\beta}:=\beta_1\times \cdots \times \beta_{n},
\]
inside the symmetric product $\Sym^{n}(\Sigma)$, where $n=g(\Sigma)+|\ws|-1$.

Ozsv\'{a}th and Szab\'{o} construct a map
\[
\frs_{\ve{w}}\colon \bT_{\alpha}\cap \bT_{\beta}\to \Spin^c(Y);
\]
 See \cite{OSDisks}*{Section~2.6} for details on the construction.

As an $\bF_2[U_{\ws}]$-module, $\CF^-(\cH,\frs)$ is defined to be the free $\bF_2[U_{\ws}]$-module generated by intersection points $\ve{x}\in \bT_{\a}\cap \bT_{\b}$ with 
\[
\frs_{\ws}(\xs)=\frs.
\]
We define
\begin{equation}
\CF^\infty(\cH,\frs):=\CF^-(\cH,\frs)\otimes_{\bF_2[U_{\ws}]} \bF_2[U_{\ws}, U_{\ws}^{-1}]\quad \text{and} \quad \CF^+(\cH,\frs):=\CF^\infty(\cH,\frs)/\CF^-(\cH,\frs).\label{eq:CFinfty/plusdef}
\end{equation}

To equip $\CF^-(\cH,\frs)$ with a differential, we  pick an auxiliary path  $(J_s)_{s\in [0,1]}$ of almost complex structures on $\Sym^{g+|\ws|-1}(\Sigma)$. We write $\cM_{J_s}(\phi)$ for the moduli space of $J_s$-holomorphic maps $u\colon [0,1]\times \R\to \Sym^{g+|\ws|-1}(\Sigma)$ which represent the class $\phi$. The moduli space $\cM_{J_s}(\phi)$ has a natural action of $\R$, corresponding to reparametrization of the source. We write
\[
\hat{\cM}_{J_s}(\phi):=\cM_{J_s}(\phi)/\R,
\]

 For a sufficiently generic $J_s$,  we  define the differential 
\[
\d_{J_s}\colon \CF^-(\cH,\frs)\to \CF^-(\cH,\frs)
\]
via the formula
\begin{equation}
\d_{J_s}(\xs)=\sum_{\substack{\phi\in \pi_2(\xs,\ys)
\\ \mu(\phi)=1}} \# \hat{\cM}_{J_s}(\phi) U_{\ws}^{n_{\ws}(\phi)} \cdot \ys,\label{eq:defdiff}
\end{equation}
extended linearly over $\bF_2[U_{\ws}]$.

The endomorphism $\d_{J_s}$ satisfies
\[
\d_{J_s}\circ \d_{J_s}=0.
\] 
We refer the reader to \cite{OSLinks}*{Lemma~4.3} for a proof.


We write $\CF^-_{J_s}(\cH,\frs)$ for the $\bF_2[U_{\ws}]$-module $\CF^-(\cH,\frs)$ equipped with differential $\d_{J_s}$. When there is no ambiguity, we will usually drop the $J_s$ subscripts from both the chain complexes and the differential.

If $(Y,\ws)$ is a disconnected, multi-pointed 3-manifold, then a Heegaard diagram for $(Y,\ws)$ is a disjoint union of Heegaard diagrams for each component. The Heegaard Floer complex for such a diagram is the tensor product over $\bF_2$ of the complexes for each component.

\subsection{Coloring the Heegaard Floer complexes}

In this section, we describe a way of algebraically modifying the Heegaard Floer complexes  to achieve functoriality of the Heegaard Floer complexes under graph cobordisms.

\begin{define}\label{def:coloring} If $X$ is a topological space, a \emph{coloring} of $X$  is a function
\[
\sigma\colon C_0(X)\to \bmP,
\]
where $C_0(X)$ denotes the set of connected components of $X$, and $\bmP$ is a finite set.
\end{define}

When $X$ is a finite set (such as a set of basepoints), we view a coloring as a map from $X$ to $\bmP$.

If $\bmP=\{p_1,\dots, p_n\}$ is a finite set, write $\cR_{\bmP}$ for the $n$-variable polynomial ring
\[
\cR_{\bmP}:=\bF_2[U_{p_1},\dots, U_{p_n}].
\]

If $\sigma\colon \ws\to \bmP$ is a coloring, then $\sigma$ induces an action of $\bF_2[U_{\ws}]$ on $\cR_{\bmP}$, defined via the formula
\[
U_w\cdot x:=U_{\sigma(w)}\cdot x,
\]
for $x\in \cR_{\bmP}^-$.  Hence, if $M$ is an $\bF_2[U_{\ws}]$-module, we can form an $\cR_{\bmP}$-module $M^\sigma$ as the tensor product
\begin{equation}
M^\sigma:=M\otimes_{\bF_2[U_{\ws}]} \cR_{\bmP},\label{eq:colormoduleM}
\end{equation}
which we think of as being obtained by formally identifying $U_{w_i}$ with $U_{\sigma(w_i)}$.

If $(Y,\ws)$ is a multi-pointed 3-manifold, and $\sigma\colon \ws\to \bmP$ is a coloring, then we write $\CF^-(\cH,\frs,\sigma)$ for the complex
\begin{equation}
\CF^-(\cH,\sigma,\frs):=\CF^-(\cH,\frs)^\sigma=\CF^-(\cH,\frs)\otimes_{\bF_2[U_{\ws}]} \cR_{\bmP}. \label{eq:color-modified-complex-def}
\end{equation}

We briefly highlight the behavior of the complexes under disjoint unions:

\begin{rem} Suppose $(Y,\{w_1,w_2\})$ is the disjoint union of $(Y_1,w_1)$ and $(Y_2,w_2)$. If $\cH_i$ is a diagram for $(Y_i,w_i)$, then $\cH_1\sqcup \cH_2$ is a diagram for $(Y,\{w_1,w_2\})$. By definition 
\[
\CF^-(\cH_1\sqcup \cH_2, \frs_1\sqcup \frs_2):= \CF^-(\cH_1,\frs_1)\otimes_{\bF_2} \CF^-(\cH_2,\frs_2),
\]
which is a module over $\bF_2[U_{w_1},U_{w_2}]$.

Let $\bmP$ be the singleton $\{p\}$, and let  $\sigma\colon \{w_1,w_2\}\to \bmP$ denote the unique map.  Let $\sigma_1$ and $\sigma_2$ denote the restrictions of $\sigma$ to $\{w_1\}$ and $\{w_2\}$. Then
\[
\CF^-(\cH_1\sqcup \cH_2,\sigma,\frs_1\sqcup \frs_2)\iso \CF^-(\cH_1,\sigma_1,\frs_1)\otimes_{\bF_2[U_{p}]} \CF^-(\cH_2,\sigma_2,\frs_2),
\]
which is a module over the single variable polynomial ring $\cR_{\bmP}=\bF_2[U_{p}]$.

\end{rem}

\subsection{Lipshitz's cylindrical reformulation}

For many technical arguments in this paper, it will be convenient to work with Lipshitz's \emph{cylindrical reformulation} of Heegaard Floer homology \cite{LipshitzCylindrical}. If $(\Sigma,\as,\bs,w)$ is a Heegaard diagram for the singly pointed 3-manifold $(Y,w)$, Lipshitz shows that instead of counting holomorphic disks in $\Sym^{g}(\Sigma)$, one can instead count holomorphic curves which map into the 4-manifold $\Sigma\times [0,1]\times \R$. The equivalence is (morally) due to a \emph{tautological correspondence} between holomorphic disks mapping into $\Sym^g(\Sigma)$ and holomorphic curves (of higher genus) mapping into $\Sigma\times [0,1]\times \R$ whose projection to $[0,1]\times \R$ are $g$-fold branched covers. Lipshitz's cylindrical setting has a similar extension for multi-pointed Heegaard diagrams; see \cite{OSLinks}*{Section~5.2}.

We now describe the almost complex structures and moduli spaces which we consider in this paper (mostly taken directly from or slightly modified from \cite{LipshitzCylindrical}).

If $(\Sigma,\as,\bs,\ws)$ is a multi-pointed Heegaard diagram, we focus on almost complex structures $J$ on $\Sigma\times [0,1]\times \R$ which satisfy the following axioms (from \cite{LipshitzCylindrical}):

\begin{enumerate}[label=($J$\arabic*), ref=$J$\arabic*]
\item\label{def:J1} $J$ is tamed by the product symplectic form.
\item\label{def:J2} There is a finite collection of points $P\subset \Sigma\setminus (\as\cup \bs)$ such that $J$ is split (i.e. equal to $\frj_\Sigma\times \frj_{\bD}$) on a cylindrical neighborhood of $P\times [0,1]\times \R$.
\item\label{def:J3} $J$ is translation invariant in the $\R$ factor.
\item\label{def:J4} $J(\d/\d s)=\d/\d t$.
\item\label{def:J5} $J$ preserves the 2-planes $T(\Sigma\times \{(s,t)\})$ for all $(s,t)\in [0,1]\times \R$.
\end{enumerate}

Sometimes we will need to consider slightly more harshly perturbed almost complex structures, satisfying the following alternative to Axiom~\eqref{def:J5}:

\begin{enumerate}[label=($J$\arabic*$'$), ref=$J$\arabic*$'$]
\setcounter{enumi}{4}
\item \label{def:J5'} There is a 2-plane distribution $\xi$ on $\Sigma\times [0,1]$ such that the restriction of $\omega$ to $\xi$ is non-degenerate, $J$ preserves $\xi$ and the restriction of $J$ to $\xi$ is compatible with $\omega$. We further assume that $\xi$ is tangent to $\Sigma\times \{pt\}$ near $(\as\cup \bs)\times [0,1]\times \R$ and near $\Sigma\times \{0,1\}\times \R$.
\end{enumerate}

Following~\cite{LipshitzCylindrical}, we will consider holomorphic curves $u:S\to \Sigma\times [0,1]\times \R$, such that $S$ is a Riemann surface with boundary, as well as positive punctures $p_1,\dots, p_n$ and  negative punctures $q_1,\dots, q_n$ (where $n=g(\Sigma)+|\ws|-1$), satisfying the following:
 \begin{enumerate}[label=($M$\arabic*), ref=$M$\arabic*]
 \item\label{def:M1} $S$ is a smooth (not nodal) Riemann surface.
 \item\label{def:M2} $u(\d S)\subset (\as\times \{1\}\times \R)\cup (\bs\times \{0\}\times \R)$.
 \item\label{def:M3} $\lim_{z\to p_i} (\pi_{\R}\circ u)(z)=-\infty$ and $\lim_{z\to q_i} (\pi_{\R}\circ u)(z)=\infty$.
 \item\label{def:M4} $u$ has finite energy.
  \item\label{def:M5} $\pi_{[0,1]\times \R}\circ u$ is locally non-constant.
 \item\label{def:M6} $u$ is an embedding.
 \end{enumerate}

When considering almost complex structures which satisfy \eqref{def:J5'} instead of \eqref{def:J5}, we will need to consider the following weaker version of the \eqref{def:M5} axiom:

\begin{enumerate}[label=($M$\arabic*$'$), ref=$M$\arabic*$'$]
\setcounter{enumi}{4}
\item \label{def:M5'} There is no non-empty open subset $U\subset S$ such that $\pi_{[0,1]\times \R}\circ u|_U$ is constant, and takes value near $\{0,1\}\times \R$ (in the sense of \eqref{def:J5'}).
\end{enumerate}

There is a similar description of holomorphic triangle maps in Lipshitz's cylindrical setting. Let $\Delta$ denote an unbounded, triangular region in the complex plane with three cylindrical ends, each identified with $[0,1]\times [0,\infty)$. Following \cite{LipshitzCylindrical}*{Section~10.2}, we consider almost complex structures on $\Sigma\times \Delta$ satisfying the following:
\begin{enumerate}[label=($J'$\arabic*), ref=$J'$\arabic*]
\item\label{def:J'1} $J$ is tamed by the split symplectic form on $\Sigma\times \Delta$.
\item \label{def:J'2}There is a finite collection of points $P\subset \Sigma\setminus(\as\cup\bs\cup\gs)$ with at least one point in each component of $\Sigma\setminus (\as\cup \bs\cup \gs)$ such that $J$ is split on product neighborhood of $P\times \Delta$. 
\item\label{def:J'3} In the cylindrical ends of $\Delta$, $J$ is equal to a cylindrical almost complex structure satisfying \eqref{def:J1}--\eqref{def:J5}.
\item\label{def:J'4} The projection map $\pi_\Delta:\Sigma\times \Delta\to \Delta$ is holomorphic and the tangent space of each fiber of $\pi_\Sigma$ is a complex line.
\end{enumerate}

At times, we will need to deal with more harshly perturbed almost complex structures on $\Sigma\times \Delta$, which satisfy the following alternatives to Axioms~\eqref{def:J'3} and \eqref{def:J'4}:
\begin{enumerate}[label=($J'$\arabic*$'$), ref=$J'$\arabic*$'$]
\setcounter{enumi}{2}
\item\label{def:J'3'} In the cylindrical ends of $\Delta$, $J$ agrees with cylindrical almost complex structures satisfying $(J1)$-$(J4)$ and $(J5')$, above.
\item\label{def:J'4'} The 2-planes of $T(\{p\}\times \Delta)$ are complex lines of $J$ for all $p\in \Sigma$.
\item\label{def:J'5'} The 2-planes $T(\Sigma\times \{d\})$, for $d\in \Delta$, are complex lines for $J$ near $(\as\cup \bs\cup \gs)\times \Delta$ and on $\Sigma\times U$ for an open subset $U\subset \Delta$ containing the three  components of $\d \Delta$.
\end{enumerate}

\subsection{Expected dimensions and transversality}

For some technical arguments, we will need compute the expected dimension of certain moduli spaces, and also know when the expected dimension is generically correct. To give precise results, our expected dimensions must take into account the source curve $S$ of a holomorphic map $u\colon S\to \Sigma\times [0,1]\times \R$, and not just the homology class. If $S$ is a topological source, and $\phi$ is a homology class, we define
\[
\cM(S,\phi)
\]
to be the set of holomorphic curves $u\colon S\times [0,1]\times \R$ which satisfy \eqref{def:M1}--\eqref{def:M5} (but possibly not~\eqref{def:M6}). Near any curve where $D\bar{\d}$ achieves transversality, the set $\cM(S,\phi)$ will be a smooth manifold of dimension equal to the Fredholm index of $D\bar{\d}$ at $u$.

Lipshitz proved that if $u\colon S\to \Sigma\times [0,1]\times \R$ is a holomorphic curve which satisfies \eqref{def:M1}--\eqref{def:M6} (in particular, $u$ is embedded), then the Fredholm index is equal to the Maslov index. More generally, at curves which only satisfy \eqref{def:M1}--\eqref{def:M5}, the Fredholm index satisfies
\[
\ind(u)=\mu(\phi)-2\Sing(u),
\]
where $\Sing(u)$ denotes the number of double points of $u$ in an equivalent singularity (see \cite{LipshitzErrata}*{Proposition~4.2$'$}). 

If $X\subset \Sym^n([0,1]\times \R)$ is a smooth submanifold, $p\in \Sigma\setminus (\as\cup \bs)$ is a point, and $\phi$ is a homology class with $n_p(\phi)=n$, there is a \emph{matched moduli space}
\[
\cM(S,\phi,X):=\{u\in \cM(S,\phi): \rho^p(u)\in X\},
\]
where and $\rho^p\colon \cM(S,\phi)\to \Sym^{n}([0,1]\times \R)$ is the map
\begin{equation}
\rho^p(u):=(u\circ \pi_{[0,1]\times \R})\left((u\circ \pi_{\Sigma})^{-1}(p)\right).\label{eq:rhopdefinition}
\end{equation}
The analysis becomes simpler if we restrict to submanifolds $X\subset \Sym^{n}([0,1]\times \R)$ which avoid the \emph{fat diagonal}, i.e. the codimension 2 subset of points with at least one repeated entry.

We state the following transversality result:

\begin{prop}\label{prop:transversality} Suppose $J$ is a generic almost complex structure on $\Sigma\times [0,1]\times \R$ satisfying \eqref{def:J1}--\eqref{def:J5}. Near any holomorphic curve $u\colon S\to \Sigma\times [0,1]\times \R$ satisfying \eqref{def:M1}--\eqref{def:M5}, the moduli space $\cM(S,\phi)$ is transversely cut out and of dimension
\[
\ind(u)=\mu(\phi)-2\Sing(u).
\]
If $X\subset \Sym^n([0,1]\times \R)$ is a submanifold which avoids the fat diagonal, then near any curve $u\in \cM(S,\phi,X)$ satisfying \eqref{def:M1}--\eqref{def:M5}, the space $\cM(S,\phi,X)$ is transversely cut out and of dimension
\[
\mu(\phi)-2 \Sing(u)-\codim(X).
\]

If $J$ is a generic almost complex structure satisfying \eqref{def:J1}--\eqref{def:J4} and \eqref{def:J5'}, then the statements hold at holomorphic curves which satisfy \eqref{def:M1}--\eqref{def:M4} and \eqref{def:M5'} with no multiply covered closed components, and with no components $S_0$ such that $\pi_{[0,1]\times \R}\circ u|_{S_0}$ is constant and takes on a value near $\{0,1\}\times \R$ (in the sense of \eqref{def:J5'}).
\end{prop}

See \cite{LipshitzCylindrical}*{Sections~3 and 4} for proofs of the statements about the unmatched moduli spaces near embedded curves, and see \cite{LipshitzErrata} for the statements about the unmatched moduli spaces near non-embedded curves. See \cite{JTNaturality}*{Section~9.3} for an account of the proof involving the matched moduli spaces.

An analog of Proposition~\ref{prop:transversality} holds for the moduli spaces of holomorphic triangles satisfying the natural analogs of \eqref{def:M1}--\eqref{def:M5} for almost complex structures satisfying \eqref{def:J'1}--\eqref{def:J'4}. Furthermore, there is also a version for more harshly perturbed almost complex structures satisfying \eqref{def:J'1}, \eqref{def:J'2}, \eqref{def:J'3'}, \eqref{def:J'4'} and \eqref{def:J'5'} using the moduli space axioms obtained by modifying \eqref{def:M1}--\eqref{def:M4}, and \eqref{def:M5'} in the obvious ways; see \cite{JTNaturality}*{Section~9.3}.

\subsection{Naturality of Heegaard Floer homology}

 For functoriality, we need to understand the relation between the chain complexes obtained from different Heegaard diagrams of the same 3-manifold.

The following is standard:
\begin{lem}\label{lem:Heegaardmoves} Any two Heegaard diagrams for a multi-pointed 3-manifold can be connected by a sequence of the following moves:
 \begin{enumerate}
\item\label{HDmove:1} Isotopies of the $\ve{\alpha}$ or $\bs$ curves not passing over the $\ve{w}$ basepoints.
\item\label{HDmove:2} Handleslides of the $\ve{\alpha}$ curves across each other or handleslides of the $\bs$ curves across each other.
\item\label{HDmove:3} Simple stabilizations or destabilizations.
\item\label{HDmove:4} Pushing forward the diagram $(\Sigma,\as,\bs,\ws)$ under an automorphism $\phi$ of $(Y,\ws)$ which is isotopic to $\id_Y$, relative to $\ws$.
\end{enumerate}
\end{lem}

We explain Move~\eqref{HDmove:3}, appearing in Lemma~\ref{lem:Heegaardmoves}. Suppose $B^3\subset Y$ is a closed, embedded 3-ball which intersects $\Sigma$ in a disk and is disjoint from $\as\cup \bs\cup \ws$. We say a diagram $(\Sigma',\as',\bs',\ws)$ is a \emph{simple stabilization} of $(\Sigma,\as,\bs,\ws)$ if $\Sigma$ agrees with $\Sigma'$ outside of $B^3$, $\as'=\as\cup \{\alpha'\}$, $\bs'=\bs\cup \{\beta'\}$, and $\Sigma'\cap B^3$ consists of a once punctured torus. Furthermore, $\alpha'$ and $\beta'$  intersect in a single, transverse intersection point, and are contained in the region $\Sigma'\cap B^3$. A \emph{simple destabilization} is the inverse of a simple stabilization.

To each of the moves appearing in Lemma~\ref{lem:Heegaardmoves}, Ozsv\'{a}th and Szab\'{o} \cite{OSDisks} associated a chain map between the corresponding Heegaard Floer complexes.

We now state the following naturality result:

\begin{prop} Suppose that $(Y,\ws)$ is a multi-pointed 3-manifold. To each pair $(\cH,J)$ and $(\cH',J')$, there is a well defined map
\[
\Psi_{(\cH,J)\to (\cH',J')}\colon \CF^-_{J}(\cH,\frs)\to \CF^-_{J'}(\cH',\frs),
\]
 which is well defined up to $\bF_2[U]$-equivariant chain homotopy. Furthermore, the following are satisfied:
\begin{enumerate}
\item If $(\cH,J)$, $(\cH',J')$ and $(\cH'',J'')$ are three diagrams with almost complex structures, then $\Psi_{(\cH,J)\to (\cH'',J'')}\simeq \Psi_{(\cH',J')\to (\cH'',J'')}\circ \Psi_{(\cH,J)\to (\cH',J')}$.
\item $\Psi_{(\cH,J)\to (\cH,J)}\simeq \id_{\CF^-_{J}(\cH,\frs)}.$
\end{enumerate}
\end{prop}
\begin{proof}The details of the proof are due to many authors. Ozsv\'{a}th and Szab\'{o} \cite{OSDisks}  constructed the transition maps $\Psi_{(\cH,J)\to (\cH',J')}$ and proved they were quasi-isomorphisms (establishing invariance of the isomorphism type of the homology groups). They also proved most of the Floer theoretic results necessary to show well-definedness of $\Psi_{(\cH,J)\to (\cH',J')}$; see \cite{OSTriangles}*{Theorem~2.1}. Lipshitz showed that the transition maps are chain homotopy equivalences \cite{LipshitzCylindrical}*{Proposition~11.4}, as opposed to quasi-isomorphisms. Juh\'{a}sz and Thurston pointed out that \cite{OSTriangles}*{Theorem~2.1} contains a gap, since the space of isotopies taking one embedded Heegaard surface to another is not connected. The proof of invariance of the transition maps is completed in \cite{JTNaturality}, by a careful topological analysis of the space of Heegaard splittings, followed by verification that the Heegaard Floer transition maps have no monodromy around a special loop of Heegaard diagrams (the  \emph{simple handleswap} loop \cite{JTNaturality}*{Definition~2.31}).
\end{proof}

We now describe the maps associated to the moves in Lemma~\ref{lem:Heegaardmoves}, as well as changes of the almost complex structure.

 If $\cH=(\Sigma,\as,\bs,\ws)$ is a diagram for $(Y,\ws)$, and $J$ and $J'$ are two almost complex structures on $\Sigma\times [0,1]\times \R$, then the transition map
$\Psi_{(\cH,J)\to (\cH,J')}$ is defined by picking a non-cylindrical almost complex structure  $\tilde{J}$ on $\Sigma\times [0,1]\times \R$ which agrees with $J$ on $\Sigma\times [0,1]\times (-\infty,-1]$ and $J'$ on $\Sigma\times [0,1]\times [1,\infty)$, and counting Maslov index 0 $\tilde{J}$-holomorphic curves in $\Sigma\times [0,1]\times \R$.

Next, we consider the maps associated to handleslides and isotopies of $\as$ and $\bs$. In this case, the transition map $\Psi_{(\cH,J)\to (\cH',J')}$ can be computed by a sequence of holomorphic triangle maps. If $\as'$ is obtained from $\as$ by a handleslide or isotopy, and the triple $(\Sigma,\as',\as,\bs,\ws)$ satisfies an admissibility condition (see Definition~\ref{def:admissibilitytriple}), the map $\Psi_{(\cH,J)\to (\cH',J')}$ can be computed by counting holomorphic triangles via the formula
\begin{equation}
\Psi_{(\cH,J)\to (\cH',J)}(-)= F_{\a',\a,\b}(\Theta_{\a',\a}^+\otimes -).\label{eq:changeofdiagmras=triangles}
\end{equation}
In Equation~\eqref{eq:changeofdiagmras=triangles}, $\Theta_{\a',\a}^+$ denotes a cycle in $\CF^-(\Sigma,\as',\as,\ws,\frs_0)$ which represents the top graded homogeneous element of $\HF^-(\Sigma,\as',\as,\ws,\frs_0)$, and $\frs_0$ denotes the torsion $\Spin^c$ structure on $(S^1\times S^2)^{\# g(\Sigma)}$.
 
If the triple $(\Sigma,\as',\as,\bs,\ws)$ is not admissible, the transition map $\Psi_{(\cH,J)\to (\cH',J')}$ is obtained by a composition of triangle maps as in Equation~\eqref{eq:changeofdiagmras=triangles}. We will often write $\Psi_{\as\to \as'}^{\bs}$ for the transition map $\Psi_{(\cH,J)\to (\cH',J')}$, in this situation.

Similarly moves of the $\bs$ curves can be computed using the holomorphic triangle maps, in an analogous fashion. If $\bs'$ is obtained from $\bs$ via a sequence of handleslides or isotopies, we write $\Psi_{\as}^{\bs\to \bs'}$ for the corresponding transition map.

If $\cH=(\Sigma,\as,\bs,\ws)$ is a Heegaard diagram for $(Y,\ws)$, and $\cH'=(\Sigma',\as\cup \{\alpha'\},\bs\cup \{\beta'\},\ws)$ is a simple stabilization, then we consider the map $\sigma\colon \CF^-(\Sigma,\as,\bs,\ws,\frs)\to \CF^-(\Sigma',\as',\bs',\ws,\frs)$ defined by the formula
\[
\sigma(\xs):=\xs\times c,
\]
extended $\bF_2[U_{\ws}]$-equivariantly, where $\{c\}=\alpha'\cap \beta'$.

If $J$ is an almost complex structure on $\Sigma\times [0,1]\times \R$ which is split in a neighborhood of the connected sum point, and $T>0$, then one can construct an almost complex structure $J(T)$ on $\Sigma'\times [0,1]\times \R$ which has a connected sum neck of length $T$ inserted. According to the proof of \cite{LipshitzCylindrical}*{Proposition~12.5}, if $T$ is sufficient large, the map $\sigma$ will satisfy
\[
\sigma\circ \d_{J}=\d_{J(T)}\circ \sigma.
\]
For appropriately large $T$, the transition map from $(\cH,J)$ to $(\cH',J(T))$ is defined to be $\sigma$. The meaning of ``appropriately large'' can be made precise: We say a neck length $T>0$ satisfies \emph{stabilization condition} \eqref{eq:stabilizationcondition1} (and can be used to compute the simple stabilization map) if for all $T'\ge T$, there is an almost complex structure $\tilde{J}$ on $\Sigma\times [0,1]\times \R$ interpolating $J(T)$ and $J(T')$ so that
\begin{equation}
\Psi_{\tilde{J}}(\xs\times c)=\xs\times c,
\tag{SC-1}\label{eq:stabilizationcondition1}
\end{equation}
for all $\xs\in \bT_{\a}\cap \bT_{\b}$. We will consider a similar stabilizing condition when we define the free-stabilization, 1-handle and 3-handle maps; see Conditions~\eqref{eq:stabilizationcondition2} and \eqref{eq:stabilizationcondition3}.

If $J'$ is a general almost complex structure on $\Sigma'\times [0,1]\times\R$ (possibly not satisfying Condition~\eqref{eq:stabilizationcondition1})  then the transition map $\Psi_{(\cH,J)\to (\cH',J')}$ is defined as the composition
\begin{align*}
\Psi_{(\cH,J)\to (\cH',J')}&:=\Psi_{(\cH',J(T))\to (\cH',J')}\circ \Psi_{(\cH,J)\to (\cH',J(T))}\\
&=\Psi_{(\cH',J(T))\to (\cH',J')}\circ \sigma,
\end{align*}
for a $T$ which satisfies Condition~\eqref{eq:stabilizationcondition1}.

Finally, if $(\cH',J')$ is obtained by pushing forward $\cH$ under a diffeomorphism $\phi$ of $(Y,\ws)$ which is isotopic to the identity, relative to $\ws$, then the transition map $\Psi_{(\cH,J)\to (\cH',J')}$ is defined to be the tautological map $\phi_*$ induced by $\phi$.

If $(Y,\ws)$ is a multi-pointed 3-manifold and $\sigma\colon \ws\to \bmP$ is a coloring, then the $\cR_{\bmP}$-modules $\CF^-(\cH,\sigma,\frs)$ from Equation~\eqref{eq:color-modified-complex-def} form a transitive system of chain complexes, for which we write $\CF^-(Y,\ws^\sigma,\frs)$.

\subsection{Admissibility of Heegaard diagrams}
\label{sec:admissibility}

In order to achieve finite counts of holomorphic disks and triangles for $\CF^-$, Ozsv\'{a}th and Szab\'{o} define several admissibility conditions (\emph{weak admissibility} and \emph{strong $\frs$-admissibility}) for singly pointed Heegaard diagrams, triples and quadruples \cite{OSDisks}*{Sections~4 and 8}. For multi-pointed diagrams of integer homology spheres, they also described a weaker version of admissibility  (\emph{weak admissibility}) \cite{OSLinks}*{Section~3.4}, though this is not sufficient for our purposes.  In this section, we extend their work on strong $\frs$-admissibility to multi-pointed diagrams of arbitrary 3-manifolds (see Definition~\ref{def:strongadmissibility}).

\begin{define} If $(\Sigma,\as,\bs,\ws)$ is a multi-pointed Heegaard diagram, a \emph{periodic domain} is an  integral 2-chain $P$ on $\Sigma$ with boundary equal to a linear combination of the $\as$ and $\bs$ curves,  with $n_w(P)=0$ for all $w\in \ws$.  A \emph{periodic class} is a homology class $\phi\in \pi_2(\xs,\xs)$ for some $\xs\in \bT_{\a}\cap \bT_{\b}$ such that $n_{w}(\phi)=0$ for all $w\in \ws$.
\end{define}

The domain of a periodic class is a periodic domain, though we usually will not make a distinction between the two.
If $R$ is a ring, we can also consider the set of $R$-valued periodic domains, for which we write $\Pi_R$.

There is a natural map
\[
H\colon \Pi_{\Z}\to H_2(Y\setminus \ws;\Z),
\]
obtained by capping  $\d P$ with an integral combination of compressing disks for the $\as$ and $\bs$ curves. The construction of $H$ works for other rings $R$, as well.

If $\phi\in \pi_2(\xs,\ys)$ is a homology class of disks, we write $D(\phi)$ for the domain of $\phi$, viewed as a 2-chain on $\Sigma$. If $\phi\in \pi_2(\xs,\xs)$ is a periodic class,  define
\[
H(\phi):=H(D(\phi)).
\]

If $P\in \Pi_{\Z}$ is a periodic domain, if $\lambda$ is an integral 1-chain on $\Sigma$, with boundary equal to an integer sum of $\ws$ basepoints, then we will write $a(\lambda,P)$ for the integer obtained by summing each local difference of the class $\lambda$ across each $\as$ curve as one traverses $\lambda$. We define an integer $b(\lambda,P)$ analogously, by summing differences of the domain $P$ across the $\bs$ curves as one traverses $\lambda$. See Section~\ref{subsec:A-lambda-construction} for more details on $a(\lambda,P)$ and $b(\lambda,P)$.


If $\lambda$ is an integral 1-chain on $\Sigma$, with boundary equal to a linear combination of the $\ws$ basepoints, then by pushing the interior of $\gamma$ into the $\as$ handlebody, we obtain the formula
\begin{equation}
a(\lambda,P)=\#(\lambda\cap H(P)),\label{eq:periodicclassalambda}
\end{equation}
Similarly, if $\gamma$ is an integral 1-cycle on $\Sigma$, then 
\begin{equation}
a(\gamma,P)=\#(\gamma\cap H(P))=\langle\PD[\gamma], H(P)\rangle.\label{eq:periodicclassalambda2}
\end{equation}

A helpful topological fact is the following:
\begin{lem}\label{lem:isotoH_2}The map $H\colon \Pi_{R}\to H_2(Y\setminus \ws;R),$
is an isomorphism for any ring $R$.
\end{lem}
\begin{proof}
Since $\Pi_{\Z}$ and $H_2(Y\setminus \ws;\Z)$ are free $\Z$-modules, it is sufficient to show the claim for $R=\Z$.

To see that $H$ is a surjection, note that $Y\setminus N(\ws)$ is obtained by attaching 2-handles to $[0,1]\times (\Sigma\setminus N(\ws))$ along $\{0\}\times \as$ and $\{1\}\times \bs$. By putting a closed 2-cycle into general position, we can ensure that it intersects the co-cores of the 2-handles transversely. After a homotopy, it becomes homologous to a 2-chain in $\Sigma\setminus N(\ws)$ with boundary equal to a linear combination of the $\as$ and $\bs$ curves, together with some linear combination of the cores of the 2-handles. Hence $H$ is surjective.

For injectivity, suppose $H(P)=0\in H_2(Y\setminus \ws;\Z)$. By Equation~\eqref{eq:periodicclassalambda}, $a(\lambda,P)=0$ for any $\lambda$ which is a closed curve or a path connecting two basepoints in $\ws$. Similarly  $b(\lambda,P)=0$ for any such $\lambda$. By considering $\lambda$ arcs or curves which are dual to $\as$ and $\bs$ curves, it follows that $H(P)$ has no changes across any of the $\as$ and $\bs$ curves. Since $P$ also has zero multiplicity at the basepoints, $P$ must be zero everywhere. 
\end{proof}

By abuse of notation, we will also write $H(P)$ or $H(\phi)$ for the induced class in $H_2(Y;\Z)$. The capping map $H$  satisfies
\begin{equation}
\mu(\phi)=\langle c_1(\frs), H(\phi)\rangle+2\sum_{w\in \ws} n_{w}(\phi).\label{eq:Chernclassformula}
\end{equation}
Equation~\eqref{eq:Chernclassformula} can be proven by a simple modification of Ozsv\'{a}th and Szab\'{o}'s proof for singly pointed diagrams \cite{OSProperties}*{Proposition~7.5}.

We make the following definition, extending Ozsv\'{a}th and Szab\'{o}'s \emph{strong $\frs$-admissibility} condition \cite{OSDisks}*{Section~4} to multi-pointed diagrams:

\begin{define}\label{def:strongadmissibility} If $\frs\in \Spin^c(Y)$, we say a Heegaard diagram $\cH=(\Sigma,\as,\bs,\ws)$ is \emph{strongly $\frs$-admissible} if for each $N>0$ and each non-trivial periodic domain $P\in \Pi_{\Z}$, the inequality
\[
\langle c_1(\frs), H(P)\rangle =2N\ge 0
\]
implies that $P$ has some multiplicity strictly greater than $N$.
\end{define}

Strong $\frs$-admissibility ensures finiteness of the differential on $\CF^-$:

\begin{lem}\label{lem:stronglyadmissiblefinitedisks}If $\cH$ is strongly $\frs$-admissible, and $j$ is fixed, then there are only finitely many homology classes $\phi\in \pi_2(\ve{x},\ve{y})$ with $\mu(\phi)=j$ and $D(\phi)\ge 0$.
\end{lem}

\begin{proof} The proof is a modification of Ozsv\'{a}th and Szab\'{o}'s proof for singly pointed diagrams \cite{OSDisks}*{Lemma~4.14}.

 Fix any $\psi\in \pi_2(\ve{x},\ve{y})$ with $\mu(\psi)=j$.  If $\phi\in \pi_2(\ve{x},\ve{y})$ is another class then we can uniquely write
\[
D(\phi)=D(\psi)+P+A
\]
where $P\in \Pi_{\Z}$, and $A$ is a $\Z$-linear combination of the components of $\Sigma\setminus \as$. If $D(\phi)\ge 0$, then
\begin{equation}
-D(\psi)\le P+A.\label{eq:psianddomains}
\end{equation}
 
Suppose to the contrary of the main statement, that there is an infinite sequence  $\phi_n$ of distinct classes in $\pi_2(\xs,\ys)$, with $\mu(\phi_n)=j$ and $D(\phi_n)\ge 0$. We obtain an infinite sequence of pairwise distinct pairs $(P_n,A_n)$ which satisfy Equation~\eqref{eq:psianddomains}.

Since there are only finitely many components of $\Sigma\setminus (\as\cup \bs)$, it follows that  
\[
\|P_n+A_n\|_\infty\to \infty,
\]
 where $\|D\|_\infty$ denotes the maximum absolute value of the components of a domain $D$.

The coefficients of $P_n+A_n$ are bounded below by the coefficients of  $-D(\psi)$. Similarly, since $P_n$ satisfies $n_w(P_n)=0$ for all $n$, it follows that $n_{w}(A_n)=n_w(\phi_n)-n_w(\psi)$. Since $D(\phi_n)\ge 0$, it follows that the coefficients of $A_n$ are also bounded from below.

Since $P_n$ is zero on $\ws$, and $A_n$ is determined by its values on $\ws$, we have
\begin{equation}
\|A_n\|_\infty\le \|P_n+A_n\|_\infty. \label{eq:An-bounded}
\end{equation}
By the triangle inequality
\begin{equation}
\|P_n\|_{\infty}\le \|A_n\|_\infty+\|P_n+A_n\|_\infty.\label{eq:Pn-bounded}
\end{equation}

 Equations~\eqref{eq:An-bounded} and ~\eqref{eq:Pn-bounded} imply $A_n/\|P_n+A_n\|_\infty$ and $P_n/\|A_n+P_n\|_\infty$ are both bounded, and hence admit subsequences which converge to real domains $A_\infty$ and $P_\infty$, respectively.

Since $\mu(\phi_n)=\mu(\psi)=j$, we have that 
\begin{equation}
\mu(P_n^{\xs}+A_n^{\xs})=0,\label{eq:Maslov-index-Pn+An}
\end{equation} where $P_n^{\xs},A_n^{\xs}\in \pi_2(\xs,\xs)$ denote the periodic classes corresponding to $P_n$ and $A_n$. Combining Equations~\eqref{eq:Chernclassformula} and~\eqref{eq:Maslov-index-Pn+An}, as well as the fact that $H(A_n)=0\in H_2(Y;\Z)$ and  $n_w(P_n)=0$, we obtain
\begin{equation}
0=\mu(P_n^{\xs}+A_n^{\xs})=\langle c_1(\frs), H(P_n)\rangle+\sum_{w\in \ws} 2n_{w}(A_n).\label{eq:theMaslovindexiszero}
\end{equation}
Taking limits in Equation~\eqref{eq:theMaslovindexiszero}, we obtain
\begin{equation}
\langle c_1(\frs), H(P_\infty)\rangle+\sum_{w\in \ws} 2n_{w}(A_\infty)=0.\label{eq:relationinvolvinglimitsdomains}
\end{equation}
 Since the multiplicities of $A_n$ and $A_n+P_n$ are bounded below, and $\|A_n+P_n\|_\infty\to \infty$, we conclude 
 \begin{equation}
 A_\infty\ge 0 \qquad \text{and}\qquad P_\infty+A_\infty\ge 0.\label{eq:A_infty>0}
 \end{equation}
If there is a real pair $(P_\infty,A_{\infty})$  satisfying Equations~\eqref{eq:relationinvolvinglimitsdomains} and ~\eqref{eq:A_infty>0}, then it is not hard to see that there is also a nearby rational domain which also satisfies Equations~\eqref{eq:relationinvolvinglimitsdomains} and \eqref{eq:A_infty>0}. By clearing denominators, we can find a pair $(P',A')$ of integral domains, satisfying the same relations.
  
  By Equations~\eqref{eq:relationinvolvinglimitsdomains} and \eqref{eq:A_infty>0},
\[
\langle c_1(\frs), H(-P')\rangle=2N\ge 0
\]
 where $N=\sum_{w\in \ws} n_{w}(A').$ Hence $-P'$ has a multiplicity which is greater than $N$ by strong admissibility. Since $A'\ge 0$, this contradicts the fact that $A'+P'\ge 0$. 
\end{proof}

We now prove that all multi-pointed Heegaard diagrams can be made admissible by performing an isotopy:

\begin{prop}\label{prop:windforadmiss}If $\cH=(\Sigma,\as,\bs,\ws)$ is a multi-pointed Heegaard diagram for $(Y,\ws)$, and $\frs$ is a fixed $\Spin^c$ structure on $Y$, then $\cH$ is isotopic to a strongly $\frs$-admissible diagram. 
\end{prop}
\begin{proof}
Our proof is a modification of Ozsv\'{a}th and Szab\'{o}'s procedure for achieving strong $\frs$-admissibility for singly pointed diagrams \cite{OSDisks}*{Lemma~5.4}, and their procedure for achieving weak admissibility for multi-pointed diagrams of integer homology spheres \cite{OSLinks}*{Proposition~3.6}.

Pick a collection of closed curves $\gamma_1,\dots, \gamma_k$ and arcs $\lambda_{1}\dots, \lambda_{n}$ on $\Sigma$ satisfying the following:
\begin{enumerate}
\item $\gamma_1,\dots, \gamma_k$ are pairwise disjoint, simple, closed curves which span $H_1(Y;\Z)$.
\item Each $\lambda_i$ is an embedded arc with boundary equal to two basepoints of $\ws$.
\item The interiors of $\lambda_i$ and $\lambda_j$ are disjoint if $i\neq j$, and $\lambda_i$ is disjoint from $\gamma_j$ for all $i$ and $j$.
\item Each basepoint of $\ws$ is in the boundary of at least one $\lambda_i$, and $\lambda_1\cup\cdots \cup \lambda_n$ is connected.
\end{enumerate}
The above conditions imply that $\gamma_1,\dots, \gamma_k,\lambda_1,\dots, \lambda_n$ span $H_1(Y,\ws;\Z)$.

Such  a collection $\gamma_1,\dots, \gamma_k,$ $\lambda_1,\dots, \lambda_n$ can be constructed as follows. Let $\as_0$ be any tuple of $g(\Sigma)$ attaching curves on $\Sigma$, which are disjoint from $\ws$ and which bound compressing disks in the handlebody $U_{\a}$, such that $\Sigma\setminus \as_0$ is a connected, planar surface. The curves $\gamma_1,\dots, \gamma_k$ can be chosen to be duals of the curves of $\as_0$. The complement of $\as_0$ is planar, connected, and contains all of the $\ws$ basepoints. The $\lambda_i$ can be chosen to form an embedded tree, such that each $\lambda_i$ connects one $\ws$ basepoint to a chosen central basepoint of $\ws$.

Write $\Pi'_{\Q}$ for the set of rational 2-chains on $\Sigma$ of the form
\[
P-\frac{\langle c_1(\frs), H(P)\rangle}{2}\cdot [\Sigma] \quad \text{ for  } \quad P\in \Pi_{\Q}.
\]
Following Ozsv\'{a}th and Szab\'{o}'s terminology, we call such domains $\frs$-\emph{renormalized} periodic domains. The groups $\Pi_{\Q}$ and $\Pi_{\Q}'$ are canonically isomorphic, so there is a well defined capping map
\[
H\colon \Pi'_{\Q}\to H_2(Y\setminus \ws; \Q),
\]
which is an isomorphism by Lemma~\ref{lem:isotoH_2}.

Let $R_1,\dots, R_n$ be a collection of pairwise disjoint, embedded rectangles in $\Sigma$, such that the following hold:
\begin{enumerate}
\item $\lambda_i\cap R_i$ is a connected arc.
\item $\lambda_i\cap R_j=\emptyset$ if $i\neq j$, and $\gamma_i\cap R_j=\emptyset$ for all $i$ and $j$.
\item If $\tau$ is an attaching curve in $\as$ or $\bs$, then $\tau\cap \lambda_i\subset R_i$. 
\end{enumerate}

Fix $N>0$. We perform the following two winding moves to $\as$  to construct a diagram $\cH_N=(\Sigma,\as_N,\bs)$:
\begin{enumerate}[ref= W-\arabic*, label= (W-\arabic*):]
\item\label{num:wind2} Let $\gamma_i^+$ and $\gamma^-_i$ denote two small, parallel pushoffs of $\gamma_i$, which are disjoint from each other and $\gamma_i$. Wind the $\as$ curves positively $N$ times around $\gamma_i^+$, and negatively $N$ times around $\gamma_i^-$, as shown in Figure~\ref{fig::61}
\item\label{num:wind1}  Along each $\lambda_i$, we perform a zig-zag move to the $\as$ curves, as shown in Figure~\ref{fig::60}. The zig-zag move is supported in the rectangle $R_i$.
\end{enumerate}

We write $\Pi_{\Q,N}'$ for the set of $\frs$-renormalized periodic domains on $\cH_N$.

The groups $\Pi_{\Q,N}'$ and $\Pi_{\Q}'$ are canonically isomorphic: indeed Lemma~\ref{lem:isotoH_2} implies both are canonically isomorphic to $H_2(Y\setminus \ws;\Q)$. Write
\[
\cW_{N}\colon \Pi_{\Q}'\to \Pi_{\Q,N}'
\]
for this isomorphism.

\begin{subclaim}\label{subclaim:wind=>admiss}For sufficiently large $N$, any nonzero $\frs$-renormalized periodic domain on $\cH_N$ has both positive and negative multiplicities.
\end{subclaim}

It is straightforward to see that Subclaim~\ref{subclaim:wind=>admiss} implies that for sufficiently large $N$, $\cH_N$ is strongly $\frs$-admissible.

We now prove Subclaim~\ref{subclaim:wind=>admiss}. Define an $L^\infty$-norm on $H_2(Y\setminus \ws;\Q)$ via the formula
\[
\|\sigma \|_{\infty}^{Y\setminus \ws} =\max\left \{|\# (\lambda_1\cap \sigma)|, \dots,  |\# (\lambda_{n}\cap \sigma)|, |\# (\gamma_1\cap \sigma)|, \dots, |\# (\gamma_k\cap \sigma)|\right\},
\]
where $\# (\lambda_i\cap \sigma)$ denotes the algebraic intersection number. Similarly, define an $L^\infty$-norm on $H_2(Y;\Q)$ via the formula
\[
\|\sigma\|_{\infty}^{Y}:=\max\left \{|\# (\gamma_1\cap \sigma)|, \dots, |\# (\gamma_k\cap \sigma)|\right\}.
\]

Let $\bS_{\Q}'$ denote the unit sphere in $\Pi_{\Q}'$ with respect to the $\|\cdot\|_{\infty}^{Y\setminus \ws}$-norm, i.e.
\[
\bS_{\Q}':=\{P\in \Pi_{\Q}': \|H(P)\|_{\infty}^{Y\setminus \ws}=1\}.
\]

It is sufficient to show Subclaim~\ref{subclaim:wind=>admiss} for $\frs$-renormalized periodic domains $P$ such that $H(P)\in \bS_{\Q}'$,  since scaling by $\Q$ preserves the property of having positive and negative coefficients.

Since $\gamma_1,\dots, \gamma_k$ span $H_1(Y;\Z)$, we can write
\begin{equation}
\PD[c_1(\frs)]=a_1[\gamma_1]+\cdots +a_k [\gamma_k],\label{eq:decompose-c1-frs-gamma-i}
\end{equation}
for some $a_i\in \Z$. Let $M$ denote the quantity
\[
M:=\sum_{i=1}^k |a_i|.
\]

Pick a real number $\epsilon$ so that 
\begin{equation}
0<\epsilon<\frac{1}{2}\qquad \text{and}\qquad  M \epsilon<\frac{1}{2}.\label{eq:epsilonsmallish}
\end{equation}
%

 Let $C$ be the maximum absolute value of any multiplicity of any domain in $\bS_{\Q}'$, and let $N>0$ be an integer which satisfies 
\begin{equation}
N\cdot \epsilon >C.\label{eq:Nlarge}
\end{equation}
 We will show that if $N$ satisfies Equation~\eqref{eq:Nlarge}, then any nonzero $\frs$-renormalized domain on $\cH_N$ has both positive and negative multiplicities.

Suppose that $P\in \bS_{\Q}'$. Note that by definition $\cW_N$ does not change the induced class in $H_2(Y\setminus \ws;\Z)$, and hence $\cW_N$ does not change either $\|\cdot\|_{\infty}^{Y\setminus \ws}$ or $\|\cdot\|_{\infty}^{Y}$. We break the argument into two cases:
\begin{enumerate}
\item\label{case:bigabshom} $\|H(P)\|_{\infty}^{Y}>\epsilon$.
\item\label{case:bigrelhom} $\|H(P)\|_{\infty}^{Y}\le \epsilon$. 
\end{enumerate}

In Case~\eqref{case:bigabshom}, let $\gamma_i$ be a curve so that $|\#(\gamma_i\cap H(P))|>\epsilon$. By Equation~\eqref{eq:periodicclassalambda},  $|a(\gamma_i, P)| >\epsilon$. Pick $x\in \gamma_i\setminus (\as\cup \bs)$ and let $x^+$ and $x^-$ be nearby  points on $\gamma^+_i\setminus (\as\cup \bs)$ and $\gamma_i^-\setminus (\as\cup \bs)$.

\begin{figure}[ht!]
\centering
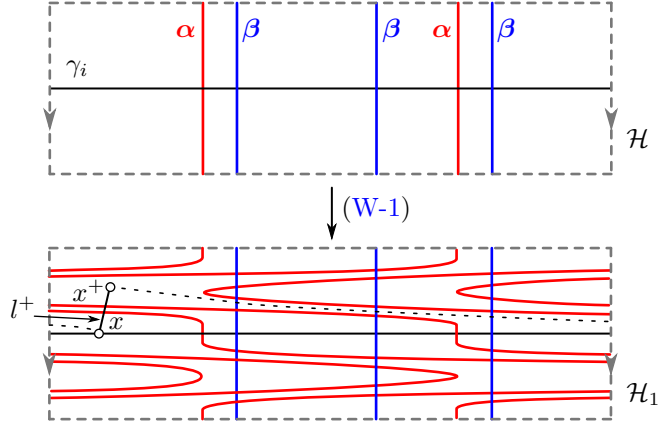
\caption{\textbf{Winding move~\eqref{num:wind2} near a closed curve $\gamma_i$ on $\Sigma$.} The arc $l^+$ is shown. The left and right sides of each rectangle are identified to form an annulus on $\Sigma$.}\label{fig::61}
\end{figure}

Let $l^+$ be an arc from $x$ to $x^+$, which is contained in a neighborhood of $\gamma_i$ and does not intersect  $\as$ or $\bs$. On $\cH_N$, $l^+$ intersects a curve in $\as_N$ with $N$ times the geometric multiplicity that $\gamma_i$ intersects the corresponding curve of $\as$. Furthermore, the arc $l^+$ can be concatenated with an arc which connects $x$ and $x^+$ and is disjoint from  $\as_N$, such that the concatenation is homologous to $N\cdot \gamma$ on $\Sigma$ (see the bottom of Figure~\ref{fig::61}). Writing $l^-$ for a similar path from $x$ to $x^-$, it follows that
\begin{equation}
a(l^+, \cW_N(P))=N\cdot a(\gamma_i, P)\quad \text{and} \quad a(l^-,\cW_N(P))=-N\cdot a(\gamma_i,P).\label{eq:al=agamma}
\end{equation}

Observing that $l^+$ and $l^-$ intersect no $\bs$ curves, it follows from Equation~\eqref{eq:al=agamma} that
\begin{equation}
n_{x^+}(\cW_{N}(P))=n_{x}(P)+N\cdot a(\gamma_i,P)\qquad \text{and} \qquad n_{x^-}(\cW_{N}(P))=n_{x}(P)-N\cdot a(\gamma_i,P).\label{eq:compn_xafterwinding}
\end{equation}
Together, Equations~\eqref{eq:Nlarge} and \eqref{eq:compn_xafterwinding} imply that $\cW_{N}(P)$ has both positive and negative multiplicities, proving Subclaim~\ref{subclaim:wind=>admiss} in Case~\eqref{case:bigabshom}.

Next, we consider Case~\eqref{case:bigrelhom}. Combining Equations~\eqref{eq:decompose-c1-frs-gamma-i} and ~\eqref{eq:epsilonsmallish} and the triangle inequality, we obtain 
\begin{equation}
|\langle c_1(\frs), H(P) \rangle|=\left|\sum_{i=1}^k a_i\langle \PD[\gamma_i], H(P) \rangle\right|\le \sum_{i=1}^k |a_i| \cdot |\# (\gamma_i\cap H(P) )|\le M\cdot \epsilon<\frac{1}{2}.\label{eq:boundchernclass}
\end{equation}


\begin{figure}[ht!]
\centering
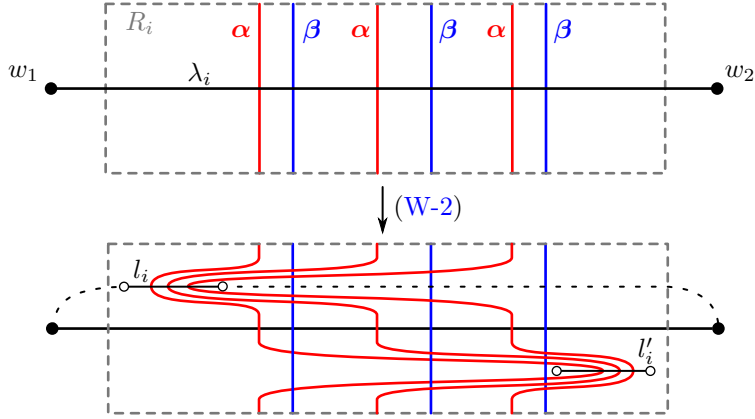
\caption{\textbf{The zig-zag move~\eqref{num:wind1} near an arc $\lambda_i$ with $\d \lambda_i=\{w_1,w_2\}.$} The arcs $l_i$ and $l_i'$ near $w_1$ and $w_2$ are shown. The move is supported in the rectangle $R_i$, whose boundary is the dashed box.}\label{fig::60}
\end{figure}

Define $m:=-\langle c_1(\frs), H(P)\rangle$, which is the multiplicity of $\cW_N(P)$ at each basepoint $w\in \ws$. By Equation~\eqref{eq:boundchernclass},   $m\in (-\tfrac{1}{2},\tfrac{1}{2})$. Furthermore, since $\| H(P)\|^{Y\setminus \ws}_\infty=1$, there is a $\lambda_i$ such that $a(\lambda_i, P)=1$.  We consider the two arcs $l_i$ and $l_i'$ on $\Sigma$, shown in Figure~\ref{fig::60}, which intersect no curves on $\Sigma$, and intersect only the $\as$ curves on $\cH_N$. Note that
\[
a(\lambda_i,\cW_N(P))=a(l_i,\cW_N(P))=a(l_i',\cW_N(P)),
\]
since $l_i$ and $l_i'$ can both be completed to curves on $\Sigma$ which are isotopic to $\lambda_i$, and whose only intersections with $\as$ occur along $l_i$ or $l_i'$ (see the dashed lines on the bottom of Figure~\ref{fig::60}). Since $l_i$ and $l_i'$ do not intersect the $\bs$ curves, it follows that the difference in multiplicity of $\cW_N(P)$ between the two points of $\d l_i$ is $a(\lambda_i,\cW_N(P))=1$, and similarly for the two points of $\d l_i'$.  Hence $\cW_{N}(P)$ has multiplicities of $m+1$ and $m-1$. Since $|m|<1/2$, it follows that $m-1<0$ and $m+1>0$. Subclaim~\ref{subclaim:wind=>admiss} follows in Case~\eqref{case:bigrelhom}, completing the proof.
\end{proof}

\subsection{Admissibility for Heegaard triples and quadruples}

We also need to consider admissibility of Heegaard triples and multi-diagrams.

If $(\Sigma,\as,\bs,\gs,\ws)$ is a Heegaard triple, Ozsv\'{a}th and Szab\'{o} construct the 4-manifold
\begin{equation}
X_{\a,\b,\g}:=(\Delta\times \Sigma)\cup (e_{\a}\times U_{\a})\cup (e_{\b}\times U_{\b})\cup (e_{\g}\times U_{\g}),\label{eq:X_abgdef}
\end{equation}
where $\Delta$ is a triangle with boundary edges $e_{\a}$, $e_{\b}$ and $e_{\g}$, and $U_{\a}$, $U_{\b}$ and $U_{\g}$ are standard handlebodies with boundary identified with $\Sigma$ \cite{OSDisks}*{Section~8.1}. The manifold $X_{\a,\b,\g}$ satisfies
\[
\d X_{\a,\b,\g}=-Y_{\a,\b}\sqcup -Y_{\b,\g}\sqcup Y_{\a,\g}.
\]
Ozsv\'{a}th and Szab\'{o} \cite{OSDisks}*{Section~8.1.4} construct  a map
\begin{equation}
\frs_{\ws}\colon \pi_2(\xs,\ys,\zs)\to \Spin^c(X_{\a,\b,\g}).\label{eq:spincmapdef}
\end{equation}

\begin{define}\label{def:admissibilitytriple}
If $(\Sigma,\as,\bs,\gs,\ws)$ is a multi-pointed Heegaard triple, we say that an integral 2-chain $P$ on $\Sigma$ is a \emph{triply periodic domain} if $\d P$ is a linear combination of the $\as$, $\bs$ and $\gs$ curves, and $n_{w}(P)=0$ for all $w\in \ws$.
\end{define}

If $\cQ=(\Sigma,\as,\bs,\gs,\ds,\ws)$ is a multi-pointed Heegaard quadruple, a \emph{quadruply periodic domain} on $\cQ$ is defined analogously.

Ozsv\'{a}th and Szab\'{o} define a notion of strong admissibility for Heegaard triples \cite{OSDisks}*{Definition~8.8}, which extends to the multi-pointed setting:

\begin{define}Suppose $\cT=(\Sigma,\as,\bs,\gs,\ws)$ is a multi-pointed Heegaard triple. If $\frs\in \Spin^c(X_{\a,\b,\g})$, we say $\cT$ is \emph{strongly $\frs$-admissible} if each nonzero triply periodic domain $P$, which can be written as a sum of doubly periodic domains
\[
P=P_{\a,\b}+P_{\b,\g}+P_{\a,\g}
\]
with
\[
\langle c_1(\frs),H(P_{\a,\b}) \rangle+\langle c_1(\frs),H(P_{\b,\g}) \rangle+\langle c_1(\frs),H(P_{\a,\g}) \rangle=2N\ge 0,
\]
has a multiplicity which is strictly greater than $N$.
\end{define}

A straightforward extension of Lemma~\ref{lem:stronglyadmissiblefinitedisks} and Proposition~\ref{prop:windforadmiss} gives the following (compare \cite{OSDisks}*{Lemmas~8.10 and 8.11}):

\begin{lem}\label{lem:finiteness-triangles} Suppose that $\cT=(\Sigma,\as,\bs,\gs,\ws)$ is a Heegaard triple.
\begin{enumerate}
\item If $\cT$ is strongly $\frs$-admissible, $\xs$, $\ys$ and $\zs$ are fixed intersection points on $\cT$, and $j$ is a fixed integer, then there are only finitely many $\psi\in \pi_2(\xs,\ys,\zs)$ with $\frs_{\ws}(\psi)=\frs$, $\mu(\psi)=j$ and $D(\psi)\ge 0$.
\item If  $\frs\in \Spin^c(X_{\a,\b,\g})$, then $\cT$ is isotopic to a strongly $\frs$-admissible Heegaard triple.
\end{enumerate}
\end{lem}

Given a Heegaard quadruple $(\Sigma,\as,\bs,\gs,\ds,\ws)$, there is a four-ended cobordism $X_{\a,\b,\g,\delta}$, as well as a map
\[
\frs_{\ws}\colon \pi_2(\ws,\xs,\ys,\zs)\to \Spin^c(X_{\a,\b,\g,\delta}).
\]
There are two natural decompositions of the 4-manifold $X_{\a,\b,\g,\delta}$:
\[
X_{\a,\b,\g,\delta}=X_{\a,\b,\delta}\cup_{Y_{\b,\delta}} X_{\b,\g,\delta}=X_{\a,\g,\delta}\cup_{Y_{\a,\g}} X_{\a,\b,\g}.
\]

For the purposes of ensuring finiteness of holomorphic rectangle counts, we need finiteness of homotopy classes in an entire $\delta H^1(Y_{\b,\delta})+\delta H^1(Y_{\a,\g})$-orbit $\frS$ in $\Spin^c(X_{\a,\b,\g,\delta})$. 

\begin{define} Suppose $\cQ=(\Sigma,\as,\bs,\gs,\ds,\ws)$ is a multi-pointed Heegaard quadruple and $\frS$ is a $\delta H^1(Y_{\b,\delta})+\delta H^1(Y_{\a,\g})$-orbit of $\Spin^c$ structures on $X_{\a,\b,\g,\delta}$. We say that $\cQ$ is \emph{strongly $\frS$-admissible} if whenever $\frs\in \frS$ and $P$ is a nonzero periodic class which can be written as a sum of doubly periodic domains 
\[
P=\sum_{\{\xi,\eta\}\subset \{\a,\b,\g,\delta\}} P_{\xi,\eta}
\]
such that
\[
 \sum_{\{\xi,\eta\}\subset \{\a,\b,\g,\delta\}} \langle c_1(\frs|_{Y_{\xi,\eta}}), H(P_{\xi,\eta})\rangle=2N\ge 0,
\]
then $P$ has a multiplicity strictly greater than $N$.
\end{define}

If $\cQ$ is strongly $\frS$-admissible Heegaard quadruple, then a finiteness result for positive rectangle classes of a fixed Maslov index, similar to Lemma~\ref{lem:finiteness-triangles} can be proven by adapting Lemma~\ref{lem:stronglyadmissiblefinitedisks}.

If $\cQ$ is an arbitrary Heegaard quadruple and $\frS$ is an $\delta H^1(Y_{\b,\delta})+\delta H^1(Y_{\a,\g})$-orbit, then strong $\frS$-admissibility can be achieved by modifying the winding and zig-zag procedure from Proposition~\ref{prop:windforadmiss} as long as the quantity
\[
 \sum_{\{\xi,\eta\}\subset \{\a,\b,\g,\delta\}} \langle c_1(\frs|_{Y_{\xi,\eta}}), H(P_{\xi,\eta})\rangle,
\]
is independent of the choice of $\frs\in \frS$. In particular, strong admissibility can be achieved by winding the curves $\as$, $\bs$, $\gs$ and $\ds$ as long as
\[
\delta H^1(Y_{\b,\delta})|_{Y_{\a,\g}}=0 \quad \text{and} \quad \delta H^1(Y_{\a,\g})|_{Y_{\b,\delta}}=0.
\]

\section{Relative homology action}
\label{sec:relhommap}

Ozsv\'{a}th and Szab\'{o} constructed an action of $\Lambda^* (H_1(Y;\Z)/\Tors)$ on the singly pointed Heegaard Floer groups $\HF^\circ(Y,w,\frs)$ \cite{OSDisks}*{Section~4.2.5}. In this section, we describe similar maps on the multi-pointed Floer complexes for closed loops in $Y$, as well as paths between pairs of basepoints. We call these maps the \emph{relative homology action}. Our construction is similar to an action of relative homology on sutured Floer homology constructed by Ni  \cite{Ni-homological}.

If $(Y,\ws)$ is a multi-pointed 3-manifold, and $\lambda$ is a path between two basepoints, $w_1,w_2\in \ws$, we  construct two maps
\[
A_{\lambda},\, B_\lambda\colon \CF^-(Y,\ws,\frs)\to \CF^-(Y,\ws,\frs),
\]
which we call the \emph{type-$A$} and \emph{type-$B$} relative homology maps, respectively.

We will show that $A_\lambda$ and $B_{\lambda}$ satisfy
\[
\d A_\lambda+A_\lambda \d=\d B_\lambda+B_{\lambda} \d=U_{w_1}+U_{w_2}.
\]
See Lemma~\ref{lem:Alambdachainhomotopy}. 

If $\gamma$ is a closed loop in $Y$, there are similar chain maps $A_\gamma$ and $B_{\gamma}$. For a closed curve $\gamma$, the maps $A_\gamma$ and $B_{\gamma}$ coincide, and both agree with the familiar action of $H_1(Y;\Z)/\Tors$. See Lemma~\ref{lem:AgBgrel}.


\subsection{Construction of the relative homology action}
\label{subsec:A-lambda-construction}

Suppose $\cH=(\Sigma,\as,\bs,\ws)$ is a multi-pointed Heegaard diagram for $(Y,\ws)$. If $\lambda$ is a path in $Y$ from $w_1$ to $w_2$, we can homotope $\lambda$ so that it has image in $\Sigma$ and is an immersion. Furthermore, we can assume that $\lambda$ intersects the $\as$ and $\bs$ curves transversely, and is disjoint from all intersections $\alpha_i\cap \beta_j$.  Let $a_1,\dots, a_m$ denote the points of intersection between $\lambda$ and the $\ve{\alpha}$ curves. Let $b_1,\dots, b_k$ denote the points of intersection between $\lambda$ and the $\bs$ curves. Given a homology class of disks $\phi\in \pi_2(\xs,\ys)$,  we let $d_{i}^{\alpha,\lambda}(\phi)$ denote the difference between the multiplicities of $\phi$ on the two sides of the point $a_i$. Similarly we let $d_{i}^{\b,\lambda}(\phi)$ denote the difference between the multiplicities of $\phi$ on the two sides of $b_i$. We define
\begin{equation}
a(\lambda,\phi)=\sum_{i=1}^m d_{i}^{\a,\lambda}(\phi) \quad \text{and} \quad b(\lambda,\phi)=\sum_{i=1}^k d_{i}^{\b,\lambda}(\phi),\label{eq:defa,blambda}
\end{equation}
which we view as elements of $\bF_2$. An orientation of $\lambda$ allows us to lift the quantities $a(\lambda,\phi)$ and $b(\lambda,\phi)$ to $\Z$, which will occasionally be useful.

We define the endomorphism $A_\lambda$ via the formula
\begin{equation}
A_{\lambda}(\ve{x})=\sum_{\ve{y}\in \bT_\alpha\cap \bT_\beta} \sum_{\substack{\phi\in \pi_2(\ve{x},\ve{y})\\ \mu(\phi)=1}}a(\lambda, \phi) \# \hat{\cM}(\phi)U_{\ve{w}}^{n_{\ve{w}}(\phi)}\cdot \ve{y}.\label{eq:Alambdadef}
\end{equation}
The map $B_{\lambda}$ similarly to Equation~\eqref{eq:Alambdadef}, by replacing the factor of $a(\lambda,\phi)$ with $b(\lambda,\phi)$.

If $\gamma$ is a closed curve in $Y$, we can similarly modify Equation~\eqref{eq:Alambdadef} to define homology actions $A_\gamma$ and $B_\gamma$ on $\CF^-(Y,\ws,\frs)$.

\subsection{Properties of the relative homology action}

In this section, we prove some basic properties of the relative homology action.

\begin{lem}\label{lem:Alambdachainhomotopy} Suppose $\cH=(\Sigma,\as,\bs,\ws)$ is a diagram for $(Y,\ws)$, and $\lambda$ is an immersed path in $\Sigma$ which connects two basepoints $w_1$ and $w_2$. Then
\[
A_{\lambda} \d+\d A_{\lambda}=B_{\lambda} \d+\d B_{\lambda}=U_{w_1}+U_{w_2}.
\]
If $\gamma$ is a closed, immersed curve on $\Sigma$, then
\[
A_\gamma \d+\d A_\gamma=B_\gamma \d+\d B_\gamma =0.
\]
\end{lem}

\begin{proof}We focus on the claim about the  map $A_\lambda$, when $\lambda$ connects two basepoints $w_1$ and $w_2$. The other claims are straightforward modifications.

Our argument proceeds by counting the ends of index 2 moduli spaces. Suppose that $\phi\in \pi_2(\ve{x},\ve{z})$ is a homology class of disks with $\mu(\phi)=2$. The 1-dimensional space $\hat{\cM}(\phi):=\cM(\phi)/\R$ admits a compactification $\bar{\hat{\cM}(\phi)}$, whose ends correspond to strip breaking and Maslov index 2 boundary degenerations. Since compact 1-manifolds have an even number of ends, we have $\# (\d \bar{\hat{\cM}(\phi)})=0$, and hence
\begin{equation}
a(\lambda, \phi)\cdot \# \left(\d\bar{\hat{\cM}(\phi)}\right)=0.\label{eq:endsofmodulispace}
\end{equation}

If $\phi_1\in \pi_2(\xs,\ys)$ and $\phi_2\in \pi_2(\ys,\zs)$ are two homology classes, then
\begin{equation}
a(\lambda,\phi_2*\phi_1)=a(\lambda,\phi_2)+a(\lambda,\phi_1).\label{eq:alambdaadditive}
\end{equation}

We consider separately the cases that $\ve{x}\neq \ve{z}$ or $\ve{x}=\ve{z}$.
Consider first the case that $\ve{x}\neq \ve{z}$. We write $\pi_2^{\a}(\xs)$ and $\pi_2^{\b}(\xs)$ for the groups of homology classes of $\as$ and $\bs$ degenerations on $(\Sigma,\as, \bs,\ws)$. We can view $\pi_2^{\a}(\xs)$ (resp. $\pi_2^{\b}(\xs)$) as the set of integral 2-chains on $\Sigma$ with boundary equal to a linear combination of the $\as$ curves (resp. $\bs$ curves).

If $w\in \ws$ and $\xs\in \bT_{\a}\cap \bT_{\b}$, then there is a unique class $A_w^{\xs}\in \pi_2^{\a}(\xs)$ whose domain has multiplicity 1 at $w$, and 0 on the other components of $\Sigma\setminus \as$. A class $B_w^{\xs}\in \pi_2^{\b}(\xs)$ is similarly specified. Furthermore, any homology class $A^{\xs}\in \pi_2^{\a}(\xs)$ decomposes as
\[
A^{\xs}=\sum_{w\in \ws} n_w(A) \cdot A_{w}^{\xs}.
\]
Using Lipshitz's formula for the Maslov index \cite{LipshitzCylindrical}*{Equation~8}, we compute 
\[
\mu(A_w^{\xs})=\mu(B_w^{\xs})=2.
\]
 Consequently, if $A^{\xs}\in \pi_2^\a(\xs)$ and $B^{\xs}\in \pi_2^{\b}(\xs)$, then
\begin{equation}
\mu(A^{\xs})=2\sum_{w\in \ws} n_w(A^{\xs})\quad \text{and} \quad \mu(B^{\xs})=2\sum_{w\in \ws} n_w(B^{\xs}). \label{eq:indexbounddegen}
\end{equation}

It follows that if $\phi$ is a Maslov index index 2 class and a broken holomorphic curve appears in $\d \bar{\hat{\cM}(\phi)}$ which contains a non-trivial boundary degeneration, then the remaining curves must have Maslov index 0, and hence must represent the constant class, by transversality. Hence, if $\ve{x}\neq \ve{z}$ then boundary degenerations cannot occur in the ends of $\hat{\cM}(\phi)$, so Equation~\eqref{eq:endsofmodulispace} implies that for each index 2 class $\phi\in \pi_2(\xs,\zs)$, 
\begin{align*}
0&=a(\lambda, \phi)\sum_{\substack{\ve{y}\in \bT_\alpha\cap \bT_\beta\\
\phi_1\in \pi_2(\ve{x},\ve{y}),\phi_2\in \pi_2(\ve{y},\ve{z})\\
\mu(\phi_1)=\mu(\phi_2)=1\\
\phi_1*\phi_2=\phi}} \# \hat{\cM}(\phi_1)\#\hat{\cM}(\phi_2) U_{\ve{w}}^{n_{\ve{w}}(\phi_1)+n_{\ve{w}}(\phi_2)}\cdot \ve{z}\\
&=\sum_{\substack{\ve{y}\in \bT_\alpha\cap \bT_\beta\\
\phi_1\in \pi_2(\ve{x},\ve{y}),\phi_2\in \pi_2(\ve{y},\ve{z})\\
\mu(\phi_1)=\mu(\phi_2)=1\\
\phi_1*\phi_2=\phi}} (a(\lambda, \phi_1)+a(\lambda,\phi_2)) \# \hat{\cM}(\phi_1)\#\hat{\cM}(\phi_2) U_{\ve{w}}^{n_{\ve{w}}(\phi_1)+n_{\ve{w}}(\phi_2)}\cdot \ve{z}.
\end{align*}

 Summing over all $\phi\in \pi_2(\ve{x},\ve{z})$ with $\mu(\phi)=2$, we get that the $\ve{z}$ component of $(A_{\lambda}\d+\d A_{\lambda})(\ve{x})$ is zero.

If $\ve{x}=\ve{z}$, there may be  ends of $\bar{\hat{\cM}(\phi)}$ corresponding to boundary degenerations. If $\phi$ is one of the classes $A_w^{\xs}$, we write $\hat{\cN}^{\a}(\phi)$ for the moduli space of cylindrical $\as$-boundary degenerations
\[
u\colon S\to \Sigma\times [0,\infty)\times \R
\]
 representing $\phi$, modulo conformal automorphisms of $[0,\infty)\times \R$.

Hence, the $\xs$ component of $(\d A_\lambda+A_\lambda \d)(\xs)$ is equal to 
\[
\sum_{w\in \ws} a(\lambda,A_{w}^{\xs})\cdot  \# \hat{\cN}^{\alpha}(A_w^{\xs})U_w\cdot \ve{x}+\sum_{w\in \ws} a(\lambda,B_w^{\xs})\cdot\# \hat{\cN}^{\beta}(B_w^{\xs}) U_w\cdot \ve{x}.
\]
We note that $a(\lambda,A_w^{\xs})=0$ unless $w\in \{w_1,w_2\}$, in which case 
\[
a(\lambda,A_{w_1}^{\xs})=a(\lambda,A_{w_2}^{\xs})=1.
\]
Furthermore, $a(\lambda, B_w^{\xs})=0$ for all $w\in \ws$. The counts of the moduli spaces of boundary degenerations were proven by Ozsv\'{a}th and Szab\'{o} \cite{OSLinks}*{Theorem~5.5}: for generic almost complex structure,
\begin{equation}
\# \hat{\cN}(A_w^{\xs})\equiv \begin{cases}1 \pmod{2}& \text{ if } |\ws|>1,\\
0& \text{ if } |\ws|=1. \label{eq:count-boundary-degenerations}
\end{cases}
\end{equation}

 Hence, the $\xs$-component of
\[
(A_{\lambda}\d+\d A_{\lambda}+U_{w_1}+U_{w_2})(\ve{x})
\]
 is zero. Combining this with the computation in the case that $\ve{x}\neq \ve{z}$, the proof is complete.
\end{proof}

We now consider the interaction between the holomorphic triangle maps and the relative homology maps. If $(\Sigma,\as,\bs,\gs,\ws)$ is a multi-pointed Heegaard triple, and $\lambda$ is an immersed path between two basepoints $w_1,w_2\in \ws$, as an extension of Equation~\eqref{eq:defa,blambda}  let $a(\lambda,\phi)$, $b(\lambda,\phi)$ and $c(\lambda,\phi)$ denote the sums of differences of the multiplicities of $\phi$ across the $\as$, $\bs$, or $\gs$ curves, respectively. Let $A_\lambda$, $B_{\lambda}$ and $C_{\lambda}$ denote the endomorphisms of $\CF^-(\Sigma,\as,\bs,\ws)$, $\CF^-(\Sigma,\bs,\gs,\ws)$ and $\CF^-(\Sigma,\as,\gs,\ws)$, defined by modifying Equation~\eqref{eq:Alambdadef}. Note that all three endomorphisms $A_\lambda,$ $B_\lambda$ and $C_\lambda$ are defined on all three complexes.

\begin{lem}\label{lem:relhomtrianglegeneral} Suppose that $(\Sigma,\as,\bs,\gs,\ws)$ is a multi-pointed Heegaard triple and $\frs\in \Spin^c(X_{\a,\b,\g})$. Then 
\begin{align*}
F_{\a,\b,\g,\frs} (A_\lambda\otimes \id)&\simeq A_\lambda\circ F_{\a,\b,\g,\frs}(\id\otimes \id)\\
 F_{\a,\b,\g,\frs}(B_\lambda\otimes \id)&\simeq F_{\a,\b,\g,\frs}(\id\otimes B_\lambda)\\
 F_{\a,\b,\g,\frs}(\id \otimes C_\lambda)&\simeq C_\lambda\circ  F_{\a,\b,\g,\frs}(\id\otimes \id),
\end{align*}
as maps from $\CF^-(\Sigma,\as,\bs,\ws,\frs|_{Y_{\a,\b}})\otimes_{\bF_2[U_{\ws}]} \CF^-(\Sigma,\bs,\gs,\ws,\frs|_{Y_{\b,\g}})$ to $\CF^-(\Sigma,\as,\gs,\ws,\frs|_{Y_{\a,\g}})$.
\end{lem}
\begin{proof} Consider the first relation, involving $A_\lambda$. The subsequent two relations involving $B_\lambda$ and $C_{\lambda}$ can be proven \emph{mutatis mutandis}. We prove the relation by counting the ends of index 1 moduli spaces of triangles. Suppose that $\psi\in \pi_2(\xs,\ys,\zs)$ is a homology class with $\frs_{\ws}(\psi)=\frs$, with $\mu(\psi)=1$. The moduli space $\cM(\psi)$ can be compactified into a compact 1-manifold $\bar{\cM(\psi)}$ whose ends consist  of pairs consisting of an index 1 holomorphic strip, and an index 0 holomorphic triangle. Since compact 1-manifolds have an even number of ends, we have
\[
\sum_{\substack{\psi\in \pi_2(\xs,\ys,\zs)\\ \mu(\psi)=1\\
\frs_{\ws}(\frs)=\frs}} a(\lambda,\psi) \# (\d \bar{\cM(\psi)})\cdot U_{\ws}^{n_{\ws}(\psi)}=0. 
\]
If $\phi$ is a homology class of disks and $\psi$ is a homology class of triangles, then similar to Equation~\eqref{eq:alambdaadditive}, 
\[
a(\lambda,\psi*\phi)=a(\lambda,\psi)+a(\lambda,\phi).
\]
It follows that 
\begin{equation}
\begin{split}
&A_\lambda\circ F_{\a,\b,\g,\frs}(\id\otimes \id)+F_{\a,\b,\g,\frs}(A_\lambda\otimes \id)+F_{\a,\b,\g,\frs}(\id \otimes A_\lambda)\\
=& H^{A,\lambda}_{\a,\b,\g,\frs}\circ (\d_{\a,\b}\otimes \id+\id\otimes \d_{\b,\g})+\d_{\a,\g}\circ H^{A,\lambda}_{\a,\b,\g,\frs},
\end{split}
\label{eq:largehomotopyinvolvingAlambda}
\end{equation}
where $H^{A,\lambda}_{\a,\b,\g,\frs}$ is the map defined on intersection points by the formula
\[
 H^{A,\lambda}_{\a,\b,\g,\frs}(\ve{x}\otimes \ve{y})=\sum_{\substack{\ve{z}\in \bT_\alpha\cap \bT_{\g}\\ \psi\in \pi_2(\ve{x},\ve{y},\ve{z})\\
\mu(\psi)=0}}a(\lambda,\psi) \# \cM(\psi)U_{\ve{w}}^{n_{\ve{w}}(\psi)}\cdot\ve{z},
\]
and extended $\bF_2[U_{\ws}]$-equivariantly.

Finally, we note that if $\phi$ is a homology class of disks on the diagram $(\Sigma,\bs,\gs,\ws)$, then the quantities $a(\lambda,\phi)$ vanish, since $\phi$ has no changes across the $\as$ curves. Hence the map
\[
A_\lambda\colon \CF^-(\Sigma,\bs,\gs,\ws,\frs|_{\b\g})\to \CF^-(\Sigma,\bs,\gs,\ws,\frs|_{\b\g})
\]
vanishes. Combining this fact with Equation~\eqref{eq:largehomotopyinvolvingAlambda}, the main statement follows.
\end{proof}

\begin{lem}\label{lem:splicingrelhom}
\begin{enumerate}
\item Suppose that $\lambda$ is a path on $\Sigma$ connecting  a pair of basepoints on $\cH$ which can be written as a concatenation $\lambda_2* \lambda_1$ of two paths which connect pairs of basepoints. Then
\[
A_{\lambda_{2}* \lambda_{1}}=A_{\lambda_{2}}+A_{\lambda_{1}}.
\]
\item Suppose that $\lambda$ is a path on $\Sigma$ connecting a pair of basepoints, and $\gamma$ is a closed loop on $\Sigma$, which has non-trivial intersection with $\gamma$. If $\lambda*\gamma$ denotes the path obtained by splicing $\gamma$ into $\lambda$, then
\[
A_{\gamma*\lambda}=A_{\gamma}+A_{\lambda}.
\]
\item If $\gamma$ and $\gamma'$ are two closed curves on $\Sigma$, then
\[
A_{\gamma'*\gamma}=A_{\gamma'}+A_{\gamma}.
\]
\end{enumerate}
The same relations hold for the type-$B$ relative homology maps.

\end{lem}
\begin{proof} The first claim follows immediately from Equation~\eqref{eq:Alambdadef} since 
\[
a(\lambda_2*\lambda_1,\phi)=a(\lambda_{2},\phi)+a(\lambda_{1},\phi),
\]
for any homology class of disks $\phi$. The second and third claims are proven similarly.
\end{proof}
%
%

We now compute the commutator of the relative homology maps:

\begin{lem}\label{lem:relhomcommutator}Suppose $\lambda_1$ and $\lambda_2$ are two paths connecting pairs of basepoints in $\ws$. Then
\[
A_{\lambda_1}A_{\lambda_2}+A_{\lambda_2}A_{\lambda_1}=\sum_{w\in \d\lambda_1\cap \d\lambda_2} U_{w}.
\]
\end{lem}

\begin{proof}Our proof proceeds by counting the ends of $\bar{\hat{\cM}(\phi)}$ for classes $\phi\in \pi_2(\ve{x},\ve{z})$ with Maslov index 2. As in Lemma~\ref{lem:Alambdachainhomotopy}, there are two cases to consider:  $\ve{x}\neq \ve{z}$ and $\ve{x}=\ve{z}$.

If $\ve{x}\neq \ve{z}$ and $\phi\in \pi_2(\xs,\zs)$ has Maslov index 2, then the ends of $\bar{\hat{\cM}(\phi)}$ all correspond to strip breaking. Summing over $\phi$, we have
\begin{equation}
0=\sum_{\substack{\phi\in \pi_2(\ve{x},\ve{z})\\
\mu(\phi)=2}}a(\lambda_1,\phi)a(\lambda_2,\phi)\sum_{\substack{\ve{y}\in \bT_\alpha\cap \bT_\beta\\
\phi_1\in \pi_2(\ve{x},\ve{y}),\phi_2\in \pi_2(\ve{y},\ve{z})\\
\mu(\phi_1)=\mu(\phi_2)=1\\
\phi_1+\phi_2=\phi}} \# \hat{\cM}(\phi_1)\#\hat{\cM}(\phi_2) U_{\ve{w}}^{n_{\ve{w}}(\phi)}\cdot\ve{z}. \label{eq:relhomcom1}
\end{equation}
 Noting that 
\[
a(\lambda_1,\phi_1+\phi_2)a(\lambda_2,\phi_1+\phi_2)
\]
\[
=\bigg(a(\lambda_1,\phi_1)a(\lambda_2,\phi_2)+a(\lambda_1,\phi_2)a(\lambda_2,\phi_1)\bigg)+\bigg(a(\lambda_1,\phi_1)a(\lambda_2,\phi_1)+a(\lambda_1,\phi_2)a(\lambda_2,\phi_2)\bigg),
\]
 we can rewrite Equation~\eqref{eq:relhomcom1} to obtain
 \begin{equation}
\begin{split}
0=&\sum_{\substack{\ve{y}\in \bT_\alpha\cap \bT_\beta\\
\phi_1\in \pi_2(\ve{x},\ve{y}),\phi_2\in \pi_2(\ve{y},\ve{z})\\
\mu(\phi_1)=\mu(\phi_2)=1}}(a(\lambda_1,\phi_1)a(\lambda_2,\phi_2)+a(\lambda_1,\phi_2)a(\lambda_2,\phi_1)) \# \hat{\cM}(\phi_1)\#\hat{\cM}(\phi_2) U_{\ve{w}}^{n_{\ve{w}}(\phi_1)+n_{\ve{w}}(\phi_2)}\cdot\ve{z}
\\
+&\sum_{\substack{\ve{y}\in \bT_\alpha\cap \bT_\beta\\
\phi_1\in \pi_2(\ve{x},\ve{y}),\phi_2\in \pi_2(\ve{y},\ve{z})\\
\mu(\phi_1)=\mu(\phi_2)=1}}(a(\lambda_1,\phi_1)a(\lambda_2,\phi_1)+a(\lambda_1,\phi_2)a(\lambda_2,\phi_2))\# \hat{\cM}(\phi_1)\#\hat{\cM}(\phi_2) U_{\ve{w}}^{n_{\ve{w}}(\phi_1)+n_{\ve{w}}(\phi_2)}\cdot\ve{z}.
\end{split}
\label{eq:relhomcom2}
\end{equation}
The right side of Equation~\eqref{eq:relhomcom2} is the $\ve{z}$ component of 
\[
(A_{\lambda_1}A_{\lambda_2}+A_{\lambda_2}A_{\lambda_1}+\d H_{\lambda_1\lambda_2}+H_{\lambda_1\lambda_2}\d)(\ve{x}),
\]
 where $H_{\lambda_1\lambda_2}$ is the map
\[
H_{\lambda_1\lambda_2}(\ve{x})=\sum_{\substack{\ve{y}\in \bT_\alpha\cap \bT_\beta\\ \phi\in \pi_2(\ve{x},\ve{y})\\
\mu(\phi)=1}} a(\lambda_1,\phi)a(\lambda_2,\phi) \# \hat{\cM}(\phi) U_{\ve{w}}^{n_{\ve{w}}(\phi)}\cdot \ve{y}.
\]

If $\ve{x}=\ve{z}$, then boundary degenerations may appear in the ends of $\bar{\hat{\cM}(\phi)}$. These make an additional contribution to Equation~\eqref{eq:relhomcom1} of
\begin{equation}
\sum_{w\in \ws} a(\lambda_1,A_w^{\xs})a(\lambda_2,A_w^{\xs}) \# \hat{\cN}^{\alpha}(A_w^{\xs})U_{w}\cdot \ve{x}+\sum_{w\in \ws} a(\lambda_1,B_w^{\xs})a(\lambda_2,B_{w}^{\xs})\# \hat{\cN}^{\beta}(B_{w}^{\xs})U_w\cdot \ve{x}. \label{eq:remaining-count-commutator-rel-hom}
\end{equation}
 We note that $a(\lambda,B_w^{\xs})=0$ for all $w$, and 
 \begin{equation}
 a(\lambda,A_w^{\xs})=\begin{cases}1& \text{ if } w\in \d \lambda,\\
 0&\text{ if } w\not\in \d \lambda. \label{eq:count-A-lambda-boundarydegen}
 \end{cases}.
 \end{equation}
 The stated formula now follows from Equations~\eqref{eq:remaining-count-commutator-rel-hom}, \eqref{eq:count-A-lambda-boundarydegen} and the count of boundary degenerations from Equation~\eqref{eq:count-boundary-degenerations}.
\end{proof}

In the context of a single basepoint in Heegaard Floer homology, the homology action squares to zero (see \cite{OSDisks}*{Proposition~4.17}, \cite{LipshitzCylindrical}*{Proposition~8.6}). We prove a similar result:

\begin{lem}\label{lem:Alambda-squares-to-zero-or-U} Suppose  $\lambda$ is an immersed path in $\Sigma$ from $w_1$ to $w_2$.  Then
\[
A_\lambda\circ A_\lambda\simeq U_{w_1}\simeq U_{w_2}.
\]
 If $\gamma$ is an immersed closed curve in $\Sigma$, then
\[
A_{\gamma}\circ A_\gamma\simeq 0.
\]

\end{lem}

\begin{proof} Note that $U_{w_1}\simeq U_{w_2}$ by Lemma \ref{lem:Alambdachainhomotopy}.  We focus on the claim that $A_\lambda\circ A_\lambda\simeq U_{w_1}$. The proof of the relation $A_\gamma\circ A_\gamma\simeq 0$ follows with only minor modification.   Pick an orientation of $\lambda$, which gives a lift of $a(\lambda,\phi)$ from $\bF_2$ to $\Z$.
 
 
 Note that both $a(\lambda,A_{w_1}^{\xs})$ and $a(\lambda,A_{w_2}^{\xs})$ are $\pm 1$, and in fact, the two quantities have opposite signs. Orient $\lambda$ so that 
 \begin{equation}
 a(\lambda,A_{w_1}^{\xs})=1\quad \text{and} \quad a(\lambda, A_{w_2}^{\xs})=-1.\label{eq:chosenorientationlambda}
 \end{equation}

Define the map
\[
H_\lambda(\ve{x})=\sum_{\substack{\ve{y}\in \bT_\alpha\cap \bT_\beta\\ \phi\in \pi_2(\ve{x},\ve{y})\\
\mu(\phi)=1}}\bigg( \frac{a(\lambda,\phi)(a(\lambda,\phi)+1)}{2}\bigg)\# \hat{\cM}(\phi)U_{\ve{w}}^{n_{\ve{w}}(\phi)} \cdot \ve{y}.
\]
 We will show that
\begin{equation}
A_\lambda^2(\ve{x})=(\d H_\lambda+H_\lambda\d+U_{w_1})(\ve{x}). \label{eq:Al^2=chainhomotopy}
\end{equation}

%

If $\zs\in \bT_\a\cap \bT_\b$, the $\ve{z}$ component of $(\d H_\lambda+H_\lambda \d+A_\lambda^2)(\ve{x})$ is
\begin{equation}
\sum_{\substack{\ve{y}\in \bT_\alpha\cap \bT_\beta\\ \phi_1\in \pi_2(\ve{x},\ve{y}),\phi_2\in \pi_2(\ve{y},\ve{z})\\
\mu(\phi_1)=\mu(\phi_2)=1}}\bigg(\frac{a(\lambda,\phi_1)(a(\lambda,\phi_1)+1)}{2}+\frac{a(\lambda,\phi_2)(a(\lambda,\phi_2)+1)}{2}\bigg) \# \hat{\cM}(\phi_1)\#\hat{\cM}(\phi_2)U_{\ve{w}}^{n_{\ve{w}}(\phi_1+\phi_2)}.
\label{eq:Alambdasquared2}
\end{equation}
 Rearranging, Equation~\eqref{eq:Alambdasquared2} becomes
\begin{equation}
\begin{split}
&\sum_{\substack{\ve{y}\in \bT_\alpha\cap \bT_\beta\\ \phi_1\in \pi_2(\ve{x},\ve{y}),\phi_2\in \pi_2(\ve{y},\ve{z})\\
\mu(\phi_1)=\mu(\phi_2)=1}}\bigg(\frac{a(\lambda,\phi_1+\phi_2)(a(\lambda,\phi_1+\phi_2)+1)}{2}\bigg) \# \hat{\cM}(\phi_1)\#\hat{\cM}(\phi_2)U_{\ve{w}}^{n_{\ve{w}}(\phi_1+\phi_2)}\\
+&\sum_{\substack{\ve{y}\in \bT_\alpha\cap \bT_\beta\\ \phi_1\in \pi_2(\ve{x},\ve{y}),\phi_2\in \pi_2(\ve{y},\ve{z})\\
\mu(\phi_1)=\mu(\phi_2)=1}}a(\lambda,\phi_1)a(\lambda,\phi_2) \# \hat{\cM}(\phi_1)\#\hat{\cM}(\phi_2)U_{\ve{w}}^{n_{\ve{w}}(\phi_1+\phi_2)}.
\end{split}
\label{eq:Alambdasquared1}
\end{equation}
The second summand of Equation~\eqref{eq:Alambdasquared1} the $\ve{z}$ coefficient of $(A_{\lambda}\circ A_{\lambda})(\ve{x})$. If $\xs\neq \zs$, then the first summand of Equation~\eqref{eq:Alambdasquared1} is equal to 
%
%
 
\begin{equation}
\sum_{\substack{\phi\in \pi_2(\xs,\zs)\\ \mu(\phi)=2}}\frac{a(\lambda,\phi)(a(\lambda,\phi)+1)}{2}\# \d \bar{\hat{\cM}(\phi)}U_{\ve{w}}^{n_{\ve{w}}(\phi)}=0.\label{eq:computecommutatorweightedsumboundary}
\end{equation}
Hence the $\ve{z}$ component of $(H_\lambda \d+d H_\lambda+A_\lambda^2)(\ve{x})$ is zero, when $\xs\neq \zs$.

We now consider the case that $\ve{z}=\ve{x}$. The expression in Equation~\eqref{eq:computecommutatorweightedsumboundary}  is equal to $(\d H_\lambda+H_\lambda \d+A_\lambda^2)(\ve{x})$ plus the following contribution due to boundary degenerations: 
\[
\sum_{w\in \ws} \frac{a(\lambda,A_{w}^{\xs})(a(\lambda,A_{w}^{\xs})+1)}{2} \# \hat{\cN}^{\alpha}(A_{w}^{\xs})U_w \cdot \ve{x}+\sum_{w\in \ws} \frac{a(\lambda,B_{w}^{\xs})(a(\lambda,B_{w}^{\xs})+1)}{2}\# \hat{\cN}^{\beta}(B_{w}^{\xs})U_w\cdot \ve{x}.
\]
Finally, we note that $a(\lambda,B_w^{\xs})=0$ for all $w\in \ws$, while $a(\lambda,A_{w}^{\xs})=0$ unless $w\in \{w_1,w_2\}$. Furthermore, from Equation~\eqref{eq:chosenorientationlambda} it follows that
\[
\frac{a(\lambda,A_{w_1}^{\xs})(a(\lambda,A_{w_1}^{\xs})+1)}{2}=1\quad \text{and} \quad \frac{a(\lambda,A_{w_2}^{\xs})(a(\lambda,A_{w_2}^{\xs})+1)}{2}=0.
\]
It follows that the $\xs$ component of $(A_\lambda^2+\d H_\lambda+H_\lambda\d)(\xs)$ is $U_{w_1}\cdot\xs$, completing the proof.

\end{proof}

Since the Heegaard surface $\Sigma$ is embedded in $Y$, there is a natural map $i_*\colon H_1(\Sigma;\Z)\to H_1(Y;\Z)$. Since $Y$ can be built by attaching 2-handles and 3-handles to $[0,1]\times \Sigma$, it follows that $i_*$ is surjective, and $\ker (i_*)=\Span\{[\alpha_1],\dots, [\alpha_n], [\beta_1],\dots, [\beta_n]\}$. Hence
\begin{equation}
H_1(Y;\Z)\iso \frac{H_1(\Sigma;\Z)}{\Span\{[\alpha_1],\dots, [\alpha_n], [\beta_1],\dots, [\beta_n]\}}. \label{eq:H1YquotientH1Sigma}
\end{equation}

\begin{lem}\label{lem:torsion} If $\gamma$ is a closed loop on $\Sigma$ such that 
\[
i_*([\gamma])=0 \in H_1(Y;\Z)/\Tors,
\]
 then
\[
A_\gamma\simeq 0.
\]

\end{lem}

\begin{proof}Our proof is a modification Ni's proof \cite{Ni-homological}*{Lemma~2.4} of a closely related result. Suppose that $\gamma$ is an integral 1-cycle on $\Sigma$ such that $i_*(k\cdot \gamma)=0$, for some integer $k\neq 0$.

 By orienting $\gamma$, we obtain a lift of the quantities $a(\gamma,\phi)$ to $\Z$. We can assume that $\gamma$ is immersed, intersects $\as$ and $\bs$ transversely, and is disjoint from any intersections of the $\as$ and $\bs$ curves. From the isomorphism in Equation~\eqref{eq:H1YquotientH1Sigma}, it follows that the class $k\cdot [\gamma]\in H_1(\Sigma;\Z)$ can be written as an integral combination of the $\as$ and $\bs$ curves. 
 
 Consequently, there is an integral 2-chain $S$ on $\Sigma$, such that
\[
\d S=k\cdot \g+  C,
\]
where $C$ is an integral 1-cycle on $\Sigma$ which consists of an integral combination of small pushoffs of the $\as$ and $\bs$ curves on $\Sigma$.

Let $n_S\colon \bT_{\a}\cap \bT_{\b}\to \Z$ denote the function
\[
n_{S}(\xs)=\left(\sum_{x\in \xs} n_x(S)\right).
\]
We claim that if $\phi\in \pi_2(\xs,\ys)$, then
\begin{equation}
-a(k\cdot \gamma+ C, \phi)=n_S(\ys)-n_S(\xs).\label{eq:relforchainhom}
\end{equation}
To establish Equation~\eqref{eq:relforchainhom}, we note that the homology class $\phi$ determines a collection of paths $g+|\ws|-1$ arcs $\ve{a}(\phi)$, which are contained in the $\as$ curves and run from the points of $\xs$ to the points of $\ys$.  The quantity $a(k\cdot \gamma+C,\phi)$ can be reinterpreted as the oriented intersection number 
\[
\#( \ve{a}(\phi)\cap (k\cdot \gamma+C)).
\]
On the other hand $\ve{a}(\phi)\cap S$ is a compact 1-manifold, and hence the algebraic count of its ends is zero. By the Leibniz rule
\[
0=\# \d (\ve{a}(\phi)\cap S)=\#(\d\ve{a}(\phi)\cap S)+ \#(\ve{a}(\phi)\cap \d S),
\]
from which Equation~\eqref{eq:relforchainhom} follows. Using additivity of $(\gamma,\phi)$ with respect to $\gamma$, from Equation~\eqref{eq:relforchainhom} we obtain
\begin{equation}
-k\cdot  a(\gamma,\phi)-a(C,\phi)=n_S(\ys)-n_S(\xs).\label{eq:expandoutintnum}
\end{equation}

If $\a'$ is a small pushoff of an $\as$ curve, then $a(\a',\phi)=0$, since there are no changes of $\phi$ across any $\as$ curves as one traverses $\a'$. Similarly if $\b'$ is a small pushoff of a $\bs$ curve, then $a(\b',\phi)=0$, since $\b'$ intersects no $\bs$ curves, so the sum of differences of $\phi$ across the $\as$ curves as one traverses $\b'$ telescopes, and is zero. With this observation, Equation~\eqref{eq:expandoutintnum} now reads
\begin{equation}
-k\cdot a(\gamma,\phi)=n_S(\ys)-n_S(\xs).\label{eq:gotridofCs}
\end{equation}
The right hand side of Equation~\eqref{eq:gotridofCs} does not depend on $\phi$. It follows that modulo $k$, the expression $n_S(\xs)$ is independent of $\ve{x}$, for $\xs$ representing a fixed $\Spin^c$ structure. By adding copies of $[\Sigma]$ to $S$, we can ensure that $n_S(\xs)$ is divisible by $k$ for each $\xs\in \bT_{\a}\cap \bT_{\b}$ representing $\frs$. Hence
\begin{equation}
-a(\gamma,\phi)=\frac{n_S(\ys)}{k}-\frac{n_S(\xs)}{k}.\label{eq:dividebyn}
\end{equation}
We define the map
\[
H_S\colon \CF^-(\Sigma,\as,\bs,\ws,\frs)\to \CF^-(\Sigma,\as,\bs,\ws,\frs)
\]
via the formula
\[
H_S(\xs)=\frac{n_{S}(\xs)}{k}\cdot \xs,
\]
extended equivariantly over $\bF_2[U_{\ws}]$.

After projecting to $\bF_2$, Equation~\eqref{eq:dividebyn} implies  that
\[
A_\gamma(\xs)=(\d H_S+H_S\d)(\xs)
\]
\end{proof}

We now describe a simple relation between the maps $A_\lambda$ and $B_{\lambda}$. To describe their relation, we must introduce a new map. If $w\in \ws$ is a basepoint, we define $\Phi_{w}\colon \CF^-(\Sigma,\as,\bs,\ws,\frs)\to \CF^-(\Sigma,\as,\bs,\ws,\frs)$ via the formula
\begin{equation}
\Phi_w(\xs):=U_w^{-1}\cdot \sum_{\ys\in \bT_{\a}\cap \bT_{\b}} \sum_{\substack{\phi\in \pi_2(\xs,\ys)\\ \mu(\phi)=1}} n_{w}(\phi) \# \hat{\cM}(\phi) U_{\ws}^{n_{\ws}(\phi)} \cdot \ys, \label{eq:defPhi}
\end{equation}
extended $\bF_2[U_{\ws}]$-equivariantly. Despite the initial factor of $U_w^{-1}$, the map $\Phi_w$ maps $\CF^-$ into $\CF^-$. Furthermore,  $\Phi_w$ is a chain map; see Lemma~\ref{lem:claim:phi0}, below.

\begin{lem}\label{lem:AgBgrel}Suppose that $(\Sigma,\as,\bs,\ws)$ is a diagram for $(Y,\ws)$, $\lambda$ is an immersed curve on $\Sigma$ with endpoints $w_1$ and $w_2$, and $\gamma$ is an immersed closed curve on $\Sigma$.
 Then
 \begin{enumerate}
 \item \label{eq:relAgBg1}$A_{\gamma}=B_{\gamma}$, and
 \item \label{eq:relAgBg2} $A_{\lambda}+B_{\lambda}=U_{w_1}\Phi_{w_1}+U_{w_2}\Phi_{w_2}.$
 \end{enumerate}
\end{lem}
\begin{proof} The map $A_\gamma$ counts holomorphic strips  weighted by the factor $a(\gamma,\phi)$, while $B_{\gamma}$ counts holomorphic disks weighted by $b(\gamma,\phi)$. The sum $a(\gamma,\phi)+b(\gamma,\phi)$ is  the total change in multiplicity of $\phi$ across all curves (either $\as$ or $\bs$), however this is zero since $\gamma$ is a closed curve. Part~\eqref{eq:relAgBg1} follows.

For an arc $\lambda$ connecting basepoints $w_1$ and $w_2$, we instead have
\[
a(\lambda,\phi)+b(\lambda,\phi)=n_{w_1}(\phi)-n_{w_2}(\phi),
\]
from which Part~\eqref{eq:relAgBg2} follows.
\end{proof}

\subsection{Naturality of the relative homology action}

 In this section, we prove the following:
 
\begin{prop}\label{prop:naturalityrelativehomology} Suppose that $\cH=(\Sigma,\as,\bs,\ws)$ is a diagram for $(Y,\ws)$ and $\lambda$ and $\lambda'$ are two immersed paths in $\Sigma$ from $w_1$ to $w_2$ which represent the same element in $H_1(Y, \{w_1,w_2\};\Z)/\Tors$. Then
\[
A_\lambda\simeq A_{\lambda'}.
\] 
Furthermore, if $\cH$ and $\cH'$ are two diagrams for $(Y,\ws)$, then
\[
A_\lambda \circ \Psi_{\cH\to \cH'}\simeq \Psi_{\cH\to \cH'}\circ A_{\lambda}.
\]
The same holds for the map $B_{\lambda}$, as well as the homology action associated to closed loops in $Y$.
\end{prop}

As a first step towards Proposition~\ref{prop:naturalityrelativehomology}, we prove that the maps $A_{\lambda}$ commute with the change of almost complex structure maps:

\begin{lem}\label{lem:relhomcommutechangeacstr} Suppose $J$ and $J'$ are two cylindrical almost complex structures on $\Sigma\times [0,1]\times \R$ which satisfy axioms \eqref{def:J1}--\eqref{def:J5}. Then
\[
A_{\lambda}\circ \Psi_{J\to J'}\simeq \Psi_{J\to J'}\circ  A_{\lambda},
\]
 where $\Psi_{J\to J'}$ is the change of almost complex structures map. The same relation holds for the type-$B$ maps.
\end{lem}
\begin{proof} To compute $\Psi_{J\to J'}$ one first picks a non-cylindrical almost complex structure $\tilde{J}$ on $\Sigma\times [0,1]\times \R$ which agrees with $J$ on $\Sigma\times [0,1]\times (-\infty,-1]$ and agrees with $J'$ on $\Sigma\times [0,1]\times [1,\infty)$. The map $\Psi_{J\to J'}$ is defined via the formula
\[
\Psi_{J\to J'}(\xs)=\sum_{\ys\in \bT_{\a}\cap \bT_{\b}} \sum_{\substack{\phi\in \pi_2(\xs,\ys)\\ \mu(\phi)=0}} \# \cM_{\tilde{J}}(\phi)U^{n_{\ws}(\phi)}_{\ws} \cdot \ys,
\]
extended linearly over $\bF_2[U_{\ws}]$. 

We define the map $H_{\tilde{J}, \lambda}^{A}\colon \CF^-(\Sigma,\as,\bs,\ws,\frs)\to \CF^-(\Sigma,\as,\bs,\ws,\frs)$ via the formula
\[
H_{\tilde{J},\lambda}^A(\xs):=\sum_{\ys\in \bT_{\a}\cap \bT_{\b}} \sum_{\substack{\phi\in \pi_2(\xs,\ys)\\ \mu(\phi)=0}} a(\lambda,\phi)\# \cM_{\tilde{J}}(\phi)U^{n_{\ws}(\phi)}_{\ws} \cdot \ys,
\]
extended linearly over $\bF_2[U_{\ws}]$.

If $\phi\in \pi_2(\xs,\ys)$ is a Maslov index 1 homology class, then the 1-dimensional moduli space $\cM_{\tilde{J}}(\phi)$ can be compactified into a compact 1-manifold, whose ends correspond to pairs of index 0 $\tilde{J}$-holomorphic curves, and index 1 $J$- or $J'$-holomorphic curves.

Using Equation~\eqref{eq:alambdaadditive} (additivity of the quantity $a(\lambda,\phi)$ with respect to $\phi$), it follows that
\[
A_{\lambda}\circ \Psi_{J\to J'}+\Psi_{J\to J'}\circ A_{\lambda}+\d_{J'}\circ  H_{\tilde{J},\lambda}^A+H_{\tilde{J},\lambda}^A\circ \d_{J}=0.
\]
 The relation for the type-$B$ homology actions is proved analogously.
\end{proof}

Towards proving that  $A_\lambda$  commutes with changes of the $\as$ curves, we prove the following:

\begin{lem}\label{lem:relativehomologycommutestriangleII} Suppose that $(\Sigma,\as',\as,\bs,\ws)$ is a Heegaard triple such that $(\Sigma,\as',\as,\ws)$ is a diagram for $(S^1\times S^2)^{\# k}$, for some $k$. Suppose that $\lambda$ is an immersed curve in $\Sigma$, which is endpoints on $w_1$ and $w_2$, or $\lambda$ is an immersed closed curve. Let $A'_\lambda$, $A_{\lambda}$ and $B_{\lambda}$ denote the relative homology maps defined by counting changes over the $\as'$, $\as$ or $\bs$ curves, respectively. Suppose that $\Theta_{\a',\a}^+\in \CF^-(\Sigma,\as',\as,\ws,\frs_0)$ is a cycle which represents the top degree generator of homology, and $\frs\in \Spin^c(X_{\a',\a,\b})$ restricts to $\frs_0$ on $Y_{\a',\a}$. Then
\begin{enumerate}
\item\label{eq:trianglesandrelhom2a} $F_{\a',\a,\b,\frs}\left(\Theta_{\a',\a}^+\otimes A_{\lambda}(-)\right)\simeq A'_{\lambda}\circ F_{\a',\a,\b,\frs}\left(\Theta_{\a',\a}^+\otimes -\right)$,
\item\label{eq:trianglesandrelhom2b} $ F_{\a',\a,\b,\frs}\left(\Theta_{\a',\a}^+\otimes B_{\lambda}(-)\right)\simeq B_{\lambda}\circ F_{\a',\a,\b,\frs}\left(\Theta_{\a',\a}^+\otimes -\right)$.
\end{enumerate}
An analogous statement holds for triples $(\Sigma,\as,\bs,\bs',\ws)$ where $(\Sigma,\bs,\bs',\ws)$ is a diagram for $(S^1\times S^2)^{\# k}$. 
\end{lem}
\begin{proof} We focus on the claim when $\lambda$ is an immersed path connecting $w_1$ and $w_2$. The claim when $\lambda$ is an immersed closed curve is a simple modification.

 Part~\eqref{eq:trianglesandrelhom2b} follows immediately from Lemma~\ref{lem:relhomtrianglegeneral}, though Part~\eqref{eq:trianglesandrelhom2a} does not follow from a symmetric argument. 
 
By using Lemmas~\ref{lem:relhomtrianglegeneral} and~\ref{lem:AgBgrel} we see 
\begin{equation}
\begin{split}
&F_{\a',\a,\b,\frs}\left(\Theta_{\a',\a}^+\otimes A_{\lambda}(-)\right)+A'_{\lambda}\circ F_{\a',\a,\b,\frs}\left(\Theta_{\a',\a}^+\otimes -\right)\\
\simeq &F_{\a',\a,\b,\frs}\left((A'_{\lambda}+A_{\lambda})(\Theta_{\a',\a}^+\right)\otimes -)\\
\simeq &F_{\a',\a,\b,\frs}\left((U_{w_1}\Phi_{w_1}+U_{w_2}\Phi_{w_2})(\Theta_{\a',\a}^+)\otimes -\right).
\label{eq:firstrelationfromtriangles}
\end{split}
\end{equation}
However, from Equation~\eqref{eq:defPhi} we see that the maps $\Phi_{w_i}$ are $+1$ graded chain maps. Since $[\Theta_{\a',\a}^+]$ is the highest graded non-zero element of $\HF^-(\Sigma,\as',\as,\ws,\frs_0)$, we must have 
\[
[\Phi_{w_1}(\Theta_{\a',\a}^+)]=[\Phi_{w_2}(\Theta_{\a',\a}^+)]=0\in \HF^-(\Sigma,\as',\as,\ws,\frs_0).
\]
The associativity relations for holomorphic triangles imply that the map $F_{\a',\a,\b,\frs}(\d \eta,-)$ is chain homotopic to the zero map, for any $\eta\in \CF^-(\Sigma,\as,\as',\ws,\frs_0)$. Hence
\begin{equation}
F_{\a',\a,\b,\frs}\left((U_{w_1}\Phi_{w_1}+U_{w_2}\Phi_{w_2})(\Theta_{\a',\a}^+)\otimes -\right)\simeq 0.\label{eq:triangleswithphis=0}
\end{equation}
Equations~\eqref{eq:firstrelationfromtriangles} and~\eqref{eq:triangleswithphis=0} imply Part~\eqref{eq:trianglesandrelhom2b}, completing the proof.
\end{proof}

\begin{cor}\label{cor:relhomcomma/bmove} Suppose $\Sigma$ is a Heegaard surface for $(Y,\ws)$, and $\lambda\subset \Sigma$ is either an immersed, closed curve, or an immersed path connecting two basepoints. If $\as$ and $\as'$ are attaching curves for the $\alpha$-handlebody, and $\bs$ and $\bs'$ are attaching curves for the $\beta$-handlebody, then
\begin{enumerate}
\item $\Psi_{\as\to \as'}^{\bs}\circ A_{\lambda} \simeq A_{\lambda} \circ \Psi_{\as\to \as'}^{\bs}$, and
\item $\Psi_{\as}^{\bs\to \bs'}\circ A_{\lambda}\simeq A_{\lambda} \circ \Psi_{\as}^{\bs\to \bs'}$.
\end{enumerate}
The same holds for the type-$B$ maps.
\end{cor}
\begin{proof} The transition maps $\Psi_{\as\to \as'}^{\bs}$ and $\Psi_{\as}^{\bs\to \bs'}$ can both be computed via a sequence of holomorphic triangle maps, so the result follows from Lemma~\ref{lem:relativehomologycommutestriangleII}.
\end{proof}

Next we consider the transition maps associated to simple stabilizations of the Heegaard surface.

\begin{lem}\label{lem:relhomologycommutesstab} Suppose that $\cH'$ is a simple stabilization of $\cH$, and let $\sigma$ denote the transition map from $\CF^-_{J}(\cH,\frs)$ to $\CF^-_{J(T)}(\cH',\frs)$. If $\lambda\subset \Sigma$ is an immersed, closed loop or path  connecting two basepoints, then
\[
A_\lambda\circ \sigma=\sigma\circ A_{\lambda}.
\]
\end{lem}
\begin{proof}The proof follows from the same count of holomorphic curves used to prove stabilization invariance \cite{LipshitzCylindrical}*{Proposition~12.5}. If $k$ is an integer, let $\phi_k\in \pi_2(c,c)$ denote the class which has multiplicity $k$ on the single domain of $(\bT^2,\alpha_0,\beta_0)$. The class $\phi_k$ has Maslov index $2k$. If $\phi\in \pi_2(\xs,\ys)$ is a homology class of disks on $\cH$, with multiplicity $k$ at the connected sum point, then using Lipshitz's formula for the Maslov index \cite{LipshitzCylindrical}*{Equation~8} together with the fact that a disk has Euler measure 1, we obtain
\[
\mu(\phi\# \phi_k)=\mu(\phi).
\]
 Furthermore, any homology class of disks on $\cH'$ can be written as such a connected sum.

 According to the proof of \cite{LipshitzCylindrical}*{Proposition~12.5}, for sufficiently large $T$, one has
 \begin{equation}
\# \hat{\cM}_{J}(\phi)=\# \hat{\cM}_{J(T)}(\phi\# \phi_k). \label{eq:stabilizedmodulicounts}
 \end{equation}  A straightforward computation shows that
 \begin{equation}
a(\lambda,\phi)=a(\lambda,\phi\# \phi_k).\label{eq:stabilizedAlambdacounts}
 \end{equation}
The claim follows immediately from Equations~\eqref{eq:stabilizedmodulicounts} and \eqref{eq:stabilizedAlambdacounts}.
\end{proof}

We now prove well-definedness of the relative homology actions:
\begin{proof}[Proof of Proposition~\ref{prop:naturalityrelativehomology}]
Suppose that $\lambda$ and $\lambda'$ are paths from $w_1$ to $w_2$ in $Y$. The claim that $A_{\lambda}\simeq A_{\lambda'}$ if $\lambda$ and $\lambda'$ represent homologous elements of $H_1(Y, \{w_1,w_2\};\Z)/\Tors$ is proven in Lemma~\ref{lem:torsion}.

To show that the map $A_\lambda$ commutes with the transition maps up to chain homotopy, it is sufficient to show that $A_{\lambda}$ commutes with the transition maps associated to changes of the almost complex structure, as well as each elementary Heegaard move from Lemma~\ref{lem:Heegaardmoves}.  Commutation of the relative homology maps with the transition maps associated to changing the almost complex structure is proven in Lemma~\ref{lem:relhomcommutechangeacstr}. Commutation with the maps associated to isotopies and handleslides of the $\as$ and $\bs$ curves is proven in Corollary~\ref{cor:relhomcomma/bmove}. Commutation with the simple stabilization maps is proven in Lemma~\ref{lem:relhomologycommutesstab}. Commutation with the maps induced by isotopies of the Heegaard surface inside of $Y$ is tautological.
\end{proof}

\section{Free-stabilization maps}
\label{sec:freestab}

In this section we describe maps for adding or removing a basepoint, which we call the \emph{free-stabilization} maps. Suppose $(Y,\ws)$ is a multi-pointed 3-manifold, $w\not \in \ws$, and 
\[
\sigma\colon \ws\to \bmP\quad  \text{and} \quad \sigma'\colon \ws\cup \{w\}\to \bmP
\]
are colorings satisfying $\sigma'|_{\ws}=\sigma$. In this section, we describe homomorphisms of $\cR_{\bmP}$-modules
\[
S_w^+ \colon \CF^-(Y,\ws^\sigma,\frs)\to \CF^-(Y,(\ws\cup \{w\})^{\sigma'},\frs),
\]
and
\[
S_w^- \colon \CF^-(Y,(\ws\cup \{w\})^{\sigma'},\frs)\to  \CF^-(Y,\ws^\sigma,\frs).
\]
Since the maps $S_{w}^+$ and $S_{w}^-$ are $\cR_{\bmP}$-equivariant, they induce maps on the $+$, $\infty$ and $\wedge$ flavors as well, by tensoring with the identity map (see Equation~\eqref{eq:CFinfty/plusdef}).

We now state the formula defining the free-stabilization maps. Suppose that $\cH=(\Sigma,\as,\bs,\ws)$ is a diagram for $(Y,\ws)$ such that $w\in \Sigma\setminus (\as\cup \bs)$. Pick a small disk $D\subset \Sigma\setminus (\as\cup \bs)$ containing the point $w$. Pick two curves $\alpha_0$ and $\beta_0$ inside of $D$, such that
\[
|\alpha_0\cap \beta_0|=2,
\]
and both $\alpha_0$ and $\beta_0$ bound a disk containing $w$.  The two intersection points of $\alpha_0\cap \beta_0$ are distinguished by their relative Maslov index. Let $\theta^+$ and $\theta^-$ denote the higher and lower graded intersection points respectively. See Figure~\ref{fig::20}.

\begin{figure}[ht!]
\centering
\begingroup%
  \makeatletter%
  \providecommand\color[2][]{%
    \errmessage{(Inkscape) Color is used for the text in Inkscape, but the package 'color.sty' is not loaded}%
    \renewcommand\color[2][]{}%
  }%
  \providecommand\transparent[1]{%
    \errmessage{(Inkscape) Transparency is used (non-zero) for the text in Inkscape, but the package 'transparent.sty' is not loaded}%
    \renewcommand\transparent[1]{}%
  }%
  \providecommand\rotatebox[2]{#2}%
  \newcommand*\fsize{\dimexpr\f@size pt\relax}%
  \newcommand*\lineheight[1]{\fontsize{\fsize}{#1\fsize}\selectfont}%
  \ifx\svgwidth\undefined%
    \setlength{\unitlength}{127.40696184bp}%
    \ifx\svgscale\undefined%
      \relax%
    \else%
      \setlength{\unitlength}{\unitlength * \real{\svgscale}}%
    \fi%
  \else%
    \setlength{\unitlength}{\svgwidth}%
  \fi%
  \global\let\svgwidth\undefined%
  \global\let\svgscale\undefined%
  \makeatother%
  \begin{picture}(1,0.80159992)%
    \lineheight{1}%
    \setlength\tabcolsep{0pt}%
    \put(0,0){\includegraphics[width=\unitlength,page=1]{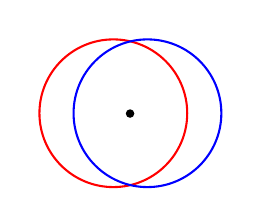}}%
    \put(0.46936006,0.15089858){\color[rgb]{0,0,0}\makebox(0,0)[lt]{\lineheight{1.25}\smash{\begin{tabular}[t]{l}$\theta^-$\end{tabular}}}}%
    \put(0.46936006,0.68094323){\color[rgb]{0,0,0}\makebox(0,0)[lt]{\lineheight{1.25}\smash{\begin{tabular}[t]{l}$\theta^+$\end{tabular}}}}%
    \put(0.51160969,0.39528047){\color[rgb]{0,0,0}\makebox(0,0)[lt]{\lineheight{1.25}\smash{\begin{tabular}[t]{l}$w$\end{tabular}}}}%
    \put(0,0){\includegraphics[width=\unitlength,page=2]{fig62.pdf}}%
    \put(0.31726606,0.65366344){\color[rgb]{1,0,0}\makebox(0,0)[rt]{\lineheight{1.25}\smash{\begin{tabular}[t]{r}$\alpha_0$\end{tabular}}}}%
    \put(0.63919207,0.65773931){\color[rgb]{0,0,1}\makebox(0,0)[lt]{\lineheight{1.25}\smash{\begin{tabular}[t]{l}$\beta_0$\end{tabular}}}}%
    \put(0.90268634,0.64644978){\color[rgb]{0,0,0}\makebox(0,0)[lt]{\lineheight{1.25}\smash{\begin{tabular}[t]{l}$p_0$\end{tabular}}}}%
  \end{picture}%
\endgroup%

\caption{\textbf{A free-stabilization.} We can think of the dashed circle as the connected sum tube. Alternatively, if we collapse the dashed circle to a point $p_0$, we get a doubly pointed diagram on $S^2$. }\label{fig::20}
\end{figure}

For appropriately chosen almost complex structures on $\Sigma\times [0,1]\times \R$ (described precisely in Section~\ref{subsec:gluingdata}), the maps $S_{w}^+$ and $S_w^-$ are defined via the formulas
\begin{equation}
\begin{split}
S_w^+(\xs)&=\xs\times \theta^+, \\
S^-_w(\ve{x}\times \theta^-)=\ve{x}\qquad &\text{ and } \qquad S^-_w(\ve{x}\times \theta^+)=0,
\end{split}\label{def:freestabilizationmaps}
\end{equation}
extended $\cR_{\bmP}$-equivariantly.

\subsection{Gluing data for stretching the neck}
\label{subsec:gluingdata}

In this section, we describe precisely which almost complex structures we use to define the free-stabilization maps.

It is convenient to view the free-stabilization operation as taking the connected sum of $\cH$ with the diagram $(S^2,\alpha_0,\beta_0,w_0,p_0)$, at the points $w$ and $p_0$. 

%
%
%
%


Fix an embedded disk $D_0\subset S^2\setminus (\alpha_0\cup \beta_0)$, centered at $p_0$. We make the following definition to parametrize the construction of a stretched almost complex structure:

\begin{define}\label{def:gluingdata}Suppose $\cH=(\Sigma,\as,\bs,\ve{w})$ is a diagram for $(Y,\ve{w})$ and $w\in \Sigma\setminus (\ws\cup \as\cup \bs)$. We call a tuple $\frd=(J^\frd,J_0^\frd,D,\iota)$ a \emph{gluing datum for free-stabilizing at $w$} if the following hold:
\begin{enumerate}

\item $D\subset \Sigma\setminus (\as\cup \bs)$ is a closed disk containing $w$.

\item $J^\frd$ is an almost complex structure on $\Sigma\times [0,1]\times \R$ which is split on $D$.

\item $J_0^\frd$ is an almost complex structure on $S^2\times [0,1]\times \R$ which is split on $D_0$.

\item $\iota\colon S^2\setminus (\tfrac{1}{2}\cdot D_0)\to D$ is an embedding which maps $w_0$ to $w$. Furthermore, $\iota$ maps the annulus $D_0\setminus (\tfrac{1}{2}\cdot D_0)$ conformally onto $D\setminus (\tfrac{1}{2}\cdot D)$. Here $\tfrac{1}{2}\cdot D_0$ denotes the subdisk of radius $\tfrac{1}{2}$ centered at $p_0$, obtained from the unique (up to rotation) conformal identification of the pair $(D_0,p_0)$ with $(\{z\in \C: |z|\le 1\}, 0)$. 
\end{enumerate}
\end{define}

Given a gluing datum $\frd$ for free-stabilizing at $w$, we can form a diagram 
\[
\cH^+:=(\Sigma, \as\cup\iota(\alpha_0), \bs\cup \iota(\beta_0), \ws\cup \{w\}),
\]
which depends on the choice of $\frd$ through the embedding $\iota$ (though we suppress this from the notation).

We define the stretched almost complex structure $J^\frd(T)$ on $\Sigma\times [0,1]\times \R$, whenever $T>0$. We begin by defining $J^\frd(T)$ when $T=2$, by letting $J^\frd(2)$ coincide with $J^\frd$ and $J_0^\frd$  on $(\Sigma\setminus (\tfrac{1}{2}\cdot D))\times [0,1]\times \R$ and $(S^2\setminus (\tfrac{1}{2}\cdot D_0))\times [0,1]\times \R$, respectively. For $T\ge 2 $, we construct an almost complex structure $J^\frd(T)$ by replacing the almost complex structure in the annulus region $D\setminus (\tfrac{1}{2}\cdot D)$ with one conformally equivalent to the annulus $D\setminus (1/T)\cdot D$. We set $J^\frd(T)=J^\frd(2)$ if $2\ge T>0$.

Recall that if $J$ and $J'$ are two cylindrical almost complex structures on $\Sigma\times [0,1]\times \R$, the transition map $\Psi_{(\cH,J)\to (\cH,J')}$ can be computed by picking a non-cylindrical almost complex structure $\tilde{J}$ on $\Sigma\times [0,1]\times \R$ which agrees with $J$ on $\Sigma\times [0,1]\times (-\infty, -1]$ and with $J'$ on $\Sigma\times [0,1]\times [1,\infty)$. The transition map is defined by counting index 0 $\tilde{J}$-holomorphic curves via the formula:
\begin{equation}
\Psi_{(\cH,J)\to (\cH,J')}(\xs):=\sum_{\substack{\phi\in \pi_2(\xs,\ys)\\ \mu(\phi)=0}} \# \cM_{\tilde{J}}(\phi)U_{\ws}^{n_{\ws}(\phi)} \cdot \ys.\label{eq:a-c-s-transitionmap}
\end{equation}
Write $\Psi_{\tilde{J}}$ for the map $\Psi_{(\cH,J)\to (\cH,J')}$ appearing in Equation~\eqref{eq:a-c-s-transitionmap}, computed using $\tilde{J}$.

\begin{define}\label{def:longenoughfreestab}Suppose $\cH$ is a Heegaard diagram and $\frd$ is a gluing datum for free-stabilizing at $w$. We say that a real number $T>0$ satisfies \emph{stabilizing condition}~\eqref{eq:stabilizationcondition2} if for any two $T_1,T_2\ge T$, there is a non-cylindrical almost complex structure $\tilde{J}$ on $\Sigma\times[0,1]\times \R$ interpolating $J^\frd(T_1)$ and $J^\frd(T_2)$, such that for all $\xs\in \bT_{\a}\cap \bT_{\b},$ we have
\begin{equation}
\begin{split}
\Psi_{\tilde{J}}(\xs\times \theta^+)&=\xs\times \theta^+,\quad\text{ and }\\
\Psi_{\tilde{J}}(\xs\times \theta^-)&=\xs\times \theta^-+\sum_{\ys\in \bT_{\a}\cap \bT_{\b}} C_{\xs,\ys}\cdot \ys\times \theta^+,
\end{split}
\tag{SC-2}\label{eq:stabilizationcondition2}
\end{equation}
for some $C_{\xs,\ys}\in \bF_2[U_{\ws}]$ (which may depend on $\frd$, $T_1$ and $T_2$).
\end{define}

We define
\[
S_{w}^+\colon \CF^-_{J^\frd}(\cH,\sigma,\frs)\to \CF^-_{J^\frd(T)}(\cH^+, \sigma',\frs),
\]
using Equation~\eqref{def:freestabilizationmaps} whenever $T$ satisfies condition~\eqref{eq:stabilizationcondition2}. The map $S_w^-$ is also defined using Equation~\eqref{def:freestabilizationmaps}, under the same assumption. 

If $J'$ is an arbitrary almost complex structure on $\Sigma\times [0,1]\times \R$, then the map $S_{w}^+$ from $\CF^-_{J^\frd}(\cH,\sigma,\frs)$ to $\CF^-_{J'}(\cH^+, \sigma',\frs)$ is defined as the composition of the map in Equation~\eqref{def:freestabilizationmaps} together with the transition map $\Psi_{J^\frd(T)\to J'}$.

\begin{prop}\label{prop:freestab-Tsufflargeexist}
If $\frd$ is a gluing datum for free-stabilizing at $w$, then there is a $T>0$ which satisfies stabilizing condition \eqref{eq:stabilizationcondition2}.
\end{prop}

Before proving Proposition~\ref{prop:freestab-Tsufflargeexist}, we prove a  Maslov index formula:

\begin{lem}\label{lem:indexdisksonsphere} Let $(S^2,\alpha_0,\beta_0,w,p_0)$ denote the diagram in Figure~\ref{fig::20}. If $x,y\in \alpha_0\cap \beta_0$ and $\phi_0\in \pi_2(x,y)$ is a homology class of disks, then
\[
\mu(\phi_0)=2n_w(\phi_0)+2n_{p_0}(\phi_0)+\gr(x,y),
\]
where $\gr(x,y)$ denotes the relative Maslov grading between $x$ and $y$.

Furthermore, if $m_1(\phi_0),$ $m_2(\phi_0),$ $m_3(\phi_0)$, and $m_4(\phi_0)$ denote the multiplicities of $\phi_0$ in the components of $S^2\setminus (\alpha_0\cup \beta_0)$, then
\[
\mu(\phi_0)=m_1(\phi_0)+m_2(\phi_0)+m_3(\phi_0)+m_4(\phi_0).
\]
\end{lem}

\begin{proof}The first formula is equivalent to 
\[
\gr(x,y)=\mu(\phi)-2n_w(\phi)-2n_{p_0}(\phi),
\]
which is the definition of the relative Maslov grading $\gr(x,y)$.

To prove the second formula, we verify it for a constant homology class $e_x\in \pi_2(x,x)$ (for which the claim is trivial), and then note that it respects slicing in bigons. Since any two classes on this diagram can be related by splicing in bigons, the formula follows in general.
\end{proof}

\begin{proof}[Proof of Proposition~\ref{prop:freestab-Tsufflargeexist}] We focus on the claim that if $T_1$ and $T_2$  satisfy \eqref{eq:stabilizationcondition2}, then $\tilde{J}$ can be chosen so that 
\[
\Psi_{\tilde{J}}(\xs\times \theta^+)=\xs\times \theta^+.
\]
The claim about $\xs\times \theta^-$ (which is the dual statement) follows by a simple modification.

We will write $p$ for the point $w$, viewed as a point on $\cH$, and write $w$ for the new basepoint on $\cH^+$.

Suppose $\phi\# \phi_0\in \pi_2(\xs\times \theta^+, \ys\times y)$ is a homology class of disks with Maslov index 0, for $y\in \{\theta^+,\theta^-\}$.  We will show that if $T_1$ and $T_2$ are sufficiently large, then $\tilde{J}$ can be chosen so that if $\phi\# \phi_0$ has a $\tilde{J}$-holomorphic representative, then  $\phi\# \phi_0$ is the constant class, $e_{\xs}\times e_{\theta^+}\in \pi_2(\xs\times \theta^+,\xs\times \theta^+)$. Furthermore, we will show that $e_{\xs}\times e_{\theta^+}$ always has a unique representative, which will imply the statement.

Suppose that $T_{1,i}$ and $T_{2,i}$ are a sequence of neck lengths which both approach $+\infty$. We can pick a sequence of interpolating almost complex structures $\tilde{J}_i$ such that $(\Sigma\times [0,1]\times \R, \tilde{J}_i)$ contains the almost complex manifold $((\Sigma\setminus N_i)\times [0,1]\times \R, J)$, where $N_i$ is a nested sequence of open balls on $\Sigma$ whose intersection is $\{p\}$. 

Suppose $u_i$ is a sequence of $\tilde{J}_i$-holomorphic curves representing $\phi\# \phi_0$. By adapting \cite{LipshitzCylindrical}*{Proposition~12.4}, we can extract a broken limiting curve on the punctured manifold $(\Sigma\setminus \{p\})\times [0,1]\times \R$. Such a holomorphic curve can be completed over $\{p\}\times [0,1]\times \R$ to obtain a (potentially broken) representative $\cU$ of the homology class $\phi$ on $(\Sigma,\as,\bs,\ws)$.

We have 
\begin{equation}
\begin{split}
\mu(\phi\# \phi_0)&=\mu(\phi)+\mu(\phi_0)-2n_{p_0}(\phi_0)\\
&=\mu(\phi)+\gr(\theta^+,y)+2n_{w}(\phi_0).\label{eq:indexbrokenlimitacstr}
\end{split}
\end{equation}
The first equality of Equation~\eqref{eq:indexbrokenlimitacstr} is justified by Lipshitz's formula for the Maslov index \cite{LipshitzCylindrical}*{Equation~8}, together with the fact that a disk has Euler measure 1. The second equality follows from  Lemma~\ref{lem:indexdisksonsphere}.

Since $\phi$ admits a broken representative for $J$,  we conclude that $\mu(\phi)\ge 0$ by transversality. Since the last line of Equation~\eqref{eq:indexbrokenlimitacstr} involves only non-negative terms, and the sum is zero, we conclude that
\[
\mu(\phi)=\gr(\theta^+,y)=n_{w}(\phi_0)=0.
\]

Since $\mu(\phi)=0$ and $\phi$ admits a broken $J$-holomorphic representative, and $J$ is cylindrical, it follows that $\xs=\ys$ and $\phi$ is the constant class, $e_{\xs}$, by transversality. Since $\gr(\theta^+,y)=0,$ it follows that $y=\theta^+$. Since $n_{p_0}(\phi_0)=n_{w}(\phi_0)=0$, as well, it follows that $\phi_0$ is the constant class $e_{\theta^+}$.

Conversely,  $e_{\xs}\times e_{\theta^+}$ admits a unique  $\tilde{J}_i$-holomorphic representative for any $i$, since each $\tilde{J}_i$ is cylindrical in a neighborhood of $(\xs\times \{\theta^+\})\times [0,1]\times \R$. The proof is complete.
\end{proof}

\subsection{Free-stabilization and the differential}

In this section, we prove that the free-stabilization maps are chain maps. The argument is essentially the same as \cite{OSLinks}*{Proposition~6.5}. We repeat the argument since we will later prove several refinements and analogous holomorphic curve counts.

\begin{prop}\label{prop:free-stabdifferential}
Suppose $\cH$ is a Heegaard diagram for $(Y,\ws)$, $\cH^+$ is its free stabilization at $w$, and $\frd$ is a gluing datum for the free-stabilization. Then for all $T$ which satisfy \eqref{eq:stabilizationcondition2},
\begin{equation}
\begin{split}
\d_{\cH^+,J^\frd(T)}(\xs\times \theta^+)&=\d_{\cH,J^\frd}(\xs)\otimes \theta^+, \quad \text{and} \\
\d_{\cH^+,J^\frd(T)}(\xs\times \theta^-)&=\d_{\cH,J^\frd}(\xs)\otimes \theta^-+\sum_{\ys\in \bT_{\a}\cap \bT_{\b}} C_{\xs,\ys}\cdot \ys\otimes \theta^+,
\end{split}
\label{eq:free-stab-diff}
\end{equation}
for $C_{\xs,\ys}\in \bF_2[U_{\ws},U_w]$ (which depend on $T$ and $\frd$).
\end{prop}

\begin{proof}
Since the transition maps $\Psi_{J^\frd(T_1)\to J^\frd(T_2)}$ are chain maps, it is easy to check that Condition~\eqref{eq:stabilizationcondition2} algebraically implies that if Equation~\eqref{eq:free-stab-diff} holds for some $T$ which satisfies Condition~\eqref{eq:stabilizationcondition2}, then it also holds for all $T$ which satisfy \eqref{eq:stabilizationcondition2}. Hence, it is sufficient to establish Equation~\eqref{eq:free-stab-diff} for any sufficiently large $T$.

Equation~\eqref{eq:free-stab-diff} is implied by the following two subclaims:
\begin{enumerate}[ref= d-\arabic*, label= (d-\arabic*):]
\item\label{num:diff-fs-1} The $\ys\times \theta^-$ coefficient of $\d_{\cH^+, J^\frd(T)}(\xs\times \theta^+)$
vanishes, whenever $T$ is sufficiently large.  
\item\label{num:diff-fs-2} If $\theta\in \{\theta^+,\theta^-\}$ and $T$ is sufficiently large, then the $\ys\times \theta$ coefficient  of $\d_{\cH^+, J^\frd(T)}(\xs\times \theta)$ (an element of $\bF_2[U_{\ws}, U_w]$) is equal to the $\ys$ coefficient of $\d_{\cH,J^\frd}(\xs)$ (an element of $\bF_2[U_{\ws}]$).
\end{enumerate}

We will write $p$ for the point $w$ on $\Sigma$, viewed as a point on $\cH$, and write $w$ for the basepoint on the free-stabilized diagram $\cH^+$.

 Let $\phi\# \phi_0\in \pi_2(\xs\times x,\ys\times y)$ be a Maslov index 1 class. Pick a sequence of neck-lengths $T_i$ approaching $+\infty$, and consider a sequence $u_i$ of $J^\frd(T_i)$-holomorphic curves representing $\phi\# \phi_0$. For such a sequence, we can extract a broken limit consisting of collections $\cU$, $\cU_m$ and $\cU_0$, where $\cU$ is a collection on $\Sigma\setminus \{p\}\times [0,1]\times \R$ whose total class is $\phi$, and $\cU_0$ is a collection on $S^2\setminus \{p_0\}\times [0,1]\times \R$ whose total class is $\phi_0$. The collection $\cU_m$ consists of curves in the tube region $S^1\times \R\times [0,1]\times \R$ (ultimately, we will rule out any non-trivial curves in $\cU_m$, due to codimension considerations).  The curves in  $\cU$ and $\cU_0$ may be completed over $p$ and $p_0$ to obtain curves on the diagrams $(\Sigma,\as,\bs)$ and $(S^2,\alpha_0,\beta_0)$. The process of obtaining limiting curves is described in \cite{LipshitzCylindrical}*{Proposition~12.4}.

Since $\phi$ has the broken holomorphic representative $\cU$, it follows from Proposition~\ref{prop:transversality} that
\begin{equation}
\mu(\phi)\ge 0.\label{eq:diffmasU>=0}
\end{equation}
On the other hand,  Equation~\eqref{eq:indexbrokenlimitacstr} implies
\begin{equation}
\mu(\phi\# \phi_0)=\mu(\phi)+\gr(x,y)+2n_{w}(\phi_0).\label{eq:diffMasexcision}
\end{equation}

We first consider Subclaim~\eqref{num:diff-fs-1}, when $x=\theta^+$ and $y=\theta^-$. In this case, we conclude from Equation~\eqref{eq:diffMasexcision} that
\[
\mu(\phi)=0 \quad \text{and} \quad n_w(\phi_0)=0.
\]
By transversality, it follows that $\phi$ is the constant class $e_{\xs}$, and $\phi_0$ is one of the two bigons in the free-stabilization region which have zero multiplicity over $w$. Both classes have unique representatives for any  almost complex structure, and hence have canceling contribution to the differential. Subclaim~\eqref{num:diff-fs-1} is established.

We now consider Subclaim~\eqref{num:diff-fs-2}. In this case, Equation~\eqref{eq:diffMasexcision} implies
\[
\mu(\phi)=1\quad  \text{and} \quad n_w(\phi_0)=0.
\]

For the moment, we trim off all ghost curves from $\cU$, $\cU_m$, and  $\cU_0$, i.e. components of the limit which have constant image in $\Sigma\times [0,1]\times \R$ or $S^2\times [0,1]\times \R$. (We will shortly prove that generically no ghost curves appear).

Having trimmed off ghost curves, we claim that transversality is achieved at the remaining curves in $\cU$. By Proposition~\ref{prop:transversality}, this amounts to showing that the limiting curves satisfy  \eqref{def:M1}--\eqref{def:M5}.  The only axiom which is non-trivial is \eqref{def:M5}, i.e. that the limiting curves have no components $v$ with $\pi_{[0,1]\times \R}\circ v$ constant. Having trimmed off ghost curves, we can assume that any curve $v$ with $\pi_{[0,1]\times \R}\circ v$ constant has $\pi_{\Sigma}\circ v$ non-constant. Such a curve $v$ has Maslov index at least 2, since its domain must be a sum of connected components of $\Sigma\setminus \as$ and $\Sigma\setminus \bs$, each weighted with a non-negative integer.  If we delete from $\cU$ all such curves $v$, then we obtain a curve at which transversality is achieved, by Proposition~\ref{prop:transversality}.  However the Maslov index of the remaining components is at most $-1$ (since $\mu(\phi)=1$ and we have removed curves whose total Maslov index is at least 2). There are no holomorphic curves with Maslov index $-1$ at which transversality is obtained, so such curves $v$ are prohibited from appearing in $\cU$.

It follows that $\cU$ (after trimming ghost curves) consists of a single curve $u\colon S\to \Sigma\times [0,1]\times \R$ satisfying \eqref{def:M1}--\eqref{def:M5}. Since $\phi$ has Maslov index 1, from Proposition~\ref{prop:transversality} it also follows that $u$ is embedded, and hence satisfies \eqref{def:M6}.

Let \[
\rho^p\colon \cM(\phi)\to \Sym^{n_p(\phi)}([0,1]\times \R)
\]
denote the map
\[
\rho^p(u):=(u\circ \pi_{[0,1]\times \R})\left((u\circ \pi_{\Sigma})^{-1}(p)\right).
\]

Consider the 1-dimensional set
\[
X(\phi):=\left\{\rho^p(u): u\in \cM(\phi)\right\}\subset \Sym^{n_p(\phi)}([0,1]\times \R).
\]

By perturbing the almost complex structure slightly near $p$, we can assume that $X(\phi)$ is disjoint from the fat diagonal in $\Sym^{n_p(\phi)}([0,1]\times \R)$, a codimension 2 subset.

We claim that $\cU_m$ consists of a union of $n_p(\phi)$ once-covered cylinders 
\[
u\colon S^1\times \R\to S^1\times \R\times [0,1]\times \R,
\]
which each have constant projection to $[0,1]\times \R$, together with some ghost curves. This follows since the maximum modulus principle implies that the projection to $[0,1]\times \R$ of any holomorphic curve in $S^1\times \R\times [0,1]\times \R$ must be constant. The asymptotics of $\cU_m$ must match those of $u$, and hence there must be exactly $n_p(\phi)$ once-covered cylinders which each project to a different point in $[0,1]\times \R$. Any additional curves must have constant image in $S^1\times \R\times [0,1]\times \R$. Write $C$ for these cylinders.

There must be a component $u_0$ of $\cU_0$ which has the same asymptotics (as a curve on $S^2\setminus \{p_0\}\times [0,1]\times \R$) at $p_0$ as the cylinders in $\cU_m$ (which we have already reasoned are the same as the asymptotics of $u$ at $p$), i.e.
\[
\rho^p(u)=\rho^{p_0}(u_0).
\]
Write $\phi_0'$ for the homology class of $u_0$. It is not hard to see that after trimming ghost curves, $u_0$ must satisfy \eqref{def:M1}--\eqref{def:M5} (embeddedness, \eqref{def:M6}, is not yet clear). Write $S_0$ for the source curve of $u_0$.

We now show that $\cU_0$ contains no curves other than $u_0$ and possibly ghost curves, and that $u_0$ is embedded (we subsequently will rule out ghost curves). 

By Proposition~\ref{prop:transversality}, for a generically chosen almost complex structure, near $u_0$ the set $\cM(S_0,\phi_0', X(\phi))$ is a manifold of dimension 
\begin{equation}
\dim \cM(S_0,\phi_0', X(\phi))=\mu(\phi_0')-\codim (X(\phi))-2\Sing(u_0).\label{eq:dimmatched}
\end{equation}
Since $D(\phi_0')\le D(\phi_0)$, Lemma~\ref{lem:indexdisksonsphere} implies that 
\begin{equation}
\mu(\phi_0')\le \mu(\phi_0) \qquad \text{and}\qquad \mu(\phi_0)=2n_{p}(\phi).\label{eq:maslov-index-inequality-0}
\end{equation}
(The first inequality of Equation~\eqref{eq:maslov-index-inequality-0} can also be proven using a transversality argument).

Since $\codim(X(\phi))=2n_p(\phi)-1$, Equation~\eqref{eq:maslov-index-inequality-0} implies
\[
\mu(\phi_0')\le \codim (X(\phi))+1.
\]
 Hence Equation~\eqref{eq:dimmatched} reduces to
\[
\dim \cM(S_0,\phi_0',X(\phi))\le 1-2\Sing(u_0)
\]
 with equality if and only if $\mu(\phi_0')=\mu(\phi_0)$. In particular $\cM(S_0,\phi_0',X(\phi))$ is generically empty unless $u_0$ is embedded and $\phi_0'=\phi_0$. It follows that $\cU_0$ consists only of the unbroken curve $u_0$, which satisfies \eqref{def:M1}--\eqref{def:M6}, as well as possibly some ghost curves.
 
We now show that generically no ghost curves appear in $\cU$, $\cU_m$ or $\cU_0$. Our argument is essentially standard; see \cite{LOTBordered}*{Lemma~5.57}. Let us write $S_i$ for the source of  $u_i$ (a curve in our original sequence of $J^\frd(T_i)$-holomorphic curves).  As one stretches the neck, the sources $S_i$ degenerate along a collection of boundary-to-boundary arcs and closed loops, as in the Deligne-Mumford compactification of stable curves. We can assume that all of the sources $S_i$ are topologically identified with a fixed surface, which we denote by $\hat{S}$. The limiting curve is a nodal curve, with nodes on the boundary or interior, corresponding to where arcs and closed curves in $\hat{S}$ collapse.  Furthermore, we can assume that all of the curves which appear in the limit are \emph{stable}, which implies that each ghost curve which is a disk or sphere has at least three nodes. Let $G_1,\dots, G_n$ denote the ghost components.

 Let us glue any nodes connecting two ghost components, and smooth any identified pair of nodes on a single ghost component. Abusing notation slightly, we write $G_1,\dots, G_n$ for the resulting surfaces, which we call the \emph{smoothed ghosts}. Note that gluing and smoothing ghost curves preserves stability.

We make the following claims about the smoothed ghosts:
\begin{enumerate}[ref= g-\arabic*, label= (g-\arabic*):]
\item\label{claim-ghost-0} There are no ghosts which have image in $\{p\}\times [0,1]\times \R$ or $\{p_0\}\times [0,1]\times \R$.
\item\label{claim-ghost-1} Each $G_i$ has exactly one node to $S$, $S_0$, or to a cylinder in $C$ in the tube region.
\item\label{claim-ghost-2} Each $G_i$ has no boundary components with no nodes.
\item\label{claim-ghost-3} Each $G_i$ has genus at least 1.
\end{enumerate}

Claim~\eqref{claim-ghost-0} follows since the limiting collections $\cU$ and $\cU_0$ are obtained by taking collections of curves in $\Sigma\setminus \{p\}\times [0,1]\times \R$ and $S^2\setminus \{p_0\}\times [0,1]\times \R$, and completing over the points $p$ and $p_0$.

Claim~\eqref{claim-ghost-1} is established as follows.  That there is at least one node is easy to see: if there are no nodes on some $G_i$, then $u_i$ must contain a closed, null-homologous component, which is prohibited by axiom \eqref{def:M5}. To see that there cannot be more than one node connecting a ghost $G_i$ to $S$ or $S_i$, we note that such a configuration would imply that $u$ or $u_0$ had a double point, which are prohibited since $u$ and $u_0$ are embedded.

We now prove Claim~\eqref{claim-ghost-2}. If $G_i$ had a boundary component with no nodes, then for large $j$ the unbroken curve $u_j$ would have a boundary component with no punctures, which must be mapped to $\Sigma\times \{0,1\}\times \R$. However such a configuration violates the maximum modulus principle applied to $\pi_{[0,1]\times \R}\circ u_j$

Finally, Claim~\eqref{claim-ghost-3} is proven by noting that each $G_i$ is stable, and has at most one node and boundary component by Claims~\eqref{claim-ghost-1} and \eqref{claim-ghost-2}, so must have genus at least 1.

Having established Claims~\eqref{claim-ghost-0}--\eqref{claim-ghost-3}, we now prove that generically ghosts do not appear. Lipshitz's formula for the Fredholm index \cite{LipshitzCylindrical}*{Equation~6} implies
\begin{equation}
\ind(S, \phi)=d-\chi(S)+2e(D(\phi))\qquad \text{and} \qquad \ind(S_0,\phi_0)=1-\chi(S_0)+2e(D(\phi_0)),\label{eq:fredholm-index-from-chi(S)}
\end{equation}
where $d=|\as|=|\bs|$, and $e(D(\phi))$ is the Euler measure of the domain of $\phi$. We have already reasoned that the Fredholm indices in Equation~\eqref{eq:fredholm-index-from-chi(S)} are $1$ and $2n_{p}(\phi)$, respectively, since $\phi$ and $\phi_0$ have Maslov indices 1 and $2n_{p}(\phi)$ respectively, and $u$ and $u_0$ are embeddings.  Similarly
\begin{equation}
\ind\left(\hat{S},\phi\# \phi_0 \right)=(d+1)-\chi(\hat{S})+2e(D(\phi\# \phi_0)),\label{eq:fredholm-index-S-hat}
\end{equation}
which also must be 1, since the $u_i$ are embedded and have Maslov index 1, by assumption. Noting that $e(D(\phi\# \phi_0))=e(D(\phi))+e(D(\phi_0))-2n_{p_0}(\phi_0)$, Equations~\eqref{eq:fredholm-index-from-chi(S)} and ~\eqref{eq:fredholm-index-S-hat} combine to imply that
\begin{equation}
\chi(\hat{S})=\chi(S)+\chi(S_0)-2n_{p_0}(\phi_0).\label{eq:Euler-characteristicS-S0-S-hat}
\end{equation}

Finally, Equation~\eqref{eq:Euler-characteristicS-S0-S-hat} prohibits any ghost components in $\cU$, $\cU_m$ or $\cU_0$: $\hat{S}$ may be topologically constructed by gluing $S$ and $S_0$ along the $n_{p_0}(\phi_0)$ punctures, as well as gluing any ghost components in. However Claims~\eqref{claim-ghost-1}, \eqref{claim-ghost-2} and \eqref{claim-ghost-3} imply that gluing in each $G_i$ drops $\chi(\hat{S})$ by at least 2, contradicting Equation~\eqref{eq:Euler-characteristicS-S0-S-hat}. We conclude that generically no ghost curves appear in $\cU$, $\cU_m$ or $\cU_0$.

 Summarizing, any sequence $u_i$ of $J^\frd(T_i)$-holomorphic curves representing $\phi\# \phi_0$ has a subsequence which converges to a pair $(u,u_0)$ which both satisfy \eqref{def:M1}--\eqref{def:M6} and $\rho^p(u)=\rho^{p_0}(u_0)$.
 
We can view such a pair $(u,u_0)$ as being a point in the compactification of the 1-dimensional moduli space
\begin{equation}
\bigcup_{T>0} \hat{\cM}_{J^\frd(T)}(\phi\# \phi_0).\label{eq:timedependentmodulispace}
\end{equation}

Standard gluing results for holomorphic curves (see \cite{LipshitzCylindrical}*{Proposition~A.2}, \cite{BourgeoisMorseBott}*{Section~5.3}) imply that  for sufficiently large $T$, near $(u,u_0)$ the compactification of the moduli space in Equation~\eqref{eq:timedependentmodulispace}  is modeled on a half open interval times a fibered product:
\begin{equation}
[0,1)\times \left(\left( \cM_{J^\frd}(\phi)\times_{\rho} \cM_{J_0^\frd}(\phi_0)\right)/\R\right). \label{eq:fiberedproduct}
\end{equation}

Since $\mu(\phi)=1$, the set $\hat{\cM}_{J^\frd}(\phi)$ is zero dimensional. Hence  Equation~\eqref{eq:fiberedproduct} implies that for sufficiently large $T$,
\begin{equation}
\# \hat{\cM}_{J^\frd(T)}(\phi\# \phi_0)\equiv\sum_{u\in \hat{\cM}(\phi)} \# \hat{\cM}_{J_0^\frd}(\phi_0, \rho^{p}(u))\pmod{2}. \label{eq:diskcountsstretchedAC}
\end{equation}

According to \cite{OSLinks}*{Lemma~6.4}, if $k>0$ and $\ve{d}\in \Sym^{k}([0,1]\times \R)$ is in the complement of the fat diagonal, then
\begin{equation}
\sum_{\substack{\phi_0\in \pi_2(\theta,\theta)\\ n_{p_0}(\phi_0)=k\\ n_{w}(\phi_0)=0}}\# \cM_{J_0^\frd}(\phi_0,\ve{d})\equiv 1\pmod{2}. \label{eq:OS'scountmatched}
\end{equation}

Subclaim~\eqref{num:diff-fs-2} now follows by combining Equations~\eqref{eq:diskcountsstretchedAC} and ~\eqref{eq:OS'scountmatched}, together with our argument which showed that the only index 1 classes $\phi\# \phi_0\in \pi_2(\xs\times \theta, \ys\times \theta)$ which have holomorphic representatives for large $T$ have $\mu(\phi)=1$ and $n_w(\phi_0)=0$.
\end{proof}

\begin{cor}Suppose $(Y,\ws)$ is a multi-pointed 3-manifold, $w\in Y\setminus \ws$, and $\sigma\colon \ws\to \bmP$ is a coloring which is extended by $\sigma'\colon \ws\cup \{w\}\to \bmP$. For sufficiently large $T$, the maps
\[
S_w^+\colon \CF^-_{J^\frd}(\cH,\sigma,\frs)\to \CF^-_{J^\frd(T)}(\cH^+,\sigma',\frs)\quad \text{and}
\]
\[
S_w^-\colon \CF^-_{J^\frd(T)}(\cH^+,\sigma',\frs)\to  \CF^-_{J^\frd}(\cH,\sigma,\frs)
\]
are chain maps.
\end{cor}
\begin{proof}The claim is a restatement of Proposition~\ref{prop:free-stabdifferential}, using the definition of the free-stabilization maps in Equation~\eqref{def:freestabilizationmaps}.
\end{proof}

\subsection{Free-stabilization and holomorphic triangles}

In this section, we prove an analog of Proposition~\ref{prop:free-stabdifferential} for holomorphic triangles. If $\cT=(\Sigma,\as,\bs,\gs,\ws)$ is a multi-pointed Heegaard triple and $w\in \Sigma\setminus (\as\cup \bs\cup \gs\cup \ws)$, consider the triple
\[
\cT^+:=(\Sigma,\as\cup \{\a_0\}, \bs\cup \{\b_0\}, \gs\cup \{\g_0\}, \ws\cup \{w\}),
\]
where $\a_0,$ $\b_0$ and $\g_0$ are new attaching curves contained in a small neighborhood of $w$ on $\Sigma$. Furthermore, assume that the free-stabilization region has the configuration shown in Figure~\ref{fig::63}. In particular,
\[
|\alpha_0\cap \beta_0|=|\beta_0\cap \gamma_0|=|\alpha_0\cap \gamma_0|=2.
\]

\begin{figure}[ht!]
\centering
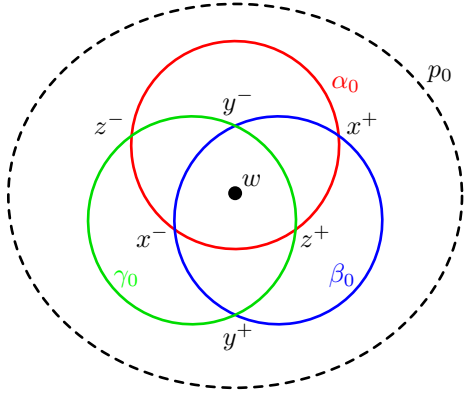
\caption{\textbf{A free-stabilized Heegaard triple.} If we collapse the dashed circle to the point $p_0$, we get the doubly based Heegaard triple $(S^2,\alpha_0,\beta_0,\gamma_0,p_0,w)$.}\label{fig::63}
\end{figure}

If $(\Sigma,\as,\bs,\gs,\ws)$ is a Heegaard triple, and $(\Sigma, \as\cup \{\alpha_0\}, \bs\cup \{\beta_0\}, \gs\cup \{\g_0\}, \ws\cup \{w\})$ is its free-stabilization, then there is a canonical diffeomorphism
\begin{equation}
X_{\a\cup \{\a_0\}, \b\cup \{\b_0\}, \g\cup \{\g_0\}}\iso X_{\a,\b,\g}, \label{eq:Xabgunchangedfreestab}
\end{equation}
where $X_{\a,\b,\g}$ is the 4-manifold defined in Equation~\eqref{eq:X_abgdef}. Equation~\eqref{eq:Xabgunchangedfreestab} follows from the fact that neither the Heegaard surface $\Sigma$ nor the handlebodies $U_{\a}$, $U_{\b}$ or $U_{\g}$ are changed by a free-stabilization.

It is straightforward to verify that the $\Spin^c$ structure map $\frs_{\ws}$ from Equation~\eqref{eq:spincmapdef} is also unchanged, in the sense that if $\psi\# \psi_0\in \pi_2(\xs\times x, \ys\times y, \zs\times z)$ is a homology class of triangles on the free-stabilized diagram, then
\[
\frs_{\ws}(\psi\# \psi_0)=\frs_{\ws}(\psi),
\]
with respect to the diffeomorphism from Equation~\eqref{eq:Xabgunchangedfreestab}.

\begin{thm}\label{thm:freestabilizetriangles}Suppose $\cT$ is a Heegaard triple, $\cT^+$ is its free-stabilization at the point $w$, and $\frd$ is a gluing datum for free-stabilization. Then for sufficiently large $T$, we have
\begin{equation*}
\begin{split}
F_{\cT^+,J^\frd(T), \frs}(S_w^+(-), S_w^+(-))&=S_w^+ F_{\cT, J^\frd, \frs}(-,-),\\
S_{w}^- F_{\cT^+, J^\frd(T), \frs}(S_w^+(-), -)&=F_{\cT,J^\frd, \frs}(-, S_w^-(-)),\\
S_{w}^-  F_{\cT^+, J^\frd(T), \frs}(-,S_w^+(-))&=F_{\cT,J^\frd, \frs}(S_w^-(-),- ).
\end{split}
\end{equation*}
\end{thm}

Before we prove Theorem~\ref{thm:freestabilizetriangles}, we prove a useful Maslov index formula:

\begin{lem}\label{lem:maslovindextriangles}Suppose that $\psi_0\in \pi_2(x,y,z)$ is a homology class of triangles on the Heegaard triple $(S^2,\alpha_0,\beta_0,\gamma_0,w,p_0)$ in Figure~\ref{fig::63}. Then
\[
\mu(\psi_0)=-\gr(x^+,x)-\gr(y^+,y)+\gr(z^+,z)+2 n_w(\psi_0)+2n_{p_0}(\psi_0).
\]
If $m_1(\psi_0),\dots, m_8(\psi_0)$ denote the multiplicities of $\psi_0$ in each of the eight regions of $S^2\setminus (\alpha_0\cup \beta_0\cup \gamma_0)$, then
\[
\mu(\psi_0)=\frac{1}{2}\left(m_1(\psi_0)+\cdots +m_8(\psi_0)-1\right).
\]
\end{lem}

\begin{proof} Both formulas can be verified using the same strategy. First, both are easily checked for the triangle class in $\pi_2(x^+,y^+,z^+)$ whose domain consists of a single component of $S^2\setminus (\a_0\cup \b_0\cup \g_0)$ with multiplicity 1. Both formulas respect splicing  in any of the 12 bigons on the diagram. Since any two triangle classes can be related by splicing in such bigons, both formulas follow in general.
\end{proof}

\begin{proof}[Proof of Theorem~\ref{thm:freestabilizetriangles}] The proof is similar to the proof of Proposition~\ref{prop:free-stabdifferential}.

Rephrasing the statement in terms of intersection points, it is sufficient to show that if $T$ is sufficiently large, then
\begin{equation}
\begin{split}
F_{\cT^+, J^\frd(T), \frs}(\xs\times x^+, \ys\times y^+)&=F_{\cT,J^\frd,\frs}(\xs, \ys)\otimes z^+,\\
  F_{\cT^+, J^\frd(T), \frs}(\xs\times x^+, \ys\times y^-)&=F_{\cT,J^\frd, \frs}(\xs, \ys)\otimes z^-+ \sum_{\zs\in \bT_{\a}\cap \bT_{\g}} C^1_{\xs,\ys,\zs}\cdot \zs\times z^+,\\
    F_{\cT^+, J^\frd(T), \frs}(\xs\times x^-, \ys\times y^+)&=F_{\cT,J^\frd, \frs}(\xs, \ys)\otimes z^-+ \sum_{\zs\in \bT_{\a}\cap \bT_{\g}} C^2_{\xs,\ys,\zs}\cdot \zs\times z^+,
\end{split}
\label{eq:free-stab-triangle}
\end{equation}
for $C^1_{\xs,\ys,\zs},C^2_{\xs,\ys,\zs}\in \bF_2[U_{\ws}, U_w]$ (which depend on $T$ and $\frd$).

Suppose that $\psi\# \psi_0\in \pi_2(\xs\times x, \ys\times y, \zs\times z)$ is a homology class of triangles with Maslov index 0. Define the quantity
\[
\delta(x,y,z):=-\gr(x^+,x)-\gr(y^+,y)+\gr(z^+,z).
\]
It is easy check that Equation~\eqref{eq:free-stab-triangle} is implied by the following two subclaims:
\begin{enumerate}[ref= f-\arabic*, label= (f-\arabic*):]
\item\label{num:TC-1} If $\delta(x,y,z)=1$, then any Maslov index 0 class in $\pi_2(\xs\times x, \ys\times y, \zs\times z)$ has no $J^\frd(T)$-holomorphic representatives if $T$ is sufficiently large.
\item\label{num:TC-2} If $\delta(x,y,z)=0$, then the $\zs\times z$ coefficient of $F_{\cT,J^\frd(T),\frs}(\xs\times x, \ys\times y)$ is equal to the $\zs$ coefficient of $F_{\cT,J^{\frd},\frs}(\xs,\ys)$.
\end{enumerate}
The classes with $\delta(x,y,z)\in \{-2,-1\}$ do not contribute to the present theorem statement.

In our proof, we will write $p\in \Sigma$ for the point $w$, viewed as a point on the unstabilized Heegaard triple, and write $w$ for the new basepoint on the free-stabilized Heegaard triple.

We compute
\begin{equation}
\begin{split}
\mu(\psi\# \psi_0)&=\mu(\psi)+\mu(\psi_0)-2n_{p_0}(\psi_0)\\
&=\mu(\psi)+\delta(x,y,z)+2n_w(\psi_0).
\end{split}
\label{eq:Masindconsumtriangles}
\end{equation}
The first equality of Equation~\eqref{eq:Masindconsumtriangles} follows from Sarkar's formula for the Maslov index \cite{SarkarMaslov}, as well as the fact that a disk has Euler measure 1. The second equality follows from Lemma~\ref{lem:maslovindextriangles}.

 Suppose  that $\psi\# \psi_0$ admits holomorphic representatives for arbitrarily large neck length parameter $T$. As in the proof of Proposition~\ref{prop:free-stabdifferential}, we may find a subsequence which converges to broken collections of curves $\cU$, $\cU_m$ and $\cU_0$, on $\Sigma\times \Delta$, $S^1\times \R\times \Delta$ and $S^2\times \Delta$, and their cylindrical ends, $\Sigma\times [0,1]\times \R$, $S^1\times \R\times [0,1]\times \R$ and $S^2\times [0,1]\times \R$. 

Consider Claim~\eqref{num:TC-1}. In particular, $\psi$ and $\psi_0$ both admit broken representatives. Using Equation~\eqref{eq:Masindconsumtriangles}, and the equalities $\mu(\psi\# \psi_0)=0$ and $\delta(x,y,z)=1$, we conclude that $\mu(\psi)=-1$. However, the triangular of Proposition~\ref{prop:transversality} implies that generically $\cM(\psi)$ is empty, contradicting the existence of the broken triangle $\cU$. Claim~\eqref{num:TC-1} follows.

We now  consider Claim~\eqref{num:TC-2}. Let $\psi\# \psi_0$ denote a class with $\delta(x,y,z)=0$, and let $\cU$, $\cU_m$ and $\cU_0$ denote broken holomorphic triangles arising in the limit. We conclude from Equation~\eqref{eq:Masindconsumtriangles} and transversality that $\mu(\psi)=0$. 

As a first step, trim off any ghost components of $\cU$, $\cU_m$ and $\cU_0$ (i.e. curves which have constant image in $\Sigma\times \Delta$ or one of its cylindrical ends). We will later see that ghost curves do not appear generically.

Next, we claim that $\cU$ cannot contain any curves $v$ mapping into $\Sigma\times \Delta$, or one of its cylindrical ends, such that $\pi_{\Delta}\circ v$ (or $\pi_{[0,1]\times \R}\circ v$) is constant, but $\pi_{\Sigma}\circ v$ is non-constant.  Such a curve $v$ must have Maslov index at least 2, since it has domain equal to a non-negative integral combination of components of $\Sigma\setminus \as$, $\Sigma\setminus \bs$ and $\Sigma\setminus \gs$. The presence of such a curve would force the remaining curves, on which transversality is obtained, to have Maslov index at most $-2$, which is impossible, generically.

Consequently, all curves in $\cU$ (except the ghost components) satisfy the analogs of \eqref{def:M1}--\eqref{def:M5}, so by the analog of Proposition~\ref{prop:transversality} for triangles, transversality is obtained at $u$. Since $\mu(\psi)=0$,  it follows from transversality that $\cU$ consists only of a single holomorphic triangle $u$ satisfying the analogs of \eqref{def:M1}--\eqref{def:M6} for holomorphic triangles (as well as possibly ghost curves). In particular, $u$ is embedded.

Arguing identically to the proof of Proposition~\ref{prop:free-stabdifferential}, the curves of $\cU_m$ generically consist of $n^p(\phi)$ once-covered cylinders, which each map to a distinct point in $\Delta$. (There may also be ghost components in $\cU_m$, though we will shortly rule those out).

There must be a curve $u_0$ in $\cU_0$ whose asymptotics (before completing over $\{p\}\times \Delta$) match those of $u$, i.e.
\begin{equation}
\rho^{p}(u)=\rho^{p_0}(u_0),\label{eq:matchingcondition}
\end{equation}
where $\rho^p\colon \cM(\psi)\to \Sym^{n_p(\psi)}(\Delta)$ is the map defined analogously to Equation~\eqref{eq:rhopdefinition}.

Let $\psi_0'$ denote the homology class of $u_0$. Since $D(\psi_0')\le D(\psi_0)$,  we conclude from Lemma~\ref{lem:maslovindextriangles} that
\begin{equation}
\mu(\psi_0')\le \mu(\psi_0)=2n_{p_0}(\psi_0).\label{eq:maslovindtridec}
\end{equation}
Define
\[
X(\psi):=\{\rho^p(u): u\in \cM(\psi)\},
\]
a 0-dimensional subset of $\Sym^{n_{p}(\psi)}(\Delta)$.

Write $S_0$ for the topological source of $u_0 $. Using Proposition~\ref{prop:transversality} and Equation~\eqref{eq:maslovindtridec}, for a generically chosen almost complex structure
\begin{equation}
\begin{split}
\dim \cM(S_0,\psi_0', X(\psi_0)) 
&=\mu(\psi_0')-\Sing(u_0)-\codim X(\psi_0)\\
&= \mu(\psi_0')-\Sing(u_0)-2n_{p_0}(u_0)
\\&\le -\Sing(u_0).
\end{split}
\end{equation}
Furthermore, equality holds if and only if $\mu(\psi_0')=\mu(\psi_0)=2n_{p_0}(u_0)$ and $\Sing(u_0)=0$. Consequently, if $\cM(S_0,\psi_0', X(\psi_0))$ is non-empty, we must generically have $\mu(\psi_0')=\mu(\psi_0)$ and $\Sing(u_0)=0$. In particular, $\cU_0$ cannot have any other components (other than ghost components), and $u_0$ must be embedded. 


The appearance of ghost components in $\cU$ and $\cU_0$, which we trimmed off earlier, can be prohibited from appearing generically by using the same index argument as in the proof of Proposition~\ref{prop:free-stabdifferential}.

Summarizing, if $T_i$ is a sequence of neck lengths approach $+\infty$, then any sequence of $J^\frd(T_i)$-holomorphic triangles $u_i$ representing $\psi\# \psi_0$ has limit equal to a pair $(u,u_0)$ of holomorphic triangles satisfying \eqref{def:M1}--\eqref{def:M6} and $\rho^p(u)=\rho^{p_0}(u_0).$

As in the proof of Proposition~\ref{prop:free-stabdifferential}, for large $T$, gluing gives a bijection
 \begin{equation}
 \cM_{J^\frd(T)}(\psi\# \psi_0)\iso \cM_{J^\frd}(\psi)\times_{\rho} \cM_{J_0^\frd}(\psi_0).\label{eq:fiberedproducttriangles}
\end{equation}
Since $\cM_{J^\frd}(\psi)$ is zero dimensional, counting $\cM_{J^\frd(T)}(\psi\# \psi_0)$ is reduced to the problem of counting the elements in the matched moduli space $\cM_{J_0^\frd}(\psi_0,\ve{d})$ for  generic $\ve{d}\in \Sym^{n}(\Delta)$. We will show that for generic $J_0^\frd$,
\begin{equation}
\sum_{\substack{\psi_0\in \pi_2(x,y,z)\\
n_{w_0}(\psi_0)=0\\
n_{p_0}(\psi_0)=n}} \# \cM_{J_0^\frd}(\psi_0,\ve{d})\equiv 1 \pmod{2},\label{eq:divisorcounttriangles}
\end{equation}
whenever $\delta(x,y,z)=0$ and $\ve{d}\in \Sym^n(\Delta)$ is not in the fat diagonal.

To establish Equation~\eqref{eq:divisorcounttriangles}, we use the following cobordism argument. Consider a path $(\ve{d}_t)_{t\in [1,\infty)}$ such that
\begin{enumerate}
\item $\ve{d}_1=\ve{d}$.
\item Each $\ve{d}_t$ is not contained in the fat diagonal.
\item As $t\to \infty$, the points in $\ve{d}_t$ all travel to $\infty$ in the $\alpha$-$\beta$ cylindrical end of $\Delta$.
\item If we identify the $\alpha$-$\beta$ cylindrical end of $\Delta$ with $[0,1]\times (-\infty,0]$, then the points of $\ve{d}_t$ are spaced at least $t$ distance apart for large $t$.
\item The $[0,1]$ component of each point of $\ve{d}_t$ approaches some fixed $s_0\in (0,1)$.
\end{enumerate}

Write
\[
\ve{\cD}:=\{ \ve{d}_t: t\in [1,\infty)\}.
\]

Consider the ends of the 1-dimensional moduli space
\[
\cM_{(x,y,z)}(\ve{\cD}):=\coprod_{\substack{\psi_0\in \pi_2(x,y,z)\\
n_{w_0}(\psi_0)=0\\
n_{p_0}(\psi_0)=n}} \cM_{J_0^\frd}(\psi_0,\ve{\cD}).
\]

There are three types of ends to consider:
\begin{enumerate}[ref= e-\arabic*, label= (e-\arabic*):]
\item\label{ends:1} Curves at $t=1$.
\item\label{ends:2} Degenerations occurring at finite $t\in (1,\infty)$.
\item\label{ends:3} Limiting curves as $t\to \infty$.
\end{enumerate}

Ends of $\cM_{(x,y,z)}(\ve{\cD})$ of type~\eqref{ends:1} correspond to the disjoint union of $\cM_{J_0^\frd}(\psi_0,\ve{d})$ over all $\psi_0\in \pi_2(x,y,z)$ with $n_w(\psi_0)=0$.

Ends of $\cM_{(x,y,z)}(\ve{\cD})$ of type~\eqref{ends:2} correspond to sequences of holomorphic triangles in $\cM_{(x,y,z)}(\ve{\cD})$ degenerating into a broken holomorphic triangle at some finite $t\in (0,\infty)$. We wish to show that such a degeneration can consist only of a holomorphic strip and triangle, both satisfying \eqref{def:M1}--\eqref{def:M6}, or their triangular analogs. Furthermore, the holomorphic strip must have Maslov index 1, and zero multiplicity on $p_0$, and the triangle must have Maslov index $2n_{p_0}(\psi_0)-1$, and must match some $\ve{d}_t$. 

Let $\cV$ denote a broken holomorphic triangle appearing as the limit of a sequence of curves in $\cM_{(x,y,z)}(\ve{\cD})$ at some $t\in (1,\infty)$. There must be holomorphic triangle $v_0$ of $\cV$ which matches $\ve{d}_t$.

Delete curves $v$ in  $\cV$ which do not map into $\Sigma\times \Delta$, or such that $\pi_{\Delta}\circ v$ is constant. Since $n_w(\psi_0)=0$, there are no non-ghost components $v$ of $\cV$ such that $\pi_{\Delta}\circ v$ is constant and equals a point of $\ve{d}_t$. Consequently, after removing any components whose projection to $\Delta$ is constant, the remaining curve $v_0$ still matches $\ve{d}_t$,  and satisfies the triangular analogs of \eqref{def:M1}--\eqref{def:M5} (embeddedness, \eqref{def:M6}, is not clear yet). Write $\psi_0'$ for the homology class of $v_0$ and $S_0$ for the  source curve of $v_0$. By the triangular analog of Proposition~\ref{prop:transversality}, the space $\cM(S_0,\psi_0', \ve{\cD})$ is transversely cut out and of dimension
\begin{equation}
\mu(\psi_0')-2 \Sing(u_0)-\codim(\ve{\cD})=\mu(\psi_0')-2 \Sing(v_0)-(2|\ve{d}|-1). \label{eq:trianglesdegenerating}
\end{equation}

The quantity $\Sing(v_0)$ takes values in $\tfrac{1}{2}\cdot \Z^{\ge 0}$; a singularity along the interior contributes 1, while a singularity along the boundary contributes $\tfrac{1}{2}$. Note that $\mu(\psi_0)=2|\ve{d}|$, and $\mu(\psi_0')\le \mu(\psi_0)$ by Lemma~\ref{lem:maslovindextriangles}. Hence Equation~\eqref{eq:trianglesdegenerating} implies that there are two possibilities which are not prohibited by dimension counts: $\mu(\psi_0')=2|\ve{d}|-1$ and $\Sing(v_0)=0$ ($v_0$ is embedded); $\mu(\psi_0')=2|\ve{d}|$ and $\Sing(v_0)=\tfrac{1}{2}$  ($v_0$ has a singular point along its boundary).

We now rule out a boundary double point. A boundary double point arises when a boundary to boundary arc $c$ in the source curve collapses. Suppose $v_i\colon S\to S^2\times \Delta$ is a sequence of curves in $\cM_{(x,y,z)}(\ve{\cD})$, all sharing the same topological source, which realize such a degeneration. Write $\d_{\alpha} S$, $\d_{\beta} S$ and $\d_{\gamma}  S$ for the subsets of $\d S$ which map to $\alpha_0$, $\beta_0$ and $\gamma_0$, respectively. If $q_1$ and $q_2$ are the two ends of $c$, then both $q_1$ and $q_2$ must be contained in a single one of $\d_{\alpha} S$, $\d_{\beta} S$ or $\d_{\gamma} S$.
Condition~\eqref{def:J'4} implies that $v_i|_{\d_{\tau} S}$ is monotonic, for $\tau\in \{\alpha,\beta,\gamma\}$. Consequently, if a boundary double point forms, subarc region between $q_1$ and $q_2$ along the boundary of $S$ must map constantly to $\Delta$, in the limit. Consequently, the limiting curve must have a component which is a boundary degeneration. Since $\ve{\cD}$ is bounded away from $\d \Delta$, and all curves in $\cM_{(x,y,z)}(\ve{\cD})$ have zero multiplicity on $w$, boundary degenerations cannot form in the ends of $\cM_{(x,y,z)}(\ve{\cD})$, since any boundary degeneration on $(S^2,\alpha_0,\beta_0,\gamma_0)$,  covers one of $p_0$ or $w$. It follows that $\Sing(v_0)=0$ and 
\[
\mu(\psi_0')\in \{2|\ve{d}|, 2|\ve{d}|-1\}.
\]

 If $\mu(\psi_0')=2|\ve{d}|=\mu(\psi_0)$, then the remaining curves in the broken triangle have Maslov index 0. Using transversality, it is not hard to see that any such curves must be constant holomorphic strips in the cylindrical ends. By definition, constant strips are not stable, so cannot appear in a broken triangle in the boundary of $\cM_{(x,y,z)}(\ve{\cD})$. Hence, no curves with $\mu(\psi_0')=2|\ve{d}|$ appear in the boundary of $\cM_{(x,y,z)}(\ve{\cD})$, for $t\in (1,\infty)$.

Next, we consider the case that $\mu(\psi_0')=2|\ve{d}|-1$. In this case, the remaining curves (call them $\cV_1$) have total Maslov index 1. It is straightforward to see that this implies that $\cV_1$ consists of a single holomorphic flowline $v_1$ on $S^2\times [0,1]\times \R$ on one of the three sub-diagrams $(S^2,\alpha_0,\beta_0)$, $(S^2,\beta_0,\gamma_0)$ or $(S^2,\alpha_0,\gamma_0)$.  The curve $v_1$ must have multiplicity zero on $p_0$ and $w$. Since the differential vanishes on all three complexes, $\hat{\CF}(S^2,\alpha_0,\beta_0,p_0,w)$, $\hat{\CF}(S^2,\alpha_0,\gamma_0,p_0,w)$ and $\hat{\CF}(S^2,\beta_0,\gamma_0,p_0,w)$, ends of type~\eqref{ends:2} come in canceling pairs.

Finally we consider the ends of type~\eqref{ends:3}. Using the Maslov index formulas from Lemmas~\ref{lem:indexdisksonsphere} and~\ref{lem:maslovindextriangles}, together with the transversality result from Proposition~\ref{prop:transversality} (as well as its triangular analog for curves in $S^2\times \Delta$), we can show that the limit of a sequence of curves with $t\to \infty$ must consist of $|\ve{d}|$ flowlines $v_1,\dots, v_{|\ve{d}|}$ on $(S^2,\alpha_0,\beta_0)$, representing Maslov index 2 classes in $\pi_2(x,x)$ with $n_{p_0}=0$, each of which matches some fixed point $d\in [0,1]\times \R$, as well as a single index 0 holomorphic triangle $v$ with Maslov index 0 and $n_w(v)=n_{p_0}(v)=0$. Using  \cite{LipshitzCylindrical}*{Proposition~A.1} to glue, it follows that such ends of $\cM_{(x,y,z)}(\ve{\cD})$ at $t=\infty$ correspond to the 0-dimensional space
\begin{equation}
\bigg(\coprod_{\substack{\psi_0^0\in \pi_2(x,y,z)\\
n_{w}(\psi^0_0)=n_{p_0}(\psi_0^0)=0}} \cM(\psi_0^0)\bigg)\times \bigg( \coprod_{\substack{\phi\in \pi_2(x,x)\\
n_{p_0}(\phi)=1\\
n_{w}(\phi)=0}}\cM(\phi, d)\bigg)^{|\ve{d}|}.\label{eq:degeneratematching}
\end{equation}

It is easy to verify that there is only one non-negative class $\psi_0^0\in \pi_2(x,y,z)$ with Maslov index 0 and $n_{w}(\psi_0^0)=n_{p_0}(\psi_0^0)=0$. Furthermore, $\psi_0^0$ has a unique representative by the Riemann mapping theorem.

Using Ozsv\'{a}th and Szab\'{o}'s count from Equation~\eqref{eq:divisorcounttriangles}, it follows that the number of elements in the set in Equation~\eqref{eq:degeneratematching} is 1, modulo 2. In particular, counting up all ends of $\cM_{(x,y,z)}(\ve{\cD})$, Equation~\eqref{eq:divisorcounttriangles} follows.

Combining Equations~\eqref{eq:fiberedproducttriangles} and~\eqref{eq:divisorcounttriangles} yields Claim~\eqref{num:TC-2}, completing the proof.
\end{proof}

\subsection{Independence from the gluing datum}


 In this section, we prove that the free-stabilization maps are independent of the choice of gluing datum:

\begin{prop}\label{prop:independencegluingdatum}
Suppose $\cH$ is diagram for $(Y,\ws)$, $\frd_1$ and $\frd_2$ are two gluing data for free-stabilizing at $w$, and  $T_1$ and $T_2$ satisfy the stabilizing condition~\eqref{eq:stabilizationcondition2} for $\frd_1$ and $\frd_2$, respectively. Write $\cH_1^+$ and $\cH_2^+$ for the two free-stabilized diagrams constructed from $\cH,$ $\frd_1$ and $\frd_2$ (note that $\cH_1^+$ and $\cH_2^+$ differ only in the embedding of $\alpha_0$ and $\beta_0$ in the free-stabilization region). The following diagram commutes up to chain homotopy:
\begin{equation}
 \begin{tikzcd}[column sep=4cm]\CF^-_{J^{\frd_1}}(\cH,\frs)\arrow{r}{\Psi_{J^{\frd_1}\to J^{\frd_2}}}\arrow{d}{S_w^+}& \CF^-_{J^{\frd_2}}(\cH,\frs\arrow{d}{S_w^+})\\
 \CF^-_{J^{\frd_1}(T_1)}(\cH_1^+,\frs)\arrow{r}{\Psi_{(\cH_1^+,J^{\frd_1}(T_1))\to (\cH_2^+,J^{\frd_2}(T_2))}} &\CF^-_{J^{\frd_2}(T_2)}(\cH_2^+,\frs).
 \end{tikzcd} 
 \label{eq:comdiagchangegluedatam}
\end{equation}
An analogous relation holds for the maps $S_w^-$.
\end{prop}
\begin{proof}
Note that if the claim holds for any particular $T_1$ and $T_2$ which satisfy \eqref{eq:stabilizationcondition2}, then it holds for all $T_1$ and $T_2$ which satisfy condition~\eqref{eq:stabilizationcondition2}.

Let us call two gluing data $\frd_1$ and $\frd_2$ \emph{equivalent}, and write $\frd_1\sim \frd_2$, if for all sufficiently large $T_1$ and $T_2$, the diagram in Equation~\eqref{eq:comdiagchangegluedatam}, as well as its analog for $S_{w}^-$, commute up to chain homotopy. The relation of two gluing data being equivalent is transitive and reflexive. It is sufficient to show that all generic gluing data are equivalent.

The proof is an elaboration of the proofs of Propositions~\ref{prop:freestab-Tsufflargeexist} and \ref{prop:free-stabdifferential}. We prove the main statement in 3 steps. Write
\[
\frd_1=(J_1,J_{0,1},D_1,\iota_1) \quad \text{and} \quad \frd_2=(J_2,J_{0,2},D_2,\iota_2).
\]

\setcounter{step}{0}

\begin{step}\label{step:ac1} If $D_1=D_2$, $J_1|_{D_i\times [0,1]\times \R}=J_2|_{D_i\times [0,1]\times \R}$, $J_{0,1}=J_{0,2}$ and $\iota_1=\iota_2$, then $\frd_1\sim \frd_2$.
\end{step}

We can pick an almost complex structure $\tilde{J}$ on $\Sigma\times [0,1]\times \R$ which interpolates $J_1$ and $J_2$ and is cylindrical on $D_i\times [0,1]\times \R$. Since $J_{0,1}=J_{0,2}$ and $\iota_1=\iota_2$, using $\tilde{J}$ we can construct almost complex structures $\tilde{J}(T)$ which interpolate $J^{\frd_1}(T)$ and $J^{\frd_2}(T)$ and agree with $\tilde{J}$ on $(\Sigma\setminus \frac{1}{T} \cdot D)\times [0,1]\times \R$.

 We claim that if $T$ is sufficiently large, then
\begin{equation}
\begin{split}
\Psi_{\tilde{J}(T)}(\xs\times \theta^+)&=\Psi_{\tilde{J}}(\xs)\times \theta^+, \quad \text{and}\\
\Psi_{\tilde{J}(T)}(\xs\times \theta^-)&=\Psi_{\tilde{J}}(\xs)\times \theta^-+\sum_{\ys\in \bT_{\a}\cap \bT_{\b}} C_{\xs,\ys}\cdot  \ys\times \theta^+,
\end{split}
\label{eq:changeacstrucfixedonD}
\end{equation}
for  $C_{\xs,\ys}\in \bF_2[U_{\ws}]$ (depending on $T$, $\frd_1$ and $\frd_2$). Note that Equation~\eqref{eq:changeacstrucfixedonD} is is equivalent to the statement that $\frd_1\sim \frd_2$.

Equation~\eqref{eq:changeacstrucfixedonD} follows from a modification of the proof of Proposition~\ref{prop:free-stabdifferential}. We briefly sketch the argument.

Suppose that $\phi\# \phi_0\in \pi_2(\xs\times x, \ys\times y)$ is a homology class of disks with Maslov index 0. If $T_i$ is a sequence of neck-lengths approaching $+\infty$, then from a sequence $u_i$ of $\tilde{J}(T_i)$ holomorphic curves representing $\phi\# \phi_0,$ we may extract two broken limiting curves, $\cU$ and $\cU_0$, representing $\phi$ and $\phi_0$. From Lemma~\ref{lem:indexdisksonsphere}, we know
\[
0=\mu(\phi\# \phi_0)=\mu(\phi)+\gr(x,y)+n_w(\phi_0).
\]
Equation~\eqref{eq:changeacstrucfixedonD} involves only classes with $\gr(x,y)\ge 0$. Since $\phi$ and $\phi_0$ both have broken representatives, we conclude that $\mu(\phi)\ge 0$. Furthermore,  if $\gr(x,y)=1$, then there are no representatives of $\phi\# \phi_0$. Hence if $\phi\# \phi_0$ has a representative for large $T$ and $\gr(x,y)\ge 0$, then
\[
0=\mu(\phi)=\gr(x,y)=n_w(\phi_0).
\]

Exactly one story of $\cU$ will be $\tilde{J}$-holomorphic, while the others will be $J^{\frd_1}$- or $J^{\frd_2}$-holomorphic. Since $\mu(\phi)=0$,  $\cU$ can consist only of a single, embedded curve $u$, which is $\tilde{J}$-holomorphic. The collection $\cU_0$ will have a matching component $u_0$ which satisfies $\rho^w(u)=\rho^{p}(u_0)$. Using a transversality argument, analogous to the one described in Proposition~\ref{prop:free-stabdifferential}, for generically chosen $J_0$, $u_0$ will be the only curve in $\cU_0$, and $u_0$ will be embedded. From here, the argument proceeds via a gluing argument, analogous to the proof of Proposition~\ref{prop:free-stabdifferential}.

\begin{step}\label{step:ac2} If $J_1=J_2$,  $D_1=D_2$, then $\frd_1\sim \frd_2$.
\end{step}

The transition map $\Psi_{(\cH_1^+,J^{\frd_1}(T_1))\to (\cH_2^+,J^{\frd_2}(T_2))}$ can be written as a composition of the map induced by a diffeomorphism $\phi\colon \Sigma\to \Sigma$ which is supported in a small neighborhood of $\tfrac{1}{2} \cdot D$, which moves the image of the $\alpha_0$ and $\beta_0$ curves under $\iota_1$ to their images under $\iota_2$, followed by a map $\Psi_{\tilde{J}}$ which counts index 0 $\tilde{J}$-holomorphic disks for an almost complex structure $\tilde{J}$ which interpolates $\phi_* J^{\frd_1}(T_1)$ and $J^{\frd_2}(T_2).$

The diagram
\[
 \begin{tikzcd}[column sep=3cm]\CF^-_{J}(\cH,\frs)\arrow{r}{\id }\arrow{d}{S_w^+}& \CF^-_{J}(\cH,\frs\arrow{d}{S_w^+})\\
 \CF^-_{J^{\frd_1}(T_1)}(\cH_1^+,\frs)\arrow{r}{\phi_*} &\CF^-_{\phi_*J^{\frd_1}(T_1)}(\cH_2^+,\frs),
 \end{tikzcd} 
\]
clearly commutes, so it is sufficient to show that the analogous diagram involving $\Psi_{\phi_* J^{\frd_1}(T_1)\to J^{\frd_2}(T_2)}$ also commutes. It is sufficient to show that the diagram commutes whenever $T_1$ and $T_2$ are sufficiently large. Furthermore, using condition~\eqref{eq:stabilizationcondition2}, we can assume $T_1=T_2$

Since $\phi_*$ is the identity outside of a neighborhood of $\tfrac{1}{2}\cdot D$, we can choose $\tilde{J}$ so that it agrees with $J^{\frd_i}(T)$ (for both $i=1,2$) on $ \Sigma\setminus \tfrac{1}{2} \cdot D$. However, equipped with the almost complex structure $J^{\frd_i}(T)$, $(\Sigma\setminus \tfrac{1}{2})\times [0,1]\times \R$ is holomorphically equivalent to $((\Sigma\setminus \tfrac{1}{T}\cdot D)\times [0,1]\times \R,J)$.

For a sequence $T_i$ approaching $+\infty$, we pick a sequence $\tilde{J}_i$ of interpolating almost complex structures between $\phi_*J^{\frd_1}(T_i)$ and $J^{\frd_2}(T_i)$. Given $\tilde{J}_i$-holomorphic representatives $u_i$ of an index 0 class $\phi\# \phi_0$, we extract  a limit to broken curves $\cU$ and $\cU_0$. As in Step~\ref{step:ac1}, this implies that there are no representatives of classes with $\gr(x,y)=1$, and the only remaining classes of interest have 
\[
\mu(\phi)=\gr(x,y)=n_w(\phi_0)=0.
\]
 However, $\cU$ is a broken $J$-holomorphic curve, which is a cylindrical almost complex structure. Since $\mu(\phi)=0$, it follows from transversality that $\cU$ can consist only of a single story, consisting of a representative of the constant class $e_{\xs}$. Since $\phi\# \phi_0$ has Maslov index 0, and $\phi=e_{\xs}$, we clearly must have $\phi_0=\{e_{\theta^+},e_{\theta^-}\}$. On the other hand, it is straightforward to see that the classes $e_{\xs}\times e_x$ will always have $\tilde{J}_i$-holomorphic representatives. The stated formula follows.

\begin{step}If $\frd_1$ and $\frd_2$ are generic, then $\frd_1\sim \frd_2$.
\end{step}

Write $J_i$ as $j_i\times j_{[0,1]\times \R}$ on $D_i$. By the uniformization theorem, there is  a conformal diffeomorphism between $(D_1,j_1)$ and $(D_2,j_2)$ which fixes $w$, which we can assume extends to a diffeomorphism from $\Sigma$ to itself and which fixes $\as\cup \bs\cup \ws\cup \{w\}$.

Writing $\phi$ also for the induced diffeomorphism of $\Sigma\times [0,1]\times \R$, we tautologically have a commutative diagram
\begin{equation}
 \begin{tikzcd}[column sep=2cm]\CF^-_{J_1}(\cH,\frs)\arrow{r}{\phi_* }\arrow{d}{S_w^+}& \CF^-_{\phi_*J_1}(\cH,\frs\arrow{d}{S_w^+})\\
 \CF^-_{J^{\frd_1}(T)}(\cH_1^+,\frs)\arrow{r}{\phi_*} &\CF^-_{J^{\phi_*\frd_1}(T)}(\cH_2^+,\frs).
 \end{tikzcd} 
 \label{eq:pushforwardgluingdataequivalent}
\end{equation}
Since both horizontal arrows in Equation~\eqref{eq:pushforwardgluingdataequivalent} agree with the maps from naturality, it follows that 
\[
\frd_1\sim \phi_* (\frd_1).
\]
Define the gluing datum $\frd':=(\phi_*( J_1), J_{0,2}, D_2, \iota_2)$.
 Step~\ref{step:ac2} shows that 
 \[
 \phi_*(\frd_1)\sim\frd'.
 \] 
 Step~\ref{step:ac1} implies that 
 \[
 \frd'\sim \frd_2.
 \]
 By transitivity of $\sim$, the proof is complete.
\end{proof}

\subsection{Commuting simple stabilizations and free-stabilizations}

\begin{lem}\label{lem:freestabsimplestab}The transition maps for a simple stabilization commute with the free-stabilization maps.
\end{lem}
\begin{proof}We will focus on the positive free-stabilization map $S_w^+$ since $S_w^-$ is the dual map. Write $\sigma$ for the simple stabilization map. Suppose that the following hold:
\begin{enumerate}
\item  $\cH=(\Sigma,\as,\bs,\ws)$ is a diagram for $(Y,\ws)$;
\item $\cH^+=(\Sigma,\as\cup \{\alpha_0\}, \bs\cup \{\beta_0\}, \ws\cup \{w\})$ is a free-stabilization of $\cH$;
\item  $\cH_\sigma=(\Sigma',\as\cup \{\alpha'\}, \bs\cup \{\beta'\}, \ws)$ is a simple-stabilization of $\cH$;
\item $\cH_\sigma^+=(\Sigma',\as\cup \{\alpha_0,\alpha'\}, \bs\cup \{\beta_0,\beta'\}, \ws\cup \{w\})$ is obtained from $\cH$ by performing both stabilizations.
\end{enumerate}
The formulas for $\sigma$ and $S_w^+$ appear to commute, however it is not clear that a single almost complex structure can be chosen to compute both maps, since $\sigma$ requires stabilizing condition~\eqref{eq:stabilizationcondition1} to be satisfied, while $S_w^+$ requires condition~\eqref{eq:stabilizationcondition2} to be satisfied.

The main claim amounts to showing the following subclaim:

\begin{subclaim}\label{subclaim:doubleneckstretch} If $\ve{T}=(T,T')$ is a pair of neck lengths, write $J(\ve{T})$ for the  almost complex structure on $\Sigma'\times [0,1]\times \R$ which has had a neck of length $T$ added in the free-stabilization region, and a neck of length $T'$ inserted in the the simple-stabilization region. If $\ve{T}_1$ and $\ve{T}_2$ are two pairs, all of whose components are sufficiently large, then there is a non-cylindrical almost complex structure $\tilde{J}$ on $\Sigma'\times [0,1]\times \R$, interpolating $J(\ve{T}_1)$ and $J(\ve{T}_2)$ such that
\begin{align*}
\Psi_{\tilde{J}}(\xs\times c\times \theta^+)&=\xs\times c\times \theta^+\\
\Psi_{\tilde{J}}(\xs\times c\times \theta^-)&=\xs\times c\times \theta^-+\sum_{\ys\in \bT_{\a}\cap \bT_{\b}} C_{\xs,\ys}\cdot \ys\times c\times \theta^+
\end{align*}
for $C_{\xs,\ys}\in \bF_2[U_{\ws}]$.
\end{subclaim}
The proof of Subclaim~\ref{subclaim:doubleneckstretch} is a double neck-stretching argument similar to the proof of Proposition~\ref{prop:freestab-Tsufflargeexist}. 

 We can write any homology class on $\cH^+_\sigma$ as $\phi\#\phi_0\# \phi'$, where $\phi\in \pi_2(\xs,\ys)$ is a class on $\cH$, $\phi_0\in \pi_2(x,y)$ is a class on $(S^2,\alpha_0,\beta_0)$, and $\phi'\in \pi_2(c,c)$ is a homology class on  $(\bT^2,\alpha',\beta')$. Adapting the proof of Lemma~\ref{lem:indexdisksonsphere}, we have
\begin{equation}
\mu(\phi\# \phi_0\# \phi')=\mu(\phi)+\gr(x,y)+2n_{w}(\phi_0).\label{eq:Maslovindexfree+simplestab}
\end{equation}

Suppose $\ve{T}_{1,i}$ and $\ve{T}_{2,i}$ are sequences of pairs of neck-lengths, all of whose components approach $+\infty$. We can pick a sequence of non-cylindrical almost complex structures $\tilde{J}_i$ interpolating $J(\ve{T}_{1,i})$ and $J(\ve{T}_{2,i})$ such that $(\Sigma'\times [0,1]\times \R, \tilde{J}_i)$ contains an almost complex submanifold equivalent to $((\Sigma\setminus N_i)\times [0,1]\times \R,J)$, for a nested sequence of open subsets $N_i\subset N_{i-1}\subset \Sigma$ such that $\bigcap_i N_i$ consists of exactly two points (the two connected sum points), and a cylindrical almost complex structure $J$ on $\Sigma\times [0,1]\times \R$.

If $\phi\# \phi_0\# \phi'$ admits a sequence $u_i$ of $\tilde{J}_i$-holomorphic representatives, then we can extract a broken $J$-holomorphic representative of $\phi$ on $\Sigma\times [0,1]\times \R$. By examining Equation~\eqref{eq:Maslovindexfree+simplestab}, we conclude that if $\gr(x,y)\ge 0$ (which are the only cases relevant to Subclaim~\ref{subclaim:doubleneckstretch}), then
\[
\mu(\phi)=n_w(\phi_0)=0.
\]
Since $\phi$ has a broken representative, and $\mu(\phi)=0$, we conclude that $\phi$ must be the constant class $e_{\xs}$, by transversality. Since $n_w(\phi_0)=0$, this also implies that $\phi_0$ is one of the constant classes $e_{\theta^+}$ or $e_{\theta^-}$, and $\phi'$ is the constant classes $e_{c}$. Conversely, the constant class $e_{\xs}\times e_{\theta^{\pm}} \times e_{c}$ always has a $\tilde{J}_i$-holomorphic representative. The subclaim follows, and so does the main claim.
\end{proof}

\subsection{Invariance of the free-stabilization maps}

In this section, we put our previously results together and prove invariance of the free-stabilization maps:

\begin{thm}\label{thm:invariancefreestab}The free-stabilization maps $S^+_w$ and $S_w^-$ defined in Section~\ref{subsec:gluingdata}, determine well defined chain maps on the level of transitive systems of chain complexes.
\end{thm}

Before we prove Theorem~\ref{thm:invariancefreestab}, we need the following topological result about embeddings of Heegaard diagrams:

\begin{lem}\label{lem:Heegaardmoveswithextrabasepoints} Suppose $(Y,\ws)$ is a multi-pointed 3-manifold, and $\ve{p}\subset Y\setminus \ws$ is a finite collection of points. If $\cH_1=(\Sigma_1,\as_1,\bs_1,\ws)$ and $\cH_2=(\Sigma_2,\as_2,\bs_2,\ws)$ are two Heegaard diagrams such that $\ve{p}\subset \Sigma_i\setminus (\as_i\cup \bs_i)$ for $i=1,2$, then $\cH_1$ and $\cH_2$ can be connected by a sequence of the following moves:
\begin{enumerate}
\item\label{move:HDmoveswithp1} Handleslides and isotopies of the $\as$ and $\bs$ curves (possibly passing over the points in $\ve{p}$).
\item\label{move:HDmoveswithp2} Simple stabilizations, away from $\ws\cup \ve{p}$.
\item\label{move:HDmoveswithp3} Isotopies $\phi_t\colon \Sigma\to Y$ of the Heegaard surface which are fixed on $\ws\cup \ve{p}$ for all $t$.
\end{enumerate} 
\end{lem}
\begin{proof} Consider Morse functions with gradient like vector fields $(f_1,v_1)$ and $(f_2,v_2)$ on $Y$ which induce  $\cH_1$ and $\cH_2$ (i.e. $\Sigma_i$ is a middle level set of $f_i$, and $\as_i$ is the intersection of the ascending manifolds of the index 1 critical points of $f_i$ with $\Sigma$, and  $\bs_i$ is the intersection of the descending manifolds of the index 2 critical points of $f_i$). We may pick a path $(f_t,v_t)_{t\in [1,2]}$ of functions with gradient like vector fields, connecting $(f_1,v_1)$ and $(f_2,v_2)$, such that the following hold:
\begin{enumerate}
\item The pair $(f_t,v_t)$ is Morse-Smale, at all but finitely many $t$.
\item At the finitely many $t$ where $f_t$ fails to be Morse, an index 1/2 birth or death singularity occurs (corresponding to a simple stabilization of the Heegaard surface). 
\item At the finitely many $t$ where the pair $(f_t,v_t)$ fails to be Smale, a handleslide between two $\as$ curves or $\bs$ curves occurs.
\end{enumerate}
This implies that the three stated moves suffice to connect $\cH_1$ and $\cH_2$, except that we have not yet ensured that an isotopy appearing in an instance of Move~\eqref{move:HDmoveswithp3} can be taken to respect the points $\ve{p}$.

 Generically, we can assume the following:
\begin{enumerate}
\item $\ve{p}$ is disjoint from the flowlines of any $v_t$ passing through $\ws$.
\item $\ve{p}$ is disjoint from the descending manifolds of any index 1 critical point of any $f_t$.
\item $\ve{p}$ is disjoint from the ascending manifolds of any index 2 critical point of any $f_t$.
\end{enumerate}
Let us write $\Sigma_t$ for a middle level set of each $f_t$. We can thus assume that for each $p\in \ve{p}$ and $t\in [1,2]$, there is a flowline $f_{t,p}$ which connects $p$ with a point on $\Sigma_t\setminus \ws$. By performing an isotopy of each $\Sigma_t$ in a neighborhood of $f_{t,p}$, we can ensure that $\ve{p}\subset \Sigma_t$ for all $t$. Consequently, if $\phi_t$ is an isotopy appearing in an instance of Move~\eqref{move:HDmoveswithp3}, we can assume that $\ve{p}\subset \phi_t(\Sigma)$.

Finally, if $\phi_t\colon \Sigma\to Y$ is an isotopy which fixes $\ws$ and such that $\ve{p}\subset \phi_t(\Sigma)$ for all $t$, there will generically be finitely many $t$ where the image of a point in $\ve{p}$ intersects an $\as$ or $\bs$ curve. Such an isotopy $\phi_t$ may thus be decomposed as a sequence of isotopies where $\ve{p}$ is fixed, as well as isotopies of the type appearing in Move~\eqref{move:HDmoveswithp1}, where the $\as$ and $\bs$ curves are moved but the Heegaard surface is fixed.
\end{proof}

\begin{proof}[Proof of Theorem~\ref{thm:invariancefreestab}]As $S_w^-$ is dual to $S_w^+$, we focus on the claim for $S_w^+$. The proof amounts to showing that if $\cH_1=(\Sigma_1,\as_1,\bs_1,\ws)$ and $\cH_2=(\Sigma_2,\as_2,\bs_2,\ws)$ are two diagrams such that $w\in \Sigma_i\setminus (\as_i\cup \bs_i)$, and $\frd_1$ and $\frd_2$ are two gluing data, then the following diagram commutes up to chain homotopy:
\begin{equation}
 \begin{tikzcd}[column sep=4cm]\CF^-_{J^{\frd_1}}(\cH_1,\frs)\arrow{r}{\Psi_{J^{\frd_1}\to J^{\frd_2}}}\arrow{d}{S_w^+}& \CF^-_{J^{\frd_2}}(\cH_2,\frs\arrow{d}{S_w^+})\\
 \CF^-_{J^{\frd_1}(T_1)}(\cH_1^+,\frs)\arrow{r}{\Psi_{(\cH_1^+,J^{\frd_1}(T_1))\to (\cH_2^+,J^{\frd_2}(T_2))}} &\CF^-_{J^{\frd_2}(T_2)}(\cH_2^+,\frs).
 \end{tikzcd}
 \label{eq:freestabnaturality}
\end{equation}
It suffices to show that the diagram in Equation~\eqref{eq:freestabnaturality} commutes when $\cH_1$ and $\cH_2$ differ by one of the moves listed in Lemma~\ref{lem:Heegaardmoveswithextrabasepoints} (for $\ve{p}=\{w\}$). First, we consider Move~\eqref{move:HDmoveswithp1}, isotopies and handleslides of the $\as$ and $\bs$ curves on a fixed Heegaard surface. The transition maps for handleslides and isotopies of the $\as$ and $\bs$ curves can be computed using a triangle map. The triangle counts from Theorem~\ref{thm:freestabilizetriangles}  imply that the free-stabilization maps commute with such transition maps.  Move~\eqref{move:HDmoveswithp2}, commutation of the free-stabilization maps with the simple stabilization maps, follows from Lemma~\ref{lem:freestabsimplestab}. Finally, commutation with respect to Move~\eqref{move:HDmoveswithp3}, isotopies of the Heegaard surface which fix $\ws\cup \{w\}$, is a tautology.
\end{proof}

\subsection{Commuting free-stabilization and relative homology maps}

\begin{lem}\label{lem:free-stab-H1-act-commute}Let $(Y,\ws)$ be a multi-pointed 3-manifold, and let $w\in Y\setminus \ws$. Suppose that $\lambda$ is either a closed path in $Y$, or a path which connects two basepoints in $\ws$. Then
\[
S_w^+\circ A_{\lambda}+A_\lambda\circ  S_w^+\simeq 0\quad \text{and} \quad S_w^-\circ A_\lambda+A_\lambda \circ S_w^-\simeq 0.
\]
\end{lem}
\begin{proof} For simplicity, let us consider the claim involving $S_w^+$, since $S_w^-$ is the dual map.

Let $\cH=(\Sigma,\as,\bs,\ws)$ be a diagram for $(Y,\ws)$, and let $\cH^+_w$ be its free-stabilization at $w$. We will show that for sufficiently stretched almost complex structure
\[
S_w^+\circ A_\lambda+A_{\lambda}\circ S_w^+=0.
\]

Suppose that $\xs\in \bT_{\a}\cap \bT_{\b}$ is an intersection point on $\cH$. By definition
\begin{equation}
(S_w^+\circ A_{\lambda})(\xs)=\sum_{\substack{\phi\in \pi_2(\xs,\ys)\\ \mu(\phi)=1}} a(\lambda, \phi) \# \hat{\cM}_{J^\frd}(\phi)\cdot U_{\ws}^{n_{\ws}(\phi)}\cdot (\ys\times \theta^+),\label{eq:freestabrelhom1}
\end{equation}
while
\begin{equation}
(A_{\lambda}\circ S_w^+)(\xs)=\sum_{\substack{\phi\# \phi_0\in \pi_2(\xs\times \theta^+, \ys\times y)\\ \mu(\phi\# \phi_0)=1}} a(\lambda,\phi\# \phi_0)  \# \hat{\cM}_{J^\frd(T)}(\phi\# \phi_0)\cdot U_{\ws}^{n_{\ws}(\phi)}U_w^{n_w(\phi_0)}\cdot (\ys\times y).\label{eq:freestabrelhom2}
\end{equation}

We can assume that $\lambda$ is disjoint from the free-stabilization region on $\Sigma$. Consequently 
\begin{equation}
a(\lambda,\phi\# \phi_0)=a(\lambda,\phi).\label{eq:alambdafreestab}
\end{equation}
By following the proof of Proposition~\ref{prop:free-stabdifferential}, we see that there are two types of index 1 classes in any $\pi_2(\xs\times \theta^+, \ys\times y)$ which can have holomorphic representatives for sufficiently large $T$.

The first type of class which could have representatives  has $y=\theta^-$, $\phi=e_{\xs}$, and has $\phi_0$ equal to a bigon supported in the free-stabilization region. For such a class, we have $a(\lambda,\phi\# \phi_0)=a(\lambda,e_{\xs})=0$, by Equation~\eqref{eq:alambdafreestab}, so such classes make no contribution to Equation~\eqref{eq:freestabrelhom2}.

The second type of class which could have holomorphic representatives has $\mu(\phi)=1$, $\phi_0\in \pi_2(\theta^+,\theta^+)$ and $n_w(\phi_0)=0$. Using the counts of holomorphic disks representing classes in $\pi_2(\xs\times \theta^+, \ys\times \theta^+)$ established in Equation~\eqref{eq:diskcountsstretchedAC} and~\eqref{eq:OS'scountmatched}, together with Equation~\eqref{eq:alambdafreestab}, we conclude that Equations~\eqref{eq:freestabrelhom1} and ~\eqref{eq:freestabrelhom2} coincide. The proof is complete.
\end{proof}

\subsection{Compositions of free-stabilizations}

We now prove that free-stabilization maps for different basepoints may be commuted.

\begin{prop}\label{prop:free-stabs-commute}For any  $\circ_1,\circ_2\in \{+,-\}$, the free-stabilization maps satisfy
 \[
S_{w_1}^{\circ_1}S_{w_2}^{\circ_2}\simeq S_{w_2}^{\circ_2}S_{w_1}^{\circ_1},
\]
as morphisms of transitive systems of chain complexes.
\end{prop}
\begin{proof}The proof follows from a double neck stretching argument similar to Lemma~\ref{lem:freestabsimplestab}. We consider a diagram on which both free-stabilizations have been performed, with underlying Heegaard surface $\Sigma$. Let $J$ be a cylindrical almost complex structure on $\Sigma\times [0,1]\times \R$, viewed as a complex structure for the unstabilized diagram. If $\ve{T}=(T_1,T_2)$ is a pair of neck lengths, let $J(\ve{T})$ denote an almost complex structure obtained by inserting two necks of length $T_1$ and $T_2$.

We recall that the free-stabilization maps require stabilizing condition~\eqref{eq:stabilizationcondition2} to be satisfied, however it is not immediately clear that condition~\eqref{eq:stabilizationcondition2} can be achieved on two necks simultaneously.

We prove the following:

\begin{subclaim}\label{subclaim:neckstretching} Suppose that $\ve{T}$ and $\ve{T}'$ are two pairs of neck-lengths for free-stabilizing at $w_1$ and $w_2$. If all of components of $\ve{T}$ and $\ve{T}'$ are sufficiently large, then a non-cylindrical almost complex structure $\tilde{J}$ may be chosen which interpolates $J(\ve{T})$ and $J(\ve{T}')$ and satisfies
\begin{equation}
\begin{split}
\Psi_{\tilde{J}}(\xs\times \theta_1^+\times \theta_2^+)&= \xs\times \theta_1^+\times \theta_2^+
\\
\Psi_{\tilde{J}}(\xs\times \theta_1^+\times \theta_2^-)&=\xs\times \theta_1^+\times \theta_2^-+\sum_{\ys\in \bT_{\a}\cap \bT_{\b}} \left( C_{\xs,\ys}^1 \cdot  \ys\times \theta_1^+\times \theta_2^++ C_{\xs,\ys}^2\cdot  \ys\times \theta_1^-\times \theta_2^+ \right)
\\
\Psi_{\tilde{J}}(\xs\times \theta_1^-\times \theta_2^+)&=\xs\times \theta_1^-\times \theta_2^++\sum_{\ys\in \bT_{\a}\cap \bT_{\b}} \left(C_{\xs,\ys}^3 \cdot \ys\times \theta_1^+\times \theta_2^++C_{\xs,\ys}^4 \cdot \ys\times \theta_1^+\times \theta_2^-\right)
\\
\Psi_{\tilde{J}}(\xs\times \theta_1^-\times \theta_2^-)&=\xs\times \theta_1^-\times \theta_2^-
\\
+\sum_{\ys\in \bT_{\a}\cap \bT_{\b}} &\left(C_{\xs,\ys}^5\cdot  \ys\times \theta_1^+\times \theta_2^-+C_{\xs,\ys}^6\cdot  \ys\times \theta_1^-\times \theta_2^++C_{\xs,\ys}^7\cdot  \ys\times \theta_1^+\times\theta_2^+\right),
\end{split}
\label{eq:double-neck-stretch-free-stab}
\end{equation}
for various $C_{\xs,\ys}^i\in \bF_2[U_{\ws}]$, which depend on $\tilde{J}$.
\end{subclaim}

Before proving Subclaim~\ref{subclaim:neckstretching}, we briefly explain why it implies the main claim. Pick $\ve{T}=(T_1,T_2)$ so that $T_2$ is large enough to compute the free-stabilization maps at $w_2$, and then pick $T_1\gg T_2$ which is large enough to compute the free-stabilization maps at $w_1$ (after having already performed the stabilization at $w_2$). Next, pick $\ve{T}'=(T_1',T_2')$ so that $T_2'\gg T_1'$, so that $T_1'$ is large enough to compute the free-stabilization  maps at $w_1$ first, and then subsequently compute the free-stabilization maps at $w_2'$. Hence, to compute $S_{w_2}^+S_{w_1}^+$ we may use $J(\ve{T}')$, while to compute $S_{w_1}^+S_{w_2}^+$, we may use $J(\ve{T})$. To compare the two compositions on the level of morphisms of transitive systems of chain complexes, we must compose with the transition map $\Psi_{J(\ve{T})\to J(\ve{T}')}$. With this in mind, the first line of Equation~\eqref{eq:double-neck-stretch-free-stab} implies that $S_{w_2}^+ S_{w_1}^+\simeq S_{w_1}^+  S_{w_2}^+$.

Similarly, the first and the second lines of Equation~\eqref{eq:double-neck-stretch-free-stab} together imply that $S_{w_1}^+  S_{w_2}^-\simeq S_{w_2}^- S_{w_1}^+$. The first and third lines together imply that $S_{w_1}^-S_{w_2}^+\simeq S_{w_2}^+ S_{w_1}^-$. Finally, all four lines together imply that $S_{w_1}^-S_{w_2}^-\simeq S_{w_2}^- S_{w_1}^-$.

We now proceed to prove Subclaim~\ref{subclaim:neckstretching}. The transition map $\Psi_{\tilde{J}}$ counts $\tilde{J}$-holomorphic representatives of Maslov index 0 homology classes of disks. If $\phi\# \phi_0^1\# \phi_0^2\in \pi_2(\xs\times x_1\times x_2, \ys\times y_1\times y_2)$ is a homology class on the doubly free-stabilized diagram, then adapting the proof of Equation~\eqref{eq:indexbrokenlimitacstr} gives
\begin{equation}
\mu(\phi\# \phi_0^1\# \phi_0^2)=\mu(\phi)+\gr(x_1,y_1)+\gr(x_2,y_2)+2 n_{w_1}(\phi)+2n_{w_2}(\phi_2).\label{eq:Maslovindexdoublyfreestab}
\end{equation}

To prove Equation~\eqref{eq:double-neck-stretch-free-stab} it suffices to make the following three observations:
\begin{enumerate}
\item\label{observation1}  Assuming $\ve{T}$ and $\ve{T}'$ are sufficiently large, if $\mu(\phi)=0$ and $\gr(x_1,y_1)+\gr(x_2,y_2)>0$, then $\cM_{\tilde{J}}(\phi)$ is empty;
\item\label{observation2} Assuming $\ve{T}$ and $\ve{T}'$ are sufficiently large,  if $\mu(\phi)=0$,  $\gr(x_1,y_1)=\gr(x_2,y_2)=0$ and $\cM_{\tilde{J}}(\phi)$ is nonempty, then  $\phi=e_{\xs}\times e_{x_1}\times e_{x_2}$.
\item If $\phi$ is a class with $\gr(x_1,y_1)+\gr(x_2,y_2)\le 0$ but $\gr(x_1,y_1)$ and $\gr(x_2,y_2)$ are not both zero, then the count of $\cM_{\tilde{J}}(\phi)$ is irrelevant to Equation~\eqref{eq:double-neck-stretch-free-stab}. 
\end{enumerate}
Observations~\eqref{observation1} and~\eqref{observation2} are proven similar to the proof of Proposition~\ref{prop:freestab-Tsufflargeexist}. For sufficiently large $\ve{T}$ and $\ve{T}'$, by picking $\tilde{J}$ appropriately, we can ensure that $\phi$ has a broken representative for a cylindrical almost complex structure. In particular $\mu(\phi)\ge 0$. Consequently, if $\gr(x_1,y_1)+\gr(x_2,y_2)>0$, then Equation~\eqref{eq:Maslovindexdoublyfreestab} can never be zero, so there are no such classes with representatives. If Equation~\eqref{eq:Maslovindexdoublyfreestab} is zero and $\gr(x_1,y_1)=\gr(x_2,y_2)=0$, then also $\mu(\phi)=0$, so $\phi$ is a constant homology class, by transversality. Using additionally the vertex multiplicities near $x_1$ and $x_2$, we conclude that $\phi_0^1$ and $\phi_0^2$ must also be constant homology classes, so the second observation follows. The third observation is straightforward. This establishes Subclaim~\ref{subclaim:neckstretching}, and hence the main claim.
\end{proof}

We now prove a simple relation involving the free-stabilization maps:

\begin{lem}\label{lem:disjoint-strand-induces-zero} Suppose $(Y,\ws)$ is a multi-pointed 3-manifold and $w\in Y\setminus \ws$. Then
\[
S_w^-S_w^+\simeq 0.
\]
\end{lem}
\begin{proof}The relation follows immediately from the formulas for the free-stabilization maps in Equation~\eqref{def:freestabilizationmaps}.
\end{proof}

\section{Graph action map}
\label{sec:graphaction}

If $\cG=(\Gamma,\ws_0,\ws_1)$  is a flow-graph in $Y$, and $\sigma$ is a coloring of $\Gamma$ (see Section~\ref{sec:flow-graph} and Definition~\ref{def:coloring} for the definitions of a flow-graph, and coloring, respectively), in this section we construct two \emph{graph action maps}
\[
\frA_{\cG}, \frB_{\cG}\colon \CF^-(Y, \ws_0^{\sigma_0}, \frs)\to \CF^-(Y,\ws_1^{\sigma_1},\frs),
\]
where $\sigma_i:=\sigma|_{\ws_i}$.

Unlike in Section~\ref{sec:flow-graph}, where flow-graphs are allowed to be immersed, we focus on embedded flow-graphs in this section. Upgrading the results to immersed flow-graphs is straightforward (the only complication is that when $\ws_0\cap\ws_1\neq \emptyset$, one must define $\frA_\cG$ and $\frB_{\cG}$ as compositions of two graph action maps).

We think of the graph action map as a restricted version of the full graph TQFT. Indeed we will later prove that the graph cobordism map for $([0,1]\times Y, \Gamma)$ is equal to the graph action map for the graph obtained by projecting $\Gamma$ into $Y$; See Part~\eqref{claim:W=[0,1]xY-project-Gamma} of Theorem~\ref{thm:A-with-enough-ends}, below.

To define the graph action map, we fix a decomposition of $\Gamma$ into edges and vertices. We write $V(\Gamma)$ and $E(\Gamma)$ for the vertices and edges of $\Gamma$, respectively. Although we require such a decomposition for our construction, the resulting graph action maps turn out to be unchanged by subdivision; see Lemma~\ref{lem:subdivision-invariance}. Our construction of $\frA_{\cG}$ and $\frB_{\cG}$ also requires picking a further decomposition of  $\cG$ into elementary flow-graphs (see Definition~\ref{def:elementary-ribbon-flow-graphs}), as well as a choice of absolute lift of the ribbon structure.

 In this section, we prove the following:

\begin{thm}\label{thm:graph-action}Suppose $\cG=(\Gamma,\ws_0,\ws_1)$ is a ribbon flow-graph in $Y$. The maps $\frA_{\cG}$ and $\frB_{\cG}$ described in this section satisfy the following:
\begin{enumerate}[label=\emph{(\alph*)}, ref=\alph*]
\item\label{thm:graph-action:a} The maps $\frA_{\cG}$ and $\frB_{\cG}$ are independent from the choice of decomposition into elementary flow-graphs.
\item\label{thm:graph-action:b} The maps $\frA_{\cG}$ and $\frB_{\cG}$ are functorial: if $\cG$ is a flow-graph from $\ws_0$ to $\ws_1$, $\cG'$ is a flow-graph from $\ws_1$ to $\ws_2$, then
\[
\frA_{\cG'}\circ \frA_{\cG}=\frA_{\cG'\cup \cG} \quad \text{and} \quad \frB_{\cG'}\circ \frB_{\cG}=\frB_{\cG'\cup \cG}.
\]
\item\label{thm:graph-action:c} The maps $\frA_{\cG}$ and $\frB_{\cG}$ are independent from the choice of absolute lift of the ribbon structure used in their construction.
\end{enumerate}

\end{thm}

\subsection{Constructing the graph action maps}
\label{sec:Cerf-decomps-graphs}

Before defining the graph action maps, we establish some terminology.

\begin{define}
 A \emph{strong} ribbon flow-graph is a ribbon flow-graph equipped with an absolute lift of the cyclic ordering at each vertex.
\end{define}

\begin{define}\label{def:elementary-ribbon-flow-graphs} A ribbon flow-graph $\cG=(\Gamma,\ws_0,\ws_1)$ is \emph{elementary} if is one of the following 3 subtypes:
\begin{enumerate}[label= ($ER$-\arabic*), ref=$ER$-\arabic*]
\item\label{ERtype:1} (\emph{translation}) $|\ws_0|=|\ws_1|$ and each edge of $\Gamma$ connects a vertex of $\ws_0$ to a vertex of $\ws_1$.
\item \label{ERtype:2} (\emph{interior vertex}) There is a single vertex $w_s$ of $\Gamma$ which is not contained in $\ws_0\cup \ws_1$, and all edges of $\Gamma$ connect either $\ws_0$ to $\ws_1$, or a point of $\ws_0\cup \ws_1$ to $w_s$. We call $w_s$ the \emph{special vertex}.
\item \label{ERtype:3} (\emph{local extrema}) $|\ws_0|=|\ws_1|\pm 2$, and all but one edge connect $\ws_0$ to $\ws_1$. There is a single edge $e_0$ which connects two vertices of $\ws_0$ to each other, or connects two distinct vertices of $\ws_1$ together. We call $e_0$ the \emph{special edge}. There are two subtypes:
\begin{enumerate}[label= ($ER$-3\roman*), ref=$ER$-3\roman*]
\item\label{ERtype:3i} (\emph{local max}) The special edge connects two vertices of $\ws_0$.
\item\label{ERtype:3ii} (\emph{local min}) The special edge connects two vertices of $\ws_1$.
\end{enumerate}
\end{enumerate}
\end{define}

\begin{define}\label{def:initialslopes} Suppose that $\cG=(\Gamma,\ws_0,\ws_1)$ is an elementary ribbon flow-graph of type~\eqref{ERtype:2}, with special vertex $w_s$. If $e$ is an edge which is adjacent $w_s$, we say that $e$ has \emph{positive initial slope} if $e$ connects $w_s$ to $\ws_1$, and we say that $e$ has \emph{negative initial slope} if $e$ connects $w_s$ to $\ws_0$.
\end{define}

\begin{figure}[ht!]
\centering
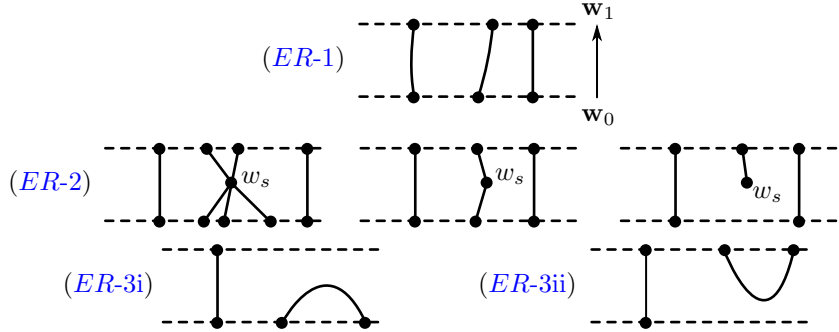
\caption{\textbf{Examples of elementary ribbon flow-graphs.} The top and bottom dashed lines indicate the incoming and outgoing vertices.}\label{fig::25}
\end{figure}

\begin{define}
\begin{enumerate}
\item We say that a (strong) ribbon flow-graph $\cG'=(\Gamma',\ws_0,\ws_1)$ is a \emph{subdivision} of $\cG=(\Gamma,\ws_0,\ws_1)$ if $\Gamma'$ is obtained by adding extra vertices to the interior of edges of $\Gamma$. If $\cG'$ is a strong ribbon flow-graph, we allow the new vertices to be added with either absolute ordering.
\item We say that two flow-graphs $\cG=(\Gamma,\ws_{0},\ws_1)$ and $\cG'=(\Gamma',\ws_0',\ws_1')$ are \emph{subdivision equivalent} if $\cG$ and $\cG'$ have a common subdivision (in particular, $\Gamma=\Gamma'$ as subsets of $Y$).
\end{enumerate}
\end{define}

We now define a notion of decomposition into elementary pieces for ribbon flow-graphs:

\begin{define} If $\cG=(\Gamma,\ws_0,\ws_1)$ is a ribbon flow-graph in $Y$, a \emph{Cerf decomposition} of $\cG$ is a sequence of ribbon flow-graphs $\cG_1,\dots, \cG_n$, where $\cG_i=(\Gamma_i,\ws_{i,0}, \ws_{i,1})$, such that the following hold:
\begin{enumerate}
\item Each $\cG_i$ is an elementary ribbon flow-graph.
\item $\ws_{1,0}=\ws_0$ and $\ws_{n,1}=\ws_1$.
\item $\ws_{0,i+1}=\ws_{1,i}$ for $i\in \{1,\dots, n-1\}$.
\item The concatenation $\cG_n\circ \cdots \circ \cG_1$ is a subdivision of $\cG$.
\item Each vertex of $V(\Gamma)\setminus (\ws_0\cup \ws_1)$ is a special vertex of some $\cG_i$. Furthermore, if $w_s$ is a special vertex of any $\cG_i$, then $w_s$ is a vertex in $V(\Gamma)$.
\end{enumerate}
\end{define}

We now define the following moves between Cerf decompositions of flow-graphs:

\begin{define}
\label{def:Cerfmoves} We say two Cerf decompositions of a flow-graph $\cG=(\Gamma,\ws_0,\ws_1)$ in $Y$ are \emph{Cerf equivalent} if they can be connected by a sequence of the following moves, or their inverses:
\begin{enumerate}[label=($CM$-\arabic*), ref=$CM$-\arabic*]
\item\label{move:CM1} (\emph{Level splitting}) Replacing an elementary flow-graph $\cG$ with a subdivision equivalent composition $\cG_2\circ \cG_1$ such that one of $\cG_1$ and $\cG_2$ is of the same type as $\cG$, and the other is of type~\eqref{ERtype:1}.
\item\label{move:CM2} (\emph{Critical point birth/death}) Replacing an elementary flow-graph $\cG$ of type \eqref{ERtype:1} with a subdivision equivalent composition $\cG_2\circ \cG_1$, such that $\cG_2$ and $\cG_1$ are of type \eqref{ERtype:3i} and \eqref{ERtype:3ii}, respectively.
\item\label{move:CM3} (\emph{Changing an initial slope})  Replacing an elementary flow-graph $\cG$ of type~\eqref{ERtype:2} with a subdivision equivalent composition $\cG_2\circ \cG_1$ such exactly one of $\cG_1$ and $\cG_2$ is of type \eqref{ERtype:3}, and the other is of type \eqref{ERtype:2} (with the same special vertex as $\cG$).
\item\label{move:CM4} (\emph{Critical value swap})  Replacing the composition of two consecutive elementary flow-graphs $\cG_2\circ \cG_1$ of type~\eqref{ERtype:3} with a subdivision equivalent composition $\cG_2'\circ \cG_1'$ such that both $\cG_2'$ and $\cG_1'$ are also of type~\eqref{ERtype:3}. Furthermore, the special edges of $\cG_2$ and $\cG_1$ are disjoint, and the special edges of $\cG_2'$ and $\cG_1'$ are disjoint.
\item\label{move:CM5} (\emph{Vertex/critical value swap)} Replacing a composition $\cG_2\circ \cG_1$, where $\cG_2$ is of type~\eqref{ERtype:2} and $\cG_1$ of type~\eqref{ERtype:3}, with a subdivision equivalent composition $\cG_2'\circ \cG_1'$ where $\cG_2'$ is of type~\eqref{ERtype:3} and $\cG_1'$ is of type~\eqref{ERtype:2}. Furthermore, we assume that  there is no path in $\cG_2\circ \cG_1$ connecting the special vertex of $\cG_2$ to the special edge of $\cG_1$.
\item\label{move:CM6} (\emph{Vertex value swap}) Replacing a composition $\cG_2\circ \cG_1$ of two elementary ribbon graphs of type~\eqref{ERtype:2} with a subdivision equivalent composition $\cG_2'\circ \cG_1'$, where $\cG_2'$ and $\cG_1'$ are also both of type~\eqref{ERtype:2}. Furthermore, we assume that there is no path in $\cG_2\circ \cG_1$ connecting the two special edges of $\cG_2$ and $\cG_1$.
\end{enumerate}
\end{define}
Examples of the moves from Definition~\ref{def:Cerfmoves} are shown in Figure~\ref{fig::26}.

\begin{figure}[ht!]
\centering
\begingroup%
  \makeatletter%
  \providecommand\color[2][]{%
    \errmessage{(Inkscape) Color is used for the text in Inkscape, but the package 'color.sty' is not loaded}%
    \renewcommand\color[2][]{}%
  }%
  \providecommand\transparent[1]{%
    \errmessage{(Inkscape) Transparency is used (non-zero) for the text in Inkscape, but the package 'transparent.sty' is not loaded}%
    \renewcommand\transparent[1]{}%
  }%
  \providecommand\rotatebox[2]{#2}%
  \newcommand*\fsize{\dimexpr\f@size pt\relax}%
  \newcommand*\lineheight[1]{\fontsize{\fsize}{#1\fsize}\selectfont}%
  \ifx\svgwidth\undefined%
    \setlength{\unitlength}{231.34274832bp}%
    \ifx\svgscale\undefined%
      \relax%
    \else%
      \setlength{\unitlength}{\unitlength * \real{\svgscale}}%
    \fi%
  \else%
    \setlength{\unitlength}{\svgwidth}%
  \fi%
  \global\let\svgwidth\undefined%
  \global\let\svgscale\undefined%
  \makeatother%
  \begin{picture}(1,1.08855625)%
    \lineheight{1}%
    \setlength\tabcolsep{0pt}%
    \put(0,0){\includegraphics[width=\unitlength,page=1]{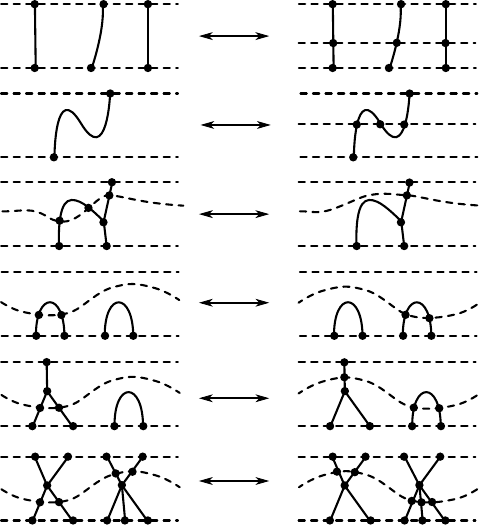}}%
    \put(0.48592612,1.03835407){\color[rgb]{0,0,0}\makebox(0,0)[t]{\lineheight{1.25}\smash{\begin{tabular}[t]{c}\eqref{move:CM1}\end{tabular}}}}%
    \put(0.48916803,0.85390814){\color[rgb]{0,0,0}\makebox(0,0)[t]{\lineheight{1.25}\smash{\begin{tabular}[t]{c}\eqref{move:CM2}\end{tabular}}}}%
    \put(0.48592612,0.6689148){\color[rgb]{0,0,0}\makebox(0,0)[t]{\lineheight{1.25}\smash{\begin{tabular}[t]{c}\eqref{move:CM3}\end{tabular}}}}%
    \put(0.48592612,0.48508816){\color[rgb]{0,0,0}\makebox(0,0)[t]{\lineheight{1.25}\smash{\begin{tabular}[t]{c}\eqref{move:CM4}\end{tabular}}}}%
    \put(0.48592612,0.28687956){\color[rgb]{0,0,0}\makebox(0,0)[t]{\lineheight{1.25}\smash{\begin{tabular}[t]{c}\eqref{move:CM5}\end{tabular}}}}%
    \put(0.48592612,0.11521622){\color[rgb]{0,0,0}\makebox(0,0)[t]{\lineheight{1.25}\smash{\begin{tabular}[t]{c}\eqref{move:CM6}\end{tabular}}}}%
    \put(0,0){\includegraphics[width=\unitlength,page=2]{fig65.pdf}}%
  \end{picture}%
\endgroup%

\caption{\textbf{Examples of the Cerf moves from Definition~\ref{def:Cerfmoves}.}}\label{fig::26}
\end{figure}

\begin{prop}\label{prop:Cerf-moves} Any two Cerf decompositions of a flow-graph $\cG$ are Cerf equivalent (i.e. can be connected by the moves in Definition~\ref{def:Cerfmoves}).
\end{prop}
\begin{proof} Suppose $\cG=(\Gamma,\ws_0,\ws_1)$ is a flow-graph in $Y$. We assume that each edge of $\Gamma$ is smoothly immersed in $Y$ (and that unless an edge has both endpoints on the same vertex, each edge is smoothly embedded). We say a function $f\colon \Gamma\to [0,1]$ is \emph{Morse} if following hold:
\begin{enumerate}
\item $f^{-1}(0)=\ws_0$ and  $f^{-1}(1)=\ws_1$.
\item If we write an edge $e$ as the image of an immersion $\hat{e}\colon [0,1]\to Y$ then $f\circ \hat{e}$ is Morse on $[0,1]$ and $0$ and $1$ are not critical points of $f\circ \hat{e}$. 
\end{enumerate}
A Cerf decomposition can be obtained by picking a generic Morse function $f\colon \Gamma\to [0,1]$, and picking a collection of values $0=t_0<t_1<\cdots < t_{n-1}<t_n=1$ such that each $f^{-1}(t_i)$ contains no critical points or vertices, and each $f^{-1}([t_i,t_{i+1}])$ contains at most one critical point of $f$ or vertex of $V(\Gamma)\setminus (\ws_0\cup \ws_1)$. Furthermore, we assume that each $f^{-1}(t)$ is nonempty for all $t\in [0,1]$. We can construct a Cerf decomposition $\cG_{n}\circ \cdots \circ \cG_{1}$ of $\cG$ by setting
\[
\cG_i=\left(f^{-1}([t_{i-1},t_i]), f^{-1}(t_{i-1}), f^{-1}(t_{i})\right).
\]
Conversely, it is not hard to see that any Cerf decomposition is induced by a Morse function, in the above sense.

Given two Cerf decompositions of $\cG$, we let $f_1$ and $f_2$ be Morse functions on $\Gamma$ which induce the two Cerf decompositions. We connect $f_1$ and $f_2$ as follows. First, we modify $f_1$ in a small neighborhood of each vertex so that it coincides with $f_2$ near $V(\Gamma)$. Write $\tilde{f}_1$ for the resulting Morse function. In particular, $\tilde{f}_1$ and $f_2$ have the same initial slopes along each edge (in the sense of Definition~\ref{def:initialslopes}). The Cerf decompositions induced by $f_1$ to $\tilde{f}_1$ are related by repeated applications of Move~\eqref{move:CM3}.

 We take a generic perturbation of a linear homotopy between $\tilde{f}_1$ and $f_2$, which is fixed in a neighborhood of the vertices $V(\Gamma)$. Write $(f_t)_{t\in [1,2]}$ for this path. Note that the critical set of $f_t$ is bounded away from $V(\Gamma)$, since $\tilde{f}_1$ and $f_2$ coincide in a neighborhood of $V(\Gamma)$. By perturbing $f_t$ slightly, the interval $[1,2]$ can be subdivided by picking
\[
1=b_1<b_2<\cdots <b_n=2.
\] 
so that one on each interval $[b_i,b_{i+1}]$, exactly one of the following holds:
 \begin{enumerate}
 \item $f_t$ is Morse for all $t\in [b_i,b_{i+1}]$, and all vertices and critical points have distinct values.
 \item $f_t$ is Morse for all $t\in [b_i,b_{i+1}]$, and all vertices and critical points have distinct values, except at a single point $t_0\in (b_i,b_{i+1})$, where two critical points or vertices exchange relative ordering.
 \item $f_t$ is Morse for all $t\in [b_i,b_{i+1}]$, except for at a single $t_0\in (b_i,b_{i+1})$, where a critical point birth-death singularity occurs along the interior of an edge. All vertices and critical points have distinct values.
 \end{enumerate}
 From these considerations, it is straightforward to see that $\tilde{f}_1$ and $f_2$ can be connected by a sequence of Moves \eqref{move:CM1}, \eqref{move:CM2}, \eqref{move:CM4}, \eqref{move:CM5} and \eqref{move:CM6}.
 
 Finally, we note one can ensure that all level sets of each function in the path $f_t$ have nonempty level sets by first modifying $f_1$ and $f_2$ near $\ws_0$, raising the value of a point near $\ws_0$ so that it is close to 1. This modification induces a sequence of Moves~\eqref{move:CM2}, \eqref{move:CM4}, and \eqref{move:CM5}.
\end{proof}

We now define the graph action map. Suppose that $\cG=(\Gamma,\ws_0,\ws_1)$ is a ribbon flow-graph in $Y$ and $\frs\in \Spin^c(Y)$. To define the map, we pick a strong ribbon structure on $\Gamma$ which lifts the ribbon structure, as well as a Cerf decomposition of $\cG$. We now define the graph action map for elementary strong ribbon flow-graphs.

If $\cG=(\Gamma,\ws_0,\ws_1)$ is an elementary ribbon flow-graph of type~\eqref{ERtype:1} or \eqref{ERtype:3}, we define the graph action map by the formula
\begin{equation}
\frA_{\cG}:=\left(\prod_{w\in \ws_0} S_w^-\right)\circ \left(\prod_{e\in E(\Gamma)} A_e\right)\circ \left(\prod_{w\in \ws_1} S_w^+\right). \label{eq:def-ER1/3-map}
\end{equation}
The map $\frB_{\cG}$ is defined using a similar formula, with the relative homology maps $B_e$ in place of $A_e$.

In an elementary flow-graph of type~\eqref{ERtype:1} or \eqref{ERtype:3}, no edges share a common vertex. Consequently, the product of the relative homology maps appearing in Equation~\eqref{eq:def-ER1/3-map} does not affect chain homotopy type of the composition by Lemma~\ref{lem:relhomcommutator}. By Proposition~\ref{prop:free-stabs-commute}, the ordering of the free-stabilization maps within the left and right factors of Equation~\eqref{eq:def-ER1/3-map} also does not affect the composition.

Next, suppose that $\cG$ is an elementary strong ribbon flow-graph of type~\eqref{ERtype:2}. Let $w_s$ denote the special vertex, and let $e_1,\dots, e_n$ denote the edges incident to $w_s$, ordered according to the strong ribbon structure. We define
\begin{equation}
\frA_{\cG}:= \left(\prod_{w\in \ws_0\cup \{w_s\}} S_{w}^- \right)\circ \left(\prod_{e\in E(\Gamma)\setminus \{e_1,\dots, e_n\} } A_e\right)\circ (A_{e_n}\circ \cdots \circ A_{e_1}) \circ \left(\prod_{w\in \ws_1\cup \{w_s\}} S_w^+\right).
\end{equation}
The  map $\frB_{\cG}$ is defined using the maps $B_e$ in place of $A_e$.

If $\cG=(\Gamma,\ws_0,\ws_1)$ is an arbitrary, strong ribbon flow-graph in $Y$, the graph action map is defined by picking a Cerf decomposition 
\[
\cG=\cG_n\circ \cdots \circ \cG_1
\]
and setting
\[
\frA_{\cG}:=\frA_{\cG_n}\circ \cdots \circ \frA_{\cG_1}.
\]
The type-$B$ map $\frB_{\cG}$ is defined similarly.

\begin{rem}
\label{rem:actual-embedding-not-important} Since the graph action map $\frA_{\cG}$ is defined using a relative homology map for each edge, the map $\frA_{\cG}$ is unchanged by replacing an edge $e$ of $\cG$ with another edge $e'$ such that $\d e=\d e'$ and $e\cup e'$ is a null-homologous loop in $Y$.
\end{rem}
\subsection{Algebraic relations in the graph TQFT}

In this section, we prove some algebraic relations involving the free-stabilization maps and the relative homology maps which will be useful in proving Theorem~\ref{thm:graph-action}.

%

We call the following relation the \emph{trivial strand relation} (the terminology is justified by Lemma~\ref{lem:full-trivial-strand-rel}, below).

\begin{lem}\label{lem:trivial-strand-rel} Suppose that $(Y,\ws)$ is a multi-pointed 3-manifold, $w\in Y\setminus \ws$ is a new basepoint, and $\lambda \subset Y$ is a path which connects $w$ to a point $w_0\in\ws$. Suppose that $\sigma\colon \ws\to \bmP$ and $\sigma'\colon \ws\cup \{w\}\to \bmP$ are colorings such that $\sigma'|_{\ws}=\sigma$ and $\sigma'(w)=\sigma'(w_0)$. Then, with respect to the complexes which are algebraically colored using $\sigma$ and $\sigma'$, we have
\[
S_w^- A_\lambda S_w^+\simeq S_w^- B_\lambda S_w^+\simeq \id_{\CF^-(Y,\ws^\sigma,\frs)}.
\]
\end{lem}
\begin{proof} We focus on the relation $S_w^- A_\lambda S_w^+\simeq \id_{\CF^-(Y,\ws^{\sigma}, \frs)}$. The relation involving $S_w^- B_\lambda S_w^+$ is proven similarly.

A single free-stabilized diagram and stretched almost complex structure can be chosen to compute $S_{w}^+$, $A_\lambda$ and $S_w^-$. Such a diagram is shown in Figure~\ref{fig:66}. 
 
Using the definition of the free-stabilization maps in Equation~\eqref{def:freestabilizationmaps}, it suffices to show that for a sufficiently stretched almost complex structure, the map $A_\lambda$ satisfies
\begin{equation}
A_\lambda(\xs\times \theta^+)=\xs\times \theta^-+\sum_{\ys\in \bT_{\a}\cap \bT_{\b}} C_{\xs,\ys} \cdot \ys\times \theta^+,\label{eq:trivial-strand-on-generators}
\end{equation}
for some $C_{\xs,\ys}\in \cR_{\bmP}$. In Equation~\eqref{eq:trivial-strand-on-generators}, $\theta^+$ and $\theta^-$ denote the two intersection points in the free-stabilized region.

Equation~\eqref{eq:trivial-strand-on-generators} follows from the proof of Proposition~\ref{prop:free-stabdifferential}. More explicitly, for sufficiently stretched almost complex structure, any Maslov index 1 class in $\pi_2(\xs\times \theta^+, \ys\times \theta^-)$ which supports holomorphic representatives has domain equal to one of the two bigons in the free-stabilization region. Writing $\phi_{\xs}^1$ and $\phi_{\xs}^2$ for these two classes, we note that both $\phi_{\xs}^1$ and $\phi_{\xs}^2$ have a unique holomorphic representative. An easy model computation (see Figure~\ref{fig:66}) shows that
\begin{equation}
a(\lambda,\phi_{\xs}^1)+a(\lambda,\phi_{\xs}^2)\equiv 1 \pmod 2. \label{eq:a(l,phi)-sum=1}
\end{equation}
Equation~\eqref{eq:trivial-strand-on-generators} now follows from the definition of the map $A_\lambda$ in Equation~\eqref{eq:Alambdadef},  completing the proof.
\end{proof}

\begin{figure}[ht!]
\centering
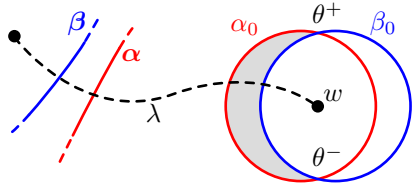
\caption{\textbf{One of the two bigons contributing to Equation~\eqref{eq:a(l,phi)-sum=1}.}}\label{fig:66}
\end{figure}

The following relation is related to subdivision invariance of the graph TQFT:

\begin{lem}\label{lem:rel-hom-relation-subdivide} Suppose that $w_0$, $w_1$ and $w_2$ are basepoints in $Y$, $w_1\not \in \{w_0,w_2\}$, and $\ws$ is a collection of basepoints containing $w_0$ and $w_2$, but not $w_1$. Suppose that $\lambda_1$ and $\lambda_2$ are paths in $Y$ satisfying $\d \lambda_i=\{w_{i-1},w_i\}$, and write $\lambda_2 * \lambda_1$ for the concatenation. Then 
\[
A_{\lambda_2 * \lambda_1}\simeq S_{w_1}^- A_{\lambda_2}A_{\lambda_1} S_{w_1}^+,
\]
as endomorphisms of $\CF^-(Y,\ws^\sigma,\frs)$, for a coloring $\sigma$ which identifies $U_{w_0}$, $U_{w_1}$ and $U_{w_2}$ with a common variable, $U$.
\end{lem}
\begin{proof} We  compute
\begin{align*}A_{\lambda_2*\lambda_1}& \simeq A_{\lambda_2* \lambda_1}(S^-_{w_1}A_{\lambda_1} S^+_{w_1})&&(\text{Lemma}~\ref{lem:trivial-strand-rel})\\
 & \simeq S^-_{w_1} A_{\lambda_2*\lambda_1}A_{\lambda_1}S^+_{w_1}&&(\text{Lemma}~\ref{lem:free-stab-H1-act-commute})\\
 &\simeq S^-_{w_1}(A_{\lambda_2}+A_{\lambda_1})A_{\lambda_1} S^+_{w_1}&&(\text{Lemma}~\ref{lem:splicingrelhom})\\
 &\simeq S^-_{w_1}A_{\lambda_2}A_{\lambda_1}S^+_{w_1}+S^-_{w_1}U S^+_{w_1}&&(\text{Lemma}~\ref{lem:Alambda-squares-to-zero-or-U})\\
 &\simeq  S^-_{w_1} A_{\lambda_2}A_{\lambda_1}S^+_{w_1}&&(\text{Lemma~\ref{lem:disjoint-strand-induces-zero}}),
\end{align*}
completing the proof.
\end{proof}

In the following lemma, and also henceforth, we write $S_{w_n w_{n-1}\cdots w_1}^{+}$ for the composition
\[
S_{w_n w_{n-1}\cdots w_1}^{+}:=S_{w_n}^+\circ S_{w_{n-1}}^+\circ \cdots \circ S_{w_1}^+,
\]
and use a similar notation for negative free-stabilizations. By Proposition~\ref{prop:free-stabs-commute}, the map $S_{w_n w_{n-1}\cdots w_1}^{+}$ is independent of the ordering of the basepoints $w_1,\dots, w_n$.

\begin{lem} \label{lem:subdivide-edge-thirds}
Suppose that $w_0,$ $w_1$, $w_2$ and $w_3$ are distinct points in $Y$, and $\lambda_1,$ $\lambda_2$ and $\lambda_3$ are paths in $Y$ such that $\d \lambda_i=\{w_{i-1},w_i\}$. See Figure~\ref{fig::67}. Furthermore, suppose that $\ws$ is collection of basepoints on $Y$, containing $w_0$ and $w_3$, but not $w_1$ and $w_2$. If $\tau\colon \{1,2,3\}\to \{1,2,3\}$ is a permutation, then
\[
S_{w_1w_2}^- A_{\lambda_{\tau(3)}} A_{\lambda_{\tau(2)}} A_{\lambda_{\tau(1)}} S_{w_1w_2}^+\simeq S_{w_1w_2}^- A_{\lambda_3} A_{\lambda_2}A_{\lambda_1} S_{w_1w_2}^+,
\]
as endomorphisms of $\CF^-(Y,\ws^\sigma,\frs)$, where $\sigma$ is a coloring which identifies $U_{w_0}$,  $U_{w_1}$, $U_{w_2}$, and  $U_{w_3}$ with a common variable,  $U$.
\end{lem}

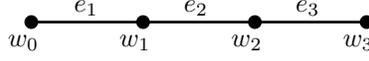
\begin{figure}[ht!]
\centering
\begingroup%
  \makeatletter%
  \providecommand\color[2][]{%
    \errmessage{(Inkscape) Color is used for the text in Inkscape, but the package 'color.sty' is not loaded}%
    \renewcommand\color[2][]{}%
  }%
  \providecommand\transparent[1]{%
    \errmessage{(Inkscape) Transparency is used (non-zero) for the text in Inkscape, but the package 'transparent.sty' is not loaded}%
    \renewcommand\transparent[1]{}%
  }%
  \providecommand\rotatebox[2]{#2}%
  \newcommand*\fsize{\dimexpr\f@size pt\relax}%
  \newcommand*\lineheight[1]{\fontsize{\fsize}{#1\fsize}\selectfont}%
  \ifx\svgwidth\undefined%
    \setlength{\unitlength}{144.90277976bp}%
    \ifx\svgscale\undefined%
      \relax%
    \else%
      \setlength{\unitlength}{\unitlength * \real{\svgscale}}%
    \fi%
  \else%
    \setlength{\unitlength}{\svgwidth}%
  \fi%
  \global\let\svgwidth\undefined%
  \global\let\svgscale\undefined%
  \makeatother%
  \begin{picture}(1,0.12827569)%
    \lineheight{1}%
    \setlength\tabcolsep{0pt}%
    \put(0,0){\includegraphics[width=\unitlength,page=1]{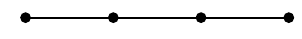}}%
    \put(0.06387688,0.00700685){\color[rgb]{0,0,0}\makebox(0,0)[t]{\lineheight{1.25}\smash{\begin{tabular}[t]{c}$w_0$\end{tabular}}}}%
    \put(0.35455456,0.00700685){\color[rgb]{0,0,0}\makebox(0,0)[t]{\lineheight{1.25}\smash{\begin{tabular}[t]{c}$w_1$\end{tabular}}}}%
    \put(0.64523209,0.00700685){\color[rgb]{0,0,0}\makebox(0,0)[t]{\lineheight{1.25}\smash{\begin{tabular}[t]{c}$w_2$\end{tabular}}}}%
    \put(0.93590961,0.00700685){\color[rgb]{0,0,0}\makebox(0,0)[t]{\lineheight{1.25}\smash{\begin{tabular}[t]{c}$w_3$\end{tabular}}}}%
    \put(0.22711332,0.09768622){\color[rgb]{0,0,0}\makebox(0,0)[t]{\lineheight{1.25}\smash{\begin{tabular}[t]{c}$e_1$\end{tabular}}}}%
    \put(0.51252623,0.09768622){\color[rgb]{0,0,0}\makebox(0,0)[t]{\lineheight{1.25}\smash{\begin{tabular}[t]{c}$e_2$\end{tabular}}}}%
    \put(0.80373687,0.09768622){\color[rgb]{0,0,0}\makebox(0,0)[t]{\lineheight{1.25}\smash{\begin{tabular}[t]{c}$e_3$\end{tabular}}}}%
  \end{picture}%
\endgroup%

\caption{\textbf{The configuration of vertices and edges in Lemma~\ref{lem:subdivide-edge-thirds}}.}\label{fig::67}
\end{figure}

\begin{proof} It is sufficient to show that if $\{i,j,k\}=\{1,2,3\}$ as sets, then the map $S_{w_1w_2}^- A_{\lambda_i}A_{\lambda_j} A_{\lambda_k} S_{w_1w_2}^+$ is invariant (up to chain homotopy) from switching the order of $\lambda_j$ and $\lambda_k$ or switching the order of $\lambda_i$ and $\lambda_j$. We will focus on the proof that the map is invariant under switching the order of $\lambda_j$ and $\lambda_k$; the proof for switching $\lambda_i$ and $\lambda_j$ is analogous.

There are two cases:
\begin{enumerate}
\item $\lambda_j$ and $\lambda_k$ are disjoint.
\item $\lambda_j$ and $\lambda_k$ share a vertex.
\end{enumerate}

If $\lambda_j$ and $\lambda_k$ are disjoint, then $A_{\lambda_j} A_{\lambda_k}\simeq A_{\lambda_k} A_{\lambda_j}$ by Lemma~\ref{lem:relhomcommutator}. Hence it remains to consider the case that $\lambda_j$ and $\lambda_k$ share a vertex $w_n\in \{w_1,w_2\}$. Let $w_m$ denote the other vertex in $\{w_1,w_2\}$. We compute
\begin{align*}
S^-_{w_1w_2}A_{\lambda_i}A_{\lambda_j}A_{\lambda_k}S^+_{w_1w_2}& \simeq  S^-_{w_1w_2}A_{\lambda_i}A_{\lambda_k}A_{\lambda_j}S^+_{w_1w_2}+U S^-_{w_1w_2}A_{\lambda_i}S^+_{w_1w_2}\\
&\simeq S^-_{w_1w_2}A_{\lambda_i}A_{\lambda_k}A_{\lambda_j}S^+_{w_1w_2}+U S^-_{w_m}S^-_{w_n}S^+_{w_n}A_{\lambda_i}S^+_{w_m}\\
&\simeq S^-_{w_1w_2}A_{\lambda_i}A_{\lambda_k}A_{\lambda_j}S^+_{w_1w_2}.
\end{align*}
The first chain homotopy is justified by Lemma~\ref{lem:relhomcommutator}. The second chain homotopy is justified by Lemma~\ref{lem:free-stab-H1-act-commute}, noting that $\lambda_i$ is disjoint from $w_n$. The final chain homotopy is justified by Lemma~\ref{lem:disjoint-strand-induces-zero}. The proof is complete.
\end{proof}

\subsection{Invariance from the Cerf decomposition}

We now prove that the graph action map is independent of the choice of decomposition into elementary flow-graphs:

\begin{proof}[Proof of Part \eqref{thm:graph-action:a} of Theorem \ref{thm:graph-action}] We focus on $\frA_{\cG}$, since the argument for $\frB_{\cG}$ is no different. By Proposition~\ref{prop:Cerf-moves}, it is sufficient to show invariance from Moves~\eqref{move:CM1}--\eqref{move:CM6}. Suppose that an absolute lift of the ribbon structure has been fixed.

We first consider Move~\eqref{move:CM1}, splitting levels. Consider first when $\cG$ is of type~\eqref{ERtype:1} (translation flow-graph). Further, restrict first to the case when $\cG$ consists of a single edge $e$, which goes from $w_0$ to $w_1$. The map $\frA_{\cG}$ is defined to be
\begin{equation}
\frA_{\cG}:=S_{w_0}^-A_e S_{w_1}^+.\label{eq:un-subdivided-AG}
\end{equation}
Write $e$ as the concatenation of two edges, $e_1$ and $e_2$, such that $e_1$ goes from $w_0$ to $w'$, and $e_2$ goes from $w'$ to $w_1$. Let $\cG_1$ denote the flow-graph  $(e_1,w_0,w')$, and $\cG_2$ denote the flow-graph $(e_2,w',w_1)$. See Figure~\ref{fig::91}.

\begin{figure}[ht!]
\centering
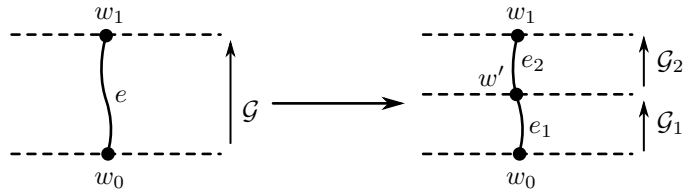
\caption{\textbf{Subdividing an elementary flow-graph of type~\eqref{ERtype:1}.} }\label{fig::91}
\end{figure}

 By definition
\begin{equation}
\frA_{\cG_2}\circ \frA_{\cG_1}:=(S_{w'}^- A_{e_2} S_{w_1}^+)(S_{w_0}^- A_{e_1} S_{w'}^+).\label{eq:level-splitting-1}
\end{equation}
Using Lemma~\ref{lem:free-stab-H1-act-commute} and Proposition~\ref{prop:free-stabs-commute}, we see Equation~\eqref{eq:level-splitting-1} is chain homotopic to
\begin{equation}
S_{w_0}^- S_{w'}^- A_{e_2}  A_{e_1} S_{w'}^+S_{w_1}^+.\label{eq:level-splitting-2}
\end{equation}
Using Lemma~\ref{lem:rel-hom-relation-subdivide}, we conclude that Equation~\eqref{eq:level-splitting-2} is chain homotopic to $S_{w_0}^- A_{e_2* e_1} S_{w_1}^+$, which coincides with the expression for $\frA_{\cG}$ in Equation~\eqref{eq:un-subdivided-AG}.

The general case of splitting an elementary flow-graph $\cG$ of type \eqref{ERtype:1} with more than one strand is no different: one applies the above manipulation to each strand, noting that the terms associated with different strands of $\cG$ commute by Lemmas~\ref{lem:free-stab-H1-act-commute} and Proposition~\ref{prop:free-stabs-commute}.

The above subdivision technique also works when $\cG$ is of type~\eqref{ERtype:2} (interior vertex flow-graphs).
 
We now consider the case that $\cG$ is of  type~\eqref{ERtype:3} (local extrema flow-graph). Let $e$ be the special edge of $\cG$, and suppose further that $\cG$ is of type~\eqref{ERtype:3i}, i.e. $e$ connects two incoming vertices, $w_1$ and $w_2$. Let $w_0$ and $w_0'$ be two new vertices on the interior of $e$, and let $e_1,$ $e_2$ and $e_3$ be the components of $e\setminus \{w_0,w_0'\}$, as shown in Figure~\ref{fig::68}.

\begin{figure}[ht!]
\centering
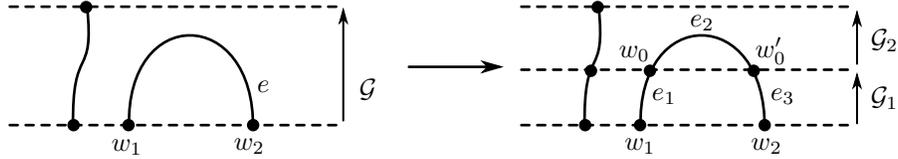
\caption{\textbf{Subdividing an elementary flow-graph of type~\eqref{ERtype:3i}.}  }\label{fig::68}
\end{figure}

We focus on the terms of $\frA_{\cG}$ and $\frA_{\cG_2}\circ \frA_{\cG_1}$ corresponding to $e$ and its subdivision. The terms corresponding to the other strands can be subdivided using the procedure described above for flow-graphs of type~\eqref{ERtype:1}. 

The terms corresponding to $e$ in $\frA_{\cG}$ are
\begin{equation}
S_{w_1w_2}^- A_e.\label{eq:subdivide-ER-type-i-1}
\end{equation}
The corresponding terms of $\frA_{\cG_2}\circ \frA_{\cG_1}$ are
\begin{equation}
S_{w_0w_0'}^- A_{e_2} S_{w_1w_2}^- A_{e_1}A_{e_3} S_{w_0w_0'}^+.\label{eq:subdivide-ER-type-i-2}
\end{equation}
We compute
\begin{align*}
S_{w_1w_2}^-A_e&\simeq S_{w_1w_2w_0'}^-A_{e_1*e_2}A_{e_3} S_{w_0'}^+&& \text{(Lemma~\ref{lem:rel-hom-relation-subdivide})} \\
&\simeq S_{w_1w_2w_0w_0'}^- A_{e_2}  A_{e_1} S_{w_0}^+A_{e_3} S_{w_0'}^+&& \text{(Lemma~\ref{lem:rel-hom-relation-subdivide})}\\
&\simeq S_{w_1w_2w_0w_0'}^- A_{e_2}A_{e_1}A_{e_3}S_{w_0'w_0}^+&&\text{(Lemma~\ref{lem:free-stab-H1-act-commute})}\\
&\simeq S_{w_0w_0'}^- A_{e_2}S_{w_1w_2}^-A_{e_1}A_{e_3}S_{w_0'w_0}^+&&\text{(Lemma~\ref{lem:free-stab-H1-act-commute})},
\end{align*}
which coincides with Equation~\eqref{eq:subdivide-ER-type-i-2}. We conclude that level splitting does not change the homotopy type of the graph action map for elementary ribbon flow-graphs of type~\eqref{ERtype:3i}. The argument for graphs of type~\eqref{ERtype:3ii} is a simple modification.

We now show independence from Move~\eqref{move:CM2}, a critical point birth along the interior of an edge. By using Move~\eqref{move:CM1} we may assume that the new critical points occur inside a flow-graph of type~\eqref{ERtype:1}. By using Lemma~\ref{lem:free-stab-H1-act-commute} and Proposition~\ref{prop:free-stabs-commute}, it is sufficient to consider the case when $\cG$ consists of a single edge $e$, and we divide $\cG$ into a composition $\cG_2\circ \cG_1$. Let $e_1$, $e_2$, $e_3$ and $e_4$ be the edges of $\cG_2\circ \cG_1$, and let $w_1,$ $w_2$, $w_3$, $w_4$ and $w_5$ be the vertices, as in Figure~\ref{fig::69}.

\begin{figure}[ht!]
\centering
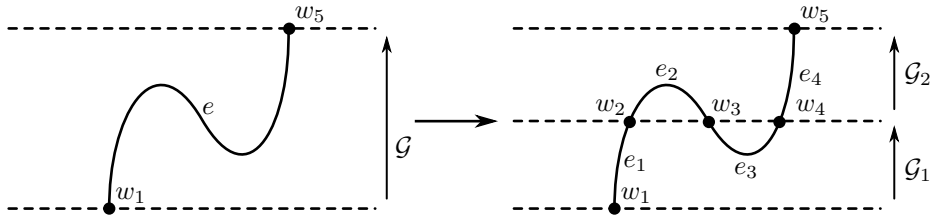
\caption{\textbf{Applying Move~\eqref{move:CM2} to a flow-graph of type~\eqref{ERtype:1}.}}\label{fig::69}
\end{figure}

By definition, 
\begin{equation}
\frA_{\cG}:=S^-_{w_1} A_e S^+_{w_5},\label{eq:pre-result-CM2}
\end{equation}
while
\begin{equation}
\frA_{\cG_2}\circ \frA_{\cG_1}:=S_{w_2w_3w_4}^-A_{e_4}A_{e_2}S^+_{w_5}S^-_{w_1}A_{e_3}A_{e_1}S^+_{w_2w_3w_4}.\label{eq:result-move-CM2}
\end{equation}
Using Lemma~\ref{lem:free-stab-H1-act-commute} and Proposition~\ref{prop:free-stabs-commute}, Equation~\eqref{eq:result-move-CM2} can be rearranged
\begin{equation}
S_{w_1w_2w_4}^- A_{e_4}(S_{w_3}^- A_{e_2}A_{e_3}S_{w_3}^+)A_{e_1}S_{w_2w_4w_5}^+.\label{eq:result-move-CM2-2}
\end{equation}
Using Lemma~\ref{lem:rel-hom-relation-subdivide}, Equation~\eqref{eq:result-move-CM2-2} becomes
\[
S_{w_1w_2w_4}^- A_{e_4}A_{e_2*e_3}A_{e_1} S_{w_2w_4w_5}^+.
\]
Proceeding similarly, using the aforementioned relations, we compute
\begin{align*}
S_{w_1w_2w_4}^- A_{e_4}A_{e_2*e_3}A_{e_1} S_{w_2w_4w_5}^+&\simeq S_{w_1w_2}^-(S_{w_4}^- A_{e_4}A_{e_2*e_3}S_{w_4}^+)A_{e_1} S_{w_2w_5}^+\\
&\simeq S_{w_1w_2}^- A_{e_2*e_3*e_4}A_{e_1} S_{w_2w_5}^+\\
&\simeq S_{w_1}^-( S_{w_3}^- A_{e_2*e_3*e_4}A_{e_1} S_{w_2}^+ )S_{w_5}^+\\
&\simeq S_{w_1}^- A_e S_{w_5}^+.
\end{align*}
Hence Equations~\eqref{eq:pre-result-CM2} and~\eqref{eq:result-move-CM2} are chain homotopic, and invariance under Move~\eqref{move:CM2} is established.

We now consider invariance under Move~\eqref{move:CM3} (changing an initial slope). Suppose $\cG$ is an elementary flow-graph of type~\eqref{ERtype:2} (an interior vertex flow-graph), with a special vertex $w_s$, and $e_1,\dots, e_n$  are the edges incident to $w_s$, ordered according to the chosen absolute lift of the cyclic ordering. Suppose we wish to change the initial slope of $e_i$, and that $e_i$ has downward initial slope in $\cG$. We decompose $\cG$ as a composition $\cG_2\circ \cG_1$, where $\cG_1$ is of type~\eqref{ERtype:2}, with special vertex $w_s$, and $\cG_2$ is of type~\eqref{ERtype:3i} (a local max). We decompose $e_i$ as the concatenation of $e$, $e'$ and $e''$, and let $w'$ and $w''$ denote the new vertices, as in Figure~\ref{fig::70}.

\begin{figure}[ht!]
\centering
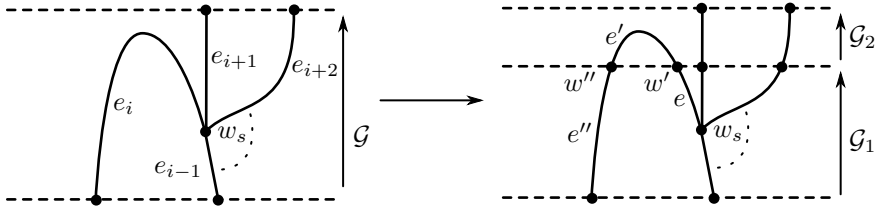
\caption{\textbf{Applying Move~\eqref{move:CM3} to a flow-graph of type~\eqref{ERtype:2}.} }\label{fig::70}
\end{figure}

Write $\ws_0$ for the incoming vertices of $\cG$, and $\ws_1$ for the outgoing vertices. By the same argument as in invariance of Move~\eqref{move:CM1}, we can reduce to the case that there are no components of $\cG$ which are disjoint from $w_s$. By definition,
\begin{equation}
\frA_{\cG}:=S_{\ws_0\cup \{w_s\}}^- A_{e_n}\cdots A_{e_i} \cdots A_{e_1} S_{\ws_1\cup \{w_s\}}^+.\label{eq:def-A_G-ER2}
\end{equation}
Using Lemmas~\ref{lem:rel-hom-relation-subdivide} and \ref{lem:free-stab-H1-act-commute}, we have
\begin{equation}
\begin{split} 
A_{e_i}&\simeq S_{w''}^-A_{e'*e}A_{e''}S_{w''}^+\\
&\simeq S_{w''w'}^-A_{e'}A_{e}S_{w'}^+ A_{e''}S_{w''}^+\\
&\simeq S_{w''w'}^- A_{e'}A_{e} A_{e''}S_{w''w'}^+
\end{split}
\label{eq:subdivide-edge-twice}
\end{equation}
Combining Equations~\eqref{eq:def-A_G-ER2} and \eqref{eq:subdivide-edge-twice} we see that 
\[
\frA_{\cG}\simeq  S_{\ws_0\cup \{w_s\}}^- A_{e_n}\cdots A_{e_{i+1}}(S_{w''w'}^- A_{e'}A_{e} A_{e''}S_{w''w'}^+)A_{e_{i-1}} \cdots A_{e_1} S_{\ws_1\cup \{w_s\}}^+.
\]
Since $e',$ $e''$, $w''$ and $w'$ are disjoint from $e_j$ when $j\neq i$, using Lemmas~\ref{lem:relhomcommutator} and \ref{lem:free-stab-H1-act-commute}, we compute that
\begin{align*}
&S_{\ws_0\cup \{w_s\}}^- A_{e_n}\cdots A_{e_{i+1}}(S_{w''w'}^- A_{e'}A_{e} A_{e''}S_{w''w'}^+)A_{e_{i-1}} \cdots A_{e_1} S_{\ws_1\cup \{w_s\}}^+\\
\simeq &(S_{w''w'}^- A_{e'}) S_{\ws_0\cup \{w_s\}}^- (A_{e_n}\cdots A_{e_{i+1}} A_{e} A_{e_{i-1}} \cdots A_{e_1})A_{e''} S_{\ws_1\cup \{w'',w',w_s\}}^+.
\end{align*}
The above expression is almost the expression for $A_{\cG_2}\circ A_{\cG_1}$, the only difference being that the remaining edges $e_1,\dots, e_{i-1}, e_{i+1},\dots e_n$ with upward initial slope have not yet been subdivided. By using Lemma~\ref{lem:rel-hom-relation-subdivide} to subdivide the remaining edges with upward initial slope, and then commuting terms exactly as in Move~\eqref{move:CM1}, we arrive at the definition of $\frA_{\cG_2}\circ \frA_{\cG_1}$, completing the proof of invariance of Move~\eqref{move:CM3}.

Invariance under Moves~\eqref{move:CM4}, \eqref{move:CM5} and \eqref{move:CM6} can be proven by adapting the above arguments. For example, invariance under Move~\eqref{move:CM4} (critical value swaps) can be proven by by using the manipulation from the proof of invariance of the maps for elementary flow-graphs of type~\eqref{ERtype:3} (local minima and maxima) under Move~\eqref{move:CM1} (level splitting) to subdivide one of the edges with a local extrema. Since the two components of the graph with local extrema are disjoint, all terms associated to one component can be commuted past the terms associated to the other component, using Lemmas~\ref{lem:relhomcommutator}, \ref{lem:free-stab-H1-act-commute} and Proposition~\ref{prop:free-stabs-commute}. A similar strategy can be used for Moves~\eqref{move:CM5} and~\eqref{move:CM6}.
\end{proof}

Since the graph action is defined by taking a decomposition of $\cG$ into elementary flow-graphs and composing the maps for each piece, functoriality (Part~\eqref{thm:graph-action:b} of Theorem~\ref{thm:graph-action}) is automatic.

\subsection{Cyclic ordering}

In our construction of the graph action map, we chose an absolute lift of the cyclic orderings. In this section, we show that the graph action map is invariant of the choice of absolute lift. The following lemma implies Part~\eqref{thm:graph-action:c} of Theorem~\ref{thm:graph-action}:

\begin{lem}\label{lem:cyclicreorder}
 Suppose that $e_1,\dots, e_n$ are edges in $Y$ such that $w_s$ has valence 1 in each $e_i$, and $e_i\cap e_j=\{w_s\}$ for all $i$ and $j$. For a coloring which identifies all of the variables $U_{w_1},\dots, U_{w_n}$ with a common variable $U$, we have
\[
S^-_{w_s} A_{e_n}A_{e_{n-1}}\cdots A_{e_1} S^+_{w_s}\simeq S^-_{w_s} A_{e_{n-1}} \cdots A_{e_1}A_{e_n}S^+_{w_s}.
\]
\end{lem}

\begin{proof} The proof is by induction. The $n=1$ case is vacuous. The $n=2$ case is easy: using Lemma~\ref{lem:relhomcommutator} we compute
\[
S^-_{w_s} A_{e_2}A_{e_1} S^+_{w_s}\simeq S^-_{w_s} U S^+_{w_s}+ S^-_{w_s} A_{e_1} A_{e_2}S^+_{w_s}\simeq  S^-_{w_s} A_{e_1} A_{e_2}S^+_{w_s},
\]

 The $n=3$ case is also relatively easy. We compute
\begin{align*}S^-_{w_s}A_{e_3}A_{e_2}A_{e_1}S_{w_s}^+&\simeq S^-_{w_s}A_{e_2}A_{e_3}A_{e_1}S_{w_s}^++ U S^-_{w_s} A_{e_1}S^+_{w_s}&&(\text{Lemma~\ref{lem:relhomcommutator}})\\
&\simeq S^-_{w_s}A_{e_2}A_{e_1}A_{e_3}S_{w_s}^++ U S^-_{w_s} (A_{e_2}+A_{e_1})S^+_{w_s}&&(\text{Lemma~\ref{lem:relhomcommutator}})\\
&\simeq S^-_{w_s}A_{e_2}A_{e_1}A_{e_3}S_{w_s}^++ U S^-_{w_s} A_{e_2*e_1} S^+_{w_s}&&(\text{Lemma~\ref{lem:splicingrelhom}})\\
&\simeq S^-_{w_s}A_{e_2}A_{e_1}A_{e_3}S_{w_s}^++U A_{e_2*e_1}S^-_{w_s} S^+_{w_s}&&(\text{Lemma~\ref{lem:free-stab-H1-act-commute}})\\
&\simeq S^-_{w_s}A_{e_2}A_{e_1}A_{e_3}S_{w_s}^+&&(\text{Lemma~\ref{lem:disjoint-strand-induces-zero}}). 
\end{align*}

We now prove the statement when $n>3$  by induction. Assume the lemma statement holds whenever $w_s$ has valence $n-1$. The idea is to replace $e_1\cup e_2$ with a $Y$-shaped graph which has valence 1 at $w_s$. See Figure~\ref{fig::71}.

Let $w_0\in Y$ be a point which is disjoint from all of the edges $e_i$. Pick a path $e$ from ${w_s}$ to $w_0$, which is also disjoint from all $e_i$. We first prove the following:
\begin{equation}
A_{e_2}A_{e_1}\simeq S^-_{w_0}A_{e_2*e}A_{e_1*e}A_{e}S^+_{w_0}. \label{eq:pull-out-a-Y-graph}
\end{equation}
 To establish Equation~\eqref{eq:pull-out-a-Y-graph}, we compute
\begin{equation}
\begin{split}
S_{w_0}^-A_{e_2*e}A_{e_1*e}A_e S^+_{w_0}&\simeq S_{w_0}^-(A_{e_2}+A_{e})(A_{e_1}+A_e)A_e S^+_{w_0}\\
&\simeq S_{w_0}^- A_{e_2}A_{e_1} A_{e}S^+_{w_0}+S_{w_0}^- (A_{e_2}A_{e}^2+A_{e}A_{e_1}A_e+A_e^3)S_{w_0}^+\\
&\simeq S^-_{w_0} A_{e_2}A_{e_1}A_eS^+_{w_0}+S_{w_0}^- (A_{e_2} U+A_{e_1}U+U A_e+UA_e)S^+_{w_0}\\
&\simeq A_{e_2}A_{e_1}S^-_{w_0}A_eS^+_{w_0}+S_{w_0}^- (A_{e_2} U+A_{e_1}U+U A_e+U A_e)S^+_{w_0}\\
&\simeq A_{e_2}A_{e_1}+S_{w_0}^- (A_{e_2} U+A_{e_1}U+U A_e+U A_e)S^+_{w_0}\\
&\simeq A_{e_2}A_{e_1}+US_{w_0}^- A_{e_2*e_1} S^+_{w_0}\\
&\simeq A_{e_2}A_{e_1}+U A_{e_2* e_1} S_{w_0}^-S_{w_0}^+\\
&\simeq A_{e_2}A_{e_1}.
\end{split}
\label{eq:pull-out-Y-graph-comp}
\end{equation}
Equation~\eqref{eq:pull-out-Y-graph-comp} is justified as follows. Line 1 is justified by additivity of the relative homology maps. Line~2 is obtained by algebra. Line~3 is justified by Lemmas~\ref{lem:relhomcommutator} and~\ref{lem:Alambda-squares-to-zero-or-U}. Line~4 is justified by Lemma~\ref{lem:free-stab-H1-act-commute}. Line~5 is justified by Lemma~\ref{lem:trivial-strand-rel}. Line~6 follows by canceling the repeated $UA_e$ terms, and using additivity of the relative homology maps. Line~7 follows from Lemma~\ref{lem:free-stab-H1-act-commute}. The final line is justified by Lemma~\ref{lem:trivial-strand-rel}. Equation~\eqref{eq:pull-out-a-Y-graph} is established.

Define new edges $e_1'$ and $e_2'$ as the concatenations
\[
e_1':=e_1* e\quad \text{and}\quad e_2':=e_2* e.
\] 
See Figure~\ref{fig::71}.

\begin{figure}[ht!]
\centering
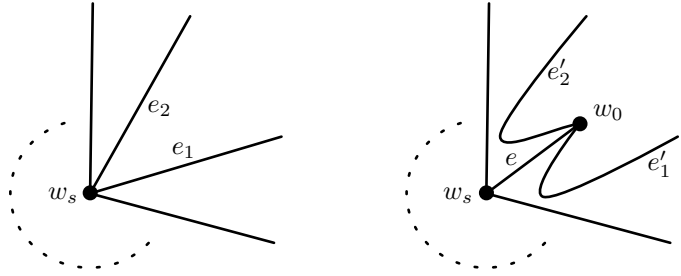
\caption{\textbf{The edges $e_1$ and $e_2$, the vertices $w_s$ and $w_0$, and the new edges $e$, $e_1',$ and $e_2'$.}}\label{fig::71}
\end{figure}

Using Equation~\eqref{eq:pull-out-a-Y-graph}, we see that
\begin{equation}
S^-_{w_s} A_{e_n}\cdots A_{e_2} A_{e_1} S^+_{w_s}\simeq S^-_{w_s} A_{e_n}\cdots A_{e_3} (S^-_{w_0} A_{e'_2} A_{e'_1} A_e S^+_{w_0}) S^+_{w_s} \label{eq:split-off-trivalent-strand1}
\end{equation}

Note that $w_0$, $e_1'$ and $e_2'$ are disjoint from $e_2,\dots, e_n$. Hence Lemmas~\ref{lem:relhomcommutator} and~\ref{lem:free-stab-H1-act-commute} and Proposition~\ref{prop:free-stabs-commute} imply that Equation~\eqref{eq:split-off-trivalent-strand1} is chain homotopic to
\begin{equation}
S_{w_0}^-A_{e_2'}A_{e'_1}(S^-_{w_s} A_{e_n}\cdots A_{e_3}A_e S^+_{w_s})S^+_{w_0}\label{eq:split-off-trivalent-strand2}
\end{equation}
By induction, we know that Equation~\eqref{eq:split-off-trivalent-strand2} is chain homotopic to
\begin{equation}
S_{w_0}^-A_{e_2'}A_{e'_1}(S^-_{w_s} A_{e_{n-1}} \cdots A_{e_3}A_e A_{e_n}S^+_{w_s})S^+_{w_0}
 \label{eq:split-off-trivalent-strand4}
\end{equation}
Commuting terms, using the same justification as above, we see that Equation~\eqref{eq:split-off-trivalent-strand4} is chain homotopic to
\begin{equation}
S^-_{w_s} A_{e_{n-1}} \cdots A_{e_3}(S_{w_0}^-A_{e_2'}A_{e'_1}A_eS^+_{w_0}) A_{e_n}S^+_{w_s}\label{eq:split-off-trivalent-strand3}
\end{equation}
Applying Equation~\eqref{eq:pull-out-a-Y-graph}, we can conclude that Equation~\eqref{eq:split-off-trivalent-strand3}
is chain homotopic to
\[
S^-_{w_s} A_{e_{n-1}} \cdots A_{e_3}A_{e_2}A_{e_1} A_{e_n}S^+_{w_s},
\]
completing the proof.
\end{proof}

\subsection{Subdivision invariance and the trivial strand relation}

In this section we prove two basic relations about the graph action map: subdivision invariance and the trivial strand relation.

\begin{lem}\label{lem:subdivision-invariance}
 Suppose $\cG=(\Gamma,\ws_0,\ws_1)$ and $\cG'=(\Gamma',\ws_0,\ws_1)$ are two ribbon flow-graphs in $Y$ such that $\Gamma'$ is obtained from $\Gamma$ by adding a vertex to the interior of an edge of $\Gamma$. Then
\[
\frA_{\cG}\simeq \frA_{\cG'}\qquad \text{and} \qquad \frB_{\cG}\simeq \frB_{\cG'}.
\]
\end{lem}
\begin{proof}We focus on the maps $\frA_{\cG}$ and $\frA_{\cG'}$. The proof is essentially the same as the argument to show invariance under level splitting (Move~\eqref{move:CM1}). Using independence of the graph action map from the choice of Cerf decomposition (Part \eqref{thm:graph-action:a} of Theorem~\ref{thm:graph-action}), it is sufficient to show the claim when $\cG$ is an elementary flow-graph of type~\eqref{ERtype:1} (a translation flow-graph).

 Suppose $e$ is the edge of $\cG$, which we wish to subdivide, and let $w_0$ and $w_1$ be the incoming and outgoing vertices of $e$. Let $w$ denote the vertex in the interior of $e$, which is added to form $\cG'$.  Note that $\cG$ is of type~\eqref{ERtype:1}, while $\cG'$ is of type~\eqref{ERtype:2}, with valence 2 special vertex $w$. Let $e_1$ and $e_2$ denote the two components of $e\setminus \{w\}$.
 
 Using Lemmas~\ref{lem:relhomcommutator}, \ref{lem:free-stab-H1-act-commute} and Proposition~\ref{prop:free-stabs-commute}, it is sufficient to show the claim when $\cG=(e,w_0,w_1)$, since the maps corresponding to other edges and vertices can be commuted past the maps for $e$.
 
   By definition,
\begin{equation}
\frA_{\cG}:=S_{w_0}^-A_e S_{w_1}^+,\label{eq:subdivide-inv-1}
\end{equation}
while
\begin{equation}
\frA_{\cG'}:=S_{w_0}^-S_{w}^- A_{e_2}A_{e_1} S_{w}^+S_{w_1}^+.\label{eq:subdivide-inv-2}
\end{equation}
Using Lemma~\ref{lem:rel-hom-relation-subdivide}, we conclude that Equations~\eqref{eq:subdivide-inv-1} and \eqref{eq:subdivide-inv-2} are chain homotopic, completing the proof.
\end{proof}

Next, we prove the \emph{trivial strand relation} for the graph action maps:

\begin{lem}\label{lem:full-trivial-strand-rel}
Suppose that $\cG=(\Gamma,\ws_0,\ws_1)$ is a ribbon flow-graph in $Y$, and $\cG'$ is obtained by adjoining a new edge $e_0$, such that $e_0\cap \Gamma$ consists of a single point, which has valence 3 in $\Gamma\cup e_0$. (We allow the new valence 3 vertex to be given either cyclic order.) Then
\[
\frA_{\cG}\simeq \frA_{\cG'}
\]
and similarly $\frB_{\cG}\simeq \frB_{\cG'}$.
\end{lem}
\begin{proof} Using independence from the choice of Cerf decomposition (Part \eqref{thm:graph-action:a} of Theorem~\ref{thm:graph-action}), it is sufficient to show the claim when $\cG$ is an elementary flow-graph of type~\eqref{ERtype:1}. As in the proof of Lemma~\ref{lem:subdivision-invariance}, it is also sufficient to consider the case when $\cG$ has one connected component. Let us write $e$ for the single edge of $\Gamma$,  $w$ for the incoming vertex of $\cG$, and $w'$ for the outgoing vertex. By definition
\begin{equation}
\frA_{\cG}= S_{w}^- A_{e}  S_{w'}^+.\label{eq:AGbefore-trivial-strand}
\end{equation}
The graph $\cG'$ can be given a Cerf decomposition into two flow-graphs $\cG_{2}\circ \cG_1$, where $\cG_2$ is an elementary flow-graph of type~\eqref{ERtype:2}, with a valence 1 special vertex, and $\cG_1$ is an elementary flow-graph of type~\eqref{ERtype:2}, with a valence 3 special vertex. Let $e_1,$ $e_2$, $e_3,$ $e_4$, $e_5$, $w_1$, $w_2$, $w_3$ and $w_4$ denote the new vertices of $\cG'$, as shown in Figure~\ref{fig::72}. 

We assume that the edges incident to the new valence 3 vertex, $w_1$, are ordered $e_1$, $e_2$ then $e_3$. The argument for the opposite cyclic ordering is a simple modification, which we leave to the reader.

\begin{figure}[ht!]
\centering
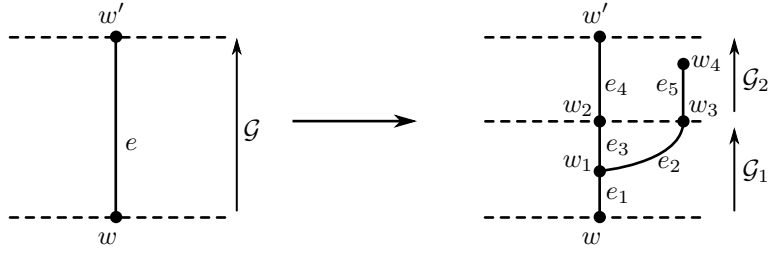
\caption{\textbf{Adding a trivial strand.} On the left is an elementary flow-graph $\cG$ of type~\eqref{ERtype:1}, with a single edge $e$. On the right is a Cerf decomposition of a graph $\cG'$ obtained by adding a trivial strand to $\cG$.}\label{fig::72}
\end{figure}

By definition,
\[
\frA_{\cG'}=S^-_{w_2w_3w_4} A_{e_4}A_{e_5}S^+_{w'w_4}S_{ww_1}^-A_{e_3}A_{e_2}A_{e_1}S^+_{w_1w_2w_3}.
\]
We manipulate the above expression for $\frA_{\cG'}$, as follows:
\begin{align*} &S^-_{w_2w_3w_4} A_{e_4}A_{e_5}S^+_{w'w_4}S_{ww_1}^-A_{e_3}A_{e_2}A_{e_1}S^+_{w_1w_2w_3}&&\\
\simeq & S^-_{w_2w_3} A_{e_4} (S_{w_4}^- A_{e_5} S_{w_4}^+) S_{w'}^+S_{ww_1}^-A_{e_3}A_{e_2}A_{e_1}S^+_{w_1w_2w_3}&&\text{(Lemma~\ref{lem:free-stab-H1-act-commute})}\\
\simeq &  S^-_{w_2w_3} A_{e_4}  S_{w'}^+S_{ww_1}^-A_{e_3}A_{e_2}A_{e_1}S^+_{w_1w_2w_3}&&\text{(Lemma~\ref{lem:trivial-strand-rel})}\\
\simeq &S^-_{ww_1w_3} (S_{w_2}^-A_{e_4}  A_{e_3}S_{w_2}^+)A_{e_2}A_{e_1}S^+_{w'w_1w_3}&&\text{(Lemma~\ref{lem:free-stab-H1-act-commute})}\\
\simeq & S^-_{ww_1w_3} A_{e_3*e_4}A_{e_2}A_{e_1}S^+_{w'w_1w_3}&&\text{(Lemma~\ref{lem:rel-hom-relation-subdivide})}\\
\simeq &S^-_{ww_1w_3} A_{e_1}A_{e_3*e_4}A_{e_2}S^+_{w'w_1w_3}&&\text{(Lemma~\ref{lem:cyclicreorder})}\\
\simeq &S^-_{ww_1} A_{e_1}A_{e_3*e_4}(S_{w_3}^-A_{e_2}S^+_{w_3})S^+_{w'w_1}&&\text{(Lemma~\ref{lem:free-stab-H1-act-commute})}\\
\simeq &S^-_{ww_1} A_{e_1}A_{e_3*e_4}S^+_{w'w_1}&&\text{(Lemma~\ref{lem:trivial-strand-rel})}\\
\simeq &S^-_{w} A_{e_1*e_3*e_4}S_{w'}^+&&\text{(Lemma~\ref{lem:rel-hom-relation-subdivide})},
\end{align*}
which agrees with the expression for $\frA_{\cG}$ in Equation~\eqref{eq:AGbefore-trivial-strand}, completing the proof.
\end{proof}

\section{1- and 3-handles}
\label{sec:1-handles}

In this section, we define maps for 4-dimensional 1-handles and 3-handles, and prove invariance. Our construction is  similar to the constructions of Ozsv\'{a}th and Szab\'{o} \cite{OSTriangles} and Juh\'{a}sz \cite{JCob}, though there are some differences. Unlike the construction from \cite{OSTriangles}, our construction allows us to consider 1-handles which join two components of a 3-manifold, or 3-handles which separate a 3-manifold into two components. Unlike Juh\'{a}sz's construction \cite{JCob}, which works only for $\hat{\CF}$, our construction applies to $\CF^-$, $\CF^+$ and $\CF^\infty$.

 In Section~\ref{sec:OS-1handle3-handle}, we show that our construction coincides with Ozsv\'{a}th and Szab\'{o}'s, for 1-handles with feet in the same component of the 3-manifold.

\subsection{Definition of the 1-handle and 3-handle maps}
\label{sec:1-handle-maps-def}
Suppose that $(Y,\ws)$ is a multi-pointed 3-manifold and $\bS^0=\{p_1,p_2\}$ is a 0-sphere in $Y\setminus \ws$. Pick a Heegaard diagram $\cH=(\Sigma,\as,\bs,\ws)$ such that
\[
\bS^0\subset \Sigma\setminus (\as\cup \bs).
\]
We construct a Heegaard surface $\hat{\Sigma}$ for the surgered 3-manifold $Y(\bS^0)$ by removing neighborhoods of $p_1$ and $p_2$ in $\Sigma$ and gluing in an annulus contained in the 1-handle region of $Y(\bS^0)$. In the  annulus region, we add two new curves $\alpha_0$ and $\beta_0$ which are homologically essential in the annulus and intersect transversely in a pair of points. The two points can be distinguished by the Maslov grading. Write $\theta^+$ and $\theta^-$ for the two points of $\alpha_0\cap \beta_0$, and $\hat{\cH}$ for the diagram $(\hat{\Sigma},\as\cup \{\alpha_0\},\bs\cup \{\beta_0\},\ws)$. 

 There is a unique $\Spin^c$ structure $\hat{\frs}\in \Spin^c(Y(\bS^0))$ which restricts to $\frs$ on $Y\setminus N(\bS^0)$ and evaluates trivially on the new 2-sphere in $Y(\bS^0)$. There is a unique 4-dimensional $\Spin^c$ structure $\frt$ on the 1-handle cobordism $W(Y,\bS^0)$ which extends $\frs$. The $\Spin^c$ structure $\frt$ restricts to $\hat{\frs}$ on $Y(\bS^0)$.

It is an easy exercise to see that if $\cH$ is strongly $\frs$-admissible then $\hat{\cH}$ is strongly $\hat{\frs}$-admissible.

 We define the 1-handle map
\[
F_{Y, \bS^0,\frt}\colon \CF^-(\cH,\sigma,\frs)\to \CF^-\left(\hat{\cH},\sigma,\hat{\frs}\right)
\]
using the formula
\begin{equation}
F_{Y,\bS^0,\frt}(\xs)=\xs\times \theta^+,\label{eq:def-1-handle}
\end{equation}
extended equivariantly over $\bF_2[U_{\ws}]$. Like the free-stabilization maps, the 1-handle and 3-handle maps require the almost complex structure to be stretched. See Definition~\eqref{def:1-handle-def-stabilizing-condition-3} for a precise definition of which almost complex structures can be chosen.

We now describe the 3-handle maps. Suppose that $\bS^2\subset Y\setminus \ws$ is an embedded 2-sphere, and $\hat{\cH}=(\hat{\Sigma}, \as\cup \{\alpha_0\}, \bs\cup \{\beta_0\},\ws)$  is a Heegaard diagram for $(Y,\ws)$ such that $\bS^2\cap \hat{\Sigma}$ consists of a circle $c$ which is disjoint from $\as\cup \bs$. Furthermore, assume that an annular neighborhood of $c$ contains both $\alpha_0$ and $\beta_0$ (which are homologically essential in this annulus) and $\alpha_0\cap \beta_0=\{\theta^+,\theta^-\}$. A diagram for $Y(\bS^2)$ may be obtained by cutting out a neighborhood of $c$, removing $\alpha_0$ and $\beta_0$, and filling in the two boundary components with disks. Write $\cH$ for the resulting diagram.

For sufficiently stretched almost complex structure, the 3-handle map 
\[
F_{Y,\bS^2,\frt}\colon \CF^-\left(\hat{\cH},\sigma,\hat{\frs}\right)\to \CF^-(\cH,\sigma,\frs)
\]
 is defined
via the formulas
\begin{equation}
F_{Y,\bS^2,\frt}(\xs\times \theta^+)=0\quad \text{and} \quad F_{Y,\bS^2,\frt}(\xs\times \theta^-)=\xs,\label{eq:def-3-handle}
\end{equation}
extended $\bF_2[U_{\ws}]$-equivariantly.

We note that the 1-handle maps and the 3-handle maps are algebraically dual, and hence any statement about the 1-handle maps has a corresponding statement about the 3-handle maps. To streamline the exposition, we focus mostly on the 1-handle maps.

\subsection{Gluing data for 1- and 3-handles}
\label{sec:gluingdata1-handles}

We now describe precisely which almost complex structures can be used to compute the 1-handle maps. The technical details are similar to the free-stabilization maps, so we will be terse.

It is convenient to view the diagram $\hat{\cH}$ for $Y(\bS^0)$, constructed in Section~\ref{sec:1-handle-maps-def}, as being obtained by connecting a diagram $(S^2,\alpha_0, \beta_0)$ to the diagram $\cH$ using two tubes attached to antipodal regions of $S^2\setminus (\alpha_0\cup \beta_0)$. Let $p_1^0$ and $p_2^0$ denote these two connected sum points on $S^2$, and fix two disks $D_{1}^0,D_2^0\subset S^2\setminus (\alpha_0\cup \beta_0)$, containing $p_1^0$ and $p_2^0$, respectively. Also fix regular neighborhoods $N(p_1)$ and $N(p_2)$ in $Y$, to construct the surgered manifold $Y(\bS^0)$.

Similar to Definition~\ref{def:gluingdata}, we make the following definition:

\begin{define}\label{def:gluing-data-1-handle}
Suppose that $\cH=(\Sigma,\as,\bs,\ws)$ is a Heegaard diagram for $(Y,\ws)$ and $p_1,p_2\in \Sigma\setminus (\ws\cup \as\cup \bs)$. We call a tuple $\frd=(J^\frd,J_0^{\frd},D_1,D_2,\iota)$ a \emph{gluing datum for a 1-handle attached at $p_1$ and $p_2$} if the following hold:
\begin{enumerate}
\item $J^{\frd}_0$ is an almost complex structure on $S^2\times [0,1]\times \R$ which is split on $D_1^0$ and $D_2^0$.
\item $D_i\subset \Sigma\setminus (\as\cup\bs\cup \ws)$ is a closed disk such that $\tfrac{1}{2}\cdot D_i$ contains $N(p_i)\cap \Sigma$, for $i\in \{1,2\}$.
\item $J^{\frd}$ is an almost complex structure on $\Sigma\times [0,1]\times \R$ which is split on $D_1$ and $D_2$.
\item $\iota\colon S^2\setminus (\tfrac{1}{2}\cdot D_{1}^0\cup \tfrac{1}{2}\cdot  D_2^0)\to Y(\bS^0)$ is a smooth embedding such that the following hold:
\begin{enumerate}
\item $\im(\iota)\cap \big(Y\setminus (N(p_1)\cup N(p_2))\big)\subset \Sigma$;
\item $\iota$ maps each annulus $D_i^0\setminus \tfrac{1}{2} \cdot D_i^0$ conformally onto  $D_i \setminus \tfrac{1}{2} \cdot D_i$.
\end{enumerate}
\end{enumerate}
\end{define}

If $\ve{T}=(T_1,T_2)$ is a pair of positive real numbers, by adapting the construction for free-stabilized almost complex structures from Section~\ref{subsec:gluingdata}, we can construct an almost complex structure $J^\frd(\ve{T})$ on $\hat{\Sigma}\times [0,1]\times \R$ with necks of length $T_1$ and $T_2$.

Analogously to Definition~\ref{def:longenoughfreestab}, we make the following definition:

\begin{define}\label{def:1-handle-def-stabilizing-condition-3}
Suppose $\cH$ is a Heegaard diagram for $(Y,\ws)$ and $\frd$ is a gluing datum for attaching a 1-handle with feet $p_1,p_2\in Y\setminus \ws$. We say a pair of neck lengths $\ve{T}$ satisfies \emph{stabilizing condition}~\eqref{eq:stabilizationcondition3} if for any two pairs of neck-lengths $\ve{T}_1$ and $\ve{T}_2$ such that $\ve{T}_1,\ve{T}_2\ge \ve{T}$, componentwise,  there is a non-cylindrical almost complex structure $\tilde{J}$ on $\Sigma\times[0,1]\times \R$, interpolating $J^{\frd}(\ve{T}_1)$ and $J^{\frd}(\ve{T}_2)$, such that for all $\xs\in \bT_{\a}\cap \bT_{\b},$ we have
\begin{equation}
\begin{split}
\Psi_{\tilde{J}}(\xs\times \theta^+)&=\xs\times \theta^+,\quad\text{ and }\\
\Psi_{\tilde{J}}(\xs\times \theta^-)&=\xs\times \theta^-+\sum_{\ys\in \bT_{\a}\cap \bT_{\b}} C_{\xs,\ys}\cdot \ys\times \theta^+,
\end{split}
\tag{SC-3}\label{eq:stabilizationcondition3}
\end{equation}
for $C_{\xs,\ys}\in \bF_2[U_{\ws}]$ (which depend on $\frd$, $\ve{T}_1$ and $\ve{T}_2$).
\end{define}

Analogous to Proposition~\ref{prop:freestab-Tsufflargeexist}, we have the following:

\begin{prop}\label{prop:1-handle-Tsufflargeexist}
If $\frd$ is a gluing datum for attaching a 1-handle at $\{p_1,p_2\}$, then there is a pair of neck lengths $\ve{T}=(T_1,T_2)$ which satisfies stabilizing condition \eqref{eq:stabilizationcondition3}.
\end{prop}

We begin with the following Maslov index formula:

\begin{lem}\label{lem:index-1-handle} Suppose $\cH=(\Sigma,\as,\bs,\ws)$ is a diagram for $(Y,\ws)$ and $\hat{\cH}=(\hat{\Sigma}, \as\cup \{\alpha_0\}, \bs\cup \{\beta_0\}, \ws)$ is a diagram for the surgered manifold $Y(\bS^0)$. If $\phi\# \phi_0\in \pi_2(\xs\times x,\ys\times y)$ is a homology class of disks on $\hat{\cH}$, where $x,y\in \{\theta^+,\theta^-\}$, then
\[
\mu(\phi\# \phi_0)=\mu(\phi)+\gr(x,y).
\]
\end{lem}
\begin{proof}
By Lemma~\ref{prop:freestab-Tsufflargeexist}, the index of $\phi_0$ is
\begin{equation}
\mu(\phi_0)=2n_{p_1^0}(\phi_0)+2n_{p_2^0}(\phi_0)+\gr(x,y).\label{eq:Maslov-ind-on-sphere-1-handle}
\end{equation}
Noting that the Euler measure of a disk is 1, Lipshitz's formula for the Maslov index \cite{LipshitzCylindrical}*{Equation~8} implies  that
\begin{equation}
\mu(\phi\# \phi_0)=\mu(\phi)+\mu(\phi_0)-2n_{p_1^0}(\phi_0)-2n_{p_2^0}(\phi_0).\label{eq:Maslov-ind-on-sphere-1-handle-2}
\end{equation}
Combining Equations~\eqref{eq:Maslov-ind-on-sphere-1-handle} and ~\eqref{eq:Maslov-ind-on-sphere-1-handle-2} implies the main statement. 
\end{proof}

With the Maslov index formula from Lemma~\ref{lem:index-1-handle},  Proposition~\ref{prop:1-handle-Tsufflargeexist} is proven by adapting the proof of Proposition~\ref{prop:freestab-Tsufflargeexist} to handle stretching two necks instead of one. The main details of the argument are unchanged, so we leave them to the reader.

\subsection{1-handles, 3-handles and the differential}

Ozsv\'{a}th and Szab\'{o} proved that their 1-handle and 3-handle maps are chain maps \cite{OSProperties}*{Proposition~6.4}. We now prove that our version of the 1-handle and 3-handle maps are also chain maps.

\begin{prop}\label{prop:1-handle-differential} Suppose that $\cH=(\Sigma,\as,\bs,\ws)$ is a diagram for $(Y,\ws)$, and  $\hat{\cH}=(\hat{\Sigma}, \as\cup \{\alpha_0\}, \bs\cup \{\beta_0\}, \ws)$  is the diagram constructed by attaching a 1-handle at $\{p_1,p_2\}\subset \Sigma\setminus(\as\cup \bs\cup \ws).$ If $\frd$ is a gluing datum for this 1-handle, and $\ve{T}$ is a pair of neck lengths satisfying stabilizing condition~\eqref{eq:stabilizationcondition3}, then
\begin{equation}
\begin{split}
\d_{\hat{\cH}, J^{\frd}( \ve{T})}(\xs\times \theta^+)&=\d_{\cH,J^{\frd}}(\xs)\otimes \theta^+,\quad \text{and}\\
\d_{\hat{\cH},J^{\frd}(\ve{T})}(\ve{x}\otimes \theta^-)&=\d_{\cH,J^{\frd}}(\xs)\otimes \theta^-+\sum_{\ys\in \bT_{\a}\cap \bT_{\b}} C_{\xs,\ys}\cdot \ys\times \theta^+,
\end{split}
\label{eq:1-handle-3-handle-differential}
\end{equation}
for $C_{\xs,\ys}\in \bF_2[U_{\ws}].$
\end{prop}
\begin{proof} The proof is essentially the same as the proof of Proposition~\ref{prop:free-stabdifferential}. 

Suppose that $\phi\# \phi_0\in \pi_2(\xs\times x,\ys\times y)$ is a class with Maslov index 1. By Lemma~\ref{lem:index-1-handle},
 \begin{equation}
 \mu(\phi\# \phi_0)=\mu(\phi)+\gr(x,y).\label{eq:Maslov-index-after-1-handle}
 \end{equation}
 
Note that Equation~\eqref{eq:1-handle-3-handle-differential} makes no claim about the counts of classes with $\gr(x,y)=-1$ (i.e. where $x=\theta^-$ and $y=\theta^+$). These counts correspond to the $C_{\xs,\ys}$ in the statement. 
 
  As in the proof of Proposition~\ref{prop:free-stabdifferential}, if $\cM_{J^{\frd}(\ve{T})}(\phi\#\phi_0)$ is nonempty for arbitrarily large $\ve{T}$, then both $\phi$ and $\phi_0$ must have broken representatives.

If $\gr(x,y)=1$, then Equation~\eqref{eq:Maslov-index-after-1-handle} implies that $\mu(\phi)=0$. Since $\phi$ has a broken representative, it follows that $\phi$ is the constant class $e_{\xs}\in \pi_2(\xs,\xs)$. In this case, $\phi\# \phi_0$ must have domain equal to one of the two bigons in the 1-handle region. These curves cancel, modulo 2, and hence make no contribution to the differential.
 
 It remains to consider classes with $\gr(x,y)=0$, i.e. $x=y=\theta^+$ or $x=y=\theta^-$. For such classes, Equation~\eqref{eq:Maslov-index-after-1-handle} implies that
 \[
\mu(\phi\# \phi_0) =\mu(\phi)=1.
 \]

Write $n_1$ and $n_2$ for $n_{p_1}(\phi)$ and $n_{p_2}(\phi)$, respectively. Consider the map
\[
\rho^{p_1,p_2}\colon \cM_{J^{\frd}}(\phi)\to \Sym^{n_1}([0,1]\times \R)\times \Sym^{n_2}([0,1]\times \R).
\] 
 Write $X(\phi)$ for the image $\rho^{p_1,p_2}(\cM(\phi))$.
 
 As in the proof of Proposition~\ref{prop:free-stabdifferential}, for large $\ve{T}$, there is a fibered product description
 \begin{equation}
 \#\hat{\cM}_{J^{\frd}(\ve{T})}(\phi\# \phi_0)\equiv \sum_{u\in \hat{\cM}(\phi)} \# \cM_{J_0^{\frd}}(\phi_0,\rho^{p_1,p_2}(u)).\label{eq:fibered-product-1-handles-disk-counts}
 \end{equation}
 Consequently, it is sufficient to show that if $\theta\in \{\theta^+,\theta^-\}$ is fixed, and $\ve{d}_1\times \ve{d}_2\in \Sym^{n_1}([0,1]\times \R)\times \Sym^{n_2}([0,1]\times \R)$ is a point with no repeated entries, then for a generic almost complex structure $J_0$ on $S^2\times [0,1]\times \R$,
 \begin{equation}
 \sum_{\substack{\phi_0\in \pi_2(\theta,\theta)\\ n_{p_1^0}(\phi_0)=n_1\\ n_{p_2^0}(\phi_0)=n_2}}\# \cM_{J_0}(\phi_0, \ve{d}_1\times \ve{d}_2)\equiv 1\pmod{2}.\label{eq:divisor-count-1-handles}
 \end{equation}

%
 
If $\ve{d}_{1}\times \ve{d}_{2}\in \Sym^{n_1}([0,1]\times \R)\times \Sym^{n_2}([0,1]\times \R)$  has no entries with the same $[0,1]$-component, consider the path $\ve{D}_T:=\ve{d}_1^T\times \ve{d}_2$ obtained by translating $\ve{d}_1$ upwards by $T$ units in the $\R$-direction. If $\ve{d}_1\times\ve{d}_2$ is not in the fat diagonal, but two elements share the same $[0,1]$-component, a perturbation of this path may be chosen which avoids the fat diagonal. The compactification of the 1-dimensional space
 \[
\bigcup_{T\in [0,\infty)}\coprod_{\phi_0\in \pi_2(\theta,\theta)} \cM_{J_0}(\phi_0, \ve{D}_T)
 \]
 has ends in bijection with the Cartesian product
 \begin{equation}
\bigg(\coprod_{\substack{\phi_0\in \pi_2(\theta,\theta)\\ n_{p_1^0}(\phi_0)=n_1\\
n_{p_2^0}(\phi_0)=0}} \cM(\phi_0, \ve{d}_1)\bigg)\times \bigg(\coprod_{\substack{\phi_0\in \pi_2(\theta,\theta)\\ n_{p_1^0}(\phi_0)=0\\
n_{p_2^0}(\phi_0)=n_2}} \cM(\phi_0, \ve{d}_2)\bigg).\label{eq:ends-of-moduli-1-handle-region}
 \end{equation}
 Equation~\eqref{eq:OS'scountmatched} implies that the count of the elements in Equation~\eqref{eq:ends-of-moduli-1-handle-region} is $1$, modulo 2. Equation~\eqref{eq:divisor-count-1-handles} follows. 
 
 Combining Equations~\eqref{eq:fibered-product-1-handles-disk-counts} and ~\eqref{eq:divisor-count-1-handles}, it follows that 
 \[
 \sum_{\substack{\phi_0\in \pi_2(\theta,\theta)\\
 n_{p_1^0}(\phi_0)=n_1\\
 n_{p_2^0}(\phi_0)=n_2}}\# \hat{\cM}(\phi\# \phi_0)\equiv \# \hat{\cM}(\phi)\pmod{2}.
\]
The main claim now follows.
\end{proof}

\begin{cor}
 The 1-handle maps $F_{Y,\bS^0,\frt}$ and the 3-handle maps $F_{Y,\bS^2,\frt}$ are chain maps.
\end{cor}
\begin{proof} The result is an immediate consequence of Proposition~\ref{prop:1-handle-differential}, together with the formulas for the 1-handle and 3-handle maps in Equations~\eqref{eq:def-1-handle} and~\eqref{eq:def-3-handle}.
\end{proof}

\subsection{1-handles, 3-handles and triangle maps}

Similar to Theorem~\ref{thm:freestabilizetriangles}, the 1-handle and 3-handle maps satisfy a useful relationship with the triangle maps.

Suppose $\cT=(\Sigma,\as,\bs,\gs,\ws)$ is a multi-pointed Heegaard triple, and $p_1,p_2\in \Sigma\setminus (\as\cup \bs\cup \gs\cup \ws)$ are two points. We can form a new Heegaard triple $\hat{\cT}=(\hat{\Sigma},\as\cup \{\alpha_0\}, \bs\cup \{\beta_0\}, \gs\cup \{\gamma_0\}, \ws)$ by cutting out a neighborhood of the points $p_1$ and $p_2$, gluing in an annulus to connect the new boundary components. We add three new attaching curves, $\alpha_0$, $\beta_0$ and $\gamma_0$, in the new annular region. We assume that $\alpha_0$, $\beta_0$ and $\gamma_0$ satisfy the configuration shown in Figure~\ref{fig::73}.

\begin{figure}[ht!]
\centering
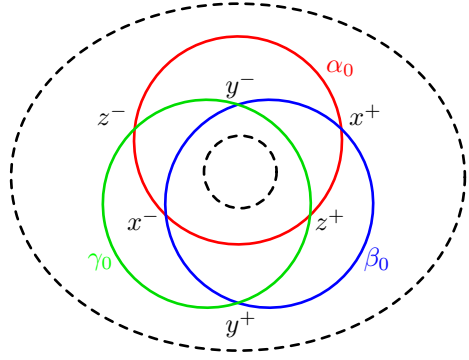
\caption{\textbf{Adding a 1-handle to a Heegaard triple.} The dashed circles are the boundaries of the new annular region.}\label{fig::73}
\end{figure}

\begin{lem}\label{lem:canonical-iso-spinc-1-handles} If $\cT=(\Sigma,\as,\bs,\gs,\ws)$ is a Heegaard triple, and $\hat{\cT}=(\hat{\Sigma},\as\cup \{\alpha_0\}, \bs\cup \{\beta_0\}, \gs\cup \{\gamma_0\}, \ws)$ is obtained by attaching a 1-handle, as above, then there is a canonical isomorphism
\begin{equation}
\Spin^c(X_{\a,\b,\g})\iso \Spin^c(X_{\a\cup \{\a_0\}, \b\cup \{\b_0\}, \g\cup \{\g_0\}}).\label{eq:canonical-iso-spinc-1-handles}
\end{equation}
\end{lem}
\begin{proof}  Write $\Sigma_0$ for the surface $\Sigma\setminus (N(p_1)\cup N(p_2))$, and write $X_{\a,\b,\g}^0$ for the 4-manifold obtained as the union
 \[
(\Delta\times\Sigma_0)\cup (e_{\a}\times U_\a^0)\cup (e_{\b}\times U_{\b}^0)\cup (e_{\g}\times U_{\g}^0),  
 \]
 where $U_{\a}^0$ is the 3-manifold with boundary and corners obtained by gluing 2-handles to $[0,1]\times \Sigma_0$ along $\as\times \{1\}$, and $U_{\b}^0$ and $U_{\g}^0$ are defined similarly. There are two restriction maps
 \begin{equation}
\Spin^c(X_{\a,\b,\g})\to \Spin^c(X^0_{\a,\b,\g})\quad \text{and} \quad  \Spin^c(X_{\a\cup \{\a_0\}, \b\cup \{\b_0\}, \g\cup \{\g_0\}})\to \Spin^c(X^0_{\a,\b,\g}).\label{eq:two-restriction-maps}
 \end{equation}
We leave it as a straightforward exercise for the reader to use the Mayer-Vietoris long exact sequences on cohomology to verify that both maps in Equation~\eqref{eq:two-restriction-maps} are isomorphisms, leading to the isomorphism in Equation~\eqref{eq:canonical-iso-spinc-1-handles}.
\end{proof}

\begin{thm}\label{thm:1-handletriangle} Suppose that $\cT$ is a Heegaard triple, and $\hat{\cT}$ is obtained by attaching a 1-handle, as in Figure~\ref{fig::73}.  Let $\frd$ be a gluing datum for the 1-handle attachment. If $\ve{T}$ is a pair of neck lengths which are sufficiently large, then with respect to the isomorphism of $\Spin^c$ structures in Equation~\eqref{eq:canonical-iso-spinc-1-handles}, we have
\begin{align*}
F_{\hat{\cT}, J^{\frd}(\ve{T}), \frs}(\xs\times x^+, \ys\times y^+)&=F_{\cT,J^\frd,\frs}(\xs, \ys)\otimes z^+,\\
  F_{\hat{\cT}, J^\frd(\ve{T}), \frs}(\xs\times x^+, \ys\times y^-)&=F_{\cT,J^\frd, \frs}(\xs, \ys)\otimes z^-+ \sum_{\zs\in \bT_{\a}\cap \bT_{\g}} C^1_{\xs,\ys,\zs}\cdot \zs\times z^+,\\
    F_{\hat{\cT}, J^\frd(\ve{T}), \frs}(\xs\times x^-, \ys\times y^+)&=F_{\cT,J^\frd, \frs}(\xs, \ys)\otimes z^-+ \sum_{\zs\in \bT_{\a}\cap \bT_{\g}} C^2_{\xs,\ys,\zs}\cdot \zs\times z^+,
\end{align*}
for $C_{\xs,\ys,\zs}^1, C_{\xs,\ys,\zs}^2\in \bF_2[U_{\ws}]$ (which depend on $\cT$,  $\frd$ and $\ve{T}$).
\end{thm}

\begin{proof} The proof follows the same line of reasoning as the proof of Theorem~\ref{thm:freestabilizetriangles}.

Suppose that $\psi\# \psi_0\in \pi_2(\xs\times x, \ys\times y, \zs\times z)$ is a homology class with Maslov index 0. Using Lemma~\ref{lem:maslovindextriangles}, Sarkar's formula for the Maslov index \cite{SarkarMaslov}, as well as the fact that the Euler measure of a disk is 1, we compute that
\begin{equation}
\mu(\psi\# \psi_0)=\mu(\psi)-\gr(x^+,x)-\gr(y^+,y)+\gr(z^+,z).\label{eq:general-maslov-index-1-handle-triangle}
\end{equation}
Define
\[
\delta(x,y,z):=-\gr(x^+,x)-\gr(y^+,y)+\gr(z^+,z).
\]
It is straightforward to check that the theorem statement follows from the following two subclaims:
\begin{enumerate}[ref= h-\arabic*, label= (h-\arabic*):]
\item\label{subclaim-1-handle-triangle-delta=1} If $\delta(x,y,z)=1$, then any Maslov index 0 class in $\pi_2(\xs\times x,\ys\times y, \zs\times z)$ has no $J^\frd(\ve{T})$-holomorphic representatives when $\ve{T}$ is sufficiently large. 
\item\label{subclaim-1-handle-triangle-delta=0} If $\delta(x,y,z)=0$, then the $\zs\times z$ coefficient of $F_{\cT,J^\frd(\ve{T}),\frs}(\xs\times x, \ys\times y)$ coincides with the $\zs$ coefficient of $F_{\hat{\cT},J^\frd,\frs}(\xs,\ys)$.
\end{enumerate}
The counts of the moduli spaces with $\delta(x,y,z)\in \{-1,-2\}$ are not relevant to the theorem statement.

Claim~\eqref{subclaim-1-handle-triangle-delta=1} is the easier of the two subclaims to verify, since it  relies only on compactness and transversality, but does not require gluing. By Equation~\eqref{eq:general-maslov-index-1-handle-triangle}, if $\delta(x,y,z)=1$, then $\mu(\psi)=-1$. If $\psi\# \psi_0$ has representatives for arbitrarily large $\ve{T}$, then $\psi$ must admit a broken  representative. However since $\mu(\psi)=-1$, there can be no broken representatives by transversality. Hence no index 0 classes with $\delta(x,y,z)=1$ have nonempty moduli spaces for sufficiently large $\ve{T}$.

We now consider Claim~\eqref{subclaim-1-handle-triangle-delta=0}. Assume $\delta(x,y,z)=0$. In this case, we have $\mu(\psi)=0$ by Equation~\eqref{eq:general-maslov-index-1-handle-triangle}. As in the proof of Theorem~\ref{thm:freestabilizetriangles}, if $\ve{T}_i$ is an increasing, unbounded sequence of pairs of neck-lengths, then any sequence of $J^{\frd}(\ve{T}_i)$-holomorphic representatives of $\psi\# \psi_0$ has a subsequence which converges to a pair $(u,u_0)$ where $u\in \cM_{J^{\frd}}(\psi)$ and $u_0\in \cM_{J^{\frd}_0}(\psi_0)$, and
\begin{equation}
\rho^{p_1,p_2}(u)=\rho^{p_1^0,p_2^0}(u_0).\label{eq:double-matching}
\end{equation}
In Equation~\eqref{eq:double-matching}, $\rho^{p_1,p_2}(u)\in  \Sym^{n_{p_1}(\psi)}(\Delta)\times \Sym^{n_{p_2}(\psi)}(\Delta)$ is the set
\[
\rho^{p_1,p_2}(u)=\bigg((\pi_\Delta\circ u)\big((\pi_\Sigma\circ u)^{-1}(p_1)\big),(\pi_\Delta\circ u)\big((\pi_\Sigma\circ u)^{-1}(p_2)\big)\bigg),
\]
and $\rho^{p_1^0,p_2^0}(u_0)$ is defined similarly.

Hence, using a gluing argument, it follows that for sufficient large $\ve{T}$ there is an identification
\[
 \cM_{J^\frd(\ve{T})}(\psi\# \psi_0)\iso \cM_{J^\frd}(\psi)\times_{\rho} \cM_{J_0^\frd}(\psi_0).
\]
Since $\cM_{J^{\frd}}(\psi)$ is zero dimensional, it is sufficient to show that if $\ve{d}_1\times \ve{d}_2\in \Sym^{n_1}(\Delta)\times \Sym^{n_2}(\Delta)$ is not in the fat diagonal, then
\begin{equation}
\sum_{\substack{\psi_0\in \pi_2(x,y,z)\\
n_{p_1^0}(\psi_0)=n_1\\
n_{p_2^0}(\psi_0)=n_2}}\# \cM(\psi_0,\ve{d}_1\times \ve{d}_2)\equiv 1, \label{eq:divisor-count-triangles-more-general}
\end{equation}
for a generic almost complex structure on $S^2\times [0,1]\times \R$.

Equation~\eqref{eq:divisor-count-triangles-more-general} is verified similarly to Equation~\eqref{eq:divisorcounttriangles}. Consider a path $(\ve{D}_t)_{t\in [1,\infty)}$ in $\Sym^{n_1}(\Delta)\times \Sym^{n_2}(\Delta)$, satisfying the following:
\begin{enumerate}
\item   $\ve{D}_1=\ve{d}_1\times \ve{d}_2$.
\item The image of $\ve{D}_t$ is disjoint from the fat diagonal.
\item All points of $\ve{D}_t$ travel into the $\alpha$-$\beta$ cylindrical end of $\Delta$.
\item For large $t$, the points of $\ve{D}_t$ are spaced at least distance $t$ apart (with respect to the identification of the $\alpha$-$\beta$ cylindrical end of $\Delta$ as $[0,1]\times (-\infty,0]$.
\item The $[0,1]$-component of all points in $\ve{D}_t$ approach a fixed $s_0\in (0,1)$.
\end{enumerate}

Write $\ve{\cD}:=\{\ve{D}_t: t\in [1,\infty)\}$ and consider the matched moduli space:
\[
\cM_{(x,y,z)}(\ve{\cD}):=\coprod_{\substack{\psi_0\in \pi_2(x,y,z)\\
n_{p_1^0}(\psi_0)=n_1\\n_{p_2^0}(\psi_0)=n_2}} \bigcup_{t\in [1,\infty)}\cM(\psi_0,\ve{D}_t).
\]
 As in the proof of Equation~\eqref{eq:divisorcounttriangles}, the space $\cM_{(x,y,z)}(\ve{\cD})$ has ends at $t=1$, at $t\in (1,\infty)$ and at $t=\infty$. The ends at $t=1$ correspond to the left hand side of Equation~\eqref{eq:divisor-count-triangles-more-general}. The ends at $t\in (1,\infty)$ correspond to index 1 strips breaking off at finite $t$, which do not pass over $p_1^0$ or $p_2^0$. The ends at $t=\infty$ correspond to the Cartesian product
\begin{equation}
\bigg(\coprod_{\substack{\psi_0^0\in \pi_2(x,y,z)\\
n_{p_1^0}(\psi^0_0)=n_{p_2^0}(\psi_0^0)=0}} \cM(\psi_0^0)\bigg)\times \bigg( \coprod_{\substack{\phi\in \pi_2(x,x)\\
n_{p_1^0}(\phi)=n_1\\
n_{p_2^0}(\phi)=0}}\cM(\phi, d)\bigg)^{n_1}\times \bigg( \coprod_{\substack{\phi\in \pi_2(x,x)\\
n_{p_1^0}(\phi)=0\\
n_{p_2^0}(\phi)=n_2}}\cM(\phi, d)\bigg)^{n_2},\label{eq:1-handle-triangle-cartesian-product}
\end{equation}
where $d\in \{s_0\}\times \R\subset [0,1]\times \R$ is a chosen point.

As in the proof of Theorem~\ref{thm:freestabilizetriangles}, the ends of $\cM_{(x,y,z)}(\ve{\cD})$ corresponding to strip breaking at finite $t$ cancel modulo 2, since $x$ is a cycle in the complex $\hat{\CF}(S^2,\alpha_0,\beta_0,p_1^0, p_2^0)$, and similarly $y$ and $z$ are cycles in their appropriate complexes.

The first factor of Equation~\eqref{eq:1-handle-triangle-cartesian-product} has total count 1, modulo 2, since it is easy to check that when $\delta(x,y,z)=0$, the only non-negative triangle class which has zero multiplicity over $p_1^0$ and $p_2^0$ is a small triangle, which clearly has a unique representative. The latter two factors of Equation~\eqref{eq:1-handle-triangle-cartesian-product} have total count 1, modulo 2, by \cite{OSLinks}*{Lemma~6.4}.

Hence, by counting the ends of $\cM_{(x,y,z)}(\ve{\cD})$, Equation~\eqref{eq:divisor-count-triangles-more-general} follows, completing the proof.
\end{proof}

\subsection{1-handles, 3-handles and gluing data}

Analogous to Proposition~\ref{prop:independencegluingdatum}, we have the following:

\begin{prop}\label{prop:independence-gluing-datum-1-handle}
Suppose $\cH=(\Sigma,\as,\bs,\ws)$ is diagram for $(Y,\ws)$, $\bS^0=\{p_1,p_2\}\subset \Sigma\setminus (\as\cup \bs\cup \ws)$ is a 0-sphere, and $\frd_1$ and $\frd_2$ are two gluing data for attaching a 1-handle at $\bS^0$. Suppose further that $\ve{T}_1$ and $\ve{T}_2$ are pairs of neck-lengths which satisfy stabilizing condition~\eqref{eq:stabilizationcondition3} for $\frd_1$ and $\frd_2$, respectively. Write $\hat{\cH}_1$ and $\hat{\cH}_2$ for the diagrams constructed by attaching a 1-handle to $\cH$ using $\frd_1$ and $\frd_2$ (note that $\hat{\cH}_1$ and $\hat{\cH}_2$ differ only by an isotopy in the 1-handle region).
The following diagram commutes up to chain homotopy:
\begin{equation}
 \begin{tikzcd}[column sep=4cm]\CF^-_{J^{\frd_1}}(\cH,\frs)\arrow{r}{\Psi_{J^{\frd_1}\to J^{\frd_2}}}\arrow{d}{F_{Y,\bS^0,\frt}}& \CF^-_{J^{\frd_2}}(\cH,\frs\arrow{d}{F_{Y,\bS^0,\frt}})\\
 \CF^-_{J^{\frd_1}(\ve{T}_1)}\left(\hat{\cH}_1,\hat{\frs}\right)\arrow{r}{\Psi_{(\hat{\cH}_1,J^{\frd_1}(\ve{T}_1))\to (\hat{\cH}_2,J^{\frd_2}(\ve{T}_2))}} &\CF^-_{J^{\frd_2}(\ve{T}_2)}\left(\hat{\cH}_2,\hat{\frs}\right).
 \end{tikzcd} 
\end{equation}
An analogous relation holds for the 3-handle maps.
\end{prop}

Proposition~\ref{prop:independence-gluing-datum-1-handle} is proven by adapting our proof of the analogous result for the free-stabilization maps in Proposition~\ref{prop:independencegluingdatum}. 

\subsection{Invariance of the 1- and 3-handle maps}

In this section, we prove that the 1-handle and 3-handle maps are invariant from the choices used in the construction.

\begin{thm}\label{thm:invariance-1-handle} 
The 1-handle and 3-handle maps determine well defined chain maps of transitive systems of chain complexes.
\end{thm}
\begin{proof}We focus on the 1-handle maps, since the 3-handle maps are dual.

Suppose $\bS^0=\{p_1,p_2\}$ is an embedded 0-sphere in $Y$, $\frt\in \Spin^c(W(Y,\bS^0))$, and  $\cH_1=(\Sigma_1,\as_1,\bs_1,\ws)$ and $\cH_2=(\Sigma_2,\as_2,\bs_2,\ws)$ are two Heegaard diagrams for $(Y,\ws)$ such that $\{p_1,p_2\}\subset \Sigma_i\setminus (\as_i\cup \bs_i\cup \ws)$ for $i\in \{1,2\}$. Let $\frd_1$ and $\frd_2$ be two  gluing data. It is sufficient to show that if $\ve{T}_1$ and $\ve{T}_2$ are two pairs of neck lengths satisfying condition~\eqref{eq:stabilizationcondition3}, then the following diagram commutes up to chain homotopy:
 \begin{equation}
\begin{tikzcd}[column sep=4cm]
\CF^-_{J^{\frd_1}}(\cH_1,\frs)\arrow{r}{\Psi_{(\cH_1,J^{\frd_1})\to (\cH_2,J^{\frd_2})}}\arrow{d}{F_{Y,\bS,\frt}} & \CF^-_{J^{\frd_2}}(\cH_2,\frs)\arrow{d}{F_{Y,\bS,\frt}}\\
\CF^-_{J^{\frd_1}(\ve{T}_1)}\left(\hat{\cH}_1, \hat{\frs}\right)\arrow{r}{\Psi_{(\hat{\cH}_2,J^{\frd_1}(\ve{T}_1))\to (\hat{\cH}_1,J^{\frd_2}(\ve{T}_2))}} & \CF^-_{J^{\frd_2}(\ve{T}_2)}\left(\hat{\cH}_2, \hat{\frs}\right).
\end{tikzcd} 
\label{eq:commutative-diagram-1-handles}
 \end{equation}

By Proposition~\ref{prop:independence-gluing-datum-1-handle}, the 1-handle maps are independent of the choice of gluing data (i.e. Equation~\eqref{eq:commutative-diagram-1-handles} commutes when $\cH_1=\cH_2$).

 By Lemma~\ref{lem:Heegaardmoveswithextrabasepoints}, we can connect $\cH_1$ and $\cH_2$ by a sequence of the following moves:
\begin{enumerate}
\item\label{move:HD-1handle1} Handleslides and isotopies of the $\as$ and $\bs$ curve (possibly passing over $p_1$ and $p_2$).
\item\label{move:HD-1handle2} Simple stabilizations, away from $\ws\cup \{p_1,p_2\}$.
\item\label{move:HD-1handle3} Changing the embedding of the Heegaard surface by an isotopy $\phi_t\colon \Sigma\to Y$ which is fixed on $\ws\cup \{p_1,p_2\}$ for all $t$.
\end{enumerate}

Since the transition maps for handleslides and isotopies can be computed using a composition of triangle maps, Theorem~\ref{thm:1-handletriangle} implies invariance under Move~\eqref{move:HD-1handle1}. Invariance under simple stabilizations away from $\ws\cup \{p_1,p_2\}$ follows from a triple neck-stretching argument, similar to Lemma~\ref{lem:freestabsimplestab}. Finally invariance under Move~\eqref{move:HD-1handle3} is tautological.
 
\end{proof}
\subsection{Further properties of the 1- and 3-handle maps}

In this section we prove several additional results about the 1-handle and 3-handle maps.

\begin{lem}\label{lem:rel-hom-1-handle-commute} Suppose that $\lambda$ is a path in $Y$, $\bS$ is an embedded 0-sphere or 2-sphere in $Y\setminus \ws$ and $\frt\in \Spin^c( W(Y,\bS))$. Suppose that $\lambda\subset Y\setminus \bS$ is a closed loop, or a path connecting two basepoints. Write $\lambda$ also for the induced path in $Y(\bS)$. Then
\[
A_{\lambda}\circ F_{Y,\bS,\frt}\simeq F_{Y,\bS,\frt}\circ A_{\lambda}.
\]
\end{lem}
\begin{proof}The proof is similar to the proof of Lemma~\ref{lem:free-stab-H1-act-commute}, and follows by examining the curves counted by the differential in Proposition~\ref{prop:1-handle-differential}, using the fact that the path $\lambda$ does not enter the 1-handle region.  We leave it to the reader to make the necessary notational modifications to the proof of Lemma~\ref{lem:free-stab-H1-act-commute}.
\end{proof}

Next, we consider commuting 1-handle and 3-handle maps amongst each other. If $\bS\subset Y$ is a union of framed $k$-spheres, we write $W(Y,\bS)$ for the cobordism from $Y$ to $Y(\bS)$ obtained by attaching a $k+1$ handles to $[0,1]\times Y$ along $\{1\}\times \bS$.  Note that if $\bS$ and $\bS'$ are two disjoint, framed spheres in $Y$, then the 4-manifolds $W(Y,\bS\cup \bS')$ and $W(Y(\bS),\bS')\cup W(Y,\bS)$
and $W(Y(\bS'),\bS)\cup W(Y,\bS')$ are diffeomorphic, via diffeomorphisms which are well defined up to isotopy.

\begin{lem}\label{lem:1-handles-commute}
 Suppose that $\bS$ and $\bS'$ are two disjoint embedded spheres of dimension 0 or 2 in $Y$. Write $\frt$ for a $\Spin^c$ structure on 
 \[
 W(Y,\bS\cup \bS')\iso W(Y(\bS),\bS')\cup W(Y,\bS)\iso W(Y(\bS'),\bS)\cup W(Y,\bS').
 \]
 Then
 \[
F_{Y(\bS),\bS', \frt|_{W(Y(\bS), \bS')}}\circ F_{Y, \bS, \frt|_{W(Y,\bS)}}\simeq 
F_{Y(\bS'),\bS,\frt|_{W(Y(\bS'),\bS)}} \circ F_{Y,\bS',\frt|_{W(Y,\bS')}}.
 \]
\end{lem}
\begin{proof} The proof follows from a quadruple neck-stretching argument, similar to the double neck-stretching argument from Proposition~\ref{prop:free-stabs-commute}, used to show that free-stabilization maps commute with each other. The present statement follows from the following subclaim:

\begin{subclaim}\label{subclaim:neck-stretching-1-handle} Suppose that $\ve{T}$ and $\ve{T}'$ are two 4-tuples of neck-lengths for attaching two 1-handles. If all of components of $\ve{T}$ and $\ve{T}'$ are sufficiently large, then a non-cylindrical almost complex structure $\tilde{J}$  interpolating $J(\ve{T})$ and $J(\ve{T}')$ may be chosen so that
\begin{equation*}
\begin{split}
\Psi_{\tilde{J}}(\xs\times \theta_1^+\times \theta_2^+)&= \xs\times \theta_1^+\times \theta_2^+
\\
\Psi_{\tilde{J}}(\xs\times \theta_1^+\times \theta_2^-)&=\xs\times \theta_1^+\times \theta_2^-+\sum_{\ys\in \bT_{\a}\cap \bT_{\b}} \left( C_{\xs,\ys}^1 \cdot  \ys\times \theta_1^+\times \theta_2^++ C_{\xs,\ys}^2\cdot  \ys\times \theta_1^-\times \theta_2^+ \right)
\\
\Psi_{\tilde{J}}(\xs\times \theta_1^-\times \theta_2^+)&=\xs\times \theta_1^-\times \theta_2^++\sum_{\ys\in \bT_{\a}\cap \bT_{\b}} \left(C_{\xs,\ys}^3 \cdot \ys\times \theta_1^+\times \theta_2^++C_{\xs,\ys}^4 \cdot \ys\times \theta_1^+\times \theta_2^-\right)
\\
\Psi_{\tilde{J}}(\xs\times \theta_1^-\times \theta_2^-)&=\xs\times \theta_1^-\times \theta_2^-
\\
+\sum_{\ys\in \bT_{\a}\cap \bT_{\b}} &\left(C_{\xs,\ys}^5\cdot  \ys\times \theta_1^+\times \theta_2^-+C_{\xs,\ys}^6\cdot  \ys\times \theta_1^-\times \theta_2^++C_{\xs,\ys}^7\cdot  \ys\times \theta_1^+\times\theta_2^+\right),
\end{split}
\end{equation*}
for various $C_{\xs,\ys}^i\in \bF_2[U_{\ws}]$, which depend on $\tilde{J}$.
\end{subclaim}

The proof Subclaim~\ref{subclaim:neck-stretching-1-handle} follows the same reasoning as the proof of Subclaim~\ref{subclaim:neckstretching}, the analogous subclaim  of Proposition~\ref{prop:free-stabs-commute}. The main difference is that we must use the index formula from Lemma~\ref{lem:index-1-handle} to compute the index of a homology class after adding two 1-handles. We leave it to the reader to verify that the argument carries over to our present context without major change.
\end{proof}

\begin{lem}\label{lem:1-handle-free-stab-commute}
 Suppose that $\bS$ is an embedded 0- or 2-sphere in a multi-pointed 3-manifold $(Y,\ws)$. If $w\in Y\setminus (\ws\cup \bS)$ is a new basepoint, then
\[
S_w^+\circ F_{Y,\bS,\frt}\simeq F_{Y,\bS,\frt}\circ S_w^+ \quad \text{and} \quad S_w^-\circ F_{Y,\bS,\frt}\simeq F_{Y,\bS,\frt}\circ S_w^-.
\]
\end{lem}
\begin{proof}  The proof follows from a triple neck-stretching argument, similar to the ones we encountered in Proposition~\ref{prop:free-stabs-commute} and Lemma~\ref{lem:1-handles-commute}. We leave the necessary modifications to the reader.
\end{proof}

\subsection{Ozsv\'{a}th and Szab\'{o}'s  1-handle and 3-handle maps}
\label{sec:OS-1handle3-handle}
Ozsv\'{a}th and Szab\'{o} originally defined the 1-handle and 3-handle maps by taking the connected sum of the Heegaard surface with a genus 1 Heegaard diagram for $S^1\times S^2$ using the same formula as in Equations~\eqref{eq:def-1-handle} and ~\eqref{eq:def-3-handle}. Morally, this amounts to picking a path between the two components of the attaching 0-sphere of a 1-handle. For showing invariance the graph cobordism maps, it it is convenient to show that our definition coincides with Ozsv\'{a}th and Szab\'{o}'s original definition~\cite{OSTriangles}*{Section~4.3}, when the feet of the 1-handle are in the same component of the 3-manifold. This amounts to the following change of almost complex structure computation:

\begin{lem}\label{lem:2-defs-1-handle-equiv}Let $\cH=(\Sigma,\as,\bs,\ws)$ denote a Heegaard diagram and let  $\cH'=(\Sigma\# \bT^2, \as\cup \{\alpha_0\}, \bs\cup \{\beta_0\},\ws)$ denote the diagram obtained by connect summing the diagram $(\bT^2,\alpha_0,\beta_0)$ for $S^1\times S^2$, as shown in Figure~\ref{fig::76}. Let $c$, $c_1$ and $c_2$ denote the three circles labeled therein. If $\ve{T}=(T,T_1,T_2)$ is a triple of positive real numbers, let $J(\ve{T})$ denotes an almost complex structure which has been stretched along $c$, $c_1$ and $c_2$, with neck-lengths $T$, $T_1$ and $T_2$. If all components of $\ve{T}$ and $\ve{T}'$ are sufficiently large, then there exists a non-cylindrical almost complex structure $\tilde{J}$ on $\Sigma\# \bT^2\times [0,1]\times \R$, interpolating $J(\ve{T})$ and $J(\ve{T}')$, such that
\begin{equation}
\begin{split}
\Psi_{\tilde{J}} (\xs\times \theta^+)&= \xs\times \theta^+\qquad \text{and}\\
\Psi_{\tilde{J}}(\ve{x}\times \theta^-)&=\ve{x}\times \theta^-+\sum_{\ys\in \bT_{\a}\cap \bT_{\b}} C_{\xs,\ys}\cdot \ys\times \theta^+,
\end{split}
\label{eq:defs-of-1-handle-maps-equivalent}
\end{equation}
for some $C_{\xs,\ys}\in \bF_2[U_{\ws}]$ (which depend on $\tilde{J}$). 
\end{lem}

\begin{figure}[ht!]
\centering
\begingroup%
  \makeatletter%
  \providecommand\color[2][]{%
    \errmessage{(Inkscape) Color is used for the text in Inkscape, but the package 'color.sty' is not loaded}%
    \renewcommand\color[2][]{}%
  }%
  \providecommand\transparent[1]{%
    \errmessage{(Inkscape) Transparency is used (non-zero) for the text in Inkscape, but the package 'transparent.sty' is not loaded}%
    \renewcommand\transparent[1]{}%
  }%
  \providecommand\rotatebox[2]{#2}%
  \newcommand*\fsize{\dimexpr\f@size pt\relax}%
  \newcommand*\lineheight[1]{\fontsize{\fsize}{#1\fsize}\selectfont}%
  \ifx\svgwidth\undefined%
    \setlength{\unitlength}{115.3428398bp}%
    \ifx\svgscale\undefined%
      \relax%
    \else%
      \setlength{\unitlength}{\unitlength * \real{\svgscale}}%
    \fi%
  \else%
    \setlength{\unitlength}{\svgwidth}%
  \fi%
  \global\let\svgwidth\undefined%
  \global\let\svgscale\undefined%
  \makeatother%
  \begin{picture}(1,0.94593635)%
    \lineheight{1}%
    \setlength\tabcolsep{0pt}%
    \put(0,0){\includegraphics[width=\unitlength,page=1]{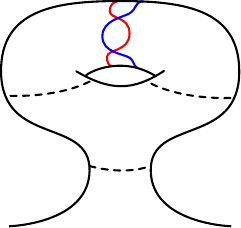}}%
    \put(0.49997281,0.27356502){\color[rgb]{0,0,0}\makebox(0,0)[t]{\lineheight{1.25}\smash{\begin{tabular}[t]{c}$c$\end{tabular}}}}%
    \put(0.24967265,0.59238642){\color[rgb]{0,0,0}\makebox(0,0)[rt]{\lineheight{1.25}\smash{\begin{tabular}[t]{r}$c_1$\end{tabular}}}}%
    \put(0.73363436,0.59238642){\color[rgb]{0,0,0}\makebox(0,0)[lt]{\lineheight{1.25}\smash{\begin{tabular}[t]{l}$c_2$\end{tabular}}}}%
    \put(0.4010187,0.78193261){\color[rgb]{0,0,1}\makebox(0,0)[rt]{\lineheight{1.25}\smash{\begin{tabular}[t]{r}$\beta_0$\end{tabular}}}}%
    \put(0.5683646,0.78279604){\color[rgb]{1,0,0}\makebox(0,0)[lt]{\lineheight{1.25}\smash{\begin{tabular}[t]{l}$\alpha_0$\end{tabular}}}}%
  \end{picture}%
\endgroup%

\caption{\textbf{The diagram $\cH'=(\Sigma\# \bT^2, \as\cup \{\alpha_0\}, \bs\cup \{\beta_0\},\ws)$ and the circles $c,$ $c_1$ and $c_2$ in Lemma~\ref{lem:2-defs-1-handle-equiv}.} }\label{fig::76}
\end{figure}

\begin{proof} First, a modification of Lemma~\ref{lem:indexdisksonsphere} implies that if $\phi\in \pi_2(\xs,\ys)$ is class of disks on $(\Sigma,\as,\bs,\ws)$, and $\phi_0\in \pi_2(x,y)$ is a class of disks on $(\bT^2,\alpha_0,\beta_0)$, then
\begin{equation}
\mu(\phi\# \phi_0)=\mu(\phi)+\gr(x,y). \label{eq:index-equiv-defs-1-handle}
\end{equation}
Note that the present lemma statement concerns only classes with $\gr(x,y)\ge 0$. 

Suppose that $\ve{T}_i$ and $\ve{T}_i'$ are two sequences of neck-lengths, all of whose components approach $+\infty$. 
We pick non-cylindrical almost complex structures $\tilde{J}_i$, interpolating $J(\ve{T}_i)$ and $J(\ve{T}_i')$, such that $(\Sigma\# \bT^2\times [0,1]\times \R, \tilde{J}_i)$ contains the almost complex submanifold $((\Sigma\setminus N_i)\times [0,1]\times \R,J)$, where $N_i$ is some nested sequence of open neighborhoods of the connected sum point $p\in \Sigma$, such that $\bigcap_{i\in \N} N_i=\{p\}$, and $J$ is a fixed, cylindrical almost complex structure on $\Sigma\times [0,1]\times \R$. Fix a Maslov index 0 class $\phi\# \phi_0$. As in the proof of Proposition~\ref{prop:freestab-Tsufflargeexist},  from a sequence $u_i$ of $\tilde{J}_i$-holomorphic representatives of $\phi\# \phi_0$, we may extract a broken representative of $\phi$. 

If $\gr(x,y)=1$ and $\phi\# \phi_0$ is a Maslov index 0 class with $\tilde{J}_i$-holomorphic representatives for large $i$, then Equation~\eqref{eq:index-equiv-defs-1-handle} implies that $\mu(\phi)=-1$. However, since $\phi$ admits a broken $J$-holomorphic representative and $J$ is cylindrical, we must have $\mu(\phi)\ge 0$, so we obtain a contradiction. Hence, such classes $\phi\# \phi_0$ have no representatives for large $i$.

If $\gr(x,y)=0$ and $\phi\# \phi_0$ is a Maslov index 0 class with $\tilde{J}_i$-holomorphic representatives for large $i$, then Equation~\eqref{eq:index-equiv-defs-1-handle} implies $\mu(\phi)=0$. Since $\phi$ admits broken holomorphic representatives for a generic, cylindrical almost complex structure, we conclude that $\phi$ is the constant class. There are no non-negative, Maslov index 0 classes on $(\bT^2,\alpha_0,\beta_0)$ with zero multiplicity over the connected sum point. Hence $\phi_0$ must also be a constant class, $e_{\theta}$ for $\theta\in \{\theta^+,\theta^-\}$. On the other hand, the class $e_{\xs\times \theta}$ always has $\tilde{J}_i$-holomorphic representatives. Equation~\eqref{eq:defs-of-1-handle-maps-equivalent} follows and the proof is complete.
\end{proof}

\section{2-handles}
\label{sec:2-handles}
In this section, we describe the cobordism maps for 2-handles. The maps we describe are essentially the same as those defined by Ozsv\'{a}th and Szab\'{o} \cite{OSTriangles}. They are also similar to the versions defined by Juh\'{a}sz \cite{JCob}  in the setting of cobordisms of sutured manifolds. 

\subsection{Definition of the 2-handle maps}

\begin{define} A \emph{framed link} $\bS^1$ in a 3-manifold $Y$ is a collection of pairwise disjoint, embedded knots $K_1,\dots, K_m\subset Y$ together with a choice framing, i.e. a choice of homology classes $\ell_i\in H_1(\d N(K_i);\Z)$ satisfying $\mu_i\cdot \ell_i=1$ for a meridian $\mu_i$ of $K_i$. 
\end{define}

Adapting Ozsv\'{a}th and Szab\'{o}'s definition in the singly pointed setting \cite{OSTriangles}*{Definition~4.1}, we make the following definition:

\begin{define}Suppose $(Y,\ws)$ is a multi-pointed 3-manifold and $\bS^1$ is a framed link in $Y$, with components $K_1,\dots, K_m$. A \emph{bouquet} $B$ of $\bS^1$ is an embedded forest in $Y$ (i.e. a collection of embedded, contractible graphs), such that each leaf of $B$ is a point on $K_i$, or a point in $\ws$. Furthermore, we assume the following:
\begin{enumerate}
\item For each $i\in \{1,\dots, m\}$, the set $B\cap K_i$ contains a single point.
\item Each connected component of $B$ intersects exactly one point in $\ws$, and furthermore that point is a leaf of $B$.
\end{enumerate} 
\end{define}

 If $B$ is a bouquet of the link $\bS^1$ in $(Y,\ws)$, then we form a manifold with boundary $Y_{B\cup \bS^1,\ws}$, by removing a regular neighborhood of $B\cup \bS^1$. We  decorate the boundary of $Y_{B\cup \bS^1,\ws}$ with a collection of \emph{sutures} (i.e. a collection of oriented, simple closed curves which divide the boundary into two subsurfaces) by adding one contractible suture for each basepoint in $\ws$.

\begin{define}Suppose that $\bS^1$ is a framed link in $Y$, with components $K_1,\dots, K_m$. We say that a Heegaard triple 
\[
(\Sigma,\as,\bs,\bs',\ws)=(\Sigma, \{\alpha_1,\dots, \alpha_n\}, \{\beta_1,\dots, \beta_n\}, \{\beta'_1,\dots, \beta_n'\}, \ws)
\]
 is \emph{subordinate} to a bouquet $B$ for the framed link $\bS^1$ in $Y$ if the following hold:
 \begin{enumerate}
 \item $(\Sigma,\as,\bs,\ws)$ is a Heegaard diagram for $(Y,\ws)$.
 \item If $\Sigma_0$ denotes the surface obtained by removing small neighborhoods of the basepoints $\ws$, then 
 \[
 (\Sigma_0,\{\alpha_1,\dots, \alpha_n\}, \{\beta_{m+1},\dots, \beta_{n}\})
 \]
 is a sutured Heegaard diagram for the sutured manifold $Y_{B\cup \bS^1,\ws}$.
 \item The curves $\beta'_{m+1},\dots, \beta'_{n}$ are small isotopies of the curves $\beta_{m+1},\dots, \beta_{n}$. Furthermore,
 \[
 |\beta_i\cap \beta_j'|=2\delta_{ij},
 \]
 whenever $m+1\le i,j\le n$.
 \item If $1\le i\le m$, then $\beta_{i}$ is a meridian of $K_i$.
 \item If $1\le i\le m$, then $\beta_{i}'$ is a longitude of $K_i$, corresponding to the framing.
 \item If $1\le i\le m$, then $\beta_{i}'$ is disjoint from the curves $\beta_{m+1},\dots, \beta_{n}$. Furthermore,
   \[
|\beta_{i}'\cap \beta_{j}| =\delta_{ij},
 \] 
 whenever $1\le i,j\le m$.
 \end{enumerate}
\end{define}

Since $Y_{\b,\b'}$ is a connected sum of $g(\Sigma)-m$ copies of $S^1\times S^2$, a theorem of Laudenbach and Po\'{e}naru \cite{LaudenbachPoenaru} implies that the diffeomorphism type of the 4-manifold obtained by attaching 3 and 4-handles to $Y_{\b,\b'}\subset \d X_{\a,\b,\b'}$ is unique. Similar to \cite{OSTriangles}*{Proposition~4.3}, we have the following simple description of the resulting 4-manifold:

\begin{lem}\label{lem:Xabb'=handle-cobordism}
Suppose $(\Sigma,\as,\bs,\bs',\ws)$ is subordinate to a bouquet for a framed link $\bS^1$ in $Y$. After filling in the boundary component $Y_{\b,\b'}\subset \d X_{\a,\b,\b'}$ with 3- and 4-handles, we obtain the handle cobordism $W(Y,\bS^1)$.
\end{lem}
\begin{proof} Add an extra product layer and view $W(Y,\bS^1)$ as the union
\begin{equation}
W(Y,\bS^1)=([1,2]\times Y(\bS^1))\cup H_2 \cup ([0,1]\times Y), \label{eq:add-product-layers}
\end{equation}
where $H_2$ is a union of copies of $D^2\times D^2$ (the 2-handles). Write
\[
Y=U_{\a}\cup ([0,1]\times \Sigma)\cup U_{\b}\qquad \text{and} \qquad Y(\bS)=U_{\a}\cup ([0,1]\times \Sigma)\cup U_{\b'},
\]
and delete 
\[
W_0:=([1,1+\epsilon]\times U_{\b'})\cup H,
\]
from $W(Y,\bS^1)$. We note that $W_0$ is a 4-dimensional handlebody. Using the description of $W(Y,\bS^1)$ from Equation~\eqref{eq:add-product-layers}, we can write
\begin{equation}
W\setminus \Int (W_0)=\left(U_{\a}\times [0,2]\right)\cup \left(U_{\b}\times [0,1]\right)\cup \left(U_{\b'}\times [1+\epsilon,2]\right)\cup ([0,2]\times [0,1]\times \Sigma).\label{eq:W(Y,S)-cut-out-middle-layer}
\end{equation}
Upon rounding corners and identifying $[0,2]\times [0,1]$ topologically with a triangle, Equation~\eqref{eq:W(Y,S)-cut-out-middle-layer} is identical to $X_{\a,\b,\b'}$, as defined in Equation~\eqref{eq:X_abgdef}.
\end{proof}

We now define the cobordism maps for 2-handle cobordisms. Suppose $\bS^1\subset Y$ is a framed link in $Y$, and $B$ is a bouquet. Let $(\Sigma,\as,\bs,\bs',\ws)$ be a Heegaard triple subordinate to $B$. There is a unique top graded intersection point $\Theta_{\beta,\beta'}^+\in \bT_{\b}\cap \bT_{\b'}$. Furthermore, it is straightforward to see that $\Theta_{\beta,\beta'}^+$ is a cycle in the complex $\CF^-(\Sigma,\bs,\bs',\ws,\frs_0)$, where $\frs_0\in \Spin^c(Y_{\b,\b'})$ is the unique torsion $\Spin^c$ structure.

By Lemma~\ref{lem:Xabb'=handle-cobordism}, there is a canonical isomorphism
\[
\Spin^c(X_{\a,\b,\b'})\iso \Spin^c(W(Y,\bS^1)).
\]
Therefore, we will not distinguish between $\Spin^c$ structures on $X_{\a,\b,\b'}$ and $W(Y,\bS^1)$.

If $\frs\in \Spin^c(W(Y,\bS^1))$, the 2-handle map
\[
F_{Y,\bS^1,\frs}\colon \CF^-(\Sigma,\as,\bs,\ws,\frs|_{Y})\to \CF^-(\Sigma,\bs,\bs',\ws,\frs|_{Y(\bS^1)})
\]
is defined as the holomorphic triangle map
\begin{equation}
F_{Y,\bS^1,\frs}(\xs):=F_{\a,\b,\b',\frs}(\xs\otimes\Theta^+_{\beta,\beta'})=\sum_{\zs\in \bT_{\a}\cap \bT_{\b'}} \sum_{\substack{\psi\in \pi_2(\xs,\Theta^+_{\b,\b'},\zs)\\
\mu(\psi)=0\\
\frs_{\ws}(\psi)=\frs}} \# \cM(\psi)U_{\ws}^{n_{\ws}(\psi)}\cdot \zs, \label{def:2-handlemap}
\end{equation}
extended $\bF_2[U_{\ws}]$-equivariantly.

\subsection{Simple stabilizations and triangle maps}

In this section, we prove that the holomorphic triangle maps are invariant under simple stabilizations. For Ozsv\'{a}th and Szab\'{o}'s original argument in the context of the symmetric product, we refer the reader to \cite{OSDisks}*{Theorem~10.4} and \cite{OSTriangles}*{Theorem~2.14}.

\begin{define} Let $(\bT^2, \delta_0,\epsilon_0,\epsilon'_0,p_0)$ denote the genus 1 Heegaard triple shown in Figure~\ref{fig::74}. If $(\Sigma,\as,\bs,\gs,\ws)$ is a Heegaard triple, we say the Heegaard triple
\[
(\Sigma', \as\cup \{\alpha_0\}, \bs\cup \{\beta_0\}, \gs\cup \{\gamma_0\}, \ws)
\]
is a \emph{simple stabilization} of $(\Sigma,\as,\bs,\gs,\ws)$ if $\Sigma'=\Sigma\# \bT^2$, and $\{\alpha_0,\beta_0,\gamma_0\}=\{\delta_0,\epsilon_0,\epsilon_0'\}$, setwise.
\end{define}

\begin{figure}[ht!]
\centering
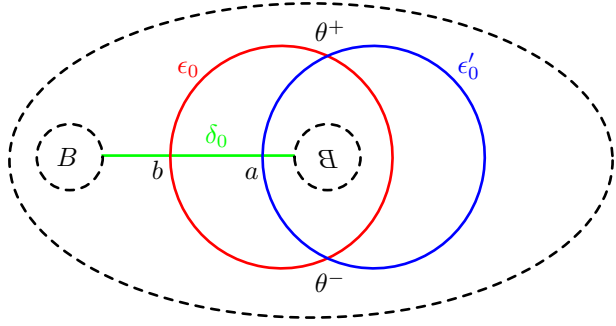
\caption{\textbf{The triple $(\bT^2,\delta_0,\epsilon_0,\epsilon_0',p_0)$ used in a simple stabilization of a Heegaard triple.} We view  $p_0$ as being the outer dashed circle, collapsed to a point. }\label{fig::74}
\end{figure}


\begin{lem}\label{lem:spinc-isomoprhism-simple-stabl}
If $(\Sigma,\as,\bs,\gs,\ws)$ is a Heegaard triple and $(\Sigma\# \bT_0, \as\cup \{\alpha_0\}, \bs\cup \{\beta_0\}, \gs\cup \{\gamma_0\}, \ws)$ is a simple stabilization, then there is a canonical isomorphism
\begin{equation}
\Spin^c(X_{\a,\b,\g})\iso \Spin^c(X_{\a\cup \{\a_0\}, \b\cup \{\b_0\}, \g\cup \{\g_0\}}),\label{eq:spin-c-canonical-isomorphism-simple-stab}
\end{equation}
\end{lem}
The proof of Lemma~\ref{lem:spinc-isomoprhism-simple-stabl} is a Mayer-Vietoris argument similar to Lemma~\ref{lem:canonical-iso-spinc-1-handles}, which we leave to the reader.

\begin{thm}\label{thm:simple-stabilization-triangle-maps} Suppose $\cT=(\Sigma,\as,\bs,\gs,\ws)$ is a Heegaard triple, $(\bT^2,\delta_0,\epsilon_0,\epsilon'_0,p_0)$ is the Heegaard triple shown in Figure~\ref{fig::74}, and $\frd$ is a gluing datum for stretching the neck. Let $a,$ $b$, $\theta^+$ and $\theta^-$ denote the intersection points shown in Figure~\ref{fig::74}. For sufficiently large $T$, we have the following:
 \begin{equation}
 \begin{split}
 F_{\a\cup \{\epsilon_0\}, \b\cup \{\epsilon_0'\}, \g\cup \{\delta_0\}, J^\frd(T), \frs}(\xs\times \theta^+, \ys\times a)&=F_{\a,\b,\g, J^\frd, \frs}(\xs,\ys)\otimes b,
 \\
 F_{\a\cup \{\delta_0\}, \b\cup \{\epsilon_0\}, \g\cup \{\epsilon_0'\}, J^\frd(T), \frs}(\xs\times b, \ys\times \theta^+)&=F_{\a,\b,\g, J^\frd, \frs}(\xs,\ys)\otimes a,
 \\
 F_{\a\cup \{\epsilon_0\}, \b\cup \{\delta_0\}, \g\cup \{\epsilon_0'\}, J^\frd(T), \frs}(\xs\times b, \ys\times a)&=F_{\a,\b,\g, J^\frd, \frs}(\xs,\ys)\otimes \theta^-+\sum_{\zs\in \bT_{\as}\cap \bT_{\gs}} C_{\xs,\ys,\zs} \cdot \zs\times \theta^+,
 \end{split}
 \label{eq:holomorphic-simple-stab-triangle-counts}
\end{equation}
for some $C_{\xs,\ys,\zs}\in \bF_2[U_{\ws}]$, which depend on $\frd$ and $T$.
\end{thm}

\begin{rem} Theorem~\ref{thm:simple-stabilization-triangle-maps} can be restated as follows. Let $\sigma_b$ denote the simple stabilization map defined on intersection points as $\sigma_b(\ve{x})= \ve{x}\times b,$ and extended equivariantly over $\bF_2[U_{\ws}]$. Define $\sigma_a$ similarly. Let $F_1^{\epsilon_0,\epsilon_0'}$ denote the 1-handle map, defined as $\ve{x}\mapsto \ve{x}\times \theta^+$, extended equivariantly over $\bF_2[U_{\ws}]$. Define the 3-handle map $F_3^{\epsilon_0,\epsilon_0'}$ via the formulas $\ve{x}\times \theta^+\mapsto 0$ and $\ve{x}\times \theta^-\mapsto \ve{x}$. Equation~\eqref{eq:holomorphic-simple-stab-triangle-counts} can be restated as
\begin{equation*}
\begin{split}
F_{\a\cup \{\epsilon_0\}, \b\cup \{\epsilon_0'\}, \g\cup \{\delta_0\}, J^\frd(T), \frs}(F_1^{\epsilon_0,\epsilon_0'}(\xs), \sigma_a(\ys))&=\sigma_b F_{\a,\b,\g, J^\frd, \frs}(\xs,\ys),
 \\
 F_{\a\cup \{\delta_0\}, \b\cup \{\epsilon_0\}, \g\cup \{\epsilon_0'\}, J^\frd(T), \frs}(\sigma_b(\xs), F_1^{\epsilon_0,\epsilon_0'}(\ys))&=\sigma_a F_{\a,\b,\g, J^\frd, \frs}(\xs,\ys),
 \\
 F_{\a\cup \{\epsilon_0\}, \b\cup \{\delta_0\}, \g\cup \{\epsilon_0'\}, J^\frd(T), \frs}(\sigma_b(\xs), \sigma_a(\ys))&=F_1^{\epsilon_0,\epsilon_0'} F_{\a,\b,\g, J^\frd, \frs}(\xs,\ys)\otimes \theta^-,
 \end{split}
\end{equation*}
\end{rem}

\begin{proof}[Proof of Theorem~\ref{thm:simple-stabilization-triangle-maps}]The proof we present is similar to the proofs of Theorems~\ref{thm:freestabilizetriangles} and~\ref{thm:1-handletriangle}. We focus on the first formula in Equation~\eqref{eq:holomorphic-simple-stab-triangle-counts}, since second and third are proven similarly.

A Maslov index computation similar to Lemma~\ref{lem:indexdisksonsphere} shows that if $\psi_0\in \pi_2(\theta^+, a, b)$, then
\[
\mu(\psi_0)=2n_{p_0}(\psi_0)+\gr(\theta^+,\theta).
\]
Consequently, if $\psi\# \psi_0\in \pi_2(\xs\times \theta,\ys\times a, \zs\times b)$, then
\begin{equation}
\mu(\psi\# \psi_0)=\mu(\psi)+\gr(\theta^+,\theta).\label{eq:Maslov-simple-stab-triangles}
\end{equation}

Next, to obtain transversality at curves which appear in our proof, we consider almost complex structures on $\Delta\times \Sigma$ satisfying~\eqref{def:J'1}, \eqref{def:J'2}, \eqref{def:J'3'}, \eqref{def:J'4'} and~\eqref{def:J'5'}; see Proposition~\ref{prop:transversality} for a precise statement of transversality. These are the same almost complex structures considered in the proof of handleswap invariance \cite{JTNaturality}*{Section~9.3}; we refer the reader there for a detailed account of a similar argument.

As in the proofs of Theorems~\ref{thm:freestabilizetriangles} and ~\ref{thm:1-handletriangle}, a sequence of $J^{\frd}(T_i)$-holomorphic representatives of $\psi\#\psi_0$ will degenerate into a pair of broken holomorphic triangles $\cU$ and $\cU_0$, representing $\psi$ and $\psi_0$, respectively. Using Equation~\eqref{eq:Maslov-simple-stab-triangles}, as well as a transversality, we conclude that $\mu(\psi)=0$. Consequently, arguing as in Theorems~\ref{thm:freestabilizetriangles} and~\ref{thm:1-handletriangle}, the broken triangles $\cU$ and $\cU_0$ each consist of a single curve, $u$ and $u_0$ respectively, which satisfy the matching condition
\[
\rho^p(u)=\rho^{p_0}(u_0),
\]
where $p$ and $p_0$ are the connected sum points. Via a gluing argument, as in the proofs of Theorems~\ref{thm:freestabilizetriangles} and ~\ref{thm:1-handletriangle}, it is sufficient to show that for a generic $\ve{d}\in \Sym^{n}(\Delta)$,  the following count holds for the matched moduli spaces on $(\bT^2, \epsilon_0,\epsilon_0',\delta_0,p_0)$:
\begin{equation}
 \sum_{\substack{\psi_0\in \pi_2(\theta^+,a,b)\\ n_{p_0}(\psi_0)=n}} \# \cM(\psi_0,\ve{d})\equiv 1 \pmod{2}.\label{eq:simple-stab-main-count}
\end{equation}
The argument to establish Equation~\eqref{eq:simple-stab-main-count} is formally similar to the argument for establishing Equations~\eqref{eq:OS'scountmatched}, \eqref{eq:divisorcounttriangles},   \eqref{eq:divisor-count-1-handles} and \eqref{eq:divisor-count-triangles-more-general}.

Pick a path $(\ve{d}_t)_{t\in [0,\infty)}$ in $\Sym^n(\Delta)$ satisfying the following:
\begin{enumerate}
\item $\ve{d}_0=\ve{d}$.
\item $\ve{d}_t$ is disjoint from the fat diagonal.
\item The points in $\ve{d}_t$ all enter into the $\epsilon_0$-$\delta_0$ cylindrical end of $\Delta\times \Sigma$.
\item For large $t$, the points in $\ve{d}_t$ are spaced at least distance $t$ apart, with respect to the Euclidean metric under the identification of the cylindrical ends as $[0,1]\times [0,\infty)$.
\item The points of $\ve{d}_t$ all approach a line $\{s_0\}\times [0,\infty)$ as $t\to \infty$.
\end{enumerate}

Write $\ve{\cD}:=\{\ve{d}_t: t\in [0,\infty)\}$, and consider the 1-dimensional moduli space 
\[
\cM_{(\theta^+,a,b)}(\ve{\cD}):=\coprod_{
\substack{
\psi_0\in \pi_2(\theta^+,a,b)\\
n_{p_0}(\psi)=n}}
\bigcup_{t\in [0,\infty)}\cM_{J^\frd(T)}(\psi_0, \ve{d}_t)
\]
We count the ends of $\cM(\ve{\cD})$. As in the proof of Equation~\eqref{eq:divisorcounttriangles}, generically, the ends correspond to the following:
\begin{enumerate}[ref= e$'$-\arabic*, label= (e$'$-\arabic*):]
\item\label{ends':1}  $t=0$.
\item\label{ends':2} Index 1 holomorphic strips breaking off at $t\in (0,\infty)$, with zero multiplicity at $p_0$.
\item\label{ends':3}  $t\to \infty$.
\end{enumerate}

The count of the ends of the form ~\eqref{ends':1} is equal to the left hand side of Equation~\eqref{eq:simple-stab-main-count}.

The ends of the form \eqref{ends':2} cancel modulo 2, since $\theta^+$ is a cycle in $\hat{\CF}(\bT^2,\epsilon_0,\epsilon'_0,p_0)$, and $a$ and $b$ are cycles in their appropriate complexes.

The ends of the form~\eqref{ends':3} correspond to the Cartesian product
\begin{equation}
\bigg(\coprod_{\substack{\psi_0^0\in \pi_2(\theta^+,a,b)\\
n_{p_0}(\psi_0^0)=0}} \cM(\psi_0^0)\bigg)\times \bigg( \coprod_{\substack{\phi\in \pi_2(b,b)\\
n_{p_0}(\phi)=n}} \cM(\phi,d)\bigg),
\label{eq:simple-stab-cartesian-product}
\end{equation}
where $d$ is a point on the line $\{s_0\}\times \R$.

We wish to show that the total count in Equation~\eqref{eq:simple-stab-cartesian-product} is 1. The total count of the left factor is 1, since there is only one class which contributes, and that class is a small triangle.  It remains to show that if $\phi$ denotes the index 2 class in $\pi_2(b,b)$ with $n_{p_0}(\phi)=1$, then
\begin{equation}
\# \cM(\phi,d)\equiv 1 \pmod 2\label{eq:Lipshitz's-counts}
\end{equation}

Note that by axiom~\eqref{def:J'3'}, $\cM(\phi,d)$ consists of holomorphic curves for an almost complex structure on $\bT^2\times [0,1]\times \R$ which satisfies \eqref{def:J1}--\eqref{def:J4} and \eqref{def:J5'}. These are precisely the almost complex structures Lipshitz used to prove stabilization invariance of the Heegaard Floer complexes. With respect to these almost complex structures, Lipshitz proves Equation~\eqref{eq:Lipshitz's-counts} while proving stabilization invariance; see \cite{LipshitzCylindrical}*{Sublemma~A.12}. We repeat Lipshitz's argument here, for the convenience of the reader.

Consider the genus 2 diagram for $S^1\times S^2$ in Figure~\ref{fig::75}. Let $\phi_1\#\phi_0$ denote the index 1 class shown. The class $\phi_1$ is a bigon, and $\phi_0$ denotes the index 2 class on $\bT^2$. By a neck stretching argument (analogous to the one above), it follows that for large neck length,
\[
\# \hat{\cM}(\phi_1\# \phi_0)=\# \hat{\cM}(\phi_1)\cdot \# \hat{\cM}(\phi_0,d).
\]
On the other hand, by invariance of Heegaard Floer homology, since the diagram represents $S^1\times S^2$, and there is only one other non-negative index 1 class (a  bigon), we must have $\# \hat{\cM}(\phi_1\# \phi_0)=1$.  Equation~\eqref{eq:Lipshitz's-counts} follows.

The other two stated formulas from Equation~\eqref{eq:holomorphic-simple-stab-triangle-counts} follow \emph{mutatis mutandis}.
\end{proof}

\begin{figure}[ht!]
\centering
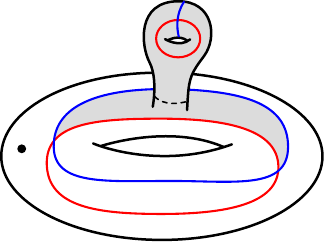
\caption{\textbf{A once stabilized diagram of $S^1\times S^2$ used in the proof of Theorem~\ref{thm:simple-stabilization-triangle-maps}.} The shaded class is $\phi_1\# \phi_0$. We stretch along the dashed line.}\label{fig::75}
\end{figure}

\subsection{Invariance of the 2-handle maps}

We now prove invariance of the 2-handle maps. Since the 2-handle maps which feature in the graph TQFT are essentially identical to those defined by Ozsv\'{a}th and Szab\'{o}, our exposition will be terse.

Extending \cite{OSTriangles}*{Lemma~4.5} to multi-pointed 3-manifolds, we have the following:
\begin{lem} \label{lem:moves-between-subordinate-triples}
 Suppose $\bS^1$ is an $m$-component framed link in $(Y,\ws)$, and $B\subset Y$ is a fixed bouquet of $\bS^1$. Then there exists a Heegaard triple $(\Sigma,\as,\bs,\bs',\ws)$ subordinate to $B$. Furthermore, two Heegaard triples subordinate to $\bS^1$ can be connected by a sequence of the following moves:
 \begin{enumerate}
 \item\label{sub-trip-move-1} An isotopy or handleslide amongst the $\as$ curves.
 \item\label{sub-trip-move-2} An isotopy or handleslide amongst the curves $\beta_{m+1},\dots, \beta_n$, followed by the corresponding move applied to $\beta_{m+1}',\dots,\beta_{n}'$.
 \item\label{sub-trip-move-3} A simple stabilization or destabilization (of the first or second types in Theorem~\ref{thm:simple-stabilization-triangle-maps}), in the complement of $B\cup \bS^1$.
 \item \label{sub-trip-move-4}For $i\in \{1,\dots, m\}$, an isotopy of $\beta_i$, or a handleslide of $\beta_i$ across one of the $\beta_{m+1},\dots, \beta_n$.
 \item\label{sub-trip-move-5} For $i\in \{1,\dots, m\}$, an isotopy of $\beta_i'$, or a handleslide of $\beta_i'$ across one of the $\beta'_{m+1},\dots, \beta_n'$. 
 \item\label{sub-trip-move-6} An isotopy $\Sigma_t$ of the Heegaard surface $\Sigma$, inside of $Y$, fixing $\ws$, such that $\Sigma_t$ intersects $B\cup \bS^1$ only along $\ws$, for all $t$.
 \end{enumerate}
\end{lem}
\begin{proof} To construct a Heegaard triple subordinate to $B$, pick a Heegaard diagram for the sutured manifold $Y_{B\cup \bS^1,\ws}$ (formed by removing a neighborhood of $B\cup \bS^1$ and adding a contractible suture for each basepoint in $\ws$). Let $(\Sigma, \alpha_1,\dots, \alpha_n, \beta_{m+1},\dots, \beta_n)$ denote this diagram. We then define $\beta_{m+1}',\dots, \beta_{n}'$ to be small isotopies of $\beta_{m+1},\dots, \beta_n$. 

A simple closed curve in $\d (Y_{B\cup \bS^1,\ws})$ which avoids the contractible regions bounded by the sutures can be projected onto $\Sigma$ to yield a curve which is in the complement of the curves $\beta_{m+1},\dots, \beta_n$. Note that the curve on $\Sigma$ obtained by projecting is only well defined up to isotopies and handleslides across the $\beta_{m+1},\dots, \beta_n$ curves. Let $\beta_{1},\dots, \beta_n\subset \Sigma$ be obtained by projecting meridians of the components of $\bS^1$ onto $\Sigma$. Let $\beta_1',\dots, \beta_n'$ denote projections of the longitudes. These curves are well defined up to Moves~\eqref{sub-trip-move-4} and~\eqref{sub-trip-move-5}.

Hence, given a sutured Heegaard diagram for $Y_{B\cup \bS^1,\ws}$ we can obtain a Heegaard triple subordinate to $B\cup \bS^1$ by the above procedure. Furthermore, any two triples constructed from a given Heegaard diagram via the above procedure can be connected by the listed moves. It remains to connect two sutured Heegaard diagrams for $Y_{B\cup \bS^1,\ws}$. Any two sutured Heegaard diagrams for $Y_{B\cup \bS^1,\ws}$ can be connected by a sequence of Heegaard moves by \cite{JDisks}*{Proposition~2.15}, which induce Moves~\eqref{sub-trip-move-1}, \eqref{sub-trip-move-2}, \eqref{sub-trip-move-3} and ~\eqref{sub-trip-move-6} on the resulting Heegaard triple. The proof is complete.
\end{proof}

We now show that the maps defined in Equation~\eqref{def:2-handlemap} are independent of the choice of bouquet, or subordinate triple:

\begin{lem} Suppose $\bS^1$ is a framed link in $(Y,\ws)$, $B_1$ and $B_2$ are two bouquets for $\bS^1$, and  $\cT_1$ and $\cT_2$ are Heegaard triples subordinate to $B_1$ and $B_2$, respectively. Write $F_{Y,\bS^1,\frs,\cT_1}$ and $F_{Y,\bS^1,\frs,\cT_2}$ for the 2-handle maps computed with $\cT_1$ and $\cT_2$. Write $\cH_1$ and $\cH_2$ for the diagrams of $Y$ induced by $\cT_1$ and $\cT_2$, respectively, and write $\cH_1'$ and $\cH_2'$ for the induced diagrams of $Y(\bS^1)$. The following diagram commutes up to homotopy:
\[
\begin{tikzcd}[column sep=2cm]\CF^-(\cH_1,\ws,\frs|_Y)\arrow{r}{\Psi_{\cH_1\to \cH_2}}\arrow{d}{F_{Y,\bS^1,\frs,\cT_1}}& \CF^-(\cH_2,\ws,\frs|_{Y})\arrow{d}{F_{Y,\bS^1,\frs,\cT_2}}\\
\CF^-(\cH_1',\ws,\frs|_{Y'})\arrow{r}{\Psi_{\cH_1'\to \cH_2'}}&\CF^-(\cH_2',\ws,\frs|_{Y'}).
\end{tikzcd}
\]
\end{lem}
\begin{proof} Our proof is no different that the original proof given by Ozsv\'{a}th and Szab\'{o} \cite{OSTriangles}*{Theorem~4.4}, so we will be terse. See also \cite{JCob}*{Theorem~6.9}.

First, we fix a bouquet $B$ for $\bS^1$ and show that $F_{Y,\bS^1,\frs,\cT}$ is independent of the triple subordinate to $B$. This amounts to proving independence from the moves in Lemma~\ref{lem:moves-between-subordinate-triples}. Invariance from Moves~\eqref{sub-trip-move-1}, \eqref{sub-trip-move-2}, \eqref{sub-trip-move-4} and~\eqref{sub-trip-move-5} all follow from a relatively straightforward argument using associativity of the triangle maps. Invariance from Move~\eqref{sub-trip-move-3} (simple stabilizations and destabilizations) follows from Theorem~\ref{thm:simple-stabilization-triangle-maps}. Finally invariance under Move~\eqref{sub-trip-move-6} (isotopies of the Heegaard surface) is tautological.

Next, one needs to prove independence from the bouquet.  The idea is that if $m+1\le i\le n$ and $1\le j\le m$, then we can handleslide $\beta_i$ twice across $\beta_j$, and handleslide $\beta_i'$ twice across $\beta_j'$ to change the bouquet, one 1-cell at a time. However the 3-manifold represented by $(\Sigma,\bs,\bs',\ws )$ is unchanged by a sequence of such handleslides, and hence the corresponding transition map will preserve the top degree generator. Consequently a simple associativity argument shows that the cobordism map $F_{Y,\bS^1,\frs,\cT}$ is unchanged.
\end{proof}

\subsection{The composition law for 2-handle maps}

We now state the $\Spin^c$ composition law for the 2-handle maps. We omit the proof, as it is identical to the proof given by Ozsv\'{a}th and Szab\'{o} for 2-handle cobordisms between singly pointed 3-manifolds \cite{OSTriangles}*{Proposition~4.9} using associativity of the holomorphic triangle maps \cite{OSDisks}*{Theorem~8.16}:

\begin{lem}\label{lem:compositionlawlinks}Suppose that $\bS_1$ and $\bS_2$ are two disjoint, framed, 1-dimensional links in $(Y,\ws)$, and $\frs_1\in \Spin^c(W(Y,\bS_1))$ and $\frs_2\in \Spin^c(W(Y(\bS_1),\bS_2)))$.
Then
\[
F_{Y(\bS_1), \bS_2,\frs_2}\circ F_{Y,\bS_1,\frs_1}=\sum_{\substack{\frs\in \Spin^c(W(Y(\bS_1),\bS_2))\\ \frs|_{W(Y,\bS_1)}=\frs_1\\
\frs|_{W(Y(\bS_1),\bS_2)}=\frs_2}} F_{Y, \bS_2\cup \bS_1,\frs}.
\]

\end{lem}

\section{Constructing the graph TQFT I}
\label{sec:constructionI}

In this section, we describe our maps for graph cobordisms which satisfy the following condition:

\begin{define}\label{def:has-enough-ends}
A cobordism $W\colon Y_0\to Y_1$ \emph{has enough ends} if each connected component of $W$ intersects both $Y_0$ and $Y_1$ non-trivially.
\end{define}

In the subsequent Section~\ref{sec:constructionII}, we define the maps cobordisms which may not have enough ends.

\subsection{Cerf theoretic preliminaries}

We need the following notion of a parametrized decomposition of a cobordism from \cite{JCob}*{Section~8.1}:

\begin{define}\label{def:enough-ends} Suppose that $W\colon Y_0\to Y_1$ is a cobordism with enough ends. A \emph{parametrized Kirby decomposition} $\cK$ of $W$ consists of the following data:
\begin{enumerate}
\item A decomposition
\[
W=W_n\circ \dots \circ W_0,
\]
where $W_i$ is a cobordism from $\cY_i$ to $\cY_{i+1}$, and $Y_0=\cY_0$ and $Y_1=\cY_{n+1}$.
\item For each $i\in \{0,\dots, n\}$, a framed link $\bS_i\subset \cY_i$, all of whose components have the same dimension.  The possibility $\bS_i=\emptyset$ is not excluded.
\item For each $i\in \{0,\dots, n\}$, a diffeomorphism $\Phi_i\colon W(\cY_i, \bS_i)\to W_i$, defined up to isotopy, such that $\Phi_i(0,y)=y$, for all $y\in \cY_i$.
\end{enumerate} Furthermore, the following are satisfied:
\begin{enumerate}
\item There is a $c\in \{0,\dots, n\}$ such that $\bS_c$ has dimension 1. 
\item  If $c>i$, then $\bS_i$ has dimension 0. If $i>c$, then $\bS_i$ has dimension 2. (Note, we allow $\emptyset$ to have any dimension).
\item If $\dim (\bS_i)=0$, then $\bS_i$ is empty or has 2 components. If $\dim(\bS_i)=2$, $\bS_i$ is empty or has 1 component.
\end{enumerate}
\end{define}

\begin{prop}\label{prop:connect-parametrized-Kirby-decomps}
Any two parametrized Kirby decompositions can be connected by a sequence of the following moves and their inverses:
\begin{enumerate}[leftmargin= 2 cm, ref= $\cK M$-\arabic*, label= ($\cK M$-\arabic*):]
\item \label{move:KM0} Adding or removing levels with $\bS=\emptyset$.
\item\label{move:KM1} Pushing $\cK$ forward under a diffeomorphism of $W$, which is the identity on $\d W$ and is isotopic to the identity relative to $\d W$.
\item\label{move:KM2} Exchanging the relative ordering of two  framed 0-spheres, which are in adjacent levels.
\item\label{move:KM3} Exchanging the relative ordering of two  framed 2-spheres, which are in adjacent levels.
\item\label{move:KM4} Handlesliding the components of $\bS_c$ (a framed 1-dimensional link) across each other.
\item\label{move:KM5} Canceling a framed 0-sphere $\bS^0$ with a framed 1-sphere $\bK^1$ in the subsequent level, if the belt sphere of $\bS^0$ intersects $\bK^1$ transversely in a single point.
\item\label{move:KM6} Canceling a framed 2-sphere $\bS^2$ with a framed 1-sphere $\bK^1$ in the previous level, if the belt sphere of $\bK^1$ intersects $\bS^2$ transversely in a single point.
\end{enumerate}
\end{prop}

We omit the proof Proposition~\ref{prop:connect-parametrized-Kirby-decomps}, and instead refer the reader to \cite{JuhaszTQFT}*{Section~2} and \cite{JCob}*{Theorem~8.9} for a careful Morse theoretic argument. We content ourselves with a more detailed topological description of  the moves in Proposition~\ref{prop:connect-parametrized-Kirby-decomps}.

Note that in Moves~\eqref{move:KM2}--\eqref{move:KM6}, we have not explained how the parametrizing diffeomorphisms are related. To this end, we note that if $\bS$ and $\bS'$ are pairwise disjoint framed spheres in $Y$, then there are canonical diffeomorphisms
\[
W(Y(\bS),\bS')\circ W(Y,\bS)\iso W(Y,\bS\cup \bS')\iso W(Y(\bS'),\bS)\circ W(Y,\bS').
\]
For example the diffeomorphism $W(Y(\bS),\bS')\circ W(Y,\bS)\iso W(Y,\bS\cup \bS')$ is obtained by noting that $W(Y(\bS),\bS')\circ W(Y,\bS)$ can be constructed by inserting a product layer $[0,1]\times Y(\bS)$ into $W(Y,\bS\cup \bS')$. Hence a diffeomorphism is obtained by picking a collar neighborhood of $Y(\bS)$ in $W(Y,\bS)$, which is unique up to isotopy. This describes the change in parametrizing diffeomorphisms in Moves~\eqref{move:KM2} and~\eqref{move:KM3}.

The change in parametrizing diffeomorphisms from  Move~\eqref{move:KM4}, a handleslide, is specified similarly. Suppose $\bK$ and $\bK'$ are two framed knots in $Y$, and $\bK''$ is obtained by handlesliding $\bK'$ across $\bK$. By our previous argument, there is a diffeomorphism between $W(Y,\bK\cup \bK')$ and $W(Y(\bK), \bK')\circ W(Y,\bK)$, which is well defined, up to isotopy. Next, $W(Y(\bK),\bK')$ and $W(Y(\bK), \bK'')$ are canonically diffeomorphic (up to isotopy), since $\bK'$ and $\bK''$ are isotopic in $Y(\bK)$. Finally, $W(Y(\bK),\bK'')\circ W(Y,\bK)$ and $W(Y,\bK\cup \bK'')$ are canonically diffeomorphic, up to isotopy, by our previous argument on reordering handles. Composing these diffeomorphisms gives a diffeomorphism between $W(Y,\bK\cup \bK')$ and $W(Y,\bK\cup \bK'')$ which is well defined up to isotopy.

For the change in parametrizing diffeomorphism after handle cancellations, Moves~\eqref{move:KM5} and \eqref{move:KM6}, we refer the reader to Juh\'{a}sz's work \cite{JuhaszTQFT}*{Definition~2.17}.

\subsection{Graphs in 4-space}

We  need the following transversality result concerning the intersection of graphs and the ascending and descending manifolds of a Morse function:

\begin{lem}\label{lem:effect-of-isotopies-on-smushed-level} Suppose that $(W,\Gamma)\colon (Y_0,\ws_0)\to (Y_1,\ws_1)$ is a graph cobordism and $f$ is a Morse function on $W$, with gradient like vector field $v$.
\begin{enumerate}
\item If $\Gamma$ is disjoint from $\Crit(f)$, then for generic $v$, $\Gamma$ is disjoint from the descending manifolds of the index 1 critical points, and the ascending manifolds of the index 3 critical points. Generically, $\Gamma$ is disjoint from both the ascending and descending manifolds of the index 2 critical points.
\item Suppose $\Gamma$ is a fixed, abstract graph, and $i_t\colon \Gamma\to W$ is a family of embeddings, whose intersection with $\d W$ is fixed for all $t$. For generic $i_t$, the image of $\Gamma$ is disjoint from $\Crit(f)$, and is disjoint from the descending manifolds of the index 1 critical points, and the ascending manifolds of the index 3 critical points. Generically, there are finitely many $t$ where $i_t(\Gamma)$ transversely intersects the ascending or descending manifold of an index 2 critical point of $f$. Generically, such intersections occur along the interior of an edge of $\Gamma$.
\end{enumerate}
\end{lem}
\begin{proof} We begin with the first claim, concerning a fixed graph.  The descending manifolds of index 1 critical points of $f$ are 1-dimensional, so generically a graph  will be disjoint, since $W$ is 4-dimensional. The same argument works for the ascending manifolds of index 3 critical points, which are also 1 dimensional. The ascending and descending manifolds of index 2 critical points are 2 dimensional, so a graph will be disjoint, generically.
 
The second claim, concerning 1-parameter families of graphs, follows from the same reasoning.
\end{proof}

\subsection{Definition of the graph cobordism maps}

Let $(W,\Gamma)\colon (Y_0,\ws_0)\to (Y_1,\ws_1)$ be a graph cobordism with enough ends. Let $\cK$ be a parametrized Kirby decomposition of $W$, which decomposes $W$ as
\[
W=W_n\circ \cdots \circ W_0.
\]
Let $c\in \{0,\dots, n\}$ denote the index of the 2-handle cobordism, and write $\cY_c$ and $\cY_{c+1}$ for the incoming and outgoing ends of $W_c$.

Suppose  $\Gamma$ is an embedded graph in $W$. The parametrizing diffeomorphisms of $\cK$ naturally equip each $W_i$ with a Morse function and gradient like vector field $(f_i,v_i)$, well defined up to isotopy in the handle attachment regions (see \cite{JuhaszTQFT}*{Lemma~2.15}). We can assume that the $(f_i,v_i)$ glue together to form a Morse function with gradient like vector field $(f,v)$ on all of $W$.


We can assume, after a small perturbation of the Morse functions, that $\Gamma$ is disjoint from $\Crit(f)$. By Lemma~\ref{lem:effect-of-isotopies-on-smushed-level}, for generically chosen $v$, the graph $\Gamma$ is disjoint from the descending manifolds of the index 1 critical points,  the ascending manifolds of the index 3 critical points, as well as both the ascending and descending manifolds of the index 2 critical points.  Flow each point of $\Gamma$ along $v$ until it hits $\cY_c$. Note that we flow $\Gamma\cap (W_{n}\circ \cdots \circ W_c)$ backwards along $v$, and we flow $\Gamma\cap (W_{c-1}\circ \cdots \circ W_0)$ forward along $v$. Write 
\[
\Gamma_c\subset \cY_c
\]
 for the resulting graph. By perturbing $v$ slightly, we may assume that $\Gamma_c$ is embedded in $\cY_c$.

Upgrade $\Gamma_c$ to a flow-graph $\cG_c$ by letting the initial and terminal vertices of $\cG_c$  be the images of the  basepoints $\ws_0\subset Y_0$ and $\ws_1\subset Y_1$, under the flow of $v$.

 Let $\bS_0,\dots, \bS_n$ denote the framed links in $\cY_0,\dots, \cY_n$, associated to $\cK$, and let $\Phi_i\colon W(\cY_i,\bS_i)\to W_i$ denote the parametrizing diffeomorphisms. Let $\phi_i\colon \cY_i(\bS_i)\to \cY_{i+1}$ denote the restriction of $\Phi_i$ to the outgoing boundary of $ W(\cY_i,\bS_i)$.

The type-$A$ graph cobordism map is defined as
\begin{equation}
F_{W,\Gamma,\frs}^{A}:=\left(\prod_{i=c+1}^{n}  (\phi_i)_*\circ F_{\cY_i, \bS_i,  \frs_i}\right)\circ \left( (\phi_i)_*\circ F_{\cY_c, \bS_c, \frs_c}\right) \circ\frA_{\cG_c}\circ \left( \prod_{i=0}^{c-1}  (\phi_i)_*\circ F_{\cY_i, \bS_i, \frs_i}\right),\label{eq:graph-cobordism-definition}
\end{equation}
where $\frs_i:=\Phi_i^*(\frs|_{W_i})$.

We define the type-$B$ maps by replacing $\frA_{\cG}$ with $\frB_{\cG}$, in Equation~\eqref{eq:graph-cobordism-definition}.

\begin{thm}\label{thm:A-with-enough-ends} Suppose that $(W,\Gamma)\colon (Y_0,\ws_0)\to (Y_1,\ws_1)$ is a graph cobordism with enough ends.
\begin{enumerate}
\item\label{claim:ind-isotopies-K-fixed} For fixed parametrized Kirby decomposition $\cK$, the maps $F_{W,\Gamma,\frs}^A$ and $F_{W,\Gamma,\frs}^B$ are unchanged by smooth isotopies of the graph $\Gamma$ (in the sense of Definition~\ref{def:smooth-graph-isotopy}).
\item\label{claim:ind-of-K}  The maps $F_{W,\Gamma,\frs}^A$ and $F_{W,\Gamma,\frs}^B$ are independent from the choice of parametrized Kirby decomposition of $W$.
\item\label{claim:W=[0,1]xY-project-Gamma} If $W=[0,1]\times Y$, then $F_{W,\Gamma,\frs}^A\simeq \frA_{\cG}$, where $\cG$ is an immersed ribbon flow-graph which is homotopic to the projection of $\Gamma$ into $Y$. Similarly $F_{W,\Gamma,\frs}^B\simeq \frB_{\cG}$.
\end{enumerate}
\end{thm}
\begin{proof} 
We first verify Claim~\eqref{claim:ind-isotopies-K-fixed}, independence from isotopies of $\Gamma$, for fixed $\cK$. Lemma~\ref{lem:effect-of-isotopies-on-smushed-level} describes the codimension 1 configurations which occur for a generic smooth isotopy $\Gamma_t$. At all but finitely many $t$, the graph $\Gamma_t$ will be disjoint from the descending manifolds of the index 2 critical points. At all but these finitely many $t$, we may flow $\Gamma_t$ along a gradient like vector field (either upwards or downwards, depending on which portion of $\cK$ the graph $\Gamma_t$ lies in) until it intersects $\cY_c$, and obtain an immersed flow-graph $\cG_t$ in $\cY_c$. Homotopies of immersed flow-graphs, which fix boundary vertices, do not affect the graph action map; see Remark~\ref{rem:actual-embedding-not-important}. There are finitely many points of time $t$ where $\Gamma_t$ crosses a descending manifold of an index 2 critical point. If $t_0$ is such a value of $t$, and $\epsilon$ is sufficiently small, the induced flow-graphs $\cG_{t_0-\epsilon}$ and $\cG_{t_0+\epsilon}$ are related by sliding an edge of $\cG_{t_0-\epsilon}$ across one of the link components of the framed link $\bS_c$.

Let $\cG^-$ denote $\cG_{t_0-\epsilon}$ and let $\cG^+$ denote $\cG_{t_0+\epsilon}$. We wish to show that
\begin{equation}
F_{\cY_c,\bS_c, \frs_c}\circ \frA_{\cG^-}\simeq F_{\cY_c,\bS_c, \frs_c}\circ \frA_{\cG^+}. \label{eq:graph-action-invariant-handleslide}
\end{equation}

By Theorem~\ref{thm:freestabilizetriangles} and Lemma~\ref{lem:relativehomologycommutestriangleII}, the graph action map commutes with 2-handle maps. However in $Y(\bS)$, the flow-graphs $\cG^{-}$ and $\cG^+$ are isotopic. Consequently,
\[
F_{\cY_c,\bS_c,\frs_c}\circ \frA_{\cG^-}\simeq \frA_{\cG^-}\circ  F_{\cY_c,\bS_c,\frs_c}\simeq \frA_{\cG^+}\circ F_{\cY_c,\bS_c,\frs_c}\simeq F_{\cY_c,\bS_c,\frs_c}\circ \frA_{\cG^+},
\]
establishing Equation~\eqref{eq:graph-action-invariant-handleslide}. It follows that for fixed $\cK$, the maps $F_{W,\Gamma,\frs}^A$ and $F_{W,\Gamma,\frs}^B$  are invariant from smooth isotopies of the graph $\Gamma$.

We now show that $F_{W,\Gamma,\frs}^A$ is invariant from the choice of parametrized Kirby decomposition. It is sufficient to show invariance from the Moves~\eqref{move:KM0}--\eqref{move:KM6}  from Proposition~\ref{prop:connect-parametrized-Kirby-decomps}.

Invariance under Move~\eqref{move:KM0} (adding identity layers) is a tautology.

Invariance from Move~\eqref{move:KM1} (pushing $\cK$ forward under an isotopy) is established as follows. Let $(\Phi_t)_{t\in [0,1]}$ be an isotopy of $W$ which is fixed on $\d W$ and satisfies $\Phi_0=\id_W$. Let $\cK$ be a parametrized Kirby decomposition, and let $\cK'$ denote the pushforward of $\cK$ under $\Phi_1$. Note that tautologically, the map $F_{W,\Gamma,\frs}^A$, computed with $\cK$, agrees with $F_{W,\Phi_1(\Gamma),\frs}^A$, computed with $\cK'$. By our previous argument, the graph cobordism maps (for fixed $\cK$) are invariant under smooth isotopies of the graph, so the map $F_{W,\Phi_1(\Gamma),\frs}^A$, computed with $\cK'$, coincides with the map $F_{W,\Gamma,\frs}^A$, computed with $\cK'$.

Invariance under Moves~\eqref{move:KM2} and~\eqref{move:KM3} (reordering disjoint 1-handles or 3-handles) follows from Proposition~\ref{lem:1-handles-commute}.

Invariance under Move~\eqref{move:KM4} (handleslides) follows from Ozsv\'{a}th and Szab\'{o}'s proof \cite{OSTriangles}*{Lemma~4.14}, and is a consequence of invariance under the Heegaard Floer groups under handleslides of the $\as$ and $\bs$, together with an associativity argument.

Our proof of invariance under Move~\eqref{move:KM5} (canceling 1- and 2-handles) is formally the same as Ozsv\'{a}th and Szab\'{o}'s original proof \cite{OSTriangles}*{Lemma~4.16}. For notational simplicity, we focus on the case when we are computing the 2-handle map of a single framed knot $\bK^1$ in $Y$ which cancels a 1-handle, attached along $\bS^0\subset Y$. The general case when $\bK^1$ is replaced by a framed link with multiple components follows from essentially the same argument. We can view $\bS^0$ and $\bK^1$ as being contained in a 3-ball in the original manifold $Y$. Hence we can find a Heegaard triple subordinate to $\bK^1$ in $Y(\bS^0)$ which locally looks like Figure~\ref{fig::77}. Fix a $\Spin^c$ structure $\frs\in \Spin^c(Y)$. Let $\hat{\frs}$ denote the unique $\Spin^c$ structure on $W(Y,\bS^0)$ extending $\frs$, and let $\frt_0\in \Spin^c(W(Y(\bS^0), \bK^1))$ denote the unique $\Spin^c$ structure which evaluates trivially on the framed 2-sphere introduced by $\bS^0$, and restricts to $\frs$ on $Y\iso Y(\bS^0)(\bK^1)$.

\begin{figure}[ht!]
\centering
\begingroup%
  \makeatletter%
  \providecommand\color[2][]{%
    \errmessage{(Inkscape) Color is used for the text in Inkscape, but the package 'color.sty' is not loaded}%
    \renewcommand\color[2][]{}%
  }%
  \providecommand\transparent[1]{%
    \errmessage{(Inkscape) Transparency is used (non-zero) for the text in Inkscape, but the package 'transparent.sty' is not loaded}%
    \renewcommand\transparent[1]{}%
  }%
  \providecommand\rotatebox[2]{#2}%
  \newcommand*\fsize{\dimexpr\f@size pt\relax}%
  \newcommand*\lineheight[1]{\fontsize{\fsize}{#1\fsize}\selectfont}%
  \ifx\svgwidth\undefined%
    \setlength{\unitlength}{115.34283122bp}%
    \ifx\svgscale\undefined%
      \relax%
    \else%
      \setlength{\unitlength}{\unitlength * \real{\svgscale}}%
    \fi%
  \else%
    \setlength{\unitlength}{\svgwidth}%
  \fi%
  \global\let\svgwidth\undefined%
  \global\let\svgscale\undefined%
  \makeatother%
  \begin{picture}(1,0.94593639)%
    \lineheight{1}%
    \setlength\tabcolsep{0pt}%
    \put(0,0){\includegraphics[width=\unitlength,page=1]{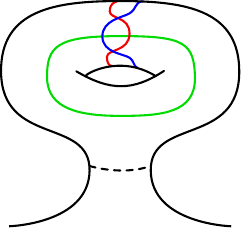}}%
    \put(0.49997281,0.27356492){\color[rgb]{0,0,0}\makebox(0,0)[t]{\lineheight{1.25}\smash{\begin{tabular}[t]{c}$c$\end{tabular}}}}%
    \put(0.4214127,0.83326221){\color[rgb]{0,0,1}\makebox(0,0)[rt]{\lineheight{1.25}\smash{\begin{tabular}[t]{r}$\beta_0$\end{tabular}}}}%
    \put(0.54797056,0.83412583){\color[rgb]{1,0,0}\makebox(0,0)[lt]{\lineheight{1.25}\smash{\begin{tabular}[t]{l}$\alpha_0$\end{tabular}}}}%
    \put(0.81217163,0.69706718){\color[rgb]{0,0.8627451,0}\makebox(0,0)[lt]{\lineheight{1.25}\smash{\begin{tabular}[t]{l}$\beta_0'$\end{tabular}}}}%
  \end{picture}%
\endgroup%

\caption{\textbf{A region of a Heegaard triple subordinate to a framed knot which cancels a 1-handle.} The almost complex structure is stretched along the curve $c$. } \label{fig::77}
\end{figure}
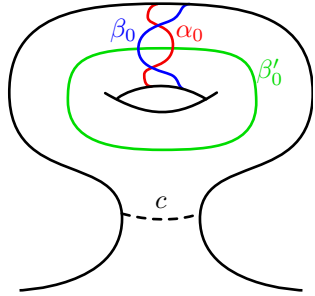

By Lemma~\ref{lem:2-defs-1-handle-equiv}, we can find almost complex structures on $\bT^2\times [0,1]\times \R$ and  $\Sigma\times [0,1]\times \R$ such that if $J(T)$ denotes the almost complex structure on $(\Sigma\# \bT^2)\times [0,1]\times \R$ obtained by gluing the two together, with neck length $T$ along the circle $c$ in Figure~\ref{fig::77}, then the 1-handle map $F_{Y,\bS^0,\hat{\frs}}$ satisfies
\begin{equation}
F_{Y,\bS^0,\hat{\frs}}(\xs)= \xs\times \theta^+_{\alpha_0,\beta_0}.\label{eq:1-handle-simple-form-with-other-ac}
\end{equation}

If $(\Sigma,\as,\bs,\ws)$ denotes the original Heegaard diagram of $Y$, let $\bs'$ denote small isotopies of $\bs$. The Heegaard triple $(\Sigma\# \bT^2, \as\cup \{\alpha_0\}, \bs\cup \{\beta_0\}, \bs'\cup \{\beta_0'\}, \ws)$ is subordinate to a bouquet for $\bK^1$. By Theorem~\ref{thm:simple-stabilization-triangle-maps} and Equation~\eqref{eq:1-handle-simple-form-with-other-ac},
\begin{equation}
F_{Y(\bS^0),\bK^1, \frt_0}(F_{Y,\bS^0,\hat{\frs}}(\xs))=F_{Y(\bS^0),\bK^1, \frt_0}(\xs\times \theta^+_{\alpha_0,\beta_0})=\sigma(\Psi_{\as}^{\bs\to \bs'}(\xs)),
\end{equation}
where $\sigma$ and $\Psi_{\as}^{\bs\to \bs'}$ denote the maps from naturality associated to a simple stabilization and to the small isotopy moving $\bs$ to $\bs'$, respectively. Invariance under Move~\eqref{move:KM5} follows.

Invariance under Move~\eqref{move:KM6} follows by turning around the above argument for Move~\eqref{move:KM5}.

We now verify Claim~\eqref{claim:W=[0,1]xY-project-Gamma}. Assume $W=[0,1]\times Y$. In this case, we can pick $\cK$ to have a single level, which has the empty framed sphere. The map $F_{W,\Gamma,\frs}^A$ is, by definition,  obtained by projecting $\Gamma$ into $Y$ and computing the map $\frA_{\cG}$. The same argument works for the type-$B$ map. The proof is complete.
\end{proof}

\section{Constructing the graph TQFT II: adding punctures}
\label{sec:constructionII}
In this section, we define the cobordism maps when a cobordism does not have enough ends (in the sense of Definition~\ref{def:has-enough-ends}). The idea is to remove small 4-balls from $W$, and connect the new copies of $S^3$ to  $\Gamma$ by adding a new edge.

Combined with our proof of invariance for graph cobordisms with enough ends (Theorem~\ref{thm:A-with-enough-ends}), this section concludes the construction of the graph cobordism maps and the proof of their invariance (Theorem~\ref{thm:A}).

 For a fixed  $\bmP$, the chain complex $\CF^-(\emptyset)$ is defined to be the ring $\cR_{\bmP}:=\bF_2[U_{p_1},\dots, U_{p_n}]$, where $\bmP=\{p_1,\dots, p_n\}$, with vanishing differential.

\subsection{0- and 4-handle maps}

Suppose that $(Y,\ws)$ is a multi-pointed 3-manifold, equipped with a coloring $\sigma\colon \ws\to \bmP$. Let $(S^3,w_0)$ denote a new copy of $S^3$, and let $\sigma'\colon \ws\cup \{w_0\}\to \bmP$ be an extension of $\sigma$.

If $(\Sigma,\as,\bs,\ws)$ is a diagram for $(Y,\ws)$, and $(S^2,w_0)$ is a genus 0 diagram for $S^3$ with no $\alpha$ or $\beta$ curves (alternatively, if one wants to avoid Heegaard diagrams with no curves, one can use a genus 1 diagram for $S^3$, with a single $\alpha$ and $\beta$ curve).

There is a canonical chain isomorphism
\begin{equation}
\CF^-(\Sigma,\as,\bs,\ws^{\sigma}, \frs)\iso \CF^-(\Sigma\cup S^2, \as,\bs, (\ws\cup \{w_0\})^{\sigma'},\frs\cup \frs_0),\label{eq:canonical-iso-0-4-handle}
\end{equation}
obtained by sending $U_{p_1}^{i_1}\cdots U_{p_n}^{i_n}\cdot \xs$ to $U_{p_1}^{i_1}\cdots U_{p_n}^{i_n}\cdot \xs$, where $\xs\in \bT_{\a}\cap \bT_{\b}$ and $U_{p_1},\dots, U_{p_n}$ are the variables in the ring $\cR_{\bmP}$.

The 4-manifold corresponding to a 0-handle attachment is the disjoint union of $[0,1]\times Y$ with $B^4$. 

Writing $\frs_0$ for the unique element of $\Spin^c(S^3)$, we define the 0-handle map
\[
F_{Y, 0}\colon \CF^-(Y,\ws^{\sigma}, \frs)\to \CF^-(Y\cup S^3, (\ws\cup \{w_0\})^{\sigma'},\frs\cup \frs_0)
\]
using the canonical isomorphism in Equation~\eqref{eq:canonical-iso-0-4-handle}.

There is also a 4-handle map
\[
F_{Y\cup S^3, 4}\colon  \CF^-(Y\cup S^3, (\ws\cup \{w_0\})^{\sigma'},\frs\cup \frs_0)\to \CF^-(Y,\ws^{\sigma}, \frs),
\]
also defined using the canonical isomorphism in Equation~\eqref{eq:canonical-iso-0-4-handle}.

\subsection{Puncturing graph cobordisms}

Suppose that $(W,\Gamma)\colon (Y_0,\ws_0)\to (Y_1,\ws_1)$ is a graph cobordism without enough ends. Construct a new 4-manifold $W'$ by removing a collection of 4-balls from $W$, which are disjoint from $\Gamma$. For each new copy of $S^3$ in $\d W'$, we add a new edge to $\Gamma$ which connects the new $S^3$ to a point along the interior of an edge of $\Gamma$. The new vertices may be labeled with any cyclic order. We designate the new copies of $S^3$ in the boundary as incoming or outgoing, in such a way that $(W',\Gamma')$ has enough ends, in the sense of Definition~\ref{def:enough-ends}. Let $F_0$ denote the 0-handle maps corresponding to the new incoming boundary components, and let $F_4$ denote the 4-handle maps corresponding to the new outgoing boundary components.

We now define
\begin{equation}
F^A_{W,\Gamma,\frs}:=F_4\circ F_{W',\Gamma',\frs|_{W'}}^A\circ F_0,\label{eq:punctured-map}
\end{equation}
and define $F_{W,\Gamma,\frs}^B$ similarly.

It remains to show that Equation~\eqref{eq:punctured-map} is well defined.  Since any two puncturings of $(W,\Gamma)$ can be related by a common puncturing (up to diffeomorphism), it is sufficient to show that the puncturing operation does not change the maps for cobordisms which already have enough boundary components:

\begin{prop} Suppose that $(W,\Gamma)\colon (Y_0,\ws_0)\to (Y_1,\ws_1)$ is a graph cobordism with enough ends, and $(W',\Gamma')$ is obtained by removing a 4-ball from $W$, and connecting the new boundary component to a point on the interior of an edge of $\Gamma$. If the new $S^3$ is designated as incoming, then
\[
F_{W,\Gamma,\frs}^A\simeq F_{W',\Gamma',\frs|_{W'}}^A\circ F_{Y_0,0}.
\]
If the new $S^3$ is designated as outgoing, then 
\[
F_{W,\Gamma,\frs}^A\simeq F_{Y_1\cup S^3,4}\circ F_{W',\Gamma',\frs|_{W'}}^A.
\]
\end{prop}
\begin{proof} We focus on the case when the new $S^3$ is designated as incoming. The argument for the other case is similar. A handle decomposition of $W'$ is obtained from a handle decomposition of $W$ by adding in a 1-handle which connects the new copy of $S^3$ to $Y_0$. Let $\bS^0$ denote the corresponding 0-sphere in $Y\cup S^3$, and let $F_{Y_0\cup S^3,\bS^0}$ denote the map for this 1-handle. Immediately from the definitions, we have
\begin{equation}
F_{Y_0\cup S^3,\bS^0}\circ F_{Y_0,0}=S_{w_0}^+,\label{label:0-handle+1-handle=free-stab}
\end{equation}
where $w_0$ is the basepoint which was added by the 0-handle map.

Let $\cY_c$ denote the level set in $W$ corresponding to the incoming boundary of the 2-handle portion of $W$, and let $\cG_c\subset \cY_c$ denote the flow-graph obtained by isotoping $\Gamma$ along the flowlines of a gradient like vector field on $W$. Let $\cG_c'$ denote the flow-graph obtained by isotoping $\Gamma'$ into $\cY_c$ using the flow of a gradient like vector field on $W'$. Let $w_0'$ denote the image of $w_0$ in $\cY_c$. We can assume that $\cG'_c$ is obtained from $\cG_c$ by connecting $w_0'$ to the interior of an edge of $\cG_c$ with an arc.

Using Equation~\eqref{label:0-handle+1-handle=free-stab} and the fact $S_{w_0}^+$ can be commuted with the 1-handle maps of $W$ using Lemma~\ref{lem:1-handle-free-stab-commute}, the main statement reduces to showing that
\[
\frA_{\cG_c}=\frA_{\cG'_c}\circ S_{w_0'}^+.
\]
We note that, by definition, the function $\frA_{\cG'_c}\circ S_{w_0'}^+$ is equal to the graph action map of $\cG'_c$, with $w_0'$ viewed as an interior vertex. By the trivial strand relation in Lemma~\ref{lem:full-trivial-strand-rel}, it follows that the induced map is chain homotopic to $\frA_{\cG_c}$, completing the proof.
\end{proof}

\section{The composition law}
\label{sec:compositionlaw}
We now prove the composition law for graph cobordisms:

\begin{customthm}{\ref{thm:C}} Suppose that a graph cobordism $(W,\Gamma)$ can be decomposed as the composition
\[
(W,\Gamma)=(W_2,\Gamma_2)\circ (W_1,\Gamma_1).
\]
If $\frs_1\in \Spin^c(W_1)$ and $\frs_2\in \Spin^c(W_2)$, then
\[
F_{W_2,\Gamma_2,\frs_2}^A\circ F_{W_1,\Gamma_1,\frs_1}^A\simeq \sum_{\substack{\frs\in \Spin^c(W)\\
\frs|_{W_1}=\frs_1\\
\frs|_{W_2}=\frs_2}} F_{W,\Gamma,\frs}^A.
\]
The same relation holds for the type-$B$ maps.
\end{customthm}
\begin{proof}First, note that the 0-handle maps of $F^A_{W_2,\Gamma_2,\frs_2}$ can be commuted to the right of all of the factors of $F_{W_1,\Gamma_1,\frs_1}^A$. Similarly the 4-handle maps of $F_{W_1,\Gamma_1,\frs_1}^A$ can be commuted to the left of all of the factors of $F_{W_2,\Gamma_2,\frs_2}^A$. Consequently, it is sufficient to prove the composition law for cobordisms with enough ends (in the sense of Definition~\ref{def:has-enough-ends}).

Pick parametrized Kirby decompositions  $\cK_1$ and $\cK_2$, for $W_1$ and $W_2$, respectively. We can stack $\cK_1$ and $\cK_2$ to get a decomposition of $\cK$ of $W$ into elementary cobordisms, though $\cK$ will not in general be a parametrized Kirby decomposition, as the handles do not appear in the correct order, and there are two levels with 2-handles. We will refer to a decomposition of $W$ into parametrized cobordisms, with at most two levels with 2-handles, as a \emph{semi-Kirby decomposition} of $W$.

Let $\cG_1$ denote the flow-graph in $W_1$ obtained by isotoping $\Gamma_1$ along the flow-lines of a gradient like vector field until it lies in the incoming boundary of the level of $W_1$ containing the 2-handles. Let $\cG_2$ denote the analogous flow-graph in a level of $W_2$.

The ascending manifolds of the index 3 critical points of $W_1$ are (after generic perturbation) disjoint from the descending manifolds of the index 1 and 2 critical points of $W_2$, as well as the flow-graph $\cG_2$. Consequently, we may topologically move all of the 3-handles of $\cK_1$ to the left of the 1-handles, 2-handles and flow-graph of $W_2$, to obtain a new semi-Kirby decomposition of $W$. Moving the 3-handles past the 1-handles and 2-handles of $W_2$ does not affect the composition by Theorem~\ref{thm:1-handletriangle} and Lemma~\ref{lem:1-handles-commute}. Commuting 3-handle maps past a graph action map for a flow-graph which is disjoint from the attaching sphere amounts to commuting the free-stabilization and relative homology maps (which are used to build the graph action map) past the 3-handle map. The free-stabilization and relative homology maps commute with the 3-handle maps by Lemmas~\ref{lem:rel-hom-1-handle-commute} and \ref{lem:1-handle-free-stab-commute}.

In an entirely analogous manner, the 1-handle maps of $W_2$ may be commuted to the right of the 3-handles, 2-handles and flow-graph of $W_1$.

Generically the ascending manifolds of the 2-handles in $W_2$ will be disjoint from $\cG_2$, and consequently we may topologically move $\cG_2$ past the 2-handles of $W_2$ and obtain a   flow-graph $\cG_2'$ in the same level as $\cG_1$. Commuting the graph action map of $\cG_2$ with the 2-handle map of $W_1$ amounts to commuting the free-stabilization and relative  homology maps past the 2-handle map, which can be done using Lemma~\ref{lem:relativehomologycommutestriangleII} and Theorem~\ref{thm:freestabilizetriangles}.

By Part~\eqref{thm:graph-action:b} of Theorem~\ref{thm:graph-action}, we have $\frA_{\cG_2'}\circ \frA_{\cG_1}=\frA_{\cG_2'\cup \cG_1}$.

In summary, we have shown that the composition $F_{W_2,\Gamma_2,\frs_2}^A\circ F_{W_1,\Gamma_1,\frs_1}^A$ is equal to the composition of handle maps and graph action map of a semi-Kirby decomposition, which fails to be a parametrized Kirby decomposition only in that it has two levels with 2-handles. Since these two levels are consecutive, by using the composition law for 2-handle maps in Lemma~\ref{lem:compositionlawlinks}, the general $\Spin^c$ composition law is obtained.
\end{proof}

\section{Path cobordisms and the normalization axiom}
\label{sec:normalizationI}

Suppose $(Y,\ws)$ is a multi-pointed 3-manifold. Using the composition law, it follows that if $(W,\Gamma)\colon (Y_0,\ws_0)\to (Y,\ws)$ is a graph cobordism and $\frs\in \Spin^c(W)$, then
\[
F_{[0,1]\times Y, [0,1]\times \ws,\frs|_{Y}}^A\circ F_{W,\Gamma,\frs}^A\simeq F_{W,\Gamma,\frs}^A.
\]
 This of course does not imply that $F_{W,\Gamma,\frs}^A$ is the identity map. In the study of TQFTs, the relation 
 \[
 F_{[0,1]\times Y,[0,1]\times \ws,\frt}^A\simeq \id_{\CF^-(Y,\ws,\frt)}
 \]
  is often referred to as the \emph{normalization axiom} (see \cite{TuraevQuantum}*{Section~1.4}). In this section, we prove the following (modulo a technical result, which we prove in Section~\ref{sec:movingbasepoints}):
 
 \begin{customthm}{\ref{thm:B}}Suppose that $(W,\Gamma)\colon (Y_0,\ve{w}_0)\to (Y_1,\ve{w}_1)$ is a graph cobordism, and $\Gamma$ is a collection of paths, each connecting $\ws_0$ to $\ws_1$.
\begin{enumerate}
\item The $A$ and $B$ versions coincide:
\[
F_{W,\Gamma,\frs}^A\simeq F_{W,\Gamma,\frs}^B.
\]
 \item  Suppose $\phi\colon (Y,\ws)\to (Y,\ws)$ is an orientation preserving diffeomorphism, and let $W(\phi)$ denote the \emph{mapping cylinder} (i.e. $[0,1]\times Y$, with $\{0\}\times Y$ identified with $Y$ via $\id_Y$ and $\{1\}\times Y$ identified with $Y$ via $\phi$). Then
 \[
F_{W(\phi), [0,1]\times \ws, \frs}^A\simeq F_{W(\phi), [0,1]\times \ws, \frs}^B\simeq \left(\phi_*\colon \CF^-(Y,\ws^{\sigma}, \frs)\to \CF^-(Y,\ws^{\phi_*\sigma}, \phi_*\frs)\right). 
 \]
 \item Suppose $(W,\gamma)\colon (Y_0,w_0)\to (Y_1,w_1)$ is a cobordism such that $W,$ $Y_0$ and $Y_1$ are nonempty and connected, and $\gamma$ is a path from $w_0$ to $w_1$. Then $F_{W,\gamma,\frs}^A\simeq F_{W,\gamma,\frs}^B$, and both maps coincide with the map defined by Ozsv\'{a}th and Szab\'{o}.
 \end{enumerate}
 \end{customthm}
\begin{proof} Using naturality of the Heegaard Floer groups, the first and second claims amount to showing that if $\cG=(\Gamma,\ws_0,\ws_1)$ is a flow-graph in $Y$ such that $\ws_0$ and $\ws_1$ are pairwise disjoint collections of points in $Y$, and $\Gamma$ consists of arcs, each connecting a point of $\ws_0$ to a point of $\ws_1$, then
\begin{equation}
\frA_{\cG}\simeq \frB_{\cG}\simeq (\phi_{\Gamma})_*\label{eq:graph-action=diffeo},
\end{equation}
where $\phi_{\Gamma}\colon (Y,\ws_0)\to (Y,\ws_1)$ is the diffeomorphism obtained by moving $\ws_0$ to $\ws_1$ along $\Gamma$.

Similarly, to prove the third claim, we note our maps $F_{W,\gamma,\frs}^{A}$ and $F_{W,\gamma,\frs}^B$ are defined in Equation~\eqref{eq:graph-cobordism-definition} by writing $W$ as a composition of 1-, 2- and 3-handles. The equivalence of our definition of the 1- and 3-handle maps with Ozsv\'{a}th and Szab\'{o}'s construction is established in Lemma~\ref{lem:2-defs-1-handle-equiv}. Our definition of the 2-handle maps coincides exactly with their definition. The only remaining difference between our map and theirs is that our map includes a graph action map between the 1- and 2-handle maps, for a flow-graph that consists of a single arc. Equation~\eqref{eq:graph-action=diffeo} implies that this graph action map is chain homotopic to a basepoint moving diffeomorphism map, from which the main claim follows.

It remains to establish Equation~\eqref{eq:graph-action=diffeo}. We leave this for the final section; see Theorem~\ref{thm:BM=graphact}.
\end{proof}

\section{Basepoint moving maps and the normalization axiom}
\label{sec:movingbasepoints}

If $\lambda$ is a path in $Y$ from $w$ to $w'$, in this section we prove that the diffeomorphism map $\lambda_*$, obtained by pushing $w$ to $w'$ along $\lambda$, satisfies
\begin{equation}
\lambda_*\simeq S_w^-A_\lambda S_{w'}^+.\label{eq:BM=graph-action}
\end{equation}
See Theorem~\ref{thm:BM=graphact}. Using Equation~\eqref{eq:BM=graph-action}, we finish our proof of Theorem~\ref{thm:B} (the normalization axiom) and prove Theorem~\ref{cor:F} (our formula for the $\pi_1$-action).

\subsection{A transition map computation}

A key ingredient in our proof of Theorem~\ref{thm:BM=graphact} is a computation of the transition map shown in Figure~\ref{fig::9}.

Suppose that $\cH_0=(\Sigma, \as,\bs,\ve{w}_0\cup \{z\})$ is a multi-pointed Heegaard diagram and $z$ is a distinguished basepoint. We consider the free-stabilized diagrams, 
\[
\cH_1=(\Sigma, \as\cup \{\alpha_1\},\bs\cup \{\beta_1\},\ve{w}_0\cup \{w,w'\})\quad  \text{and}\quad  \cH_2=(\Sigma, \as'\cup \{\alpha_2\},\bs'\cup \{\beta_2\},\ve{w}_0\cup \{w,w'\}),
\]
 obtained by removing $z$ from $\cH_0$, and inserting the diagrams shown in Figure~\ref{fig::9} into the region which contained $z$. The curves $\as'$ and $\bs'$ are small Hamiltonian isotopies of $\as$ and $\bs$. The case that $\ve{w}_0=\emptyset$ is not excluded. A key step in our proof of Equation~\eqref{eq:BM=graph-action} is to compute the transition map between $\cH_1$ and $\cH_2$.

\begin{figure}[ht!]
\centering
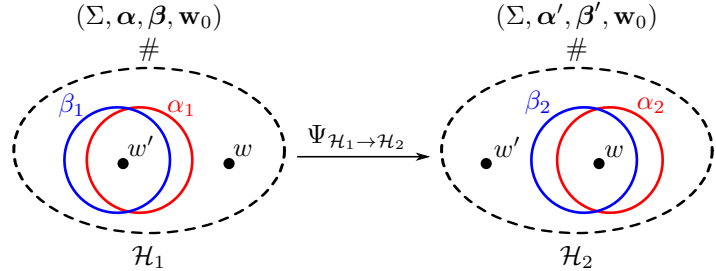
\caption{\textbf{The diagrams $\cH_1$ and $\cH_2$ considered in Theorem~\ref{thm:transitionmapcomp}.}}\label{fig::9}
\end{figure}

We introduce some notation. If $C$ is a module over the ring $\bF_2[U_z]$, we write $C^{U_z\to U_{w}}$ for the tensor product:
\[
C^{U_z\to U_w}:=C\otimes_{\bF_2[U_z]} \bF_2[U_z,U_w]/(U_w-U_z),
\]
which we think of as the module obtained by formally setting $U_z$ to the variable $U_w$.
If $F\colon C_1\to C_2$ is a map of  $\bF_2[ U_z]$-modules,  write $F^{U_{z}\to U_w}$ for the induced map
\[
F^{U_z\to U_{w}}:=F\otimes \id_{\bF_2[U_z,U_w]/(U_w-U_z)}.
\]

If $R$ is a ring of characteristic 2 and $F\colon C_1\to C_2$ is a map of $R$-modules,  write $(F)_{U_w}$ for the map of $R\otimes_{\bF_2} \bF_2[U_w]$-modules
\begin{equation}
F_{U_w}:=F\otimes \id_{\bF_2[U_w]}\colon C_1\otimes_{\bF_2} \bF_2[U_w]\to C_2\otimes_{\bF_2} \bF_2[U_w]\label{eq:extension-of-scalars}
\end{equation}

Finally, we introduce a convenient matrix notation. Let $V$ denote the 2-dimensional vector space $\langle \theta^+,\theta^-\rangle$. If $F$ is a homomorphism
\[
F\colon C_1\otimes_{\bF_2} V\to C_2\otimes_{\bF_2} V,
\]
then we will write $F$ as a $2\times 2$ block matrix. We always use the ordered basis $(\theta^+,\theta^-)$, so the first row and column of such a matrix correspond to $\theta^+$, and the second row and column to $\theta^-$.

\begin{thm}\label{thm:transitionmapcomp} Let $\cH_1$ and $\cH_2$ denote the free-stabilized diagrams  in Figure \ref{fig::9}.  There are choices of almost complex structures $J_1$ and $J_2$ on $\cH_1$  and $\cH_2$, respectively, such that $J_1$  can be used to compute $S_{w}^+$ and $S_{w}^-$, and $J_2$ can be used to compute $S_{w'}^+$ and $S_{w'}^-$, such that (for some choice of intermediate diagrams and almost complex structures)
\[
\Psi_{(\cH_1,J_1)\to(\cH_2,J_2)}= \begin{pmatrix}(\Psi_{\as\to \as'}^{\bs'})^{U_z\to U_{w'}}_{U_w}\circ (\Psi_{\as}^{\bs\to \bs'})^{U_{z}\to U_{w}}_{U_{w'}}& 0\\
*&(\Psi_{\as\to \as'}^{\bs'})^{U_{z}\to U_{w'}}_{U_{w}}\circ  (\Psi_{\as}^{\bs\to \bs'})^{U_{z}\to U_{w}}_{U_{w'}}
\end{pmatrix}.
\]
 Here $\Psi_{\as}^{\bs\to \bs'}$ denotes the transition map from $\CF^-(\Sigma,\as,\bs,\ws_0\cup \{z\},\frs)$ to $\CF^-(\Sigma, \as, \bs', \ws_0\cup \{z\},\frs)$, and $\Psi_{\as\to \as'}^{\bs'}$ is defined similarly.
\end{thm}

 We remark that the component marked $*$ can be computed explicitly. It is equal to
\begin{equation}
(\Phi_{\as\to \as'}^{\bs'})^{U_z\to U_{w'}}_{U_w}\circ \left(\sum_{i,j\ge 0} U_w^{i} U_{w'}^j (\d_{i+j+1})_{U_w,U_{w'}}\right)\circ (\Phi_{\as}^{\bs\to \bs'})^{U_z\to U_{w}}_{U_{w'}}, \label{eq:asterisk-concrete}
\end{equation}
 though we do not need this fact to prove Theorem~\ref{thm:BM=graphact}. Equation~\eqref{eq:asterisk-concrete} can be proven by a small modification of Lemma~\ref{lem:explicitchangeofalmostcomplexstructure}, below. In Equation~\eqref{eq:asterisk-concrete}, the expression $\d_{i+j+1}$ is obtained by writing the differential on $\CF^-(\cH_0,\frs)$ as
\[
(\d_{\cH_0})=\sum_{i=0}^\infty \d_i U^i_z.
\]

The proof of Theorem~\ref{thm:transitionmapcomp} involves computing several  holomorphic triangle maps and several non-cylindrical holomorphic strip counts.

\begin{figure}[ht!]
\centering
\begingroup%
  \makeatletter%
  \providecommand\color[2][]{%
    \errmessage{(Inkscape) Color is used for the text in Inkscape, but the package 'color.sty' is not loaded}%
    \renewcommand\color[2][]{}%
  }%
  \providecommand\transparent[1]{%
    \errmessage{(Inkscape) Transparency is used (non-zero) for the text in Inkscape, but the package 'transparent.sty' is not loaded}%
    \renewcommand\transparent[1]{}%
  }%
  \providecommand\rotatebox[2]{#2}%
  \newcommand*\fsize{\dimexpr\f@size pt\relax}%
  \newcommand*\lineheight[1]{\fontsize{\fsize}{#1\fsize}\selectfont}%
  \ifx\svgwidth\undefined%
    \setlength{\unitlength}{123.66934157bp}%
    \ifx\svgscale\undefined%
      \relax%
    \else%
      \setlength{\unitlength}{\unitlength * \real{\svgscale}}%
    \fi%
  \else%
    \setlength{\unitlength}{\svgwidth}%
  \fi%
  \global\let\svgwidth\undefined%
  \global\let\svgscale\undefined%
  \makeatother%
  \begin{picture}(1,0.81696)%
    \lineheight{1}%
    \setlength\tabcolsep{0pt}%
    \put(0.23662178,0.48515744){\color[rgb]{1,0,0}\makebox(0,0)[rt]{\lineheight{0}\smash{\begin{tabular}[t]{r}$\alpha_1$\end{tabular}}}}%
    \put(0.75999433,0.4851572){\color[rgb]{0,0,1}\makebox(0,0)[lt]{\lineheight{0}\smash{\begin{tabular}[t]{l}$\beta_2$\end{tabular}}}}%
    \put(0.49994734,0.76095759){\color[rgb]{0,0,0}\makebox(0,0)[t]{\lineheight{0}\smash{\begin{tabular}[t]{c}$(\Sigma, \ve{\alpha},\ve{\beta}',\ve{w}_0)$\end{tabular}}}}%
    \put(0,0){\includegraphics[width=\unitlength,page=1]{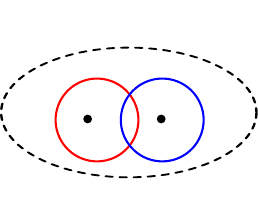}}%
    \put(0.67654959,0.38414117){\color[rgb]{0,0,0}\makebox(0,0)[lt]{\lineheight{0}\smash{\begin{tabular}[t]{l}$w$\end{tabular}}}}%
    \put(0.35477998,0.38414117){\color[rgb]{0,0,0}\makebox(0,0)[lt]{\lineheight{0}\smash{\begin{tabular}[t]{l}$w'$\end{tabular}}}}%
    \put(0.49980456,0.01392953){\color[rgb]{0,0,0}\makebox(0,0)[t]{\lineheight{0}\smash{\begin{tabular}[t]{c}$\cH_{1.5}$\end{tabular}}}}%
    \put(0.49991311,0.67871127){\color[rgb]{0,0,0}\makebox(0,0)[t]{\lineheight{1.25}\smash{\begin{tabular}[t]{c}$\#$\end{tabular}}}}%
    \put(0,0){\includegraphics[width=\unitlength,page=2]{fig79.pdf}}%
  \end{picture}%
\endgroup%

\caption{\textbf{The intermediate diagram $\cH_{1.5}$.}}\label{fig::47}
\end{figure}

On  $\cH_1$, $\cH_{1.5}$, and $\cH_2$ we will write $J_\alpha$ for an almost complex structure which is stretched along the circles $c$ and $c_\alpha$ in Figure \ref{fig::48}. We write $J_\beta$ for an almost complex structure which is stretched along the circles $c$ and  $c_\beta$ in Figure \ref{fig::48}.

\begin{figure}[ht!]
\centering
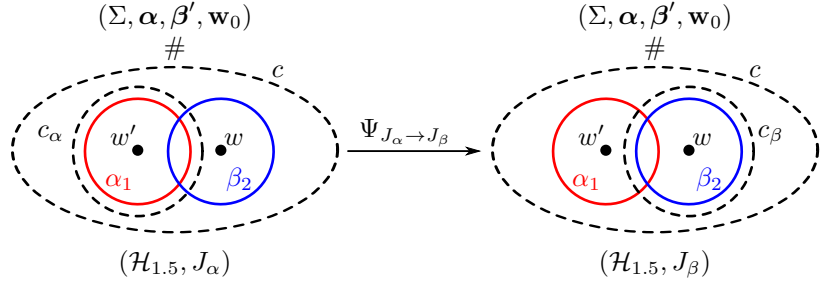
\caption{\textbf{The almost complex structures $J_\alpha$ and $J_\beta$ for $\cH_{1.5}$.}}\label{fig::48}
\end{figure}

To help simplify the statements of some of the results in this section, we prove the following:

\begin{lem}\label{lem:relative-lengths-unimportant} Let $\cH$ be one of $\cH_1,$ $\cH_{1.5}$ or $\cH_{2}$. If $\ve{T}=(T,T')$,  write $J_{\alpha}(\ve{T})$ for an almost complex structure which has been stretched along $c$ and $c_{\alpha}$, with neck-lengths $T$ and $T'$. There is a constant $N$ such that if $\ve{T}_1$ and $\ve{T}_2$ are two pairs of neck lengths, all of whose components are greater than $N$, then there is a non-cylindrical almost complex structure $\tilde{J}$ interpolating $J_\alpha(\ve{T}_1)$ and $J_{\alpha}(\ve{T}_2)$ satisfying
\[
\Psi_{J_\alpha(\ve{T}_1)\to J_\alpha(\ve{T}_2)}:=\Psi_{\tilde{J}}=\begin{pmatrix}\id &0\\
0& \id
\end{pmatrix}.
\]
\end{lem}
\begin{proof} We focus on the case that $\cH=\cH_1$. The statement is proven similarly for $\cH_{1.5}$ and $\cH_2$. The proof is a double neck stretching argument. Suppose that $\ve{T}_{1,i}$ and $\ve{T}_{2,i}$ are two sequences of pairs of neck-lengths, all of whose components approach $+\infty$.    Write $\ve{T}_{1,i}=(T_{1,i},T_{1,i}')$ and $\ve{T}_{2,i}=(T_{2,i}, T_{2,i}')$. Define
 \[
 T_{\min,i}=\min(T_{1,i},T_{2,i})\qquad \text{and}\qquad T_{\min,i}'=\min(T_{1,i}',T_{2,i}').
 \]
 We decompose neighborhoods of $c$ and $c_\alpha$ as unions of three annuli, as shown in Figure~\ref{fig::83}:
 \[
N(c)=N_1\cup N_2\cup N_3\qquad \text{and} \qquad N(c_\alpha)=N_1'\cup N_2'\cup N_3'. 
 \]
 
Construct interpolating almost complex structures $\tilde{J}_i$ between $J_{\alpha}(\ve{T}_{1,i})$ and $J_{\alpha}(\ve{T}_{2,i})$.  We require that the almost complex structures $\tilde{J}_i$ be split on each of $N_1,$ $N_3$, $N_1'$ and $N_3'$ (in particular,  $\tilde{J}_i$ is cylindrical on these regions). Furthermore, we require that $\tilde{J}_i$ be chosen so that $N_1$ and $N_3$ are both conformally equivalent to $S^1\times [0,T_{\min,i}/3]$, and so that $N_1'$ and $N_3'$ are conformally equivalent to $S^1\times [0,T_{\min,i}'/3]$. Further, we assume that $\tilde{J}_i$ is only non-cylindrical in the regions $N_2$ and $N_2'$. We can pick $\tilde{J}_i$ so that  projection from $\Sigma\times [0,1]\times \R$ to $[0,1]\times \R$ is holomorphic.
 
 \begin{figure}[ht!]
\centering
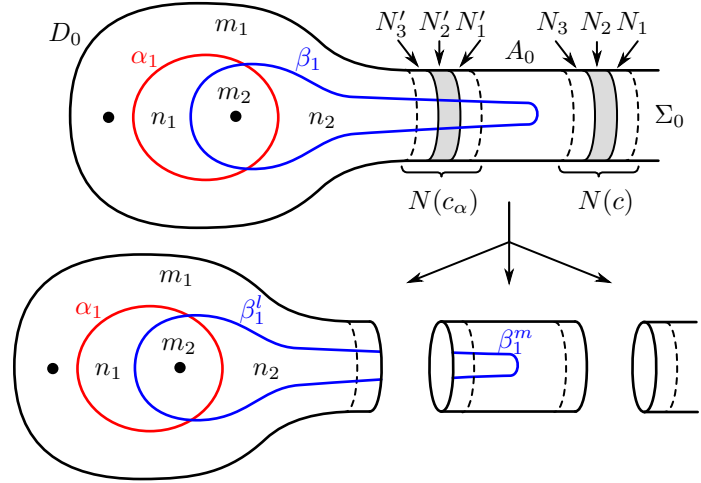
\caption{\textbf{Decomposing the stretching regions in Lemma~\ref{lem:relative-lengths-unimportant}.}}\label{fig::83}
\end{figure}

 Write 
\begin{equation}
\Psi_{\tilde{J}_i}=\begin{pmatrix} A_i& B_i\\
C_i& D_i
\end{pmatrix}.
\label{eq:transition-matrix-i}
\end{equation}
If $\phi\# \phi_0\in \pi_2(\xs\times x,\ys\times y)$ is a Maslov index 0 class of disks, then Equation~\eqref{eq:diffMasexcision} implies
\begin{equation}
\mu(\phi\# \phi_0)=\mu(\phi)+\gr(x,y)+2m_2(\phi_0).\label{eq:maslov-index-double-neck}
\end{equation}
Classes with $\gr(x,y)=1$ contribute to $C_i$, classes with $\gr(x,y)=0$ contribute to $A_i$ or $D_i$, and classes with $\gr(x,y)=-1$ contribute to $B_i$.

Given a sequence $u_i$ of $\tilde{J}_i$-holomorphic curves representing $\phi\# \phi_0$, we may extract a subsequence which converges to broken holomorphic curves on the diagrams $(S^2,\alpha_1,\beta_1^l)$, $(S^2,\beta_1^m)$ and $(\Sigma,\as,\bs)$, whose total homology class is $\phi\# \phi_0$. The curves $\beta_1^l$ and $\beta_1^m$ are the result cutting $\beta_1$ along its intersection with $c_\alpha$, and then collapsing the ends to a point; see Figure~\ref{fig::83}. Furthermore, the broken limiting curves are pseudo-holomorphic for cylindrical almost complex structures. 

In somewhat more detail, to construct such a convergent subsequence, let $D_0$ denote the disk component of $\Sigma\setminus N(c_{\alpha})$ containing $\alpha_1$, and  let $A_0$ denote the annulus between $N(c)$ and $N(c_{\alpha})$, as in Figure~\ref{fig::83}. 
Let $\Sigma_0$ denote the component of $\Sigma\setminus N(c)$ which is disjoint from $w$ and $w'$.

If $u_i$ is $\tilde{J}_i$-holomorphic, we take preimages of $u_i$ into subregions of $\Sigma\times [0,1]\times \R$ to construct holomorphic curves with additional boundary circles, $u_i^l$, $u_i^m$ and $u_i^r$, which map into the 4-manifolds $(D_0\cup N_3')\times [0,1]\times \R$, $(N_1'\cup A_0\cup N_3)\times [0,1]\times \R$ and $ (N_1\cup \Sigma_0)\times [0,1]\times \R$, respectively. These are holomorphic curves for cylindrical almost complex structures. We can view $((D_0\cup N_3')\times [0,1]\times \R, \tilde{J}_i)$ as being contained in $((D_0\cup N_3')\times [0,1]\times \R, \tilde{J}_k)$, whenever $T_{\min,i}\le T_{\min,k}$. Consequently, given such a sequence $u_i$, we may find a subsequence such that $u_i^l$, $u_i^m$ and $u_i^r$ each converge to curves in the punctured manifolds $S^2\setminus \{p_0\}\times [0,1]\times \R$, $S^1\times \R\times [0,1]\times \R$ and $\Sigma\setminus \{p\}\times [0,1]\times \R$, where $p_0$ and $p$ denote the connected sum  points corresponding to the circles $c_\alpha$ and $c$, respectively.

Consequently, $\phi$ and $\phi_0$ admit broken homomorphic representatives on $(\Sigma,\as,\bs)$ and $(S^2,\alpha_1,\beta_1^l)$.  In particular $\mu(\phi)\ge 0$, so Equation~\eqref{eq:maslov-index-double-neck} implies
\[
m_2(\phi_0)=0\qquad \text{and} \qquad \gr(x,y)\le 0.
\]
 Hence,  $B_i=0$ for large $i$, since $B_i$ counts curves with $\gr(x,y)=+1$

Next, we consider classes with $\gr(x,y)=0$ (which contribute to $A_i$ and $D_i$). Since $\phi$ admits a broken representative and has index 0 by Equation~\eqref{eq:maslov-index-double-neck}, we conclude that $\phi$ must represent the constant class $e_{\xs}$. It is straightforward to see that this also implies that $\phi_0$ is the constant class $e_x$. Conversely, since $\tilde{J}_i$ is cylindrical in a neighborhood of all the intersection points, the index 0 classes $e_{\xs}\times  e_{x}$ have unique $\tilde{J}_i$-holomorphic representatives for all $i$. Consequently $A_i=D_i=\id$.

We now consider classes with $\gr(x,y)=-1$, which contribute to $C_i$. For such classes, Equation~\eqref{eq:maslov-index-double-neck} and the inequality $\mu(\phi)\ge 0$ imply
\[
\mu(\phi)=1\qquad \text{and}\qquad  m_2(\phi_0)=0.
\]
 Since $\mu(\phi)=1$, transversality implies that the limiting curve on $(\Sigma,\as,\bs)$  is a non-broken index 1 flow line. Since projection to $[0,1]\times \R$ is $\tilde{J}_i$-holomorphic, the asymptotics of the curves on $(S^2,\beta_1^m)$ must match the curve on $(\Sigma,\as,\bs)$ at the connected sum point. Since the limiting curve on $(\Sigma,\as,\bs)$ is a genuine flow line, its asymptotics at the connected sum point consist of $m_1(\phi_0)$ points in $[0,1]\times \R$ (which we can assume are distinct, for a generically chosen almost complex structure). The only curve on the partial diagram $(S^2,\beta_1^m)$ which could match $m_1(\phi_0)$ distinct points in $[0,1]\times \R$ consists of $m_1(\phi_0)$ copies of $S^2$ which each map constantly to the $[0,1]\times \R$ factor, as well as some collection of $\beta_1^m$-boundary degenerations, which have $m_1=0$. The 2-spheres contribute equally to $m_1(\phi_0)$ and $n_2(\phi_0)$, while the boundary degenerations only contribute to $n_2(\phi_0)$. Consequently, 
\begin{equation}
n_2(\phi_0)\ge m_1(\phi_0). \label{eq:spheres-contribute-equally}
\end{equation}
Since $\phi_0\in \pi_2(\theta^-,\theta^+)$ and $m_2(\phi_0)=0$, we must have
\begin{equation}
m_1(\phi_0)+1=n_1(\phi_0)+n_2(\phi_0).\label{eq:vertex-multiplicity}
\end{equation}
(Equation~\eqref{eq:vertex-multiplicity} is satisfied for all classes $\phi_0\in \pi_2(\theta^-,\theta^+)$ with $m_2(\phi_0)=0$).
Equation~\eqref{eq:vertex-multiplicity} contradicts Equation~\eqref{eq:spheres-contribute-equally}, so $C_i$ must be zero for large $i$.
\end{proof}

 We now describe a refinement of the differential computation of Proposition~\ref{prop:free-stabdifferential}. The following computation is similar to \cite{OSLinks}*{Proposition~6.5}, though the placement of basepoints and choices of almost complex structures are different than they consider. Recall that $\cH_{1.5}$ is the stabilization of $\cH_0$ shown in Figure~\ref{fig::47}.

 \begin{lem}\label{lem:differentialcompdegen} Let $J_{\alpha}$ denote an almost complex structure on $\cH_{1.5}$ which is stretched along $c$ and $c_{\alpha}$. For sufficiently large neck lengths along $c$ and $c_{\alpha}$, we have 
 \[
\d_{\cH_{1.5},J_\alpha}=\begin{pmatrix}(\d_{\cH_0})^{U_z\to U_w}& U_w+U_{w'}\\
 0& (\d_{\cH_0})^{U_z\to U_w}
 \end{pmatrix}.
\]
If $J_{\beta}$ denotes an analogous almost complex structure stretched sufficiently along $c$ and $c_{\beta}$, then
 \[
\d_{\cH_{1.5},J_\beta}=\begin{pmatrix}(\d_{\cH_0})^{U_z\to U_{w'}}& U_w+U_{w'}\\
  0& (\d_{\cH_0})^{U_z\to U_{w'}}
  \end{pmatrix}.
\]
 \end{lem}

\begin{proof}

We focus on computing $\d_{\cH_{1.5},J_\alpha}$; the computation for $J_\beta$ is a straightforward modification.  Write
\[
\d_{\cH_{1.5},J_\alpha}=\begin{pmatrix}A& B\\
C& D
\end{pmatrix}.
\]
By Lemma~\ref{lem:relative-lengths-unimportant}, if the claim holds any pair of arbitrarily large neck lengths, then it holds for all sufficiently large pairs of neck-lengths (regardless of their relative lengths).

Let $m_1,$ $m_2,$ $n_1$ and $n_2$ denote the multiplicities of the regions of $(S^2,\alpha_1,\beta_2,w,w')$ shown in Figure \ref{fig::49}.

\begin{figure}[ht!]
\centering
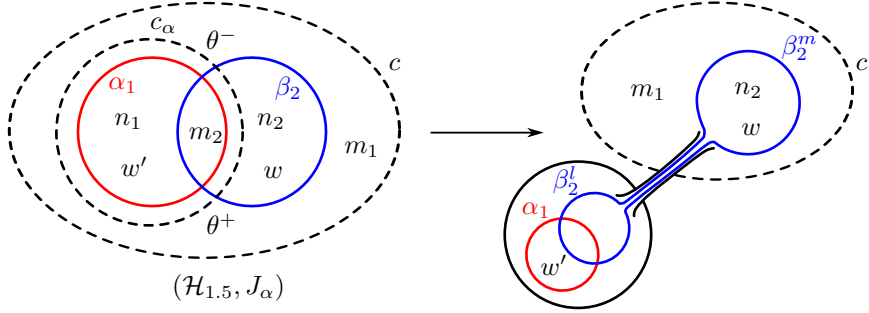
\caption{\textbf{Stretching along $c_\alpha$.} On the left are multiplicities $m_1$, $m_2$, $n_1$ and $n_2$. On the right we illustrate how the almost complex structure degenerates as we stretch $J_\alpha$ along $c_\alpha$.}\label{fig::49}
\end{figure}

 Suppose $\phi\# \phi_0\in \pi_2(\xs\times x, \ys\times y)$ is a homology class of disks on $\cH_{1.5}$.  Lemma \ref{lem:indexdisksonsphere} adapts to show 
\begin{equation}
\mu(\phi\# \phi_0)=\mu(\phi)-\gr(x,y)+2m_2(\phi_0).\label{eq:Maslov-index-H1.5}
\end{equation}

Classes with $\gr(x,y)=1$ contribute to $C$. Classes with $\gr(x,y)=0$ contribute to $A$ or $D$. Classes with $\gr(x,y)=-1$ contribute to $B$.

We begin by computing $B$. Suppose $\phi\# \phi_0\in \pi_2(\xs\times \theta^-,\ys\times \theta^+)$ is a class counted by $B$. By stretching sufficiently along $c$, we may assume that $\phi$ has a broken representative, and consequently $\mu(\phi)\ge 0$. Equation~\eqref{eq:Maslov-index-H1.5} implies that $\mu(\phi)=m_2(\phi_0)=0$. Since $\phi$ has a broken representative, it follows that $\phi$ is a constant class $e_{\xs}$. Since $m_2(\phi_0)=m_1(\phi_0)=0$, we conclude that $\phi_0$ must have domain equal to one of the two bigons going over $w$ or $w'$. These classes also have a unique holomorphic representative. Consequently, the map $B$ is multiplication by $U_w+U_{w'}$, as claimed.

We now compute the maps $A$ and $D$, which count classes with $\gr(x,y)=0$. Equation~\eqref{eq:Maslov-index-H1.5} implies that $\mu(\phi)=1$ and $m_2(\phi_0)=0$.  Let us fix a neck length along $c_\alpha$.  By the proof of Proposition~\ref{prop:free-stabdifferential}, for sufficiently large neck-length along $c$, there is an identification 

\begin{equation}
 \cM(\phi\# \phi_0)\iso  \left\{(u, u_0)\in \cM(\phi)\times \cM(\phi_0): \rho^{p}(u)=\rho^{p_0}(u_0)\right\},\label{eq:fibered-product-H_1.5}
\end{equation}
 where $p$ and $p_0$ denote the connected sum points on $\Sigma$ and $S^2$. Note that the required neck-length along $c$ for Equation~\eqref{eq:fibered-product-H_1.5} to hold may depend on the fixed neck-length along $c_\alpha$.
 
 By the proof of Proposition~\ref{prop:free-stabdifferential}, we know if $\theta\in \{\theta^+,\theta^-\}$ and $\ve{d}\in \Sym^{m}([0,1]\times \R)$ is a fixed, generic element, then
 \begin{equation}
\sum_{\substack{\phi_0\in \pi_2(x,x) \\m_2(\phi_0)=0 \\ m_1(\phi_0)=m}} \# \cM(\phi_0,\ve{d})\equiv 1 \pmod{2}.\label{eq:count-H1.5}
\end{equation}
Equations~\eqref{eq:fibered-product-H_1.5} and~\eqref{eq:count-H1.5} together imply that if $\mu(\phi)=1$ and $\theta\in \{\theta^+,\theta^-\}$, then
\begin{equation}
\# \hat{\cM}(\phi)\equiv \sum_{\substack{\phi_0\in \pi_2(\theta,\theta)\\ m_2(\phi_0)=0\\ m_1(\phi_0)=m_1(\phi)}} \# \hat{\cM}(\phi\# \phi_0)\pmod{2}.\label{eq:count-index-1-classes-H1.5}
\end{equation}
Note that Equation~\eqref{eq:count-index-1-classes-H1.5} almost gives us the desired identification of $A$ and $D$, except it does not inform us about the multiplicity on the basepoints $w'$ and $w$. By considering the multiplicities in the four regions around $\theta^+$, we know if  $\phi_0\in \pi_2(\theta,\theta)$ has $m_2(\phi)=0$, then
\begin{equation}
m_1(\phi_0)=n_1(\phi_0)+n_2(\phi_0).\label{eq:m1=n1+n2}
\end{equation}
The integer $m_1(\phi_0)$ is the power of $U_z$ contributed by $\phi$ to $\d_{\cH_0}$, while $n_1(\phi_0)$ and $n_2(\phi_0)$ are the powers of $U_w$ and $U_{w'}$, respectively, contributed by $\phi\# \phi_0$ to $\d_{\cH_{1.5}}$.

We finish our claim about the maps $A$ and $D$, it is sufficient to show that if $\ve{d}\in \Sym^{m}([0,1]\times \R)$ is a fixed point, then for almost complex structure sufficiently stretched along $c_{\alpha}$, the only classes $\phi_0\in \pi_2(\theta,\theta)$ with non-empty $\cM(\phi_0,\ve{d})$ have $n_2(\phi_0)=m_1(\phi_0)$ and $n_1(\phi_0)=0$. This implies that any curve which makes non-trivial contribution to $A$ or $D$ is counted with a  factor of $U_w^{m_1(\phi)}$ and no factor of $U_{w'}$. We prove this in the following subclaim:

\begin{subclaim} Suppose $\ve{d}\in \Sym^{m}([0,1]\times \R)$ is not contained in the fat-diagonal, and also does not contain any points of $\{0,1\}\times \R$. Suppose $\phi_0\in \pi_2(\theta,\theta)$ has $m_2(\phi_0)=0$, $m_1(\phi_0)=m$ and $n_1(\phi_0)> 0$. If the almost complex structure on $S^2\times [0,1]\times \R$  is sufficiently stretched along $c_{\alpha}$, then $\cM(\phi_0,\ve{d})$ is empty.
\end{subclaim}

\begin{proof} The argument is similar to our proof of Lemma~\ref{lem:relative-lengths-unimportant}. As we let the stretching parameter approach $+\infty$, we can extract broken limiting curves on $(S^2,\alpha_1,\beta_2^l)$ and the degenerate diagram $(S^2,\beta_2^m)$. See Figure~\ref{fig::49}. The curves $\beta_2^l$ and $\beta_2^m$ are obtained by cutting $\beta_2$ along $c_\alpha$, and collapsing the endpoints.

 Consider the limiting curves on $(S^2,\beta_2^m)$. Since these arose as the limit of curves which matched $\ve{d}$, the limiting curves must also match $\ve{d}$ at the point $p_0$ (recall that $p_0$ corresponds to the circle $c$ in Figure~\ref{fig::49}). There are no $\as$ curves on the degenerate diagram $(S^2,\beta_2^m)$. Consequently, any limiting curve either has no boundary, or has boundary which maps to $\beta_2^m$. Any curve which has boundary on $\beta_2^m$ must map locally constantly to $\{0\}\times \R\subset [0,1]\times \R$, by the maximum principle. Consequently, the curves which match $\ve{d}$ can only be spheres which map to $S^2\times \{d\}\subset S^2\times [0,1]\times \R$ for $d\in \ve{d}$, together with boundary degenerations which do not cover $p_0$.  Such a sphere contributes equally to $m_1$ and $n_2$, while a boundary degeneration not covering $p_0$ only contributes to $n_2$. Consequently, if $\cM(\phi_0,\ve{d})$ is non-empty for sufficiently stretched along complex structure, then
\begin{equation}
m_1(\phi_0)\le n_2(\phi_0).\label{eq:inequality-multiplicity}
\end{equation}
Combined with Equation~\eqref{eq:m1=n1+n2}, since all multiplicities are non-negative, Equation~\eqref{eq:inequality-multiplicity} implies that if $\cM(\phi_0,\ve{d})$ is nonempty, then $m_1(\phi_0)=n_2(\phi_0)$ and $n_1(\phi_0)=0$, completing the proof.
\end{proof}

We now prove that $C$ is zero (recall $C$ counts representatives of classes in $\pi_2(\xs\times \theta^+,\ys\times \theta^-)$). For such classes, Equation~\eqref{eq:Maslov-index-H1.5} constrains $m_2(\phi_0)$ to be in $\{0,1\}$. 

If $m_2(\phi_0)=1$, then Equation~\eqref{eq:Maslov-index-H1.5} implies that $\mu(\phi)=0$, forcing $\phi$ to be a constant class, and $\phi_0$ to be the bigon $m_2(\phi_0)=1$. This class has a unique holomorphic representative.  If $m_2(\phi_0)=0$, then Equation~\eqref{eq:Maslov-index-H1.5} implies that if the almost complex structure is sufficiently stretched along $c$, then $\mu(\phi)=2$. It remains to count such classes and show that their total counts cancel the bigons with $m_2(\phi_0)=1$.

We prove a helpful subclaim:

\begin{subclaim}\label{subclaim:d_inearbeta} Suppose that $d_i\in [0,1]\times \R$ is a sequence of points approaching the line $\{0\}\times \R$, and suppose $\phi\in \pi_2(\xs,\ys)$ is a Maslov index 2 class on $(\Sigma,\as,\bs')$. If the matched moduli space $\cM(\phi,d_i)$ is nonempty for arbitrarily large $i$, then $\phi$ has domain equal to a connected component of $\Sigma\setminus \bs$.
\end{subclaim}

\begin{proof}
Let $u_i$ be a sequence of curves in $\cM(\phi,d_i)$. Since $p$ is contained in $\Sigma\setminus (\as\cup \bs)$, any curve $u$ in the limit with $n_p(u)\neq 0$ must have $\pi_{[0,1]\times \R}\circ u$ constant, by the maximum principle. Consequently, $u$ must have domain equal to $[\Sigma]$, or to a $\beta$-boundary degeneration. Since such curves have Maslov index at least 2, and any other curves which appear will achieve transversality by Proposition~\ref{prop:transversality}, no additional curves may appear in the limit. The proof is complete.
\end{proof}
%
 
  The only non-negative class $\phi_0\in \pi_2(\theta^+,\theta^-)$ with $m_1(\phi_0)=1$ and $m_2(\phi_0)=0$ is the bigon with $m_1(\phi_0)=1$. Define the point
  \[
  d(\phi_0):=(u\circ \pi_{[0,1]\times \R})\big( (u\circ \pi_{S^2})^{-1}(p_0)\big)\in [0,1]\times \R,
  \]
  where $u$ is a representative of the bigon $\phi_0$ (well defined up to $\R$ translation). Note that $d(\phi_0)$ depends on the choice of almost complex structure on $S^2\times [0,1]\times \R$. By stretching along $c_{\alpha}$, we can make $d(\phi_0)$ arbitrarily close to $\{0\}\times \R$. Let us fix a neck length along $c_\alpha$, so the conclusion of Subclaim~\ref{subclaim:d_inearbeta} holds.  
  
Having fixed a neck-length along $c_\alpha$, we now stretch along $c$. Suppose $\phi\# \phi_0$ admits a sequence of holomorphic representatives for a sequence of neck lengths along $c$ which approach $+\infty$. We can extract broken limits. The curve appearing on $S^2\times [0,1]\times \R$  will be a representative of the bigon $\phi_0$. On $\Sigma\times [0,1]\times \R$, there will be a curve in $\cM(\phi',d(\phi_0))$, for some class $\phi'$. We claim that $\phi'$ must actually be equal to the class $\phi$. Using Proposition~\ref{prop:transversality}, it is straightforward to show that any curves appearing in the limit either achieve transversality, or have the homology class of a boundary degeneration or have domain $[\Sigma]$. Classes which have domain $[\Sigma]$ or which have the domain of a boundary degeneration have Maslov index at least 2 by Equation~\eqref{eq:indexbounddegen}. Hence, other curves are prohibited from appearing in the limit by a dimension count, since $\mu(\phi)=2$. Consequently, $\phi=\phi'$.
  
Subclaim~\ref{subclaim:d_inearbeta} implies that since $\cM(\phi,d(\phi_0))$ is nonempty, $\phi$ must have domain equal to a connected component of $\Sigma\setminus \bs$.

Having restricted the homology classes which can contribute to $C$, we now count their holomorphic representatives. Let $\phi\# \phi_0\in \pi_2(\ve{x}\times \theta^+,\ve{x}\times \theta^-)$ denote an index 1 homology class with $m_1(\phi_0)=1$ and $m_2(\phi_0)=0$, such that the domain of $\phi$ is a connected component of $\Sigma\setminus \bs$.

 Let $\phi'$ denote the Maslov index 1 bigon going over $w'$. Splicing $\phi\# \phi_0$ and $\phi'$ together, we get the Maslov index 2 homology class $B'_{\xs\times \theta^+}=\phi'* (\phi\# \phi_0)\in \pi_2(\ve{x}\times \theta^+,\ve{x}\times \theta^+)$. The moduli space $\hat{\cM}(B'_{\xs\times \theta^+})$ is 1-dimensional. We count its ends. One end corresponds to the boundary degenerations in $\hat{\cN}^\beta(B'_{\xs\times \theta^+})$, which have total count equal to 1 by \cite{OSLinks}*{Theorem 5.5}. The other ends correspond to representatives of $B'_{\xs\times \theta^+}$ breaking into two holomorphic strips. Our previous argument implies that the only other end of  $\hat{\cM}(B)$ consists of the broken holomorphic strip consisting of a representative of $\phi\# \phi_0$  and a representative of the bigon $\phi'$. 
 
Summing over the ends of $\hat{\cM}(B'_{\xs\times \theta^+})$, we conclude
\[
\#\hat{\cN}^\beta(B'_{\xs\times \theta^+})+\#\hat{\cM}(\phi\# \phi_0)\cdot \# \hat{\cM}(\phi')=0.
\]
 Since $\# \hat{\cN}^\beta(B'_{\xs\times \theta^+})=\#\hat{\cM}(\phi')=1$, we conclude that $\# \hat{\cM}(\phi\# \phi_0)=1$.
 
 It follows that $C=0$, completing the proof.
\end{proof}

\begin{figure}[ht!]
\centering
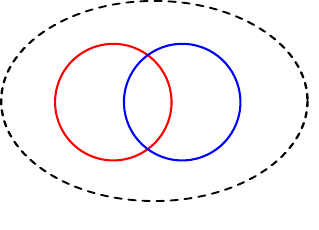
\caption{\textbf{Multiplicities on $\cH_1$.}}\label{fig::56}
\end{figure}

In a similar manner to Lemma~\ref{lem:differentialcompdegen}, we may compute the differentials on $\cH_1$ and $\cH_2$:

\begin{lem}\label{lem:differentialonH_1H_2} The differentials $\d_{\cH_1,J_\alpha}$ and $\d_{\cH_1,J_\beta}$ have the same form:
\[
\d_{\cH_1,J_\alpha}=\d_{\cH_1,J_\beta}=\begin{pmatrix}(\d_{\cH_0})^{U_z\to U_w}& U_w+U_{w'}\\
0& (\d_{\cH_0})^{U_z\to U_{w}}
\end{pmatrix}.
\]
 Furthermore, using the multiplicities from Figure~\ref{fig::56}, the index 1 $J_{\alpha}$-holomorphic curves counted by $\d_{\cH_1,J_{\alpha}}$ satisfy the following:

\begin{enumerate}
\item Any Maslov index 1 class in $\pi_2(\ve{x}\times \theta^+,\ve{y}\times \theta^+)$ or $\pi_2(\ve{x}\times \theta^-,\ve{y}\times \theta^-)$ with holomorphic representatives  has $n_1=m_2=0$ and $n_2=m_1$.
\item The index 1 classes in $\pi_2(\ve{x}\times \theta^+,\ve{y}\times \theta^-)$ with holomorphic representatives have domain equal to one of the two bigons with  $n_1=1$ or $n_2=1$ (and all other multiplicities zero).
\item The index 1 classes in $\pi_2(\ve{x}\times \theta^-,\ve{y}\times \theta^+)$ with holomorphic representatives have one of two domains. One domain is a bigon on  $(S^2,\alpha_1,\beta_1)$ with $m_2=1$ (and all other multiplicities zero). The other domain is the connected sum of the bigon on $(S^2,\alpha_1,\beta_1)$ with $m_1=1$, together with the domain on $(\Sigma,\as,\bs)$ consisting of the connected component of $\Sigma\setminus \bs$ which contains the connected sum point. Any class with either domain has one representative, modulo 2.
\end{enumerate}  
Similar statements hold for $\d_{\cH_1,J_{\beta}},$ $\d_{\cH_2,J_\alpha}$ and $\d_{\cH_2,J_{\beta}}$.
\end{lem}
\begin{proof}The differential $\d_{\cH_1,J_{\alpha}}$ counts the same holomorphic curves as $\d_{\cH_{1.5},J_{\alpha}}$, which is analyzed in Lemma~\ref{lem:differentialcompdegen}. The present claim is proven by repeating the argument therein, while keeping track of the multiplicities over the basepoints, as they appear in $\cH_1$.
\end{proof}

We now compute the change of almost complex structure map $\Psi_{J_\alpha\to J_\beta}$ on $\CF^-(\cH_{1.5},\frs)$.

\begin{lem}\label{lem:changealmostcomplexstructure} Consider the almost complex structures $J_\alpha$ and $J_{\beta}$ on $\cH_{1.5}$, obtained by stretching along $c$ and $c_\alpha$, or $c$ and $c_{\beta}$, respectively.  Whenever the necks along $c$ are sufficiently large,
\begin{equation}
\Psi_{J_\alpha\to J_\beta}=\begin{pmatrix}\id &0\\
*&\id 
\end{pmatrix}.
\label{eq:J_a->J_b-transition-matrix}
\end{equation}
If all necks are sufficiently long, then the $*$ component may be identified with
\[
*=\sum_{i,j\ge 0} U_w^{i} U_{w'}^j (\d_{i+j+1})_{U_{w},U_{w'}}.
\]
\end{lem}
\begin{proof}Lemma~\ref{lem:relative-lengths-unimportant} implies that the relative lengths along $c$ and $c_{\alpha}$ for $J_\alpha$ do not affect the transition map, and similarly the relative lengths along $c$ and $c_{\beta}$ for $J_{\beta}$ do not affect the transition map.

We fix a neck length along $c_{\alpha}$ for $J_{\alpha}$, and a neck length along $c_{\beta}$ for $J_{\beta}.$ The computation of the three components not marked with a $*$ follows from a small adaptation to Lemma~\ref{prop:freestab-Tsufflargeexist}.

To compute the entry marked with a $*$, we  stretch along all necks so that the differentials $\d_{\cH_{1.5}, J_{\alpha}}$ and $\d_{\cH_{1.5},J_{\beta}}$ take the form described in Lemma~\ref{lem:differentialcompdegen}. Write $C$ for the entry labeled $*$ in Equation~\eqref{eq:J_a->J_b-transition-matrix}. The transition map  $\Psi_{J_{\alpha}\to J_{\beta}}$ is a chain map. We view the relation
\[
\Psi_{J_{\alpha}\to J_{\beta}} \circ \d_{\cH_{1.5}, J_{\alpha}}+\d_{\cH_{1.5}, J_{\beta}}\circ \Psi_{J_{\alpha}\to J_{\beta}}=0
\]
as a matrix involving two-by-two matrices. The diagonal entries give
\[
(\d_{\cH_0})^{U_z\to U_w}+(\d_{\cH_0})^{U_z\to U_{w'}}=(U_{w}+U_{w'})\cdot C,
\]
which algebraically implies the stated form of the map $C$.

\end{proof}

We now perform a triangle map computation for a triple which has been stabilized as in Figure \ref{fig::50}.

\begin{figure}[ht!]
\centering
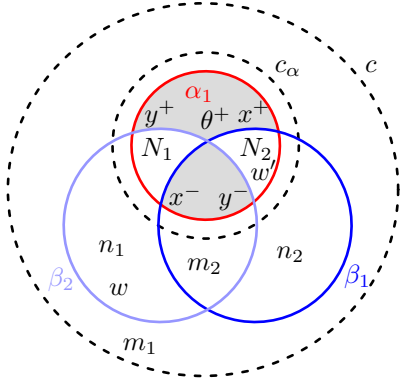
\caption{\textbf{The triple $\hat{\cT}$ in Proposition~\ref{prop:trianglecount1}}, a stabilization of $\cT$. The shaded regions aretwo  examples of small triangles, which might be counted.}\label{fig::50}
\end{figure}

\begin{prop}\label{prop:trianglecount1}Suppose that $\cT=(\Sigma, \as,\bs,\bs',\ve{w}_0\cup \{z\})$ is a Heegaard triple with a distinguished basepoint $z$,  and $\bs'$ are small isotopies of $\bs$, satisfying $|\beta_i\cap \beta_j'|=2\delta_{ij}$. Let $\hat{\cT}=(\Sigma, \as\cup \{\alpha_1\},\bs\cup \{\beta_1\},\bs'\cup \{\beta_2\},\ve{w}_0\cup \{w,w'\})$ be the Heegaard triple obtained by replacing a neighborhood of $z$ with the region shown in Figure~\ref{fig::50}. Let $J$ be an almost complex structure on $\cT$, and $J(T)$  an almost complex structure on $\hat{\cT}$ which is stretched along $c_\alpha$. Using matrix notation, we have
\[
F_{\hat{\cT},J(T)}(\ve{x}\times-, \Theta_{\beta,\beta'}^+\times \theta^+)=\begin{pmatrix} F_{\cT,J}(\ve{x},\Theta_{\beta,\beta'}^+)^{U_z\to U_{w}}&0\\
0& F_{\cT,J}(\ve{x},\Theta_{\beta,\beta'}^+)^{U_z\to U_{w}}
\end{pmatrix},
\]
for sufficiently large $T$.
\end{prop}

The following lemma will be useful in the proof of Proposition~\ref{prop:trianglecount1}:

\begin{lem}\label{lem:smallisotopiesMaslov} Suppose that $(\Sigma, \bs,\bs',\ve{w})$ is a diagram for $(S^1\times S^2)^{\# g(\Sigma)}$ such that $\bs'$ are small isotopies if $\bs$, and $|\beta_i\cap \beta_j'|=2\delta_{ij}$. Let $\Theta_{\beta,\beta'}^+$ denote the top graded intersection point of $\bT_{\beta}\cap \bT_{\beta'}$. If $\phi\in \pi_2(\Theta_{\beta,\beta'}^+,\ve{y})$ is a non-negative homology class, then $\mu(\phi)-n_{w_0}(\phi)\ge 0$ for all $w_0\in \ve{w}$. Furthermore,  $\mu(\phi)-n_{w_0}(\phi)=0$ for some $w_0\in \ws$ if and only if $\phi$ is the constant homology class $e_{\Theta_{\b,\b'}^+}$.
\end{lem}
\begin{proof}By the formula for the relative Maslov grading in \cite{OSDisks}, we have 
\begin{equation}
 \gr(\Theta_{\b,\b'}^+,\ve{y})=\mu(\phi)-2\sum_{w\in \ve{w}} n_w(\phi)\ge 0.\label{eq:grading-definition}
\end{equation}
 Hence, if $\phi$ is non-negative
\begin{equation}
\begin{split}
\mu(\phi)-n_{w_0}(\phi)&=\gr(\Theta_{\b,\b'}^+,\ys)+n_{w_0}(\phi)+2\sum_{w\in \ws\setminus \{w_0\}} n_w(\phi)\\
&\ge n_{w_0}(\phi)+2\sum_{w\in \ws\setminus \{w_0\}} n_w(\phi).\\
&\ge 0
\end{split} \label{eq:simple-diagram-maslov>pts}
\end{equation}
In particular, $\mu(\phi)-n_{w_0}(\phi)\ge 0$. Furthermore,  Equation~\eqref{eq:simple-diagram-maslov>pts} implies that if $\mu(\phi)-n_{w_0}(\phi)=0$ for some $w_0$ then $n_w(\phi)=0$ for all $w\in \ws$, and $\mu(\phi)=0$.  From Equation~\eqref{eq:grading-definition}, we see $\gr(\Theta_{\b,\b'}^+,\ys)=0$, so $\ys=\Theta_{\b,\b'}^+$.  By diagrammatic inspection, the only non-negative class $\phi\in \pi_2(\Theta^+_{\b,\b'}, \Theta^+_{\b,\b'})$ with $\mu(\phi)=0$ and $n_{w}(\phi)=0$ for all $w$ is the constant class.
\end{proof}

\begin{proof}[Proof of Proposition~\ref{prop:trianglecount1}]

We begin with a Maslov index computation:

\begin{subclaim} Suppose  $x\in \{x^+,x^-\}$ and $y\in \{y^+,y^-\}$, and $\psi_0\in \pi_2( x, \theta^+, y)$ is a triangle class on the triple $(S^2,\alpha_1,\beta_1,\beta_2)$ shown in  Figure~\ref{fig::50}. We claim
\begin{equation}
\mu(\psi_0)=(m_1+m_2+N_1+N_2)(\psi_0).\label{eq:odd-Maslov-index-formula}
\end{equation}
\end{subclaim}
\begin{proof} Equation~\eqref{eq:odd-Maslov-index-formula} holds for both of the two small triangle classes shaded in Figure~\ref{fig::50}. Furthermore, Equation~\eqref{eq:odd-Maslov-index-formula} respects splicing in doubly periodic domains on $(S^2,\beta_1,\beta_2)$ as well as bigons on $(S^2,\alpha_1,\beta_1)$ and $(S^2,\alpha_1,\beta_2)$, implying the formula in general.
\end{proof}
 
Suppose $\psi\in \pi_2(\xs,\Theta^+_{\b,\b'},\ys)$ is a class on $\cT$, and $\psi_0\in \pi_2(x,\theta^+,y)$ is a class on $(S^2,\alpha_1,\beta_1,\beta_2)$. We form the class $\psi\# \psi_0$ on $\hat{\cT}$ by taking the connected sum of $\psi$ and $\psi_0$, along $c$.  We claim
\begin{equation}\begin{split}\mu(\psi\# \psi_0)&=\mu(\psi)+\mu(\psi_0)-2m_1(\psi_0)\\
&=\mu(\psi)+(m_2-m_1+N_1+N_2)(\psi_0).\end{split}\label{eq:Maslovindexquasitriangles}
\end{equation}
The first equality of Equation~\eqref{eq:Maslovindexquasitriangles} follows from Sarkar's formula for the Maslov index \cite{SarkarMaslov} and the fact that a disk has Euler measure 1. The second equality follows from Equation~\eqref{eq:odd-Maslov-index-formula}.

Let $p\in \Sigma$ and $p_0\in S^2$ denote the points corresponding to $c_\alpha$, which arise after we cut $\Sigma$ along $c_\alpha$, and collapse each of the resulting boundary components. Write 
\[
\Sigma_0=\Sigma\setminus \{p\}\qquad \text{and}\qquad S^2_0=S^2\setminus \{p_0\}.
\]
 Let $\beta_1^r$ and $\beta_2^r$ denote the resulting arcs on $\Sigma_0$. 
 
 We will also be interested in the tube region, for which we write $(S^1\times \R, \beta_1^m,\beta_2^m)$. We write $p_m^l$ and $p_m^r$ for the two punctures of $S^1\times \R$.  See Figure~\ref{fig::92} for a schematic.
 
 If $T_i$ is a sequence of neck-lengths approaching $\infty$, and $J(T_i)$ denotes an almost complex structure on $\Sigma\times \Delta$ with neck-length $T_i$ along $c_\alpha$, then a sequence $u_i$ of $J(T_i)$-holomorphic triangles representing $\psi\# \psi_0$ has a subsequence which converges to three collections, $\cU$, $\cU_m$ and $\cU_0$, where $\cU$ is a broken representative  of $\psi$ on $(\Sigma_0, \as,\bs\cup \{\beta_1^r\},\bs'\cup \{\beta_2^r\})$, $\cU_0$ is a broken representative of $\psi_0$ on $(S^2_0,\alpha_1,\beta_1^l,\beta_2^l)$, and  $\cU_m$ is a collection of holomorphic curves in the tube region $(S^1\times \R,\beta_1^m,\beta_2^m)$. The collections $\cU$, $\cU_0$ and $\cU_m$ may contain both holomorphic curves mapping into $\Sigma\times \Delta$, $S^2\times \Delta$ or $S^1\times \R\times \Delta$, and holomorphic  curves mapping into $\Sigma\times [0,1]\times \R$, $S^2\times [0,1]\times \R$ or $S^1\times \R\times [0,1]\times \R$.

 \begin{figure}[ht!]
 \centering
 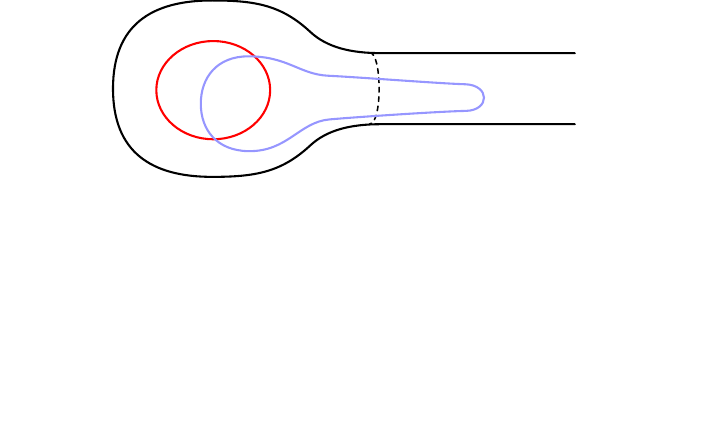
 \caption{\textbf{Decomposing $\hat{\cT}$ along $c_{\alpha}$ and the broken curves $\cU_0$, $\cU_m$ and $\cU$.}}\label{fig::92}
 \end{figure}

We now describe the possible asymptotic behavior of the curves appearing in  $\cU$, $\cU_m$ and $\cU_0$. For definiteness, let us focus on $\cU$. View a neigborhood of the puncture on $\Sigma_0$ as $S^1\times [0,\infty)$. If $u\in \cU$, then the source of $u$ may have punctures along the boundary, at which $u$ asymptotic to an intersection point on the Heegaard diagram. Additionally, a curve $u$ in $\cU$ may have the following types of asymptotics:
\begin{enumerate}[ref= a-\arabic*, label= (a-\arabic*):]
\item\label{num:asymp-1} The source of $u$ may have an interior puncture which is asymptotic to an orbit $S^1\times \{p\}\subset S^1\times \Delta$ for $p\in \Delta$ (multiple covers of such orbits are allowed). Similar asymptotics may occur for curves of $\cU$ mapping into $\Sigma\times [0,1]\times \R$.
\item\label{num:asymp-2} The source $u$ may have a boundary puncture which is asymptotic to a chord  $a\times \{p\}\subset S^1\times \Delta$, where $p$ is a point in one of the three components of $\d \Delta$, and $a$ is a subarc of $S^1$ (thought of as the boundary of $\Sigma$ with a point removed). The arc $a$ connects $\beta_1^r$ to $\beta_1^r$ or $\beta_2^r$ to $\beta_2^r$. The arc $a$ may wind multiple times around $S^1$. Similar asymptotics could also appear in curves in $\cU$ which map into $\Sigma\times [0,1]\times \R$. (These asymptotics are studied in \cite{LOTBordered}.)
\item\label{num:asymp-3}  The source of $u$ may have a boundary puncture which is asymptotic to a chord which connects $\beta_1^r$ and $\beta_2^r$. At this boundary puncture, $\pi_{\Delta}\circ u$ approaches $\infty$ in one of the cylindrical ends of $\Delta$. Also, $\pi_\Sigma\circ u$ is asymptotic to an arc which connects $\beta_1^r$ and $\beta_2^r$ (perhaps winding many times). Similar asymptotics could also appear in curves in $\cU$ which map into $\Sigma\times [0,1]\times \R$.
\end{enumerate}

 Examples of asymptotics~\eqref{num:asymp-2} and \eqref{num:asymp-3} are shown in Figure~\ref{fig::94}. We will rule out asymptotics of type~\eqref{num:asymp-2} and~\eqref{num:asymp-3} from appearing generically in $\cU$, $\cU_m$ or $\cU_0$.
 
 \begin{figure}[ht!]
 \centering
 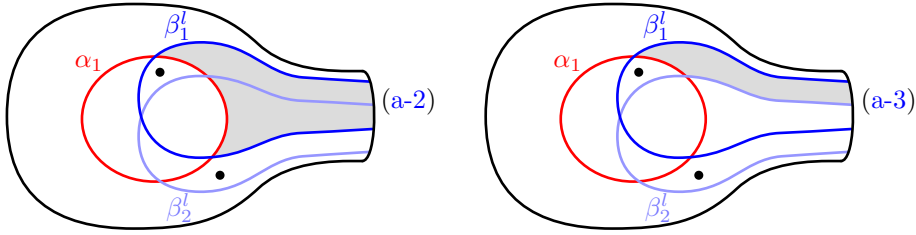
 \caption{\textbf{Domains of troublesome curves which have asymptotics of type~\eqref{num:asymp-2} and \eqref{num:asymp-3}.} We rule out such curves from appearing.}\label{fig::94}
 \end{figure}

We first consider $\cU$, the curves on $\Sigma_0$:
\begin{subclaim}\label{subclaim:U-simple} If $u\in \cU$, and $S_0$ is a connected component of the source of $u$, then $u|_{S_0}$ cannot have boundary on both $\as$ and $\beta_1^r$ or $\beta_2^r$.
\end{subclaim}
\begin{proof}
The claim follows from the maximum modulus principle  (compare the proof of \cite{MOIntSurg}*{Proposition 5.2}). Suppose $u\colon S\to \Sigma_0\times \Delta$ is a holomorphic triangle with boundary on $\as$ and at least one of $\beta_1^r$ or $\beta_2^r$. Let $\bar{S}$ denote the surface obtained by compactifying $S$ with at its boundary punctures. There must be a connected component of $\d \bar{S}$ which is mapped by $\pi_\Sigma \circ u$ to only $\beta_1^r$ and $\beta_2^r$, since there are no intersection points between $\beta_1^r$ or $\beta_2^r$ and $\as$, $\bs$ or $\bs'$. By the maximum modulus principle, $\pi_\Delta\circ u$  must be constant, which is a contradiction. A similar argument applies if $u$ is a holomorphic strip on one of the subdiagrams $(\Sigma_0,\as,\bs\cup \{\beta_1^r\})$ or $(\Sigma_0,\as,\bs'\cup \{\beta_2^r\})$.
\end{proof}

It follows from Subclaim~\ref{subclaim:U-simple} that $\cU$ may be arranged into the following collections:
\begin{enumerate}
\item A broken holomorphic triangle $u_\Sigma$ on $(\Sigma, \as,\bs,\bs')$ representing a class  $\psi_\Sigma\in \pi_2(\ve{x},\Theta',\ve{y})$.
\item A broken holomorphic strip $u_{\beta,\beta'}$ on $(\Sigma, \bs\cup\{\beta_1^r\},\bs'\cup \{\beta_2^r\})$ representing a  class $\phi_{\beta,\beta'}$ in $\pi_2(\Theta_{\beta,\beta'}^+\times p,\Theta'\times p)$.
\end{enumerate}

\begin{subclaim}\label{subclaim:cU=usigma} The class $\phi_{\beta,\beta'}$ must be the constant class, and $\cU$ consists only of the curve $u_\Sigma$ (and potentially some ghost curves). The class $\psi_\Sigma$ is equal to $\psi$, and has Maslov index 0. The asymptotics of $u_\Sigma$ at $p$ consist only of orbits of type~\eqref{num:asymp-1}.  After perturbing the almost complex structure, the orbit asymptotics of $u_\Sigma$ project to $m_1(\psi)$ distinct points in the interior of $\Delta$. After completing over the puncture $p$, the curve $u_{\Sigma}$ satisfies \eqref{def:M1}--\eqref{def:M6}.
\end{subclaim}
\begin{proof}
The proof is essentially combinatorial, and is based on obtaining convenient index formulas.

Let $\phi_{\beta,\beta'}^0$ denote the homology class in $\pi_2(\Theta_{\beta,\beta'}^+,\Theta')$ induced by $\phi_{\beta,\beta'}$ on $(\Sigma,\bs,\bs')$, obtained by removing a disk containing the $\beta_1^r$ and $\beta_2^r$ curves, and collapsing the resulting boundary component to a point.

Note that $\psi=\psi_\Sigma+\phi_{\beta,\beta'}^0$, since all multiplicities are represented. Hence
\begin{equation}
\mu(\psi)=\mu(\psi_\Sigma)+\mu(\phi_{\beta,\beta'}^0).\label{eq:additive-index-weird-triangles}
\end{equation}
We claim
\begin{equation}
\begin{split}
(m_2-m_1)(\psi_0)&=(m_2-m_1)(\psi_\Sigma)+(m_2-m_1)(\phi_{\beta,\beta'})\\
&=(m_2-m_1)(\phi_{\beta,\beta'}).
\end{split}
\label{eq:diff-in-multiplicity}
\end{equation}
The first equality of Equation~\eqref{eq:diff-in-multiplicity} follows since $\psi_0$ represents the entire homology class of the limiting curves on $(S^2_0,\alpha_1,\beta_1^l,\beta_2^l)$, while $\psi_\Sigma$ and $\phi_{\beta,\beta'}$ represent the entire homology class of the limiting curves on $(\Sigma_0,\as,\bs\cup \{\beta_1^r\}, \bs'\cup \{\beta_2^r\})$. The second equality of Equation~\eqref{eq:diff-in-multiplicity} follows from the fact that $m_2(\psi_\Sigma)=m_1(\psi_\Sigma)$, since $u_{\Sigma}$ has no boundary mapping to $\beta_1$ or $\beta_2$.

Combining Equations~\eqref{eq:Maslovindexquasitriangles}, ~\eqref{eq:additive-index-weird-triangles} and~\eqref{eq:diff-in-multiplicity}, we obtain
\begin{equation}
\begin{split}
\mu(\psi\# \psi_0)&= \mu(\psi)+(m_2-m_1+N_1+N_2)(\psi_0)\\
&=\mu(\psi_\Sigma)+\mu(\phi_{\beta,\beta'}^0)+(m_2-m_1+N_1+N_2)(\psi_0)\\
&=\mu(\psi_\Sigma)+\left(\mu(\phi_{\beta,\beta'}^0)-m_1(\phi_{\beta,\beta'})\right)+m_2(\phi_{\beta,\beta'})+(N_1+N_2)(\psi_0).
\end{split}
\label{eq:Maslov-index-sum-nonnegative-terms}
\end{equation}

By construction of the class $\phi_{\beta,\beta'}^0$, we have
 \begin{equation}
 m_1(\phi_{\beta,\beta'})=m_{1}(\phi_{\beta,\beta'}^0). \label{eq:m1phibb=m1phibb0}
 \end{equation}
Equation~\eqref{eq:m1phibb=m1phibb0} and  Lemma~\ref{lem:smallisotopiesMaslov} imply
  \[
 \mu(\phi_{\beta,\beta'}^0)-m_1(\phi_{\beta,\beta'})= \mu(\phi_{\beta,\beta'}^0)-m_1(\phi_{\beta,\beta'}^0)\ge 0.
  \]
   Consequently, Equation~\eqref{eq:Maslov-index-sum-nonnegative-terms} implies
\begin{equation}
\begin{split}
0&=\mu(\psi_\Sigma)\\
&=\mu(\phi_{\beta,\beta'}^0)-m_1(\phi_{\beta,\beta'})\\
&=m_2(\phi_{\beta,\beta'})\\
&=N_1(\psi_0)\\
&=N_2(\psi_0).
\end{split}
\label{eq:all-terms-zero}
\end{equation}

In particular, Equation~\eqref{eq:all-terms-zero} implies that
\[
\mu(\phi_{\beta,\beta'}^0)-m_1(\phi_{\beta,\beta'}^0)=\mu(\phi_{\beta,\beta'}^0)-m_1(\phi_{\beta,\beta'})=0,
\]
so Lemma~\ref{lem:smallisotopiesMaslov} implies that $\phi_{\beta,\beta'}^0$ must be a constant class.

Next, if $n\in \{n_1,n_2,m_1,m_2\}$, we note that
\[
n(\phi_{\beta,\beta'})+m_1(\psi_\Sigma)=n(\psi_0),
\]
since $\phi_{\beta,\beta'}$, $\psi_0$ and $\psi_{\Sigma}$, when spliced together, represent the entire class $\psi\# \psi_0$. Since $u_\Sigma$ has no boundary components on $\beta_1^r$ or $\beta_2^r$, it follows that
\[
n_1(\psi_\Sigma)=n_2(\psi_\Sigma)=m_1(\psi_\Sigma)=m_2(\psi_\Sigma).
\] 
Consequently, the class $\phi_{\beta,\beta'}$ satisfies the same vertex relation at $\theta^-$ that the class  $\psi_0$ satisfies:
 \begin{equation}
 (n_1+n_2)(\phi_{\beta,\beta'})=(m_1+m_2)(\phi_{\beta,\beta'}).\label{eq:vertex-multiplicities-phi-b-b'}
 \end{equation}
Since $\phi_{\beta,\beta'}^0$ is the constant class, we know  $m_1(\phi_{\beta,\beta'}^0)=0$. Equation~\eqref{eq:m1phibb=m1phibb0} implies that $m_1(\phi_{\beta,\beta'})=0$. From Equation~\eqref{eq:all-terms-zero} we know that $m_2(\phi_{\beta,\beta'})=0$. Combined with Equation~\eqref{eq:vertex-multiplicities-phi-b-b'}, we conclude
\[
n_1(\phi_{\beta,\beta'})=n_2(\phi_{\beta,\beta'})=m_1(\phi_{\beta,\beta'})=m_2(\phi_{\beta,\beta'})=0,
\]
so $\phi_{\beta,\beta'}$ is a constant class, as well.

Since $\phi_{\beta,\beta'}$ is the constant class, it follows that $\psi=\psi_{\Sigma}$.

Since $u_\Sigma$ has no boundary mapping to $\beta_1^r$ or $\beta_2^r$, it can only be asymptotic at $p$ to orbits of type~\eqref{num:asymp-1}. From Equation~\eqref{eq:all-terms-zero} we know that $\mu(\psi)=0$. Hence, for generic almost complex structure, after completing over $p$, $u_\Sigma$ will satisfy \eqref{def:M1}--\eqref{def:M6}.

Finally, by perturbing the almost complex structure slightly, we can ensure that the asymptotics of $u_\Sigma$ at $p$ consist of $m_1(\psi)$ once-covered orbits of type~\eqref{num:asymp-1}, which each project to a different point of $\Delta$. Such a perturbation can be done concretely, by perturbing the placement of the connected sum point so that it is disjoint from the image of the branch set of $\pi_{\Sigma}\circ v$, for any Maslov index 0 holomorphic triangle $v$ (a codimension 2 subset of $\Sigma$).
\end{proof}

We now analyze the curves in $\cU_0$ and $\cU_m$:

\begin{subclaim}\label{subclaim:Um}
The curves in $\cU_0$ can only be asymptotic at the puncture $p_0$ to an orbit of type~\eqref{num:asymp-1}. The curves in $\cU_m$ consist only of once-covered cylinders of the form $S^1\times \R\times \{d\}$, where $d$ is an interior point of $\Delta$, as well as possibly some ghost curves.
\end{subclaim}
\begin{proof}
 A version of compactness in symplectic field theory is described in \cite{BEHWZ}, where a linear ordering of levels appears. Our present situation is more similar to the compactification via holomorphic combs which appears in bordered Floer homology \cite{LOTBordered}*{Section~5.4}, since holomorphic curves can degenerate into the three cylindrical ends of $\Delta$, and also into the tube region which results from cutting along $c_\alpha$. Hence, the limiting curves may be arranged into a 2-component level structure. We refer to one component as the \emph{$\Delta$-level}, and the other is the \emph{$\Sigma$-level} (compare \cite{LOTBordered}*{Definition~5.20}). Furthermore, there is a single $\Delta$-level in each $\Sigma$-level which consists of curves mapping into $S^2_0\times \Delta$, $S^1\times \R\times \Delta$ or $\Sigma_0\times \Delta$. We refer to this level as the \emph{central} $\Delta$-level. All other curves map into one of the cylindrical 4-manifolds appearing in the ends.
 
 We begin at the central $\Delta$-level of $\cU$, which consists only of the curve $u_\Sigma$ by Subclaim~\ref{subclaim:cU=usigma} (and possibly some ghost curves, which we will later rule out). We proved that the asymptotics of $u_\Sigma$ consisted of $m_1(\psi)$  orbits of type~\eqref{num:asymp-1}, each of which projects to a distinct point in the interior of $\Delta$.

 There is a story $\cV_1$ of $\cU_m$ which matches $u_\Sigma$. Since $\cV_1$ matches $m_1(\psi)$ once-covered orbits of type~\eqref{num:asymp-1}, which each project to an interior point of $\Delta$, $\cV_1$ must also be in the central $\Delta$-level. Since there are no $\as$ curves on $(S^1\times \R,\beta_1^m,\beta_2^m)$, any holomorphic curve with connected source which has such an orbit at $p_m^r$ must project constantly to $\Delta$. Consequently, such a curve must also be asymptotic to a once-covered orbit at $p_m^l$. According to Subclaim~\ref{subclaim:cU=usigma}, we have 
 \[
n_1(\psi_0)=n_2(\psi_0)=m_1(\psi_0)=m_2(\psi_0)=m_1(\psi).
 \]
 Since $\cV_1$ must contain $m_1(\psi)$ once-covered cylinders, which each project constantly to $\Delta$, there can be no other holomorphic curves with non-constant image, since they would raise the multiplicity in some region of $\Sigma$ too high.
 
 We now consider the level $\cV_2$ of $\cU_m$ which matches the asymptotics of $\cV_1$. The previous argument implies that $\cV_2$ consists only of $m_1(\psi)$ once-covered cylinders, and some ghost curves. We continue in this manner until we reach the central $\Delta$-level of $\cU_0$. Since adjacent levels must have matching asymptotics, we conclude that the central $\Delta$-level of $\cU_0$ has asymptotics at $p_0$ consisting only of $m_1(\psi_0)$ once-covered orbits. Furthermore, the asymptotics match those of $u_\Sigma$.
\end{proof}

\begin{subclaim}\label{subclaim:U0} The collection $\cU_0$ consists of a single holomorphic triangle $u_0$, satisfying \eqref{def:M1}--\eqref{def:M6}, as well as possibly some ghost curves. Furthermore
\[
\rho^p(u_\Sigma)=\rho^{p_0}(u_0).
\]
\end{subclaim}
\begin{proof}
Subclaim~\ref{subclaim:U0} constrains the asymptotics at $p_0$ of the curves in $\cU_0$ to satisfy \eqref{num:asymp-1}. The Maslov index of $\psi_0$ is $2m_1(\psi_0)$, since $\mu(\psi)=0$ by the proof of Subclaim~\ref{subclaim:cU=usigma}. The set $X(\psi)\subset \Sym^{m_1(\psi_0)}(\Delta)$ defined by
\[
X(\psi):=\{\rho^{p}(u): u\in \cM(\psi)\}
\]
is finite. Consequently, as in the proof of Proposition~\ref{prop:free-stabdifferential}, dimension counting using Proposition~\ref{prop:transversality} implies that $\cU_0$ consists of a single holomorphic triangle $u_0$ which satisfies \eqref{def:M1}--\eqref{def:M6}. 
\end{proof}

Having constrained the curves of $\cU$, $\cU_m$ and $\cU_0$ in Subclaims~\ref{subclaim:cU=usigma}, \ref{subclaim:Um} and \ref{subclaim:U0},  the index argument in the proof of Proposition~\ref{prop:free-stabdifferential} applies to show that generically no ghost curves appear.

The  class $\psi_0$ is completely determined: since $N_1(\psi_0)=N_2(\psi_0)=0$ by Equation~\eqref{eq:all-terms-zero}, it follows that $\psi_0$ has domain consisting of one of the two shaded small triangles in Figure~\eqref{fig::50} together with $k$ copies of the component of $S^2\setminus \alpha_0$ which has non-zero multiplicity on the region marked $m_1$. In particular, this implies that $\psi_0\in \pi_2(x^+,\theta^+,y^+)$ or $\psi_0\in \pi_2(x^-,\theta^+,y^-)$, and hence the off-diagonal entries of the map $F_{\hat{\cT},J(T)}(\ve{x}\times -, \Theta_{\b,\b'}^+\times \theta^+)$ are zero. Let us write
\[
\psi_0^{k,+}\in \pi_2(x^+,\theta^+,y^+)\qquad \text{and} \qquad \psi_0^{k,-}\in \pi_2(x^-,\theta^+,y^+),
\]
for these two triangle classes.

Finally, it remains to count representatives of the class $\psi\# \psi_0$, when $\mu(\psi)=0$ and $\psi_0\in \left\{\psi_0^{m_1(\psi),+}, \psi_0^{m_1(\psi),-}\right\}$. If $u_\Sigma$ denotes the holomorphic triangle in $\cU$, and $u_0$ the triangle in $\cU_0$, then we can view the pair $(u_\Sigma,u_0)$ as a point in the compactification of the space
\[
\bigcup_{T>0} \cM_{J(T)}(\psi\# \psi_0).
\]
Gluing gives an identification of a neighborhood of the set of such pairs $(u_\Sigma,u_0)$ with the Cartesian product
\begin{equation}
\left\{(u_\Sigma,u_0)\in \cM(\psi)\times\cM^{p_0}(\psi_0): \rho^{p}(u_\Sigma)=\rho^{p_0}(u_0)\right\}\times [0,1), \label{eq:fibered-prod-0}
\end{equation}
where  $\cM^{p_0}(\psi_0)$ denotes set of holomorphic curves $u_0\colon S_0\to S^2_0\times [0,1]\times \R$, representing the class $\psi_0$ on $(S^2_0,\alpha_1,\beta_1^l,\beta_2^l)$, such that $(\pi_{S^2}\circ u_0)^{-1}(p_0)\cap \d S=\emptyset$ (i.e. the curves which only have asymptotics of type~\eqref{num:asymp-1} at $p_0$, and none of type~\eqref{num:asymp-2} or \eqref{num:asymp-3}).

Since  $\mu(\psi)=0$, it suffices to count the number of elements in the matched moduli space $\cM^{p_0}(\psi_0,\ve{d})$, where $\ve{d}\in \Sym^{m_1(\psi_0)}(\Delta)$ is a generic point. The strategy used by Ozsv\'{a}th and Szab\'{o} \cite{OSLinks}*{Lemma~6.4} to prove Equation~\eqref{eq:OS'scountmatched} now readily adapts to our present situation. Pick a path of points $\ve{d}_t\colon [0,\infty)\to \Sym^{m_1(\psi_0)}(\Delta)$, disjoint from the fat diagonal, consisting of $m_1(\psi_0)$ points which all travel into the $\alpha_1$-$\beta_1$ cylindrical end of $\Delta$, such that the points of $\ve{d}_t$ are spaced at least $t$ distance apart (with respect to a metric obtained by embedding $\Delta$ conformally in the complex plane, so that each cylindrical end is identified with $[0,1]\times [0,\infty)$).

\begin{subclaim}
 The only ends of the 1-dimensional moduli space $\bigcup_{t\in [0,\infty)}\cM^{p_0}(\psi_0,\ve{d}_t)$ at finite $t$ correspond to $\cM^{p_0}(\psi_0,\ve{d}_0)$.
 \end{subclaim} 
\begin{proof} 
	The proof is similar to our proof of Equation~\eqref{eq:divisorcounttriangles}. Degenerations may be analyzed by considering possible curves and arcs collapsing in the source curve. Such degenerations may be broadly classified into one of the following phenomena: holomorphic strips breaking off into one of the three cylindrical ends, boundary degenerations bubbling off, curves breaking off toward the puncture $p_0$, or the source curve becoming nodal.
	
	One example of a curve breaking off to $p_0$ would be a slit along $\beta_1^l$ or $\beta_2^l$ traveling out towards $p_0$. 

Strip breaking is prohibited as follows. Suppose $v\colon S\to S^2\times [0,1]\times \R$ is a hypothetical strip appearing in degeneration in $\cM^{p_0}(\psi_0,\ve{d}_t)$ at some finite $t$. Since $N_1(v)=N_2(v)=0$, the domain of $v$ must have non-zero multiplicity in one of the four regions adjacent to $p_0$. Consequently any holomorphic triangle also arising in the degeneration could not match any $\ve{d}_t$, since $\ve{d}_t$ is bounded away from the three components of $\d\Delta$.

 Boundary degenerations are prohibited similarly, since the only possible class of the resulting boundary degeneration would have non-zero multiplicity around $p_0$, prohibiting the triangular component in the limit from matching $\ve{d}_t$.

Curves breaking off in the direction of $p_0$ are also prohibited, since the resulting curve mapping into $S^1\times \R\times \Delta$ must match $\ve{d}_t$ at the puncture $p_m^r$. As argued in the proof of Subclaim~\ref{subclaim:cU=usigma}, this constrains any curves appearing in $S^1\times \R\times \Delta$ to consist only of cylinders, and possibly ghost curves. Ghost curves are prohibited from dimension counts, as in the proof of Proposition~\ref{prop:free-stabdifferential}.

Finally, the formation of nodal singularities is prohibited similarly to the proof of Equation~\eqref{eq:divisorcounttriangles}. Double points appearing on the interior of the source are prohibited by dimension counts. Boundary double points are prohibited since they result in the formation of a boundary degeneration, which we have already prohibited.

\end{proof}

Consequently, the limiting curves which appear in the ends of $\cM^{p_0}(\psi_0,\ve{d}_t)$ as $t\to \infty$ take the following form:
\begin{enumerate}
\item A single index 0 holomorphic triangle (with domain equal  to one of the two shaded regions in Figure~\ref{fig::50}), together with
\item\label{enum:curve-type-2}  $m_1(\psi_0)$ Maslov index 2 holomorphic curves with domain $A$ equal to the component of $S^2\setminus \alpha_0$ with $m_1(A)=1$. Each curve matches a single point $d\in [0,1]\times \R$.
\end{enumerate}
If $x\in \{x^+,x^-\}$,  write $A_x$ for the Maslov index 2 homology class in $\pi_2(x,x)$ corresponding to the curves of type-\eqref{enum:curve-type-2}. It remains to count the matched moduli space $\cM^{p_0}(A_x,d)$, for a generic point $d\in [0,1]\times \R$. This is achieved by letting $d$ approach $\{1\}\times \R$ (the $\alpha_1$-boundary of $[0,1]\times \R$). The limiting curve then consists of a single $\alpha_1$ boundary degeneration, as well as a representative of the constant class, $e_x$. There are no other ends. 

By \cite{OSLinks}*{Theorem~5.5}, the count of $\alpha_1$ boundary degenerations in the class $A_x$ is $1$, modulo 2. By the Riemann mapping theorem, each of the shaded triangles in Figure~\ref{fig::50} has a unique representative. Consequently, gluing together these curves, we conclude that $\cM^{p_0}(\psi_0,\ve{d})$ has a single element, modulo 2, for any generic $\ve{d}\in \Sym^{m_1(\psi_0)}(\Delta)$.

Using Equation~\eqref{eq:fibered-prod-0}, we conclude that for large $T$
\[
\# \cM_J(\psi)=\# \cM_{J(T)}(\psi\# \psi_0).
\]

The class $\psi\# \psi_0$ has multiplicity $m_1(\psi)$ over $w$, and zero multiplicity on $w'$. Hence any representative of $\psi\# \psi_0$ is counted with a factor of $U_{w'}^{n_z(\psi)}$, and no factor of $U_w$, completing the proof
\end{proof}

\begin{figure}[ht!]
\centering
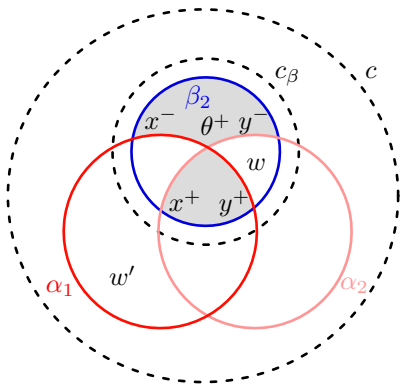
\caption{\textbf{The Heegaard triple $\hat{\cT}$ in Proposition \ref{prop:trianglecount2}.} The shaded regions are each examples of small triangles which might be counted.}\label{fig::51}
\end{figure}

\begin{prop}\label{prop:trianglecount2}
Suppose that $\cT=(\Sigma, \as',\as,\bs',\ve{w}_0\cup \{z\})$ is a Heegaard triple with a distinguished basepoint $z$, and $\as'$ are small Hamiltonian isotopies of $\as$, satisfying $|\alpha_i'\cap \alpha_j|=2\delta_{ij}$.  Let $\hat{\cT}=(\Sigma, \as'\cup \{\alpha_2\},\as\cup \{\alpha_1\},\bs'\cup \{\beta_2\},\ve{w}_0\cup \{w,w'\})$  be the Heegaard triple shown in Figure \ref{fig::51}. Write $J$ for an almost complex structure on $\Sigma\times \Delta$, for $\cT$, and write $J(T)$ for an almost complex structure for $\hat{\cT}$, which has had a neck of length $T$ inserted along $c_{\beta}$. For sufficiently large $T$,
\[
F_{\hat{\cT},J(T)}( \Theta_{\alpha',\alpha}^+\times \theta^+,\ve{x}\times -)=\begin{pmatrix} F_{\cT,J}(\Theta_{\alpha',\alpha}^+,\ve{x})^{U_z\to U_{w'}}&0\\
0& F_{\cT}(\Theta_{\alpha',\alpha}^+,\ve{x})^{U_z\to U_{w'}}.
\end{pmatrix}
\]
\end{prop}
\begin{proof}The proof is identical to the proof of Proposition~\ref{prop:trianglecount1}.
\end{proof}

\begin{lem}\label{lem:picknicecomplexstructure}Almost complex structures $J_\alpha$ and $J_\beta$ can be chosen on $\cH_1$, $\cH_{1.5}$ and $\cH_{2}$ so that the following hold:
\begin{enumerate}
\item Lemma~\ref{lem:differentialcompdegen} applies to compute the differentials on $\cH_{1.5}$. Lemma~\ref{lem:differentialonH_1H_2} applies to compute the differential on $\cH_{1}$ and $\cH_2$.
\item Proposition~\ref{prop:trianglecount1} applies to compute the transition map $\Psi_{(\cH_1,J_\alpha)\to (\cH_{1.5},J_{\alpha})}$ and Proposition~\ref{prop:trianglecount2} applies to compute $\Psi_{(\cH_{1.5}, J_{\beta})\to (\cH_{2}, J_{\beta})}$.
\item Lemma~\ref{lem:changealmostcomplexstructure} applies to compute  $\Psi_{(\cH_{1.5}, J_{\alpha})\to (\cH_{1.5},J_{\beta})}$.
\item $J_\alpha$ can be used to compute  $S_{w'}^+$ on $\cH_1$, and $J_{\beta}$ can be used to compute $S_{w}^-$ on  $\cH_2$.
\end{enumerate}
\end{lem}

\begin{proof}All of the results follow from the aforementioned results, together with Lemma~\ref{lem:relative-lengths-unimportant}, which shows that the transition map for changing between two choices of $J_{\alpha}$ with different relative neck lengths is the identity map, on the level of chain complexes.
\end{proof}

We now  prove Theorem~\ref{thm:transitionmapcomp}:
\begin{proof}[Proof of Theorem \ref{thm:transitionmapcomp}] Let $J_1$ and $J_2$ be the almost complex structures $J_\alpha$ and $J_\beta$ on $\cH_1$ and $\cH_2$, chosen so that the statements of  Lemma~\ref{lem:picknicecomplexstructure} hold.

 Decompose $\Psi_{(\cH_1,J_\alpha)\to (\cH_2,J_\beta)}$ as
\begin{equation}
\Psi_{(\cH_1,J_\alpha)\to (\cH_2,J_\beta)}=\Psi_{(\cH_{1.5},J_\beta)\to (\cH_2,J_\beta)}\circ \Psi_{(\cH_{1.5},J_\alpha)\to (\cH_{1.5}, J_\beta)} \circ \Psi_{(\cH_1,J_\alpha)\to (\cH_{1.5},J_\alpha)}.\label{eq:decompose-transitionH1H2}
\end{equation}
 Using Lemma~\ref{lem:picknicecomplexstructure}, Equation~\eqref{eq:decompose-transitionH1H2} implies that $\Psi_{(\cH_1,J_\alpha)\to (\cH_2,J_\beta)}$ is equal to, in matrix notation
\[
\begin{pmatrix}(\Psi_{\as\to \as'}^{\bs'})^{U_z\to U_{w'}}&0\\
0& (\Psi_{\as\to \as'}^{\bs'})^{U_z\to U_{w'}}
\end{pmatrix} \begin{pmatrix}\id &0\\
*& \id
\end{pmatrix} \begin{pmatrix}(\Psi_{\as}^{\bs\to \bs'})^{U_z\to U_{w}}&0\\
0& (\Psi_{\as}^{\bs\to \bs'})^{U_z\to U_{w}}
\end{pmatrix}
\]
\[
=\begin{pmatrix}(\Psi_{\as\to \as'}^{\bs'})^{U_z\to U_{w'}}\circ (\Psi_{\as}^{\bs\to \bs'})^{U_z\to U_{w}}& 0\\
*&(\Psi_{\as\to \as'}^{\bs'})^{U_z\to U_{w'}}\circ  (\Psi_{\as}^{\bs\to \bs'})^{U_z\to U_{w}}
\end{pmatrix},
\]
 as claimed. Furthermore, according to Lemma~\ref{lem:picknicecomplexstructure}, the almost complex structure $J_\alpha$ can be used to compute $S^{+}_{w'}$ on $\cH_1$, and  $J_\beta$ can be used to compute $S_{w}^{-}$ on $\cH_2$.
\end{proof}

\subsection{Basepoint moving maps and the $\pi_1$-action}
In this section, we prove our proposed formulas for the basepoint moving diffeomorphism maps.

We first need a computation of the relative homology map for the diagrams $\cH_1$ and $\cH_2$ considered in the previous section. In the diagrams $\cH_1$ and $\cH_{2}$, let $\lambda$ be the path shown in Figure \ref{fig::86}.

\begin{figure}[ht!]
\centering
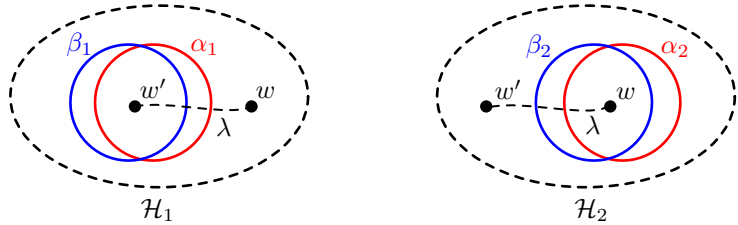
\caption{\textbf{The path $\lambda$ in the diagrams $\cH_1$ and $\cH_2$.}}\label{fig::86}
\end{figure}

\begin{lem}\label{lem:A_lambdacomputation} On $(\cH_{1},J_\alpha)$, the relative homology map $A_\lambda$ takes the form
\[
(A_\lambda)_{\cH_1,J_\alpha}=\begin{pmatrix}U_w\cdot(\Phi_z)^{U_z\to U_w}& U_{w}\\
\id& U_w\cdot(\Phi_z)^{U_z\to U_w}
\end{pmatrix},
\]
where $\Phi_z$ is the endomorphism of $\CF^-(\Sigma, \as,\bs,\ve{w}_0\cup \{z\})$
\[
\Phi_z(\ve{x})=\sum_{\ve{y}\in \bT_\alpha\cap \bT_\beta} \sum_{\substack{\phi\in \pi_2(\ve{x},\ve{y})\\ \mu(\phi)=1}}n_z(\phi) \# \hat{\cM}(\phi) U_{\ve{w}_0}^{n_{\ve{w}_0}(\phi)}U_z^{n_{z}(\phi)-1}\cdot \ve{y}.
\]

On $(\cH_2,J_\beta)$, the map $A_\lambda$ takes the form
\[
(A_\lambda)_{\cH_2,J_\beta}=\begin{pmatrix} 0& U_{w}\\
\id& 0\end{pmatrix}.
\]
\end{lem}

\begin{proof}The statement follows from the holomorphic disk counts in Lemma~\ref{lem:differentialonH_1H_2}, as we now explain.

We first consider the diagonal entries. The curves which contribute to the diagonal entries of $(A)_{\cH_1,J_\alpha}$ represent classes of the form $\phi\# \phi_0$ in $ \pi_2(\xs\times \theta,\ys\times \theta)$ for $\theta\in \{\theta^+,\theta^-\}$. Lemma~\ref{lem:differentialonH_1H_2} implies that these have no change across $\beta_1$ curve, and have a change of $n_{z}(\phi)$ across $\alpha_1$. The curves are counted by $(A_\lambda)_{\cH_1,J_\alpha}$ with an additional factor of $a(\lambda,\phi\# \phi_0),$ which is thus $n_z(\phi)$. It follows that the diagonal entries are equal to $U_w\cdot (\Phi_z)^{U_z\to U_{w}}$.  We leave it to the reader to verify using Lemma~\ref{lem:differentialonH_1H_2} that the off diagonal entries are as claimed.

An entirely analogous argument works to compute $(A_\lambda)_{\cH_1,J_\beta}$.
\end{proof}

\begin{thm}\label{thm:BM=graphact} Suppose  $\lambda$ is a path from $w$ to $w'$, and $\ws_0\subset Y$ is a (possibly empty) collection of basepoints which contains neither $w$ nor $w'$. 
The diffeomorphism map
\[
\lambda_*\colon \CF^-(Y,(\ws_0\cup \{w\})^\sigma, \frs)\to \CF^-(Y,(\ws_0\cup \{w'\})^{\sigma'}, \frs)
\]
satisfies
\[
\lambda_*\simeq S^-_w A_\lambda S^+_{w'},
\]
when $\sigma'$ is the coloring obtained by pushing forward $\sigma$ under $\lambda$.
\end{thm}

\begin{proof}Assume that a diagram $(\Sigma, \as,\bs,\ve{w}_0\cup \{w\})$ is chosen so that $\lambda$ is embedded in $\Sigma$ and is disjoint from $\as$ and $\bs$. Let $\phi_t\colon [0,1]\times \Sigma\to \Sigma$ be an isotopy of $\Sigma$, supported in a small neighborhood of $\lambda$, such that $\phi_0=\id_\Sigma$ and $\phi_1(w)=w'$. Further, assume that $\phi_t$ is the identity outside a small neighborhood of $\lambda$. If $J$ is an almost complex structure on $\Sigma\times [0,1]\times \R$, then pushing forward under $\phi_1\times \id$ yields a tautological map
\[
T\colon  \CF^-_{J}(\Sigma, \as,\bs,\ve{w}_0\cup\{w\})\to  \CF^-_{(\phi_1)*J}(\Sigma, \as,\bs,\ve{w}_0\cup\{w'\}).
\]
 By definition,
\[
\lambda_*\colon \CF^-_{J}(\Sigma, \as,\bs,\ve{w}_0\cup\{w\})\to \CF^-_{J}(\Sigma, \as,\bs,\ve{w}_0\cup\{w'\})
\]
is the composition
\begin{equation}
\lambda_*=\Psi_{(\phi_1)_* J\to J}\circ T.\label{eq:trivial-push-forward}
\end{equation}

As a first step, we show that the expression for $\lambda_*$ in Equation~\eqref{eq:trivial-push-forward} is equal to the tautological map on intersection points. The map $\Psi_{(\phi_1)_* J\to J}$ is obtained by counting Maslov index zero holomorphic disks for a dynamic almost complex structure which interpolates $(\phi_1)_* J$ and $J$. The isotopy $\phi_t$ induces an automorphism $\Phi$ of $\Sigma\times [0,1]\times \R$:
\[
\Phi(x,s,t)=(\phi_t(x),s,t).
\]
 We push $J$ forward along $\Phi$ to get a dynamic almost complex structure, interpolating $J$ and $(\phi_1)_* J$. However,  $\Phi_*(J)$-holomorphic disks on $\Sigma\times [0,1]\times \R$ pull back under $\Phi$ to $J$-holomorphic disks on $\Sigma\times [0,1]\times \R$. (The assumption that $\phi_t$ is fixed near $\as$ and $\bs$ is necessary since it implies that the pullback of a $\Phi_*(J)$-holomorphic disk has boundary mapping to the cylinders $\as\times \{0\}\times \R$ and $\bs\times \{0\}\times \R$). By transversality for $J$, a Maslov index zero $\Phi_*(J)$-holomorphic disk must be a constant disk. Hence, $\Psi_{(\phi_1)_* J\to J}$ is the tautological map on intersection points.

Hence, $\lambda_*$, being the composition of two tautological maps, is a tautological map.

We now consider the composition $S_w^- A_\lambda S_{w'}^+$. Let $\cH_1$ and $\cH_2$ be diagrams for $(Y,\ws_0\cup \{w,w'\})$ which are free-stabilizations of $(\Sigma,\as,\bs,\ws_0\cup \{w\})$ or $(\Sigma,\as,\bs,\ws_0\cup \{w'\})$, respectively, as in Figure~\ref{fig::86}. Similarly to Figure~\ref{fig::48}, let $J_\alpha$ and $J_\beta$ be almost complex structures on $\cH_1$ and $\cH_2$, obtained by stretching parallel to $\alpha_1$ and $\beta_2$, respectively.

 The map
$S_{w}^- A_\lambda S_{w'}^+$, viewed as a map from $\CF^-_{J}(\Sigma,\as,\bs,\ws_0\cup \{w\})$ to $\CF^-_{J}(\Sigma,\as,\bs,\ws_0\cup \{w'\})$, can be computed as the composition
\begin{equation}
S_{w}^- A_\lambda S_{w'}^+=\Psi_{\as'\to \as}^{\bs'\to \bs}\circ S_{w}^- \circ(A_\lambda)_{\cH_2,J_\beta}\circ \Psi_{(\cH_1,J_\alpha)\to (\cH_2, J_\beta)} \circ S_{w'}^+.\label{eq:basepoint-moving-penultimate-form}
\end{equation}

Using matrix notation, we can write
\[
S_{w'}^+=\begin{pmatrix} \id\\ 0\end{pmatrix} \qquad \text{and}\qquad S_{w}^-=\begin{pmatrix} 0& \id\end{pmatrix}.
\]
Using matrix notation, we rewrite Equation~\eqref{eq:basepoint-moving-penultimate-form} using  Theorem~\ref{thm:transitionmapcomp}, Lemma~\ref{lem:changealmostcomplexstructure}, and Lemma~\ref{lem:A_lambdacomputation} as
\begin{equation}
 S_{w}^- A_\lambda S_{w'}^+\simeq \Psi_{\as'\to \as}^{\bs'\to \bs} \begin{pmatrix} 0& \id\end{pmatrix}\begin{pmatrix} 0& U_{w}\\
\id& 0\end{pmatrix} \begin{pmatrix}\Psi_{\as\to \as'}^{\bs'}\circ  \Psi_{\as}^{\bs\to \bs'}& 0\\
*&\Psi_{\as\to \as'}^{\bs'} \circ \Psi_{\as}^{\bs\to \bs'}
\end{pmatrix}\begin{pmatrix}\id\\0
\end{pmatrix}.
\label{eq:matrix-expansion-basepoint-moving}
\end{equation}

Matrix multiplication reduces Equation~\eqref{eq:matrix-expansion-basepoint-moving} to  $\Psi_{\as'\to \as}^{\bs'\to \bs}\circ \Psi_{\as\to \as'}^{\bs'}\circ \Psi_{\as}^{\bs\to \bs'}$, which is chain homotopic to the tautological map by naturality. Since we already showed that $\lambda_*$ is the tautological map, the proof is complete.
\end{proof}

\subsection{The $\Phi_w$ map}

If $(\Sigma,\as,\bs,\ws)$ is a Heegaard diagram and  $w\in \ve{w}$, then we write $\Phi_w$ for the endomorphism
\begin{equation}
\Phi_w(\ve{x})=U_w^{-1}\sum_{\substack{\phi\in \pi_2(\ve{x},\ve{y})\\ \mu(\phi)=1}}n_w(\phi)\# \hat{\cM}(\phi) U_{\ve{w}}^{n_{\ve{w}}(\phi)}\cdot \ve{y}, \label{eq:Phi-map-def}
\end{equation}
of $\CF^-(\Sigma,\as,\bs,\ws,\frs)$.

\begin{rem}\label{rem:Phi=formal-derivative} The map $\Phi_w$ can be conveniently described algebraically as the \emph{formal derivative} of the differential. Viewing $\CF^-(\cH,\frs)$ as a free $\bF_2[U_{\ws}]$-module with basis $\bT_{\a}\cap \bT_{\b}$, we can write $\d$ as a matrix with entries in $\bF_2[U_{\ws}]$. The map $\Phi_w$ is obtained by differentiating each entry with respect to $U_w$.
\end{rem}

We now prove several properties about the endomorphism $\Phi_w$.

\begin{lem}\label{lem:claim:phi0} The map $\Phi_w$ is a chain map.
\end{lem}
\begin{proof} View the differential $\d$ as a square matrix with entries in $\bF_2[U_{\ws}]$. Differentiating $\d\circ \d=0$ with respect to $d/d U_w$, using the Leibniz rule, yields 
\[
\d \Phi_w+\Phi_w \d=0.
\]
\end{proof}

\begin{lem}\label{lem:claim:phi1} The map $\Phi_w$ commutes with transition maps up to chain homotopy. 
\end{lem}
\begin{proof} The transition maps  satisfy
\begin{equation}
\d\Psi_{\cH_1\to \cH_2}+\Psi_{\cH_1\to \cH_2}\d=0.\label{eq:change-of-diagrams-chain-maps}
\end{equation}
We apply $d/d U_w$ to Equation~\eqref{eq:change-of-diagrams-chain-maps}. The Leibniz rule implies
\[
\Phi_w \Psi_{\cH_1\to \cH_2}+\Psi_{\cH_1\to \cH_2}\Phi_w+\d \Psi_{\cH_1\to \cH_2}'+\Psi_{\cH_1\to \cH_2}'\d=0,
\]
which says that $\Phi_w$ commutes with $\Psi_{\cH_1\to \cH_2}$ up to chain homotopy.
\end{proof}

\begin{lem}\label{lem:claim:phi2}  If $w\in \ws$ and $|\ws|>1$ (so $S_w^+S_w^-$ is defined) then
\[
\Phi_w\simeq S^+_wS_w^-.
\]
\end{lem}
\begin{proof} Pick a diagram $\cH$ for $(Y,\ve{w})$ which is a free-stabilization of a diagram $\cH_0$ for $(Y,\ws\setminus \{w\})$ at $w$. Furthermore, assume that the free-stabilization is adjacent to another basepoint $w'\in \ws$. Using Lemma~ \ref{lem:differentialonH_1H_2}, we may pick an almost complex structure so that
\begin{equation}
\d_{\cH}=\begin{pmatrix}(\d_{\cH_0})_{U_{w}}&U_w+U_{w'}\\
0& (\d_{\cH_0})_{U_w}
\end{pmatrix},\label{eq:differentialPhi=S-S+}
\end{equation}
where  $(\d_{\cH_0})_{U_w}$ is the map obtained by extending $\d_{\cH_0}$ linearly over $U_w$, as in Equation~\eqref{eq:extension-of-scalars}.

 Differentiating Equation~\eqref{eq:differentialPhi=S-S+} with respect to $U_w$ yields
\[
\Phi_w=\begin{pmatrix}0& \id\\
0&0
\end{pmatrix},
\]
 which is $S_w^+S_w^-$.
\end{proof}

\begin{lem}\label{lem:claim:phi3} Suppose $\lambda$ is a path between two distinct basepoints in $\ws$. If $w\in \d \lambda$, then
\[
A_\lambda\Phi_{w}+\Phi_{w}A_\lambda\simeq \id.
\]
If $w\not\in \d\lambda$, then
\[
A_\lambda\Phi_w+\Phi_w A_\lambda\simeq 0.
\]
\end{lem}
\begin{proof} Assume first that $w\in \d \lambda$.  We apply $d/d U_w$ to the equation
 \[
\d A_\lambda+A_{\lambda}\d=U_w+U_{w'}
 \]
from  Lemma~\ref{lem:Alambdachainhomotopy}, obtaining the relation
\[
A_\lambda \Phi_w+\Phi_w A_\lambda+A'_\lambda \d+\d A_{\lambda}'=\id.
\]
This proves the claim when $w\in \d \lambda$. The claim when $w\not\in \d \lambda$ is proven similarly.
\end{proof}

\begin{lem}\label{lem:claim:phi4} If $w\neq w'$, then $\Phi_w S_{w'}^\circ+S_{w'}^\circ \Phi_w\simeq 0$, for $\circ\in \{+,-\}$.
\end{lem}
\begin{proof} 
Similar to Lemma~\ref{lem:claim:phi3}, the claim is proven by differentiating the expression $\d S_{w'}^\circ+S_{w'}^\circ \d=0$ with respect to $U_w$.
\end{proof}

\begin{lem}\label{lem:claim:phi5} The map $\Phi_w$ satisfies $(\Phi_w)^2\simeq 0$,  ($\bF_2[U_{\ve{w}}]$-equivariantly).
\end{lem}

\begin{proof} Write $\d=\sum_{n=0}^\infty \d_n U_w^n$, where $\d_n$ does not involve any powers of $U_w$. Define
\[
H:=\sum_{n=2}^\infty \frac{n(n-1)}{2} \d_n U^{n-2}_w,
\]
 (the quantity $n(n-1)/2$ is computed over $\Z$, then projected to $\bF_2$). The equality $\d^2=0$ implies that
\[
\sum_{i+j=k} \d_i \d_j =0,
\]
 for each $k\ge 0$. We now compute:
\begin{align*}\d H+H\d &=\sum_{i,j\ge 0} \left(\frac{i(i-1)}{2}+\frac{j(j-1)}{2} \right)\d_i\d_j U_w^{i+j-2}\\
&=\sum_{k=2}^\infty \left(\frac{k(k-1)}{2}\sum_{i+j=k} \d_i \d_j\right)U_w^{k-2} +\sum_{i,j\ge 0} ij \d_i \d_j U^{i+j-2}_w\\
&=0+\Phi_w^2,
\end{align*}
 completing the proof.
\end{proof}

\subsection{The $\pi_1$-action}

We  now prove our formula for the $\pi_1$-action:

\begin{customthm}{\ref{thm:D}}The action of $\gamma\in \pi_1(Y,w)$ satisfies
\[
\gamma_*\simeq \id+\Phi_w\circ A_\gamma.
\]
\end{customthm}
\begin{proof}Break $\gamma$ into the concatenation of two paths, and write $\gamma=\lambda_2*\lambda_1$, where $\lambda_1$ is a path from $w$ to $w'$ and $\lambda_2$ is a path from $w'$ to $w$. Using Theorem \ref{thm:BM=graphact}, we have
\begin{equation}
(\gamma)_*\simeq (\lambda_2)_*\circ (\lambda_1)_*\simeq (S_{w'}^- A_{\lambda_2} S_w^+)(S_w^- A_{\lambda_1} S_{w'}^+).\label{eq:pi-1-in-3-pieces}
\end{equation}
 We manipulate Equation~\eqref{eq:pi-1-in-3-pieces} as follows:
\begin{align*}(\gamma)_*&\simeq  S_{w'}^- A_{\lambda_2} \Phi_w A_{\lambda_1} S_{w'}^+&& \text{(Lemma \ref{lem:claim:phi2})}\\
&\simeq S_{w'}^- (\Phi_w A_{\lambda_2}+\id)A_{\lambda_1}S_{w'}^+&& \text{(Lemma \ref{lem:claim:phi3})}\\
&\simeq S_{w'}^- \Phi_ w A_{\lambda_2} A_{\lambda_1} S_{w'}^+ +S_{w'}^-A_{\lambda_1}S_{w'}^+&&\text{(algebra)}\\
&\simeq \Phi_ wS_{w'}^-  A_{\lambda_2} A_{\lambda_1} S_{w'}^++S_{w'}^-A_{\lambda_1}S_{w'}^+&&\text{(Lemma \ref{lem:claim:phi4})}\\
&\simeq  \Phi_w A_\gamma+S_{w'}^- A_{\lambda_1} S_{w'}^+&&\text{(Lemma \ref{lem:rel-hom-relation-subdivide})}\\
&\simeq  \Phi_w A_\gamma+\id &&\text{(Lemma \ref{lem:trivial-strand-rel})},
\end{align*}
completing the proof.

%

\end{proof}

We now consider the triviality of the $\pi_1$-action on homology. By Remark~\ref{rem:Phi=formal-derivative}, the map $\Phi_w$ can be viewed as the derivative of the differential, with respect to $U_w$. Note that to apply the map to  $d/d U_w$ to an element $\d(x)$, we must use the Leibniz rule:
\[
\frac{d}{d U_w} (\d (x))=\left(\frac{d}{d U_w} \d\right)(x)+\d \left(\frac{d}{d U_w} x\right).
\]
 Rearranging, we obtain
\begin{equation}\Phi_w=\frac{d}{d U_w}\circ \d+\d \circ\frac{d}{d U_w}. \label{eq:chainhomotopy}\end{equation}

The map $\Phi_w$ is an $\bF_2[U_{\ve{w}}]$-module homomorphism and hence induces an $\cR_{\bmP}$-module homomorphism after tensoring $\CF^-$ with the $\bF_2[U_{\ws}]$-module $\cR_{\bmP}$. The derivative $d/d U_w$ is not an $\bF_2[U_{\ve{w}}]$-module homomorphism and hence does not induce a chain null-homotopy on homology, unless the coloring is simple. Similarly the map $d/d U_w$ does not always persist to $\hat{\CF}$.

We summarize with the following:

\begin{cor}\label{cor:pi_1-actionvanishesnonequivariantly} The $\pi_1(Y,w)$-action vanishes on $\HF^\circ(Y,\ve{w}^\sigma,\frs)$ for $\circ\in \{+,-,\infty\}$ as long as $w$ has a coloring distinct from all the other basepoints. In particular, it vanishes on the uncolored modules $\HF^\circ(Y,\ve{w},\frs)$ for $\circ\in \{+,-,\infty\}$.
\end{cor}

 Corollary~\ref{cor:F} is also a consequence:

\begin{customcor}{\ref{cor:F}}If $(W,\gamma)$ is path cobordism between two singly based 3-manifolds $(Y_0,w_0)$ and $(Y_1,w_1)$, then the cobordism map
\[
F_{W,\gamma,\frs}\colon \HF^\circ(Y_0,w_0,\frs|_{Y_0})\to \HF^\circ(Y_1,w_1,\frs|_{Y_1})
\]
is independent of $\gamma$ if $\circ\in \{-,\infty,+\}$.
\end{customcor}
\begin{proof} Suppose that $\gamma$ and $\gamma'$ are two paths in $W$, connecting $w_0$ and $w_1$. Decompose $W$ as $W_3\circ W_2\circ W_1$, where $W_i$ is obtained by attaching $i$-handles. Let $\cY$ denote $W_2\cap W_1$. Let $w_0',w_1'\in \cY$ denote the images of $w_0$ and $w_1$, under the flow of gradient like vector fields on $W_3\circ W_2$, and $W_1$. By flowing $\gamma$ and $\gamma'$ downwards in $W_3\circ W_2$, and upwards in $W_1$, we obtain two paths, $\lambda$ and $\lambda'$, from $w_0'$ to $w_1'$ in $\cY$. We can decompose the two cobordism maps as
\[
F_{W,\gamma,\frs}=F_3\circ F_2\circ \lambda_*\circ F_1\qquad \text{and} \qquad F_{W,\gamma',\frs}=F_3\circ F_2\circ \lambda'_*\circ F_1,
\]
where $F_i$ denotes the $i$-handle map for $W_i$, and $\lambda_*$ and $\lambda_*'$ denote the two basepoint moving diffeomorphism maps. That $\lambda_*=\lambda_*'$ as homomorphisms from $\HF^\circ(\cY,w_0',\frs|_{\cY})$ to $\HF^\circ(\cY,w_1',\frs|_{\cY})$ follows from Corollary~\ref{cor:pi_1-actionvanishesnonequivariantly}.
\end{proof}

%
%
%

\subsection{3-manifolds with non-vanishing $\pi_1$-action}

We are now equipped to prove Corollary~\ref{cor:E}, and give examples where the $\pi_1$-action is non-vanishing.

\begin{customcor}{\ref{cor:E}}
Let $Y$ be a 3-manifold and $\frs\in \Spin^c(Y)$.
\begin{enumerate}
\item \label{claim:pi1!=01} If $\frs$ is torsion and there is an $x\in \HF^+(Y,w,\frs)$ such that 
\[
U\cdot x=0 \quad \text{ and }\quad x\not \in U\cdot \HF^+(Y,w,\frs),
\]
then $\pi_1(Y\# S^1\times S^2,w)$ acts non-trivially on
\[
\hat{\HF}(Y\# S^1\times S^2,w).
\]
\item\label{claim:pi1!=02} Suppose $[\gamma]\in H_1(Y;\Z)$ is a class whose action on $\HF^+(Y,w,\frs)$ does not vanish. If $|\ws|\ge 2$, then the diffeomorphism map $\gamma_*$ acts non-trivially on the $\bF_2[U]$-module
\[
\HF^-(Y,\ws^\sigma,\frs),
\]
where $\sigma$ denotes the coloring which assigns all basepoints the variable $U$.
\end{enumerate}
\end{customcor}
\begin{proof}
We begin with Claim~\eqref{claim:pi1!=01}. The classification theorem for finitely generated chain complexes over a PID (see, e.g. \cite{HMZConnectedSum}*{Lemma~6.1}) implies that $\CF^-(Y,w,\frs)$ can be written as a direct sum of 1-step complexes (complexes with a single generator over $\bF_2[U]$, and vanishing differential) and 2-step complexes $a\xrightarrow{p(U)}b$ (i.e. complexes with two generators over $\bF_2[U]$, $a$ and $b$, with $\d(b)=0$ and $\d(a)=p(U)\cdot b$).  

If $\frs\in \Spin^c(Y)$ is torsion, then $\CF^-(Y,w,\frs)$ admits a relative $\Z$ grading \cite{OSDisks}*{Section~4.2}. Consequently, an algebraic argument (see \cite{HMZConnectedSum}*{Lemma~6.2}) implies each of the 2-step complexes appearing must have $p(U)=U^i$ for some $i$.

The existence of an $x\in \HF^+(Y,w,\frs)$ such that $U\cdot x=0$ and $x\not \in U\cdot \HF^+(Y,w,\frs)$ implies that at least one summand of $\CF^-(Y,w,\frs)$ must be a 2-step complex of the form $a\xrightarrow{U} b$. The $\Phi_w$ map for such a complex satisfies $\Phi_w(a)=b$. In particular, on $\hat{\HF}(Y,w, \frs)$ the map $\Phi_w$ is non-zero.

We view $\hat{\HF}(S^1\times S^2,\frs_0)$ as the 2-dimensional vector space $V=\langle \theta^+, \theta^-\rangle$. If $\gamma\subset S^1\times S^2$ is a curve which generates $H_1(S^1\times S^2)$, then
\[
A_\gamma(\theta^+)=\theta^-\qquad \text{and} \qquad A_{\gamma}(\theta^-)=0.
\]

Using the connected sum formula, $\hat{\HF}(Y\# S^1\times S^2, w,\frs\# \frs_0)$ is isomorphic to $\hat{\HF}(Y,w, \frs)\otimes V$. From Theorem~\ref{thm:D} and the connected sum formula, the map $\gamma_*$ satisfies
\[
\gamma_*(a\otimes \theta^+)=a\otimes \theta^++b\otimes \theta^-,\qquad \text{and}\qquad \gamma_*(a\otimes \theta^-)=a\otimes \theta^-,
\]
which is not equal to the identity map.

We now consider Claim~\eqref{claim:pi1!=02}. Since $|\ws|\ge 2$,  pick a $w\in \ws$ and write $\ws_0:=\ws\setminus \{w\}$. By Lemma~\ref{lem:differentialonH_1H_2}, we obtain the decomposition
\[
\HF^-(Y,\ws^\sigma,\frs)\iso \HF^-(Y, \ws_0^\sigma, \frs)\otimes V.
\]
Analyzing the curve counts appearing in Lemma~\ref{lem:differentialonH_1H_2}, we conclude that the maps $\Phi_w$ and $A_\gamma$ have the matrix descriptions
\begin{equation}
\Phi_w=\begin{pmatrix} 0& \id\\
0& 0
\end{pmatrix}\qquad \text{and} \qquad A_\gamma=\begin{pmatrix} A_\gamma & 0\\
0& A_\gamma
\end{pmatrix}.\label{eq:formula-stabilized-pi_1}
\end{equation}
In Equation~\eqref{eq:formula-stabilized-pi_1}, we are abusing notation slightly and writing $A_\gamma$ for the induced map on both $\HF^-(Y,\ws^\sigma,\frs)$ and $\HF^-(Y,\ws^\sigma_0,\frs)$.

Consequently
\begin{equation}
\gamma_*=\id+\Phi_wA_\gamma=\begin{pmatrix}\id& A_\gamma\\
0& \id
\end{pmatrix},\label{eq:free-stab-pi-1}
\end{equation}
which is not the identity if $A_\gamma$ acts non-trivially on $\HF^-(Y,\ws_0^\sigma,\frs)$. Using the formula from Equation~\eqref{eq:formula-stabilized-pi_1}, we see that $A_\gamma$ acts non-trivially on $\HF^-(Y,\ws_0^\sigma,\frs)$ if and only if it acts non-trivially on the singly based complex. The main statement now follows.
\end{proof}

\subsection{Illustrating the $\pi_1$-action on multi-pointed diagrams}
\label{subsec:alternateproof}

In this section, we sketch an alternate proof of the formula for the $\pi_1$-action, when there are at least two basepoints. In contrast to our proof of the normalization axiom (in particular Theorem~\ref{thm:BM=graphact}),  the proof we describe in this section is a direct holomorphic curve count, and does not use any functorial aspects of the graph TQFT.

Suppose $\ws_0$ is a non-empty collection of basepoints in $Y$, and suppose  $p\in Y\setminus \ws_0$ is a new basepoint. If $\cH=(\Sigma,\as,\bs,\ws_0)$ is a Heegaard diagram for $(Y,\ws_0)$, let 
\[
\cH^+_p=(\Sigma,\as\cup \{\alpha_0\}, \bs\cup \{\beta_0\}, \ws_0\cup \{p\})
\]
 denote the free-stabilization of $\cH$ at $p$. As we observed in the proof of Corollary~\ref{cor:E} (see Equation~\eqref{eq:free-stab-pi-1}), the $\pi_1$-action takes a simple form for free-stabilized diagrams. We now show that Equation~\eqref{eq:free-stab-pi-1} can be proven using a direct holomorphic curve count:

\begin{prop}\label{prop:alternate-pi-1} Suppose $\cH$ is a diagram for $(Y,\ws_0)$ and $\cH^+_w$ is obtained by free-stabilizing at $w$.  If $\gamma$ is a closed loop based at $w$, then
\begin{equation}
\gamma_*\simeq \begin{pmatrix}\id& A_\gamma\\
0& \id
\end{pmatrix},\label{eq:matrix-formula-pi-1}
\end{equation}
with respect to the matrix decomposition $\CF^-(\cH_w^+,\frs)\iso \CF^-(\cH,\frs)\otimes_{\bF_2} V$, where $V=\langle \theta^+,\theta^- \rangle$ is a 2-dimensional vector space.
\end{prop}

If $\cH=(\Sigma,\as,\bs,\ws)$ is a Heegaard diagram, and $p\in \Sigma\setminus \as\cup \bs$, we can define an endomorphism  $\Omega_p\colon \CF^-(\cH,\frs)\to \CF^-(\cH,\frs)$ via the formula
\begin{equation}
\Omega_p(\ve{x}):=\sum_{\substack{\phi\in \pi_2(\ve{x},\ve{y})\\ \mu(\phi)=1}} n_p(\phi) \# \hat{\cM}(\phi) U_{\ve{w}_0}^{n_{\ve{w}_0}(\phi)} \cdot \ve{y}. \label{eq:Omega-p-def}
\end{equation}

\begin{rem}The map $\Omega_p$ appearing in Equation~\eqref{eq:Omega-p-def} may not be a chain map. Instead $\Omega_p$ is a chain homotopy between the actions $U_{w_\alpha}$ and $U_{w_\beta}$, where $w_\alpha\in \ws_0$ is the basepoint in the component of $\Sigma\setminus \as$ containing $p$, and $w_\beta$ is the basepoint in the component of $\Sigma\setminus \bs$ containing $p$.
\end{rem}

We need the following change of almost complex structure map computation:

\begin{prop}\label{lem:explicitchangeofalmostcomplexstructure} Suppose $\cH$ is a Heegaard diagram for $(Y,\ws_0)$, and $\cH_p^+$ is the free-stabilization of $\cH$ at the point $p$. Let $J_\alpha$ and $J_\beta$ denote almost complex structures which have been stretched as in Figure~\ref{fig::48}. For sufficiently long necks, the transition map $\Psi_{J_\beta\to J_\alpha}$ takes the form
\[
\Psi_{J_\beta\to J_\alpha}=\begin{pmatrix}\id & (\Omega_p)_{U_p}\\
0& \id
\end{pmatrix},
\]
 where $\Omega_p\colon \CF^-(\cH,\frs)\to \CF^-(\cH,\frs)$ is the map defined in Equation~\eqref{eq:Omega-p-def}.
\end{prop}

\begin{proof}[Proof of Proposition \ref{lem:explicitchangeofalmostcomplexstructure}] Similar to Lemma~\ref{lem:changealmostcomplexstructure}, we fix two neck lengths along $c_\alpha$ and $c_\beta$ for $J_\alpha$ and $J_{\beta}$, respectively, and then stretch along $c$.  Adapting the argument therein, for an appropriately chosen almost complex structure $\tilde{J}$ on $\Sigma\times [0,1]\times \R$, interpolating $J_{\alpha}$ and $J_{\beta}$ with sufficiently large neck length along $c$, the transition map satisfies 
\begin{equation}
\Psi_{J_\beta\to J_\alpha}:=\Psi_{\tilde{J}}=\begin{pmatrix}\id & *\\
0& \id
\end{pmatrix}.\label{eq:Psi-with-*}
\end{equation}
The $*$ appears in the upper right corner, whereas in Lemma~\ref{lem:changealmostcomplexstructure} it appeared in the lower left corner, since the basepoint configuration is now different, and the Maslov grading of the points of $\alpha_0\cap \beta_0$ is now reversed from Lemma~\ref{lem:changealmostcomplexstructure}.
 
  We now show $*=\Omega_p$.   This amounts to counting $\tilde{J}$-holomorphic representatives of index 0 classes
 $\phi\# \phi_0\in \pi_2(\ve{x}\times \theta^-,\ve{y}\times \theta^+)$. Let $n_1(\phi_0)$, $n_2(\phi_0)$, $m_1(\phi_0)$ and $n_p(\phi_0)$ denote the multiplicities shown in Figure~\ref{fig::57}.

\begin{figure}[ht!]
\centering
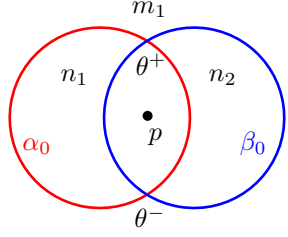
\caption{\textbf{Multiplicities in the free-stabilization region.}}\label{fig::57}
\end{figure} 
 
  For classes $\phi\# \phi_0\in \pi_2(\xs\times \theta^-, \ys\times \theta^+)$, Equation~\eqref{eq:indexbrokenlimitacstr} implies
 \begin{equation}
\mu(\phi\# \phi_0)=\mu(\phi)+2n_p(\phi_0)-1. \label{eq:Maslov-index-Omega-p}
\end{equation}

 If $\phi\# \phi_0$ is a class with $\tilde{J}$-holomorphic representatives with arbitrarily large neck length along $c$, then a limiting argument as in Lemma~\ref{lem:relative-lengths-unimportant} shows that $\phi$ has a broken representative for a cylindrical almost complex structure. Hence $\mu(\phi)\ge 0$, by transversality. Consequently, Equation~\eqref{eq:Maslov-index-Omega-p} implies that if $\phi\# \phi_0\in \pi_2(\xs\times \theta^-,\ys\times \theta^+)$ has Maslov index 0 and has representatives for arbitrarily long neck length along $c$, then 
 \[
 \mu(\phi)=1\qquad \text{and} \qquad n_{p}(\phi_0)=0.
 \]

 If $k$ is a fixed non-negative integer, there are $k$ non-negative homology classes in $\pi_2(\theta^-,\theta^+)$ with $n_p=0$ and $m_1=k$. If $k_1,k_2$ are non-negative integers with $k_1+k_2=k-1$, we write $\phi_0^{k_1,k_2}$ for the class in $\pi_2(\theta^-,\theta^+)$ with
 \begin{equation}
n_p(\phi_0^{k_1,k_2})=0,\qquad m_1(\phi_0^{k_1,k_2})=k,\qquad n_1(\phi_0^{k_1,k_2})=k_1,\qquad \text{ and } \qquad n_2(\phi_0^{k_1,k_2})=k_2.\label{eq:mult-phik1k2}
\end{equation}
 
Hence, to compute the component marked $*$ in Equation~\eqref{eq:Psi-with-*}, it is  sufficient to compute
\[
\#\cM_{\tilde{J}}(\phi\# \phi_0^{k_1,k_2})
\]
when $\phi$ has Maslov index 1, and $k_1+k_2=m_1(\phi)-1$.   To this end, we will show that
\begin{equation}
\# \cM_{\tilde{J}}(\phi\# \phi_0^{k_1,k_2})\equiv \#\hat{\cM}(\phi)\pmod{ 2} \label{eq:countphiphik1k2}
\end{equation}
 for each non-negative pair $k_1,k_2$ with $k_1+k_2=k-1$. Note that the main statement follows quickly from Equation~\eqref{eq:countphiphik1k2}.

  Similar to the classes $\phi_0^{k_1,k_2}$, there are $k+1$ non-negative homology classes in $\pi_2(\theta^+,\theta^+)$ with $n_p=0$ and $m_1=k$.  If $K_1+K_2=k$, write $\xi_{0}^{K_1,K_2}$ for the homology class in $\pi_2(\theta^+,\theta^+)$ defined similarly to Equation~\eqref{eq:mult-phik1k2}. To establish Equation~\eqref{eq:countphiphik1k2}, we count the ends of the 1-dimensional spaces $\cM_{\tilde{J}}(\phi\# \xi_{0}^{K_1,K_2})$.

As in the proof that $\Psi_{\tilde{J}}$ is a chain map, the ends of $\cM_{\tilde{J}}(\phi\#\xi_{0}^{K_1,K_2})$ all result from strip breaking, i.e. correspond to pairs $(\tilde{u}_1,u_2^{\beta})$ or $(u_1^{\alpha}, \tilde{u}_2)$ where $\tilde{u}_i$ denotes an index 0 $\tilde{J}$-holomorphic curve, and $u_2^{\beta}$ and  $u_1^{\alpha}$ denote index 1 $J_{\beta}$- or $J_{\alpha}$-holomorphic curves.  Furthermore, each end consists of one the following configurations:
\begin{enumerate}[ref= e-\arabic*, label= (e-\arabic*):]
\item\label{case:ut1-u2a:1} A pair $(\tilde{u}_1, u_2^\alpha)$ where $\tilde{u}_1$ represents an index 0 class in $\pi_2(\xs\times \theta^+, \zs\times \theta^+)$ and $u_2^\alpha$ represents an index 1 class in $\pi_2(\zs\times \theta^+, \ys\times \theta^+)$.
\item\label{case:ut1-u2a:2} A pair $(\tilde{u}_1,u_2^{\alpha})$ where $\tilde{u}_1$ represents an index 0 class in $\pi_2(\xs\times \theta^+, \zs\times \theta^-)$ and $u_2^\alpha$ represents an index 1 class in $\pi_2(\zs\times \theta^-, \ys\times \theta^+)$.
\item\label{case:ub-ut:1} A pair $(u_1^\beta,\tilde{u}_2)$ where $u_1^\beta$ represents an index 1 class in $\pi_2(\xs\times \theta^+, \zs\times \theta^+)$ and $\tilde{u}_2$ represents an index 0 class in $\pi_2(\zs\times \theta^+,\ys\times \theta^+)$.
\item\label{case:ub-ut:2} A pair $(u_1^\beta,\tilde{u}_2)$ where $u_1^\beta$ represents an index 1 class in $\pi_2(\xs\times \theta^+, \zs\times \theta^-)$ and $\tilde{u}_2$ represents an index 0 class in $\pi_2(\zs\times \theta^-, \ys\times \theta^+)$.
\end{enumerate}

In case~\eqref{case:ut1-u2a:1}, the curve $\tilde{u}_1$ contributes to the first diagonal entry of $\Psi_{\tilde{J}}$. The argument used to establish Equation~\eqref{eq:Psi-with-*} implies that $\tilde{u}_1$ must represent a constant class $e_{\xs}\times e_{\theta^+}$.

In case~\eqref{case:ut1-u2a:2}, the curve $\tilde{u}_1$ contributes to the lower left entry of $\Psi_{\tilde{J}}$. Our argument to establish Equation~\eqref{eq:Psi-with-*} showed that no such curves $\tilde{u}_1$ exist, prohibiting the existence of these ends.

In case~\eqref{case:ub-ut:1}, the curve $\tilde{u}_2$ contributes to the first diagonal entry of $\Psi_{\tilde{J}}$. As with case~\eqref{case:ut1-u2a:1}, the curve $\tilde{u}_2$ must represent the constant class $e_{\ys}\times e_{\theta^+}$.

In case~\eqref{case:ub-ut:2} the curve $u_1^\beta$ contributes to the lower left entry of the differential $\d_{\cH_p^+, J_\alpha}$, which vanishes. Hence $u_1^\beta$ represents one of two bigon classes in the free-stabilization region. See the proof of Lemma~\ref{lem:differentialcompdegen}. The class $\tilde{u}_2$ contributes to the upper right component of the matrix for $\Psi_{\tilde{J}}$ (which we are trying to compute). We write $B_{0,1}$ and $B_{1,0}$ for these bigons, chosen so that $n_1(B_{1,0})=1$ and $n_2(B_{0,1})=1$.

In summary, the ends of $\cM_{\tilde{J}}(\phi\# \xi_{0}^{K_1,K_2})$ are constrained to the following:
\begin{enumerate}
\item A $\tilde{J}$-holomorphic representative of $e_{\xs}\times e_{\theta^+}$ and a $J_{\alpha}$-holomorphic representative of $\phi\# \xi_{0}^{K_1,K_2}$.
\item A $\tilde{J}$-holomorphic representative of $e_{\ys}\times e_{\theta^+}$ and a $J_{\beta}$-holomorphic representative of $\phi\#\xi_{0}^{K_1,K_2}$.
\item A $\tilde{J}$-holomorphic representative of $\phi\# \phi_0^{K_1-1,K_2}$ and a $J_{\beta}$-holomorphic representative of $B_{1,0}$.
\item A $\tilde{J}$-holomorphic representative of $\phi\# \phi_0^{K_1,K_2-1}$ and a $J_{\beta}$-holomorphic representative of $B_{0,1}$.
\end{enumerate}

We can write out the ends as

\begin{equation}
\begin{split}
\d \bar{\cM}_{\tilde{J}}(\phi\# \xi^{K_1,K_2}_{0})&=\hat{\cM}_{J_\alpha}(\phi\#\xi^{K_1,K_2}_{0})\times \cM_{\tilde{J}}(e_{\ve{x}}\times e_{\theta^+})\,\,
\sqcup \,\,  \cM_{\tilde{J}}(e_{\ve{y}}\times e_{\theta^+})\times \hat{\cM}_{J_\beta}(\phi\#\xi^{K_1,K_2}_{0})\\ &\,\, \sqcup\,\, \hat{\cM}_{J_\beta}(B_{1,0})\times  \cM_{\tilde{J}}(\phi\#\phi_{0}^{K_1-1,K_2})\,\,\sqcup \,\,  \hat{\cM}_{J_\beta}(B_{0,1})\times \cM_{\tilde{J}}(\phi\#\phi_0^{K_1,K_2-1}).
\end{split}
\label{eq:count-ends-phiphiK1K2}
\end{equation}

The bigons and constant classes appearing in Equation~\eqref{eq:count-ends-phiphiK1K2} have unique holomorphic representatives. Since the total ends of a compact 1-manifold are 0, Equation~\eqref{eq:count-ends-phiphiK1K2} implies 
\begin{equation}
\#\hat{\cM}_{J_\alpha}(\phi\#\xi^{K_1,K_2}_{0})+\#\hat{\cM}_{J_\beta}(\phi\#\xi^{K_1,K_2}_{0})+\#\cM_{\tilde{J}}(\phi\#\phi_0^{K_1-1,K_2})+ \#\cM_{\tilde{J}}(\phi\#\phi_0^{K_1,K_2-1})=0.
\label{eq:recursion}
\end{equation}
The proof of Lemma~\ref{lem:differentialonH_1H_2} adapts immediately to show that
\begin{equation}
\#\hat{\cM}_{J_\beta}(\phi\# \xi^{K_1,K_2}_{0})=\begin{cases}\# \hat{\cM}(\phi)& \text{ if } (K_1,K_2)=(k,0)\\
0& \text{ otherwise},
\end{cases}
\label{eq:recursion-initial}
\end{equation}
 and
\begin{equation}
\#\hat{\cM}_{J_\alpha}(\phi\# \xi^{K_1,K_2}_{0})=\begin{cases}\# \hat{\cM}(\phi)& \text{ if } (K_1,K_2)=(0,k)\\
0& \text{ otherwise}.
\end{cases}
\label{eq:recursion-final}
\end{equation}

By starting at $(K_1,K_2)=(k,0)$ and using Equations~\eqref{eq:recursion}, \eqref{eq:recursion-initial} and~\eqref{eq:recursion-final} to go through all pairs $(K_1,K_2)$ with $K_1+K_2=k$ we obtain  Equation~\eqref{eq:countphiphik1k2}, completing the proof.
\end{proof}

\begin{figure}[ht!]
\centering
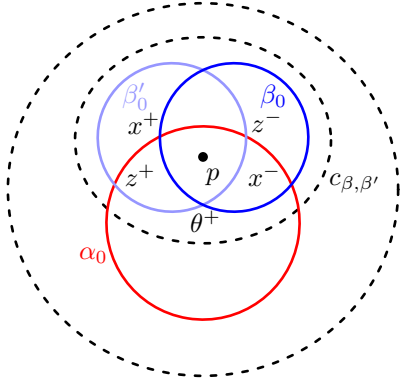
\caption{\textbf{A Heegaard triple which has been free-stabilized.} We stretch the almost complex structure along $c_{\beta,\beta'}$.}\label{fig::53}
\end{figure}

We need the following triangle count:

\begin{lem}\label{lem:quasi-stabtrianglemap}Suppose $\cT=(\Sigma,\as,\bs,\bs',\ve{w}_0)$ is a Heegaard triple (we do not assume $\bs'$ are small isotopies of $\bs$). Let $\cT^+_p$ be the Heegaard triple obtained by free-stabilizing at $p$, as in Figure~\ref{fig::53}. Let $J$ be an almost complex structure on $\Sigma\times [0,1]\times \R$, and let $J(T)$ denote an almost complex structure with a neck of length $T$ added along $c_{\beta,\beta'}$. For sufficiently large $T$,
\[
F_{\cT^+_p, J(T)}(\ve{x}\times -, \ve{y}\times \theta^+)=\begin{pmatrix}F_{\cT,J}(\ve{x},\ve{y})&0\\
0& F_{\cT,J}(\ve{x},\ve{y})
\end{pmatrix}.
\]
\end{lem}
\begin{proof}
 The statement is a special case of \cite{MOIntSurg}*{Proposition 5.2}.
\end{proof}

We now sketch our proof of the formula for the $\pi_1$-action when there is more than one basepoint:

\begin{proof}[Proof of Proposition~\ref{prop:alternate-pi-1}] 
 Immerse  $\gamma$ in $\Sigma$, and break $\gamma$ into a sequence of embedded arcs, $\lambda_1,\dots, \lambda_n$, such that each $\lambda_i$ crosses a single $\as$ or $\bs$ curve, exactly once. See~Figure~\ref{fig::87}.

\begin{figure}[ht!]
\centering
\begingroup%
  \makeatletter%
  \providecommand\color[2][]{%
    \errmessage{(Inkscape) Color is used for the text in Inkscape, but the package 'color.sty' is not loaded}%
    \renewcommand\color[2][]{}%
  }%
  \providecommand\transparent[1]{%
    \errmessage{(Inkscape) Transparency is used (non-zero) for the text in Inkscape, but the package 'transparent.sty' is not loaded}%
    \renewcommand\transparent[1]{}%
  }%
  \providecommand\rotatebox[2]{#2}%
  \newcommand*\fsize{\dimexpr\f@size pt\relax}%
  \newcommand*\lineheight[1]{\fontsize{\fsize}{#1\fsize}\selectfont}%
  \ifx\svgwidth\undefined%
    \setlength{\unitlength}{238.62819115bp}%
    \ifx\svgscale\undefined%
      \relax%
    \else%
      \setlength{\unitlength}{\unitlength * \real{\svgscale}}%
    \fi%
  \else%
    \setlength{\unitlength}{\svgwidth}%
  \fi%
  \global\let\svgwidth\undefined%
  \global\let\svgscale\undefined%
  \makeatother%
  \begin{picture}(1,0.52824002)%
    \lineheight{1}%
    \setlength\tabcolsep{0pt}%
    \put(0,0){\includegraphics[width=\unitlength,page=1]{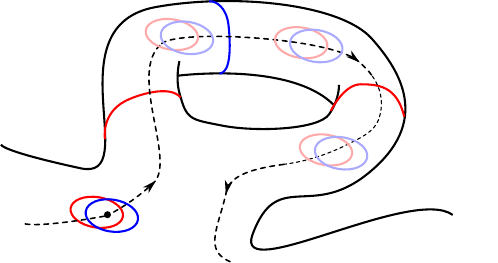}}%
    \put(0.46042208,0.10140971){\color[rgb]{0,0,0}\makebox(0,0)[lt]{\lineheight{0}\smash{\begin{tabular}[t]{l}$\gamma$\end{tabular}}}}%
    \put(0.33272783,0.21617283){\color[rgb]{0,0,0}\makebox(0,0)[lt]{\lineheight{0}\smash{\begin{tabular}[t]{l}$\lambda_1$\end{tabular}}}}%
    \put(0.47416116,0.46267098){\color[rgb]{0,0,0}\makebox(0,0)[lt]{\lineheight{0}\smash{\begin{tabular}[t]{l}$\lambda_2$\end{tabular}}}}%
    \put(0.74717336,0.27428749){\color[rgb]{0,0,0}\makebox(0,0)[rt]{\lineheight{0}\smash{\begin{tabular}[t]{r}$\lambda_3$\end{tabular}}}}%
    \put(0.12138259,0.14072937){\color[rgb]{0,0,0}\makebox(0,0)[rt]{\lineheight{0}\smash{\begin{tabular}[t]{r}$w$\end{tabular}}}}%
    \put(0,0){\includegraphics[width=\unitlength,page=2]{fig87.pdf}}%
    \put(0.84868343,0.12314245){\color[rgb]{0,0,0}\makebox(0,0)[lt]{\lineheight{0}\smash{\begin{tabular}[t]{l}$\Sigma$\end{tabular}}}}%
    \put(0,0){\includegraphics[width=\unitlength,page=3]{fig87.pdf}}%
  \end{picture}%
\endgroup%

\caption{\textbf{Breaking the closed curve $\gamma$ into a sequence of arcs $\lambda_1,\dots, \lambda_n$.}}\label{fig::87}
\end{figure}

Write  $\d \lambda_i=\{p_i,p_{i+1}\}$. Note that $p_1=p_{n+1}=w$. Let $\cH^+_{p_i}$ denote a free-stabilization of $\cH$ at $p_i$.

Let $J_{\alpha,i}$ denote an almost complex structure for $\cH^+_{p_i}$ stretched along a circle $c_\alpha$ parallel to the $\alpha$ circle in the free-stabilization region, and let $J_{\beta,i}$ denote one stretched along a circle $c_{\beta}$ parallel to the $\beta$ circle.

  Each arc $\lambda_i$ induces a diffeomorphism map
  \begin{equation}
(\lambda_i)_*\colon \CF^-_{J_{\beta,i}}(\cH_{p_i}^+,\frs)\to \CF^-_{J_{\beta,i+1}}(\cH_{p_{i+1}}^+,\frs).\label{eq:lambda-diffeo-domainJbeta}
  \end{equation}

First, we claim that if $\lambda_i$ intersects a single $\bs$ curve (and no $\as$ curves), then $(\lambda_i)_*$ is chain homotopic to
\begin{equation}
(\lambda_i)_*\simeq \begin{pmatrix}\id& 0\\
0& \id
\end{pmatrix}.
\label{eq:diffeomorphism-map-if-single-beta}
\end{equation}

To establish Equation~\eqref{eq:diffeomorphism-map-if-single-beta}, we decompose $\lambda_i$ into two maps, a holomorphic triangle count to handleslide a $\bs$ curve over $p_i$, followed by a tautological diffeomorphism map to move $p_{i}$ to $p_{i+1}$. See Figure~\ref{fig::90}.

\begin{figure}[ht!]
\centering
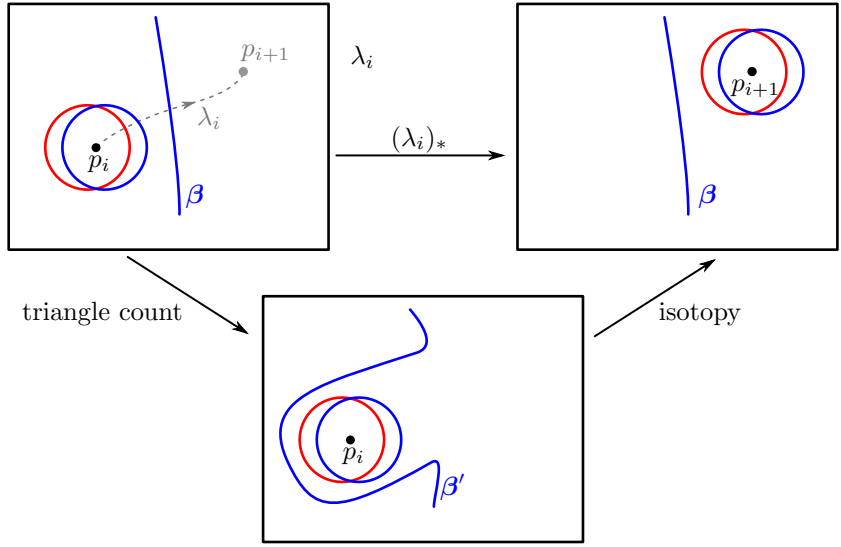
\caption{\textbf{Decomposing $(\lambda_i)_*$ into a triangle map followed by a tautological diffeomorphism map.}}\label{fig::90}
\end{figure}

Equation~\eqref{eq:diffeomorphism-map-if-single-beta} follows from Lemma~\ref{lem:quasi-stabtrianglemap}, since we are using the almost complex structure $J_{\beta,i}$. Hence the triangle counts on the free-stabilized Heegaard triple for  the handleslide are obtained from the triangle counts on the unstabilized diagram. Once we compose with the tautological map for an isotopy, the diagonal entries are chain homotopic to the identity, by naturality.

Next, we claim that if $\lambda_i$ intersects a single $\as$ curve (and no $\bs$ curves), then $(\lambda_i)_*$ is chain homotopic to
\begin{equation}
(\lambda_i)_*\simeq \begin{pmatrix}\id& \Omega_{p_i}+\Omega_{p_{i+1}}\\
0& \id
\end{pmatrix}.
\label{eq:diffeomorphism-map-if-single-alpha}
\end{equation}

Equation~\eqref{eq:diffeomorphism-map-if-single-alpha} is established as follows. By simply switching the roles of the $\as$ and $\bs$ curves, our previous argument shows that the map
\[
(\lambda_i)_*\colon \CF^-_{J_{\alpha,i}}(\cH_{p_i}^+,\frs)\to \CF^-_{J_{\alpha,i}}(\cH_{p_{i+1}}, \frs)
\]
takes the same form as in Equation~\eqref{eq:diffeomorphism-map-if-single-beta}. To obtain the form induced map with the same choice of almost complex structures as in Equation~\eqref{eq:lambda-diffeo-domainJbeta}, we must pre-compose with $\Psi_{J_{\beta,i}\to J_{\alpha,i}}$ and post-compose with $\Psi_{J_{\alpha,{i+1}}\to J_{\beta,i+1}}$. These two transition maps may be computed using Lemma~\ref{lem:explicitchangeofalmostcomplexstructure}. Upon multiplying the three matrices together, we obtain Equation~\eqref{eq:diffeomorphism-map-if-single-alpha}.

Finally, to compute
\[
\gamma_*\colon \CF^-_{J_{\beta,1}}(\cH_w^+,\frs)\to \CF^-_{J_{\beta,1}}(\cH_w^+,\frs),
\]
we write
\[
\gamma_*=(\lambda_{n})_*\circ \cdots \circ (\lambda_1)_*,
\]
and use the expressions for $\lambda_i$ from Equations~\eqref{eq:diffeomorphism-map-if-single-beta} and ~\eqref{eq:diffeomorphism-map-if-single-alpha}. Upon multiplying all the matrices out, we obtain that
\begin{equation}
\gamma_*=\begin{pmatrix}\id & \displaystyle\sum_{\substack{i\in\{1,\dots, n\}\\
 \lambda_i\cap \as\neq \emptyset}} (\Omega_{p_{i}}+\Omega_{p_{i+1}})\\
 0
 & \id
\end{pmatrix}
\label{eq:almost-final-form-2-basepoint-pi-1}
\end{equation}

The map
\[
\displaystyle\sum_{\substack{i\in\{1,\dots, n\}\\
 \lambda_i\cap \as\neq \emptyset}} (\Omega_{p_{i}}+\Omega_{p_{i+1}})
 \]
 is exactly equal to the map which counts index 1 holomorphic disks, with a factor equal to their total change across the $\as$ curves. This is exactly the homology action $A_\gamma$. Consequently Equation~\eqref{eq:almost-final-form-2-basepoint-pi-1} coincides with Equation~\eqref{eq:matrix-formula-pi-1}, and the proof is complete.
\end{proof}

\subsection{Basepoint swapping diffeomorphism}
We now compute an additional diffeomorphism map. Suppose $w_1,w_2\in \ws$ are two distinct basepoints,  and $\lambda$ is a path connecting $w_1$ and $w_2$. Let
\[
\Sw_\lambda\colon (Y,\ws)\to (Y,\ws)
\]
denote the diffeomorphism corresponding to swapping $w_1$ and $w_2$, along $\lambda$. The diffeomorphism $\Sw_\lambda$ is well defined up to isotopy, since $Y$ is 3-dimensional.

The diffeomorphism $\Sw_{\lambda}$ is the map induced by the graph cobordism $([0,1]\times Y,\Gamma)$ where $\Gamma$ is the union of the arcs $[0,1]\times \{w\}$ for $w\in \ws\setminus \{w_1,w_2\}$, as well as two arcs, one connecting $\{w_1\}\times \{0\}$ to $\{w_2\}\times \{1\}$ and the other connecting $\{w_2\}\times \{0\}$ to $\{w_1\}\times \{1\}$, both of which project to $\lambda$, up to isotopy.

Note that $\Sw_\lambda$ only induces an $\cR_{\bmP}$-equivariant map on $\CF^-(Y,\ve{w}^\sigma,\frs)$ for colorings $\sigma$  satisfying $\sigma(w_1)=\sigma(w_2)$.

\begin{prop}\label{prop:basepointswappingmap}Let $\Sw_\lambda\colon (Y,\ve{w})\to (Y,\ve{w})$ be the basepoint swapping diffeomorphism. The induced map $(\Sw_{\lambda})_*$ satisfies
\begin{equation}
(\Sw_{\lambda})_*\simeq \Phi_{w_1} A_\lambda+A_\lambda \Phi_{w_2}\simeq A_{\lambda}\Phi_{w_1}+\Phi_{w_2}A_{\lambda}.
\label{eq:swap-diffeo-formula}
\end{equation}
\end{prop}
\begin{proof}The second chain homotopy in Equation~\eqref{eq:swap-diffeo-formula} can be obtained from Lemma~\ref{lem:claim:phi3}, so we focus on the first.

 Let $w'\not\in \ve{w}$ be a new basepoint, given the same color as $w_1$ and $w_2$. Pick paths $\lambda_1$ and $\lambda_2$ from $w_1$ to $w'$ and from $w'$ to $w_2$ (respectively) such that the concatenation $\lambda_1*\lambda_2$ is isotopic to $\lambda$. Break $\Sw_\lambda$ into the composition of the following three diffeomorphisms:
 \begin{enumerate}
 \item An isotopy of $w_1$ to $w'$, along $\lambda_1$.
 \item An isotopy of $w_2$ to $w_1$, along $\lambda$.
 \item An isotopy of $w'$ to $w_2$, along $\lambda_2$.
 \end{enumerate}
 
 Using Theorem~\ref{thm:BM=graphact}, we obtain
\begin{equation}
(\Sw_{\lambda})_*\simeq (S^-_{w'} A_{\lambda_2} S_{w_2}^+)(S_{w_2}^- A_\lambda S_{w_1}^+)(S_{w_1}^- A_{\lambda_1} S_{w'}^+).\label{eq:basepoint-swapping-map-1}
\end{equation}

 Using Lemmas~\ref{lem:claim:phi0}--\ref{lem:claim:phi5}, we manipulate Equation~\eqref{eq:basepoint-swapping-map-1} as follows:
\begin{equation}
\begin{split}(\Sw_{\lambda})_*&\simeq S_{w'}^- A_{\lambda_2}\Phi_{w_2} A_\lambda \Phi_{w_1} A_{\lambda_1} S_{w'}^+\\
&\simeq \Phi_{w_2} S_{w'}^- A_{\lambda_2} A_{\lambda} \Phi_{w_1} A_{\lambda_1} S^+_{w'}+S_{w'}^-A_{\lambda} \Phi_{w_1} A_{\lambda_1} S_{w'}^+\\
&\simeq \Phi_{w_2} S_{w'}^- A_{\lambda_2} A_{\lambda} A_{\lambda_1} S_{w'}^+ \Phi_{w_1}+\Phi_{w_2}S_{w'}^{-} A_{\lambda_2} A_{\lambda} S_{w'}^+ + S_{w'}^-A_{\lambda} \Phi_{w_1} A_{\lambda_1} S_{w'}^+\\
&\simeq \Phi_{w_2} S_{w'}^- A_{\lambda_2} A_{\lambda} A_{\lambda_1} S_{w'}^+ \Phi_{w_1}+\Phi_{w_2}S_{w'}^{-} A_{\lambda_2} A_{\lambda} S_{w'}^++\Phi_{w_1} S^{-}_{w'} A_{\lambda} A_{\lambda_1} S_{w'}^++S_{w'}^-A_{\lambda_1} S_{w'}^+.
\end{split}
\label{eq:Swlambda=first-manipulation}
\end{equation}

We focus on the first term of the last line of Equation~\eqref{eq:Swlambda=first-manipulation}. By Lemma~\ref{lem:splicingrelhom}, $A_\lambda=A_{\lambda_2}+A_{\lambda_1}$ when all are defined. Also $A_{\lambda_i}^2\simeq U$ (where $U$ denotes any of $U_{w_1}$, $U_{w_2}$, $U_{w'}$, which are identified by the coloring) by Lemma~\ref{lem:Alambda-squares-to-zero-or-U}. Hence
\begin{equation}
\begin{split}
&A_{\lambda_2}A_{\lambda}A_{\lambda_1}\\
\simeq &A_{\lambda_2}(A_{\lambda_1}+A_{\lambda_2})A_{\lambda_1}\\
\simeq &A_{\lambda_2} A_{\lambda_1}^2+A_{\lambda_2}^2A_{\lambda_1}\\
\simeq &U(A_{\lambda_1}+A_{\lambda_2})\\
\simeq &UA_\lambda.
\end{split}
\label{eq:A-l2-l-l1}
\end{equation}
Using Equation~\eqref{eq:A-l2-l-l1}, as well as Lemmas~\ref{lem:free-stab-H1-act-commute} and~\ref{lem:disjoint-strand-induces-zero}, we see
\begin{equation}
\begin{split}
&\Phi_{w_2} S_{w'}^- A_{\lambda_2} A_{\lambda} A_{\lambda_1} S_{w'}^+ \Phi_{w_1}\\
\simeq &U S_{w'}^- A_\lambda S_{w'}^+ \Phi_{w_1}\\
\simeq  &UA_\lambda S_{w'}^-  S_{w'}^+ \Phi_{w_1}\\
\simeq &0.
\end{split}
\label{eq:Sw-l-first-term-0}
\end{equation}

We now consider the second term of the last line of Equation~\eqref{eq:Swlambda=first-manipulation}. Lemma~\ref{lem:free-stab-H1-act-commute} and \ref{lem:trivial-strand-rel} imply
\begin{equation}
\begin{split}
&\Phi_{w_2}S_{w'}^{-} A_{\lambda_2} A_{\lambda} S_{w'}^+
\\
\simeq &\Phi_{w_2} S_{w'}^-A_{\lambda_2} S_{w'}^+ A_\lambda\\
\simeq &\Phi_{w_2}A_{\lambda}.
\end{split}
\label{eq:Sw-l-2nd-term}
\end{equation}

Similarly, the third term of the last line of Equation~\eqref{eq:Swlambda=first-manipulation} becomes 
\begin{equation}
\begin{split}
&\Phi_{w_1} S^{-}_{w'} A_{\lambda} A_{\lambda_1} S_{w'}^+
\\ \simeq &\Phi_{w_1}A_{\lambda}S_{w'}^-A_{\lambda_1} S_{w'}^+\\
\simeq &\Phi_{w_1}A_\lambda.
\end{split}
\label{eq:Sw-l-3rd-term}
\end{equation}

The last term of the last line of Equation~\eqref{eq:Swlambda=first-manipulation} satisfies
\begin{equation}
S_{w'}^+A_{\lambda_1} S_{w'}^+\simeq \id,
\label{eq:Sw-l-4th-term}
\end{equation}
by Lemma~\ref{lem:trivial-strand-rel}.

 Combining Equations~~\eqref{eq:Swlambda=first-manipulation}, \eqref{eq:Sw-l-first-term-0}, ~\eqref{eq:Sw-l-2nd-term}, ~\eqref{eq:Sw-l-3rd-term} and ~\eqref{eq:Sw-l-4th-term}, we obtain
\begin{equation}
(\Sw_{\lambda})_*\simeq \Phi_{w_2}A_{\lambda}+\Phi_{w_1}A_\lambda+\id.\label{eq:basepoint-swapping-almost}
\end{equation}
Equation~\eqref{eq:basepoint-swapping-almost} is chain homotopic to the expression in the statement since $\Phi_{w_1}A_{\lambda}+A_{\lambda}\Phi_{w_1}\simeq \id$ by Lemma~\ref{lem:claim:phi3}.
\end{proof}

\bibliographystyle{custom}
\bibliography{biblio}

\end{document}